\documentclass[reqno]{amsart}
\usepackage[utf8]{inputenc}
\usepackage{amssymb, amsmath, amsthm, color, enumerate, mathabx, mathrsfs, tikz}
\usepackage{bbm}
\usepackage{url} % link to Tao's notes in reference
\usepackage{pdfpages}
\usepackage{soul} % \hl

\usepackage{esint} % for \fint

\usepackage{tcolorbox}

\numberwithin{equation}{section}

\newcommand{\B}{\mathbb{B}}
\newcommand{\C}{\mathbb{C}}
\newcommand{\D}{\mathbb{D}}
\newcommand{\E}{\mathbb{E}}

\newcommand{\N}{\mathbb{N}}

\newcommand{\R}{\mathbb{R}}
\newcommand{\T}{\mathbb{T}}

\newcommand{\Z}{\mathbb{Z}}

\newcommand{\cB}{\mathcal{B}}
\newcommand{\cC}{\mathcal{C}}
\newcommand{\cD}{\mathcal{D}}
\newcommand{\cE}{\mathcal{E}}
\newcommand{\cG}{\mathcal{G}}

\newcommand{\cH}{\mathcal{H}}

\newcommand{\cI}{\mathcal{I}}
\newcommand{\cJ}{\mathcal{J}}
\newcommand{\cK}{\mathcal{K}}

\newcommand{\cL}{\mathcal{L}}
\newcommand{\cN}{\mathcal{N}}
\newcommand{\cO}{\mathcal{O}}
\newcommand{\cP}{\mathcal{P}}
\newcommand{\cQ}{\mathcal{Q}}
\newcommand{\cS}{\mathcal{S}}

\newcommand{\cZ}{\mathcal{Z}}

\newcommand{\ud}{\mathrm{d}}

\newcommand{\by}{\mathbf{x}}
\newcommand{\bx}{\mathbf{x}}
\newcommand{\bz}{\mathbf{z}}

\newcommand{\bA}{\mathbf{A}}
\newcommand{\bB}{\mathbf{B}}
\newcommand{\bC}{\mathbf{C}}

\newcommand{\bH}{\mathbf{H}}
\newcommand{\bI}{\mathbf{I}}
\newcommand{\bK}{\mathbf{K}}

\newcommand{\bM}{\mathbf{M}}
\newcommand{\bN}{\mathbf{N}}
\newcommand{\bO}{\mathbf{O}}

\newcommand{\bS}{\mathbf{S}}
\newcommand{\bT}{\mathbf{T}}
\newcommand{\bZ}{\mathbf{Z}}

\newcommand{\bLS}{\mathbf{LS}}

\newcommand{\sE}{\mathscr{E}}

\newcommand{\bdGamma}{\boldsymbol{\Gamma}}

\newcommand{\fB}{\mathfrak{B}}
\newcommand{\fC}{\mathfrak{C}}
\newcommand{\fd}{\mathfrak{d}}
\newcommand{\fD}{\mathfrak{D}}
\newcommand{\fG}{\mathfrak{G}}

\newcommand{\fJ}{\mathfrak{J}}

\newcommand{\fM}{\mathfrak{M}}
\newcommand{\fN}{\mathfrak{N}}
\newcommand{\fp}{\mathfrak{p}}
\newcommand{\fP}{\mathfrak{P}}
\newcommand{\fR}{\mathfrak{R}}

\def\inn#1#2{\langle#1,#2\rangle}

\theoremstyle{plain}

\newtheorem{theorem}{Theorem}[section]
\newtheorem{lemma}[theorem]{Lemma}
\newtheorem{claim}[theorem]{Claim}

\newtheorem{proposition}[theorem]{Proposition}
\newtheorem{corollary}[theorem]{Corollary}

\theoremstyle{definition}
\newtheorem{definition}[theorem]{Definition}
\newtheorem{example}[theorem]{Example}
\newtheorem{remark}[theorem]{Remark}

\newcommand{\tK}{\mathrm{K}}
\newcommand{\tN}{\mathrm{N}}

\newcommand{\rank}{\mathrm{rank}\,}
\newcommand{\proj}{\mathrm{proj}}

\newcommand{\trans}{\mathrm{trans}}
\newcommand{\cell}{\mathrm{cell}}
\newcommand{\ttiny}{\mathrm{tiny}}
\newcommand{\alg}{\mathrm{alg}}
\newcommand{\wall}{\mathrm{wall}}

\newcommand{\ctr}{\mathrm{ctr}}
\newcommand{\dir}{\mathrm{dir}}

\newcommand{\crit}{\mathrm{crit}}

\newcommand{\tI}{\mathrm{I}}
\newcommand{\tII}{\mathrm{II}}

\newcommand{\ceil}[1]{\lceil #1 \rceil }

\newcommand{\CC}{\mathrm{CC}}

\newcommand{\supp}{\mathrm{supp}\,}

\newcommand{\diam}{\mathrm{diam}}
\newcommand{\dist}{\mathrm{dist}}

\usepackage[framemethod=TikZ]{mdframed}

\mdfdefinestyle{MyFrame}{%
    linecolor=blue,
    outerlinewidth=1.5pt,
    roundcorner=20pt,
    innertopmargin=5pt,
    innerbottommargin=\baselineskip,
    innerrightmargin=20pt,
    innerleftmargin=20pt,
    backgroundcolor=gray!10!white}

\usepackage[normalem]{ulem}

\begin{document}

\title[Oscillatory integral operators and variable propagators]{Oscillatory integral operators and variable Schr\"odinger propagators: beyond the universal estimates}

\date{}

\begin{abstract} We consider a class of H\"ormander-type oscillatory integral operators in $\R^n$ for $n \geq 3$ odd with real analytic phase. We derive weak conditions on the phase which ensure $L^p$ bounds beyond the universal $p \geq 2 \cdot \frac{n+1}{n-1}$ range guaranteed by Stein's oscillatory integral theorem. This expands and elucidates pioneering work of Bourgain from the early 1990s. We also consider a closely related class of variable coefficient Schr\"odinger propagator-type operators, and show that the corresponding theory differs significantly from that of the H\"ormander-type operators. The main ingredient in the proof is a curved Kakeya/Nikodym maximal function estimate. This is established by combining the polynomial method with certain uniform sublevel set estimates for real analytic functions. The sublevel set estimates are the main novelty in the argument and can be interpreted as a form of quantification of linear independence in the $C^{\omega}$ category. 
\end{abstract}

\author[M. Chen]{Mingfeng Chen}
\address{Mingfeng Chen: Department of Mathematics, University of Wisconsin-Madison, Madison, WI 53706, USA.}
\email{mchen454@wisc.edu}

\author[S. Gan]{Shengwen Gan}
\address{Shengwen Gan: Department of Mathematics, University of Wisconsin-Madison, Madison, WI 53706, USA.}
\email{sgan7@wisc.edu}

\author[S. Guo]{Shaoming Guo}
\address{Shaoming Guo: Chern Institute of Mathematics and LPMC, Nankai University, Tianjin, China\\
and \\
Department of Mathematics, University of Wisconsin-Madison, Madison, WI-53706, USA}
\email{shaomingguo2018@gmail.com}

\author[J. Hickman]{Jonathan Hickman}
\address{Jonathan Hickman: School of Mathematics and Maxwell Institute for Mathematical Sciences, James Clerk Maxwell Building, The King's Buildings, Peter Guthrie Tait Road, Edinburgh, EH9 3FD, UK.}
\email{jonathan.hickman@ed.ac.uk}

\author[M. Iliopoulou]{Marina Iliopoulou}
\address{Marina Iliopoulou: Department of Mathematics, National and Kapodistrian University of Athens, Panepistimioupolis, Zografou 157 84, Greece.}
\email{miliopoulou@math.uoa.gr}

\author[J. Wright]{James Wright}
\address{James Wright: School of Mathematics and Maxwell Institute for Mathematical Sciences, James Clerk Maxwell Building, The King's Buildings, Peter Guthrie Tait Road, Edinburgh, EH9 3FD, UK.}
\email{j.r.wright@ed.ac.uk}

\maketitle

\tableofcontents

%%%%%%%%%%%%%%%%%%%%%%%%%%%%%%%%%%%%%%%%%%%%%%%%%%%%%%%%%%%%%%%%%%%%%%%%%%%%%%%%%%%%%%%%%%%%%%%%

%    Background

%%%%%%%%%%%%%%%%%%%%%%%%%%%%%%%%%%%%%%%%%%%%%%%%%%%%%%%%%%%%%%%%%%%%%%%%%%%%%%%%%%%%%%%%%%%%%%%%

\section{Background}\label{sec: background}

%%%%%%%%%%%%%%%%%%%%%%%%%%%%%%%%%%%%%%%%%%%%%%%%%%%%%%%%%%%%%%%%%%%%%%%%%%%%%%%%%%%%%%%%%%%%%%%%

%    Oscillatory integral operators 

%%%%%%%%%%%%%%%%%%%%%%%%%%%%%%%%%%%%%%%%%%%%%%%%%%%%%%%%%%%%%%%%%%%%%%%%%%%%%%%%%%%%%%%%%%%%%%%%

\subsection{Oscillatory integral operators and variable Schr\"odinger propagators}%\label{subsec: osc} 
For $n \geq 2$ and $\rho > 0$ write $\D^n_{\rho}:= \B_{\rho}^{n-1} \times \B^1_{\rho} \times \B_{\rho}^{n-1}$, where  $\B^d_{\rho}$ denotes the open ball in $\R^d$ of radius $\rho$ centred at the origin. When $\rho = 1$, we abbreviate $\D^n_{\rho}$ and $\B^n_{\rho}$ to $\D^n$ and $\B^n$, respectively.  Let $\phi \in C^{\infty}(\D^n)$ be a real-valued function and consider the following conditions:
\begin{itemize}
\item[H1)] $\det \partial_{x y}^2 \phi(x,t;y) \neq 0$ for all $(x,t;y) \in \D^n$.
\item[H2)] Defining the map $G \colon \D^n \to S^{n-1}$ by 
\begin{equation}\label{eq: Gauss map}
G(x,t;y) := \frac{G_0(x,t;y)}{|G_0(x,t;y) |} \quad \textrm{where} \quad G_0(x,t;y) := \bigwedge_{j=1}^{n-1} \partial_{y_j} \partial_{x,t}\phi(x,t;y),
\end{equation}
the curvature condition
\begin{equation*}
\det \partial^2_{y y} \langle \partial_{x,t}\phi(x,t;y),G(x, t; y_0)\rangle|_{y = y_0} \neq 0
\end{equation*}
holds for all $(x,t; y_0) \in \D^n$.
\end{itemize}
If $\phi$ satisfies H1), then we say $\phi$ is a \textit{non-degenerate phase}. If $\phi$ satisfies both H1) and H2), then we say $\phi$ is a \textit{H\"ormander-type phase}. A prototypical example of a H\"ormander-type phase is 
\begin{equation}\label{eq: prototype}
    \phi_{\mathrm{const}}(x,t;y) := \inn{x}{y} + t|y|^2;
\end{equation}
this is \textit{constant coefficient}, which in the present context means that the phase is linear in the $(x,t)$ variables. Often we will pair the phase with a smooth cutoff function $a \in C_c^{\infty}(\D^n)$, in which case we refer to the pairing $[\phi; a]$ as a \textit{phase-amplitude pair}. 

It is convenient to work with a rescaled setup. For any $\lambda \geq 1$ and phase-amplitude pair $[\phi; a]$, let 
\begin{equation*}
  \phi^{\lambda}(x,t;y) :=\lambda\phi(x/\lambda, t/\lambda;y) \qquad \textrm{and} \qquad   a^{\lambda}(x, t; y) := a(x/\lambda, t/\lambda; y).
\end{equation*}
With this, we introduce two distinct operators associated to $[\phi; a]$:

\begin{enumerate}[I)]
    \item The \textbf{H\"ormander-type oscillatory integral operator} $S^{\lambda} = S^{\lambda}[\phi; a]$ associated to $[\phi; a]$ is given by
\begin{equation*}
S^{\lambda}f(x,t) :=  \int_{\B^{n-1}} e^{i \phi^{\lambda}(x,t; y)}a^{\lambda}(x,t;y) f(y)\,\ud y,
\end{equation*}
defined initially as acting on $f \in C^{\infty}_c(\B^{n-1})$.
    \item The  \textbf{variable Schr\"odinger propagator} $U^{\lambda} = U^{\lambda}[\phi; a]$ associated to $[\phi; a]$ is given by
\begin{equation*}
U^{\lambda}f(x,t) :=  \int_{\widehat{\R}^{n-1}} e^{i \phi^{\lambda}(x,t; \xi)}a^{\lambda}(x,t;\xi) \hat{f}(\xi)\,\ud \xi,
\end{equation*}
defined initially as acting on $f \in \cS(\R^{n-1})$.
\end{enumerate}

We are interested in certain $L^p \to L^q$ estimates for these operators. 

\begin{definition} Let $\phi \colon \D^n \to \R$ be a H\"ormander-type phase, $1 \leq p, q \leq \infty$ and $\alpha \geq 0$.

\begin{enumerate}[I)]
    \item We let $\bH_{p \to q}(\phi; \alpha)$ denote the statement: \medskip
    
    \noindent There exists $\rho = \rho(\phi) >0$ such that for all choices of $a \in C^{\infty}_c(\D^n_{\rho})$  the inequality
    \begin{equation*}
        \|S^{\lambda}[\phi; a]f\|_{L^q(B_R)} \lesssim_{\phi, a} R^{\alpha} \|f\|_{L^p(\B^{n-1})}
    \end{equation*}
    holds uniformly over $1 \leq R \leq \lambda$, all $R$-balls $B_R \subseteq \R^n$ and all $f \in C^{\infty}_c(\B^{n-1})$.\medskip

    \noindent When $\alpha = 0$, we write $\bH_{p \to q}(\phi)$ in place of $\bH_{p \to q}(\phi; 0)$.\medskip

    \item We let $\bLS_q(\phi; \alpha)$ denote the statement: \medskip
    
    \noindent 
    There exists $\rho = \rho(\phi) >0$ such that for all choices of $a \in C^{\infty}_c(\D^n_{\rho})$ and all  $\varepsilon > 0$, the inequality
    \begin{equation*}
        \|U^{\lambda}[\phi; a]f\|_{L^q(\R^n)} \lesssim_{\varepsilon, \phi, a} \lambda^{\alpha + \varepsilon}\|f\|_{L^q(\R^{n-1})}
    \end{equation*}
     holds uniformly over all $\lambda \geq 1$ and all $f \in \cS(\R^{n-1})$. 
\end{enumerate} 
\end{definition}

 For the prototypical phase $\phi_{\mathrm{const}}$ from \eqref{eq: prototype}, the estimates in I) and II) are of intense interest. Indeed, the famous Fourier restriction conjecture concerns the sharp range of $p$ and $q$ for which $\bH_{p \to q}(\phi_{\mathrm{const}})$ holds. On the other hand, the local smoothing conjecture for the Schr\"odinger equation concerns the sharp range of $q$ and $\alpha$ for which $\bLS_q(\phi_{\mathrm{const}}; \alpha)$ holds. These two conjectures are closely related: it was shown in \cite[Theorem 9]{Rogers2008} that the Fourier restriction conjecture formally implies the local smoothing conjecture for the Schr\"odinger equation. 
For more general phases, determining whether $\bH_{p \to q}(\phi)$ is valid for a given H\"ormander-type phase $\phi$ is a longstanding problem, dating back to early work of H\"ormander~\cite{Hormander1973}. It has a long and rich history with notable contributions appearing in, for instance, \cite{Bourgain1991, Wisewell2005, Lee2006, BG2011, GHI2019, guo2022dichotomy}. The theory of variable coefficient local smoothing estimates $\bLS_q(\phi; \alpha)$ is less well-developed, and one of the goals of this paper is to contrast this against the H\"ormander problem.  
%%%%%%%%%%%%%%%%%%%%%%%%%%%%%%%%%%%%%%%%%%%%%%%%%%%%%%%%%%%%%%%%%%%%%%%%%%%%%%%%%%%%%%%%%%%%%%%%

%    Geometric maximal functions

%%%%%%%%%%%%%%%%%%%%%%%%%%%%%%%%%%%%%%%%%%%%%%%%%%%%%%%%%%%%%%%%%%%%%%%%%%%%%%%%%%%%%%%%%%%%%%%%

\subsection{Geometric maximal functions}\label{subsec: geometric max}
Any non-degenerate phase $\phi \colon \D^n \to \R$ defines a family of curves $\Gamma_{y,\omega}$ in $\R^n$, parameterised by $(y,\omega) \in \B_{\rho}^{n-1} \times \B_{\rho}^{n-1}$ for $0 < \rho < 1$ sufficiently small. In particular, provided $0 < \rho(\phi) < 1$ is suitably chosen, we may implicitly define a function $\Psi \colon \D_{\rho(\phi)}^n \to \R^{n-1}$ by 
\begin{equation}\label{eq: Psi}
    \partial_y \phi(\Psi(\omega; t ;y), t ; y) = \omega.
\end{equation}
Fixing $0 < \rho < \rho(\phi)$, we let $\Omega_{\phi} = Y_{\phi} := \{x \in \R^{n-1} : |x| \leq \rho\}$ and $I_{\phi}  := [-\rho, \rho]$.

For each pair $(y, \omega)\in Y_{\phi} \times \Omega_{\phi}$, we let $\gamma_{y,\omega}(t) :=  \Psi(\omega; t ;y)$ for $t \in I_{\phi}$. The curve $\Gamma_{y,\omega}$ is then given by the graph 
\begin{equation*}
 \Gamma_{y,\omega}  := \big\{ (\gamma_{y,\omega}(t),t) : t \in I_{\phi} \big\}.
\end{equation*}
We define the associated curved \textit{$\delta$-tubes}
\begin{equation}\label{eq: delta tube}
T_{y,\omega}^{\delta}  := \big\{ (x,t) \in \B^{n-1} \times I_{\phi} : |x - \gamma_{y,\omega}(t)| < \delta \big\}
\end{equation}
and refer to $\omega$ and $y$ as the \textit{centre} and \textit{direction} of $T_{y,\omega}^{\delta}$, respectively, and $\Gamma_{y, \omega}$ as the \textit{core curve} of $T_{y,\omega}^{\delta}$. With this, we introduce two distinct maximal functions associated to $\phi$:

\begin{enumerate}[I)]
    \item The \textbf{Kakeya maximal function} $\cK^{\delta} = \cK^{\delta}[\phi]$ associated to $\phi$ and $0 < \delta < 1$ is given by
\begin{equation*}
\cK^{\delta}g(y) := \sup_{\omega \in \Omega_{\phi}} \fint_{ T_{y,\omega}^{\delta}} |g| \qquad \textrm{for $y \in Y_{\phi}$ and $g \in L^1_{\mathrm{loc}}(\R^n)$.}
\end{equation*}

    \item The \textbf{Nikodym maximal function} $\cN^{\delta} = \cN^{\delta}[\phi]$ associated to $\phi$ and $0 < \delta < 1$ is given by
\begin{equation*}
\cN^{\delta}g(\omega) := \sup_{y \in Y_{\phi}} \fint_{ T_{y,\omega}^{\delta}} |g| \qquad \textrm{for all $\omega \in \Omega_{\phi}$ and $g \in L^1_{\mathrm{loc}}(\R^n)$.}
\end{equation*}
For notational convenience, we extend $\cK^{\delta}g$ and $\cN^{\delta}g$ to functions on $\B^{n-1}$ by setting $\cK^{\delta}g(y) := 0$ if $y \in \B^{n-1} \setminus Y_{\phi}$ and $\cN^{\delta}g(\omega) := 0$ if $\omega \in \B^{n-1} \setminus \Omega_{\phi}$.
\end{enumerate}

We are interested in certain $L^p \to L^s$ estimates for these operators, with sharp dependence on the $\delta$ parameter. 

\begin{definition}\label{defKN} Let $\phi \colon \D^n \to \R$ be a non-degenerate phase, $1 \leq p, s \leq \infty$ and $\beta \geq 0$. 
\begin{enumerate}[I)]
    \item We let $\bK_{p \to s}(\phi; \beta)$ denote the statement:\medskip
    
    \noindent For all $\varepsilon > 0$, the inequality
    \begin{equation*}
        \|\cK^{\delta}[\phi]g\|_{L^s(\B^{n-1})} \lesssim_{\phi, \varepsilon} \delta^{-\beta - \varepsilon}\|g\|_{L^p(\R^n)}
    \end{equation*}
    holds uniformly over all $0 < \delta < 1$ and all $g \in L^1_{\mathrm{loc}}(\R^n)$.  
    \item We let $\bN_{p \to s}(\phi;\beta)$ denote the statement:\medskip
    
    \noindent For all $\varepsilon > 0$, the inequality
    \begin{equation*}
        \|\cN^{\delta}[\phi]g\|_{L^s(\B^{n-1})} \lesssim_{\phi, \varepsilon} \delta^{-\beta - \varepsilon}\|g\|_{L^p(\R^n)}
    \end{equation*}
     holds uniformly over all $0 < \delta < 1$ and all $g \in L^1_{\mathrm{loc}}(\R^n)$. 
\end{enumerate} 
\end{definition}

If we take $\phi$ to be the prototypical phase $\phi_{\mathrm{const}}$ from \eqref{eq: prototype}, then the corresponding maximal functions are essentially the classical Kakeya and Nikodym maximal functions for tubes formed by line segments. In particular, determining the range of $p$ and $q$ for which $\bK_{p \to q}(\phi_{\mathrm{const}};0)$ and $\bN_{p \to q}(\phi_{\mathrm{const}};0)$ hold is equivalent to the classical Kakeya and Nikodym maximal conjectures, respectively. We remark that these two conjectures are known to be equivalent: see \cite{Tao1999}.

%%%%%%%%%%%%%%%%%%%%%%%%%%%%%%%%%%%%%%%%%%%%%%%%%%%%%%%%%%%%%%%%%%%%%%%%%%%%%%%%%%%%%%%%%%%%%%%%

%    Universal estimates

%%%%%%%%%%%%%%%%%%%%%%%%%%%%%%%%%%%%%%%%%%%%%%%%%%%%%%%%%%%%%%%%%%%%%%%%%%%%%%%%%%%%%%%%%%%%%%%%

\subsection{Universal estimates} All the operators introduced satisfy certain \textit{universal} estimates, which are valid for all choices of H\"ormander-type phase. 

%%%%%%%%%%%%%%%%%%%%%%%%%%%%%%%%%%%%%%%%%%%%%%%%%%%%%%%%%%%%%%%%%%%%%%%%%%%%%%%%%%%%%%%%%%%%%%%%

%    Universal estimates: Oscillatory integral operators

%%%%%%%%%%%%%%%%%%%%%%%%%%%%%%%%%%%%%%%%%%%%%%%%%%%%%%%%%%%%%%%%%%%%%%%%%%%%%%%%%%%%%%%%%%%%%%%%

\subsubsection*{Universal estimates: Oscillatory integral operators}%\label{subsec: universal} 

Combining classical results of H\"ormander~\cite{Hormander1973} and Stein~\cite{Stein1986}, it is well known that any H\"ormander-type phase $\phi$ satisfies $\bH_{p \to q}(\phi;\alpha)$ and $\bLS_q(\phi; \alpha)$ for a certain \textit{universal} range. This is defined in relation to the \textit{Stein--Tomas} exponents
\begin{equation*}
  q(n) :=   2 \cdot \frac{n+1}{n-1}
\end{equation*}
and
\begin{align}\label{eq: H alpha def}
   \alpha_{\mathrm{H}}(n, q) &:=
   \begin{cases}
       \frac{1}{2} -\frac{n+1}{2}\Big(\frac{1}{2} - \frac{1}{q}\Big) & \textrm{if $2 \leq q \leq q(n)$} \\
       0 & \textrm{if $q(n) \leq q \leq \infty$}
   \end{cases}, \\
\label{eq: LS alpha def}
 \alpha_{\mathrm{LS}}(n;q) &:= \alpha_{\mathrm{H}}(n,q) + (n-1)\Big(\frac{1}{2} - \frac{1}{q}\Big).
\end{align}
In particular, we have the following theorem.

\begin{theorem}[Universal oscillatory estimate \cite{Stein1986}]\label{thm: uni osc} For $n \geq 2$ and $\phi \colon \D^n \to \R$ a H\"ormander-type phase:
\begin{enumerate}[I)]
    \item $\bH_{2 \to q}(\phi; \alpha)$ holds for all $\alpha \geq \alpha_{\mathrm{H}}(n;q)$ and $2 \leq q \leq \infty$. In particular, $\bH_{2 \to q}(\phi)$ holds for all $q \geq q(n)$.
    \item $\bLS_q(\phi; \alpha)$ holds for all $\alpha \geq \alpha_{\mathrm{LS}}(n;q)$ and all $2 \leq q \leq \infty$. 
\end{enumerate}
\end{theorem}

At the level of the universal estimates, H\"ormander-type oscillatory integrals and propagators essentially share the same theory. We remark, in particular, that part II) of Theorem~\ref{thm: uni osc} is a direct consequence of part I) combined with Plancherel's theorem, exploiting a certain local property of the propagator. For completeness, we present the details of the proof of part II) in \S\ref{subsec: loc osc} below.

%%%%%%%%%%%%%%%%%%%%%%%%%%%%%%%%%%%%%%%%%%%%%%%%%%%%%%%%%%%%%%%%%%%%%%%%%%%%%%%%%%%%%%%%%%%%%%%%

%    Geometric maximal functions

%%%%%%%%%%%%%%%%%%%%%%%%%%%%%%%%%%%%%%%%%%%%%%%%%%%%%%%%%%%%%%%%%%%%%%%%%%%%%%%%%%%%%%%%%%%%%%%%

\subsubsection*{Universal estimates: geometric maximal functions} There are also universal estimates at the level of the underlying geometric maximal functions, which are once again true for all phase functions satisfying H1) and H2). These are defined in relation to the \textit{bush} exponents
\begin{equation}\label{eq: bush exp}
  p(n) :=   \frac{n+1}{2},  \qquad \beta(n; p) := 
  \begin{cases}
      \frac{n}{p} - 1 & \textrm{for $1 \leq p \leq p(n)$} \\
      \frac{n-1}{2p} & \textrm{for $p(n) \leq p \leq \infty$}
  \end{cases},
\end{equation}
and
\begin{equation}\label{eq: other bush exp}
  s(n; p) := 
  \begin{cases}
       (n-1) p' & \textrm{for $1 \leq p \leq p(n)$} \\
      2p & \textrm{for $p(n) \leq p \leq \infty$}
  \end{cases}.
\end{equation}
In particular, we have the following theorem. 

\begin{proposition}[Universal geometric estimate]\label{prop: univ Kak/Nik} For $n \geq 2$ and $\phi \colon \D^n \to \R$ a H\"ormander-type phase, 
\begin{enumerate}[I)]
    \item $\bK_{p \to s}\big(\phi;  \beta\big)$,
    \item $\bN_{p \to s}\big(\phi;  \beta\big)$
\end{enumerate}
both hold for all $\beta \geq \beta(n; p)$ and $1 \leq p \leq \infty$ and $1 \leq s \leq s(n;p)$.
\end{proposition}

By testing the inequalities against the characteristic function of a $\delta$ ball, it is easy to see that Proposition~\ref{prop: univ Kak/Nik} cannot be improved for any $1 \leq p \leq p(n)$, in the sense that neither $\bK_{p \to 1}\big(\phi;  \beta\big)$ nor $\bN_{p \to 1}\big(\phi;  \beta\big)$ can hold for any $\beta < \beta(n; p)$ in this range. 

%%%%%%%%%%%%%%%%%%%%%%%%%%%%%%%%%%%%%%%%%%%%%%%%%%%%%%%%%%%%%%%%%%%%%%%%%%%%%%%%%%%%%%%%%%%%%%%%

%    Sharpness of the universal bounds

%%%%%%%%%%%%%%%%%%%%%%%%%%%%%%%%%%%%%%%%%%%%%%%%%%%%%%%%%%%%%%%%%%%%%%%%%%%%%%%%%%%%%%%%%%%%%%%%

\subsubsection*{Sharpness of the universal bounds} Without further hypotheses on the phase, it is a celebrated observation of Bourgain~\cite{Bourgain1991} that both part I) of Theorem~\ref{thm: uni osc} and part I) of Proposition~\ref{prop: univ Kak/Nik} are sharp in odd dimensions. In particular, the exponents $q(n)$ and $p(n)$ represent the limit of the universal theory. 

We describe Bourgain's result in more detail. In particular, let $n \geq 3$ be odd and define a phase function 
\begin{equation*}
\phi_{\star} \colon \D^n \to \R, \qquad \phi_{\star}(x,t; y) := \inn{x}{y} +\frac{1}{2} \inn{\bA(t)y}{y}, \end{equation*}
where
\begin{equation*}
\bA(t) := \underbrace{A(t) \oplus \cdots \oplus A(t)}_{\frac{n-1}{2}-\textrm{fold}} \qquad \textrm{for} \qquad A(t):=
\begin{pmatrix}
0 &t
\\
t &t^2
\end{pmatrix}.
\end{equation*}

Fixing $a \in C^{\infty}_c(\D^n)$ satisfying $a(0,0;0) = 1$, we let \begin{equation*}
    S_{\star}^{\lambda} := S^{\lambda}[\phi_{\star}; a], \quad U_{\star}^{\lambda} := U^{\lambda}[\phi_{\star}; a], \quad \cK^{\delta}_{\star} := \cK^{\delta}[\phi_{\star}], \quad  \cN^{\delta}_{\star} := \cN^{\delta}[\phi_{\star}] 
\end{equation*}
for all $\lambda \geq 1$. With these definitions, the oscillatory integral operator $S_{\star}^{\lambda}$ and corresponding maximal function $\cK^{\delta}_{\star}$ exhibit the following unfavourable behaviour.

\begin{theorem}[Bourgain \cite{Bourgain1991}]\label{thm: univ sharp} For all $n \geq 3$ odd, the following hold:
\begin{enumerate}[i)]
    \item For all $2 \leq q \leq q(n)$, we have 
\begin{equation}\label{eq: univ sharp}
\|S_{\star}^{\lambda}\|_{L^{\infty}(\B^{n-1}) \to L^q(\R^n)} \gtrsim \lambda^{\alpha_{\mathrm{H}}(n,q)} \qquad \textrm{for all $\lambda \geq 1$.}
\end{equation}
In particular, $\bH_{p \to q}(\phi_{\star})$ \textbf{fails} for all $1 \leq p \leq \infty$ whenever $2 \leq q < q(n)$.
    \item For all $p(n) \leq p \leq \infty$, we have
\begin{equation*}
    \|\cK^{\delta}_{\star}\|_{L^p(\R^n) \to L^1(\B^{n-1})} \gtrsim \delta^{-\beta(n; p)} \qquad \textrm{for all $0 < \delta < 1$.}
\end{equation*}
Thus, $\bK_{p \to 1}(\phi_{\star}, \beta)$ \textbf{fails} to hold for any $\beta < \beta(n; p)$ whenever $p(n) < p \leq \infty$. 
\end{enumerate}
\end{theorem}

In particular, the operators $S_{\star}^{\lambda}$ and $\cK^{\delta}_{\star}$ do not satisfy any bounds beyond those guaranteed by the universal estimates in Theorem~\ref{thm: uni osc} and Proposition~\ref{prop: univ Kak/Nik}. Theorem~\ref{thm: univ sharp} demonstrates the worst possible behaviour for a H\"ormander-type operator: indeed, Theorem~\ref{thm: uni osc} shows the reverse inequality in \eqref{eq: univ sharp} always holds.

An interesting observation, essentially due to Wisewell \cite{Wisewell2005}, is that the corresponding propagator and Nikodym maximal function exhibit \textit{favourable} behaviour. 

\begin{theorem}[Wisewell \cite{Wisewell2005}]\label{thm: Wisewell} For all $n \geq 3$ odd, the following hold:
\begin{enumerate}[i)]
    \item For all $2 < q < q(n)$, there exists some $\alpha(q) < \alpha_{\mathrm{LS}}(n; q)$ such that $\bLS_q(\phi_{\star}; \alpha)$ \textbf{holds} for all $\alpha \geq \alpha(q)$.
    \item For all $p(n) < p < \infty$, there exists some $\beta(p) < \beta(n;p)$ such that $\bN_{p \to 1}(\phi_{\star}, \beta)$ \textbf{holds} for all $\beta \geq \beta(p)$.
\end{enumerate}
\end{theorem}

In particular, the operators $U_{\star}^{\lambda}$ and $\cN^{\delta}_{\star}$ satisfy bounds beyond those guaranteed by the universal estimates in Theorem~\ref{thm: uni osc} and Proposition~\ref{prop: univ Kak/Nik}. We remark that part ii) of Theorem~\ref{thm: Wisewell} is a special case of \cite[Theorem 12]{Wisewell2005}, which also gives an explicit value for $\beta(p)$. Part i) of  Theorem~\ref{thm: Wisewell} does not explicitly appear in \cite{Wisewell2005}, but it nevertheless follows from part ii) using a straightforward variant of a well-known argument from \cite{Bourgain1991} (see also Proposition~\ref{prop: geom red} below). 

We emphasise the contrast between the prototypical case of the constant coefficient phase $\phi_{\mathrm{const}}$ and the general situation for variable coefficient phases. For instance, recall that \cite{Tao1999} shows that, in the constant coefficient case, the $L^p$ theory of the classical Kakeya maximal function is equivalent to that of the classical Nikodym maximal function. However, for a general (variable coefficient) phase $\phi$, Theorem~\ref{thm: univ sharp} and Theorem~\ref{thm: Wisewell} show that these theories can differ markedly. 

We remark that there do exist examples of operators $U_{\dagger}^{\lambda}$ and $\cN_{\dagger}^{\lambda}$ which do not satisfy any bounds beyond those guaranteed by the universal estimates in Theorem~\ref{thm: uni osc} and Proposition~\ref{prop: univ Kak/Nik}. That is, there is a counterpart of Bourgain's construction Theorem~\ref{thm: univ sharp} in the propagator / Nikodym setting. Such examples have not previously appeared in the literature, and so for completeness we briefly describe one construction in \S\ref{subsec: sharpness universal}.

%%%%%%%%%%%%%%%%%%%%%%%%%%%%%%%%%%%%%%%%%%%%%%%%%%%%%%%%%%%%%%%%%%%%%%%%%%%%%%%%%%%%%%%%%%%%%%%%

%    Main results

%%%%%%%%%%%%%%%%%%%%%%%%%%%%%%%%%%%%%%%%%%%%%%%%%%%%%%%%%%%%%%%%%%%%%%%%%%%%%%%%%%%%%%%%%%%%%%%%

\section{Main results and overview of the argument}\label{sec: main results}

%%%%%%%%%%%%%%%%%%%%%%%%%%%%%%%%%%%%%%%%%%%%%%%%%%%%%%%%%%%%%%%%%%%%%%%%%%%%%%%%%%%%%%%%%%%%%%%%

%    Overview

%%%%%%%%%%%%%%%%%%%%%%%%%%%%%%%%%%%%%%%%%%%%%%%%%%%%%%%%%%%%%%%%%%%%%%%%%%%%%%%%%%%%%%%%%%%%%%%%

\subsection{Overview} The goal of this article is to determine weak conditions on $\phi$ which guarantee estimates hold for the corresponding operators $S^{\lambda}$ and $U^{\lambda}$ beyond the universal range of Theorem~\ref{thm: uni osc}. There are already some instances where such improvements are known:
\begin{enumerate}[a)]
    \item If $n \geq 2$ is even \cite{BG2011};
    \item If $\phi$ satisfies a \textit{positive definite} (or, more generally, a \textit{positive signature}) condition \cite{GHI2019, HI2022}.
\end{enumerate}
In light of a), throughout this paper we focus on the case $n \geq 3$ odd, although many subsidiary results also apply (and are of interest) in even dimensions. On the other hand, our main results are of interest in the signature $0$ case, although we are also able to prove results which are of interest in the positive definite setting, going beyond the range described in \cite{GHI2019}. This is discussed in more detail in \S\ref{240511subsection2_1} below.

The considerations of \cite{BG2011, GHI2019, HI2022} which lead to the improvements in a) and b) above are somewhat different in flavour from our analysis. Here we explicitly study the geometry of the underlying Kakeya/Nikodym sets of curves, establishing estimates for $\cK^{\delta}$ and $\cN^{\delta}$ beyond the universal range of Proposition~\ref{prop: univ Kak/Nik}. These geometric maximal estimates then imply their oscillatory counterparts. For H\"ormander-type operators $S^{\lambda}$ and Kakeya maximal functions $\cK^{\delta}$, a programme of establishing estimates beyond the universal range was initiated in the pioneering work of Bourgain~\cite{Bourgain1991}. Our methods and results extend and clarify those of \cite{Bourgain1991}, combining ideas from this classic work with more modern harmonic analysis techniques.

%%%%%%%%%%%%%%%%%%%%%%%%%%%%%%%%%%%%%%%%%%%%%%%%%%%%%%%%%%%%%%%%%%%%%%%%%%%%%%%%%%%%%%%%%%%%%%%%

%    Beyond the universal estimates: the $n=3$ case

%%%%%%%%%%%%%%%%%%%%%%%%%%%%%%%%%%%%%%%%%%%%%%%%%%%%%%%%%%%%%%%%%%%%%%%%%%%%%%%%%%%%%%%%%%%%%%%%

\subsection{Beyond the universal estimates: the $n=3$ case}\label{240511subsection2_1}
In light of Theorem~\ref{thm: univ sharp}, to go beyond the universal estimate from Theorem~\ref{thm: uni osc} when $n \geq 3$ is odd, it is necessary to impose additional hypotheses on $\phi$, beyond H1) and H2). When studying H\"ormander-type operators and Kakeya sets of curves, to simplify matters, we  restrict to the class of \textit{translation-invariant} phases.

\begin{definition} A non-degenerate phase $\phi \colon \D^n \to \R$ is \textit{translation-invariant} if it is of the form 
\begin{equation}\label{eq: trans inv}
    \phi(x,t;y) = \inn{x}{y} + \psi(t;y)
\end{equation}
for some $\psi \in C^{\infty}((-1,1) \times \B^{n-1})$ satisfying $\psi(0; y) = 0$ for all $y \in \B^{n-1}$. 
\end{definition}

Any phase of the form \eqref{eq: trans inv} automatically satisfies H1), whilst in this setting the condition H2) becomes 
\begin{equation*}
    \det \partial_{yy}^2 \partial_t\psi(t;y) \neq 0 \qquad \textrm{for $(t,y) \in (-1,1) \times \B^{n-1}$.}
\end{equation*}
An advantage of working with phases of the form \eqref{eq: trans inv} is that the curves $\gamma_{y,\omega}$ are given explicitly by
\begin{equation*}
    \gamma_{y,\omega}(t) = \omega - \partial_y \psi(0, t; y);
\end{equation*}
note that for fixed $y$ the $\gamma_{y,\omega}$ are simply translates of the curve $\gamma_{y,0}$. We remark that the phase $\phi_{\star}$ from Theorem~\ref{thm: univ sharp} is translation-invariant, and so further assumptions are required if one wishes to prove oscillatory integral estimates which go beyond the universal range in Theorem~\ref{thm: uni osc}.

Comparing Theorem~\ref{thm: univ sharp} and Theorem~\ref{thm: Wisewell}, it is natural to impose \textit{different} conditions for H\"ormander-type operators and variable Schr\"odinger propagators. We begin by describing our hypotheses and results in the $n = 3$ case, where the setup simplifies. 

\begin{definition}\label{dfn: n=3 lin comb K} Let $\phi \colon \D^3 \to \R$ be a translation-invariant phase. When studying the associated H\"ormander-type operator, we work under the following additional hypothesis.\medskip

\noindent \textbf{Hypothesis I): Kakeya non-compression.} Let $\cZ_{\mathrm{K}}(\phi)$ denote the set of all $y \in \B^2$ for which there exist scalars $\mu_{i,j}$, $1 \leq i, j \leq 2$, such that
   \begin{equation*}
    \det \partial_{yy}^2\phi(0,t;y) = \sum_{1 \leq i, j \leq 2} \mu_{i,j} \partial_{y_iy_j}^2 \phi(0,t;y) \qquad \textrm{for all $t \in (-1,1)$.}
\end{equation*} 
Then $\cL^2(\cZ_{\mathrm{K}}(\phi)) = 0$. 
\end{definition}

Before stating the alternative hypothesis used in the variable Schr\"odinger propagator case, we introduce some notation. Given $d \in \N_0$, an open interval $I \subset \R$, a function $f \in C^d(I)$ and $s \in I$, we write $P_d[f,s]$ for the order $d$ Taylor polynomial of $f$ centred at $s$.

\begin{definition}\label{dfn: n=3 lin comb N} Let $\phi \colon \D^3 \to \R$ be a non-degenerate phase. When studying the associated variable Schr\"odinger propagator, we work under the following additional hypothesis.\medskip

\noindent \textbf{Hypothesis II): Nikodym non-compression.} Define functions
\begin{equation*}
    g(t; \omega; y) := \det (\partial_{yy}^2\phi)(\Psi(\omega; t; y),t;y), \quad  g_{i, j}(t; \omega; y) := (\partial_{y_iy_j}^2 \phi)(\Psi(\omega; t; y),t;y)
\end{equation*}
for $1 \leq i, j \leq 2$. The hypothesis now involves two distinct conditions:
\begin{enumerate}[a)]
    \item For $s \in (-1,1)$ and $d \in \N$, let $\cZ_{\mathrm{N}}(\phi;s;d)$ denote the set of all $(\omega, y) \in \B^2 \times \B^2$ for which there exist scalars $\mu$, $\mu_{i,j}$, $1 \leq i, j \leq 2$ such that
   \begin{equation*}
    1 = \mu P_d[g(\,\cdot\,; \omega; y), s](t)+ \sum_{1 \leq i, j \leq 2} \mu_{i,j} P_d[g_{i,j}(\,\cdot\,; \omega; y), s](t)
\end{equation*} 
for all $t \in (-1,1)$. Then there exists some $d \in \N$ such that $\cZ_{\mathrm{N}}(\phi; s; d) = \emptyset$ for all $s \in (-1,1)$.
\item There is an integer $0 \leq r \leq 4$ such that the $\infty \times 4$ matrix
\begin{equation*}
 \begin{bmatrix}
         g (t; \omega ;y) & g_{1,1} (t; \omega ;y) & g_{1,2} (t; \omega ;y) &  g_{2,2} (t; \omega ;y) \\
        (\partial_t g) (t; \omega ;y) & (\partial_t g_{1,1}) (t; \omega ;y) & (\partial_t g_{1,2}) (t; \omega ;y) & (\partial_t  g_{2,2}) (t; \omega ;y) \\
        (\partial_t^2 g) (t; \omega ;y) & (\partial_t^2 g_{1,1}) (t; \omega ;y) & (\partial_t^2 g_{1,2}) (t; \omega ;y) & (\partial_t^2  g_{2,2}) (t; \omega ;y) \\
        \vdots & \vdots & \vdots & \vdots
    \end{bmatrix}
\end{equation*}
has rank $r$ for all $(\omega, t;y) \in \D^3$. 
\end{enumerate}
\end{definition}
We emphasise that Definition~\ref{dfn: n=3 lin comb N} does not require translation invariance of the phase function. 

With these definitions, the $n = 3$ case of our main theorem reads thus. 

\begin{theorem}\label{thm: n=3 osc improve}  Let $n =3$ and $\phi \colon \D^3 \to \R$ be a real analytic H\"ormander-type phase. 
\begin{enumerate}[I)]
    \item Suppose $\phi$ is translation-invariant and satisfies Hypothesis I). Then there exists $q_{\phi} < 4$ such that $\bH_{\infty \to q}(\phi)$ holds for $q \geq q_{\phi}$.
    \item Suppose $\phi$ satisfies Hypothesis II). For all $2 < q < 4$ there exists $\alpha_{\phi}(q) < 1/2$ such that $\bLS_q(\phi; \alpha)$ holds for $\alpha \geq \alpha_{\phi}(q)$. 
\end{enumerate}
\end{theorem}

As remarked at the beginning of the section, Theorem~\ref{thm: n=3 osc improve} (and its higher dimensional generalisation, Theorem~\ref{thm: osc improve} below) are already known to hold for H\"ormander-type phases under certain \textit{signature} hypotheses. For example, we say a H\"ormander-type phase $\phi$ is \textit{positive-definite} if it satisfies the following strengthening of H2):
\begin{itemize}
\item[H2$^+$)] For all $(x, t; y_0) \in \D^n$, the matrix 
\begin{equation*}
\partial^2_{yy} \langle \partial_x\phi(x, t; y),G(x, t;y_0)\rangle|_{y = y_0}
\end{equation*}
is positive-definite.
\end{itemize}
It then follows from work of Lee~\cite{Lee2006} (see also \cite{BG2011, GHI2019, HI2022}) that for $n = 3$ and $\phi$ a positive-definite H\"ormander-type phase, we have $\bH_{\infty \to q}(\phi)$ for all $q > 10/3$. We remark that this is sharp, in the sense that there exists a positive-definite H\"ormander-type phase $\phi$ for which $\bH_{\infty \to q}(\phi)$ fails for all $q < 10/3$.

In light of the above, Theorem~\ref{thm: n=3 osc improve} is of primary interest in cases where H2)$^+$ fails. We remark, however, that our methods can be used to prove new estimates in the positive-definite case. As an example, we have the following partial strengthening of Lee's theorem~\cite{Lee2006}.

\begin{theorem}\label{thm: pos-def}  Let $n = 3$ and $\phi \colon \D^3 \to \R$ be a real analytic, positive-definite H\"ormander-type phase. Suppose $\phi$ is translation-invariant and satisfies Hypothesis I). Then there exists $q_{\phi} < 10/3$ such that $\bH_{\infty \to q}(\phi)$ holds for $q \geq q_{\phi}$.
\end{theorem}

Theorem~\ref{thm: pos-def} follows by combining the methods used to prove Theorem~\ref{thm: n=3 osc improve} (more precisely, the underlying geometric maximal estimate from Theorem~\ref{thm: geom improve} below) with a variable coefficient generalisation of the the refined \textit{broad-narrow} argument of Bourgain--Guth~\cite[\S4]{BG2011}. We shall not present the proof of Theorem~\ref{thm: pos-def}, but include the statement to highlight the scope of our methods. 

%%%%%%%%%%%%%%%%%%%%%%%%%%%%%%%%%%%%%%%%%%%%%%%%%%%%%%%%%%%%%%%%%%%%%%%%%%%%%%%%%%%%%%%%%%%%%%%%

%    Beyond the universal estimates in higher dimensions

%%%%%%%%%%%%%%%%%%%%%%%%%%%%%%%%%%%%%%%%%%%%%%%%%%%%%%%%%%%%%%%%%%%%%%%%%%%%%%%%%%%%%%%%%%%%%%%%

\subsection{Beyond the universal estimates in higher dimensions}\label{subsec: beyond higher} Theorem~\ref{thm: n=3 osc improve} extends to all odd dimensions $n \geq 3$, but the setup is a little more involved. To describe the general result, we begin by introducing some preliminary notation. 
\begin{itemize}
    \item For $1 \leq m \leq d$, we let $\cP(d,m)$ denote the collection of all subsets of $\{1, \dots, d\}$ of cardinality $m$. We also let $\cP(\alpha)$ denote the power set of a set $\alpha \subseteq \N$ and, given $\alpha$, $\beta \subseteq \N$, we define
    \begin{equation*}
        \cI(\alpha,\beta) := \{(\alpha',\beta') \in \cP(\alpha) \times \cP(\beta) : 0 < |\alpha'| = |\beta'|\}.
    \end{equation*} 
    \item Given a matrix $M \in \mathrm{Mat}(\R, d)$ and $\alpha \in \cP(d, k)$, $\beta \in \cP(d,\ell)$ for $1 \leq k, \ell \leq d$, we let $[M]_{\alpha, \beta} \in \mathrm{Mat}(\R, k \times \ell)$ denote the matrix formed by removing the $i$th row of $M$ for all $i \in \{1, \dots, d\} \setminus \alpha$ and the $j$th column of $M$ for all $j \in \{1, \dots, d\} \setminus \beta$.
    \item For $n \geq 3$, define the \textit{critical dimension} by
\begin{equation*}
    d_{\crit}(n) :=
    \begin{cases}
        \frac{n+1}{2} & \textrm{if $n$ is odd,} \\
        \frac{n+2}{2} & \textrm{if $n$ is even.}
    \end{cases}
\end{equation*}
We remark that $d_{\crit}(n)$ corresponds to the minimum possible dimension of either a Kakeya or Nikodym set of curves in $\R^n$, as dictated by the results of \cite{Wisewell2005, BG2011}. 
\end{itemize}
With this notation in place, we can extend Definition~\ref{dfn: n=3 lin comb K} to higher dimensions.

\begin{definition}\label{dfn: lin comb K} Let $\phi \colon \D^n \to \R$ be a translation-invariant phase.  When studying the associated H\"ormander-type operator, we work under the following additional hypothesis.\medskip

\noindent \textbf{Hypothesis I): Kakeya non-compression.} There exist $\alpha$, $\beta \in \cP(n-1, d_{\crit}(n))$ such that the following holds. Let $\cZ_{\mathrm{K}}(\phi; \alpha, \beta)$ denote the set of all $y \in \B^{n-1}$ for which there exist scalars $\mu_{\alpha',\beta'}$, indexed by $(\alpha', \beta') \in \cI(\alpha,\beta)$, such that
   \begin{equation*}
    \det\big[ \partial_{yy}^2\phi(0,t;y)\big]_{\alpha,\beta} = \sum_{(\alpha', \beta') \in \cI(\alpha,\beta)} \mu_{\alpha',\beta'} \det\big[\partial_{yy}^2 \phi(0,t;y)\big]_{\alpha',\beta'}
\end{equation*} 
for all $t \in (-1,1)$. Then $\cL^{n-1}(\cZ_{\mathrm{K}}(\phi; \alpha, \beta)) = 0$.
\end{definition}

Similarly, we extend Definition~\ref{dfn: n=3 lin comb N} to higher dimensions. 

\begin{definition}\label{dfn: lin comb N}  Let $\phi \colon \D^n \to \R$ be a non-degenerate phase. When studying the associated variable Schr\"odinger propagator, we work under the following additional hypothesis.\medskip

\noindent \textbf{Hypothesis II): Nikodym non-compression.} Given $\alpha, \beta \subseteq \{1, \dots, n-1\}$ with $0 < |\alpha| = |\beta|$, define the function
\begin{equation*} 
    g_{\alpha, \beta}(t; \omega; y) := \det[(\partial_{yy}^2 \phi)(\Psi(\omega; t; y),t;y)]_{\alpha, \beta}.
\end{equation*}
The hypothesis now involves two distinct conditions:
\begin{enumerate}[a)]
    \item There exist $\alpha$, $\beta \in \cP(n-1, d_{\crit}(n))$ and $d \in \N$ such that the following holds. Let $\cZ_{\mathrm{N}}(\phi;\alpha, \beta; d)$ denote the set of all $(\omega, y) \in \B^{n-1} \times \B^{n-1}$ for which there exist some $s \in (-1,1)$ and scalars $\mu_{\alpha',\beta'}$, indexed by $(\alpha', \beta') \in \cI(\alpha,\beta)$, such that
   \begin{equation*}
    1 = \sum_{(\alpha', \beta') \in \cI(\alpha,\beta)} \mu_{\alpha', \beta'} P_d[g_{\alpha', \beta'}(\,\cdot\,; \omega; y), s](t) \qquad \textrm{for all $t \in (-1,1)$.}
\end{equation*} 
Then $\cZ_{\mathrm{N}}(\phi; \alpha, \beta; d) = \emptyset$.
\item There is an integer $0 \leq r \leq |\cI(\alpha,\beta)|$ such that the $\infty \times |\cI(\alpha,\beta)|$ matrix with columns
\begin{equation*}
 \begin{bmatrix}
         g_{\alpha', \beta'} (t; \omega ;y)  \\
        (\partial_t g_{\alpha', \beta'}) (t; \omega ;y) \\
        (\partial_t^2 g_{\alpha', \beta'}) (t; \omega ;y) \\
        \vdots 
    \end{bmatrix} \qquad \textrm{for $(\alpha', \beta') \in \cI(\alpha,\beta)$}
\end{equation*}
has rank $r$ for all $(\omega;t;y) \in \D^n$. 
\end{enumerate}
\end{definition}

With these definitions, the main theorem in general odd dimensions reads thus. 

\begin{theorem}\label{thm: osc improve}  Let $n \geq 3$ be odd and $\phi \colon \D^n \to \R$ be a real analytic H\"ormander-type phase. 
\begin{enumerate}[I)]
    \item Suppose $\phi$ is translation-invariant and satisfies Hypothesis I). Then there exists $q_{\phi} < q(n)$ such that $\bH_{\infty \to q}(\phi)$ holds for $q \geq q_{\phi}$.
    \item Suppose $\phi$ satisfies Hypothesis II). For all $2 < q < q(n)$ there exists $\alpha_{\phi}(q) < \alpha_{\mathrm{LS}}(n;q)$ such that $\bLS_q(\phi; \alpha)$ holds for $\alpha \geq \alpha_{\phi}(q)$. 
\end{enumerate}
\end{theorem}

Theorem~\ref{thm: osc improve} is true for even dimensions $n \geq 2$, which follows from the same proof. However, the even dimensional case is not of direct interest since (a more quantitative form of) the conclusion of the theorem is already known to hold for $n \geq 2$ even under hypotheses H1) and H2) alone~\cite{BG2011}.

 Underpinning the main oscillatory theorem is a purely geometric result for the corresponding maximal functions. 

\begin{theorem}\label{thm: geom improve} Let $n \geq 3$ be odd and $\phi \colon \D^n \to \R$ be a real analytic H\"ormander-type phase. 
\begin{enumerate}[I)]
    \item Suppose $\phi$ is translation-invariant and satisfies Hypothesis I). For all $p(n) < p < \infty$ there exists $\beta_{\phi}(p) < \beta(n; p)$ such that $\bK_{p \to p}(\phi; \beta)$ holds for $\beta \geq \beta_{\phi}(p)$;
    \item Suppose $\phi$ satisfies Hypothesis II). For all $p(n) < p < \infty$ there exists $\beta_{\phi}(p) < \beta(n; p)$ such that $\bN_{p \to p}(\phi; \beta)$ holds for $\beta \geq \beta_{\phi}(p)$.
\end{enumerate}
\end{theorem}

Thus, if $\phi$ is translation-invariant and satisfies Hypothesis I), then both the corresponding H\"ormander-type operator and Kakeya maximal function satisfy estimates beyond the universal range. Similarly, if Hypothesis II) holds, then the corresponding variable Schr\"odinger propagator and Nikodym maximal operator satisfy estimates beyond the universal range.

 We remark that Theorem~\ref{thm: osc improve} is in fact a consequence of Theorem~\ref{thm: geom improve}, due to a classical square function argument of Bourgain~\cite{Bourgain1991}: see Proposition~\ref{prop: geom red} below.

%%%%%%%%%%%%%%%%%%%%%%%%%%%%%%%%%%%%%%%%%%%%%%%%%%%%%%%%%%%%%%%%%%%%%%%%%%%%%%%%%%%%%%%%%%%%%%%

%    Context for the results

%%%%%%%%%%%%%%%%%%%%%%%%%%%%%%%%%%%%%%%%%%%%%%%%%%%%%%%%%%%%%%%%%%%%%%%%%%%%%%%%%%%%%%%%%%%%%%%%

\subsection{Context for the results}\label{subsec: context} Here we provide some remarks to contextualise and motivate Theorem~\ref{thm: osc improve}.\medskip

\noindent i) \textit{Addressing an anomaly in \cite{Bourgain1991}}. A similar result to Theorem~\ref{thm: n=3 osc improve} I) is presented in \cite[Theorem, p.366]{Bourgain1991}, with different hypotheses on the phase. A variant of the $n = 3$ case of Theorem~\ref{thm: geom improve} I) is also implicit in \cite{Bourgain1991}, although our method of proof for Theorem~\ref{thm: geom improve} differs significantly from that used in \cite{Bourgain1991} (see item iii) below). Both here and in \cite{Bourgain1991} an additional hypothesis (in our case Hypothesis I)) is used to prove sublevel set estimates which control the extent to which the tubes $T_{y,\omega}$ can compress into small sets. However, the sublevel set theory appears to be significantly more involved than the argument in \cite{Bourgain1991} suggests (see \cite[pp.364-365]{Bourgain1991}). Moreover, the hypothesis used in \cite{Bourgain1991} is insufficient to guarantee the conclusions of the argument: we present an explicit counterexample in \S\ref{subsec: counterexample} which shows that some key conclusions of \cite{Bourgain1991} are not valid. One motivation for this work is to clarify this anomaly. In doing so, we develop a rather rich and intricate sublevel set theory which goes far beyond the arguments presented in \cite{Bourgain1991}.\medskip

\noindent ii) \textit{Contrasting the theory of H\"ormander-type operators and variable Schr\"odinger propagators}. In addition to clarifying \cite{Bourgain1991}, our results strike out in a number of new directions. The most obvious new feature is that both Theorem~\ref{thm: osc improve} and Theorem~\ref{thm: geom improve} give results for odd dimensions $n \geq 3$. The elementary methods used in \cite{Bourgain1991} appear to be specific to the $n=3$ case: see item iii) below. Furthermore, here a principal motivation is to contrast the theory of H\"ormander-type operators with that of the variable coefficient Schr\"odinger propagators introduced above. Differences between the two theories are hinted at by Wisewell~\cite{Wisewell2005} (which focuses on the underlying Kakeya and Nikodym maximal functions), and here we aim to bring them to the fore. We anticipate that the propagator operators will have many applications, which we aim to explore in future work.\medskip

\noindent iii) \textit{New techniques and results in higher dimensions}. Our main result is Theorem~\ref{thm: geom improve}: using a classical argument from \cite{Bourgain1991}, the oscillatory estimates in Theorem~\ref{thm: osc improve} are a consequence of the geometric maximal function estimates from Theorem~\ref{thm: geom improve}. Whilst this initial reduction to geometrical estimates is adapted from \cite{Bourgain1991}, the argument used to prove the geometrical estimates themselves in Theorem~\ref{thm: geom improve} differs significantly from corresponding arguments appearing in \cite{Bourgain1991}.
\begin{itemize}
    \item In \cite{Bourgain1991}, a somewhat \textit{ad hoc} method is used to prove the underlying Kakeya estimate, which is based on C\'ordoba's $L^2$ argument \cite{Cordoba1977} (and also has some similarities with the proof of the circular maximal function from \cite{Bourgain1986}). This approach appears to be heavily tied to the case $n = 3$, since it exploits the fact that the critical bush exponent satisfies $p(3) = 2$. 
    \item By contrast, here we use a more systematic approach to Kakeya and Nikodym maximal estimates, based on modern polynomial partitioning techniques introduced by Guth~\cite{Guth2016}. One advantage of this approach is that it is not tied down to the $n=3$ case.
\end{itemize}
It is interesting to note that both  the $L^2$-based approach of \cite{Bourgain1991} and the polynomial partitioning approach lead one to consider the same (crucial) sublevel sets, where Hypothesis I) and Hypothesis II) play a role. As discussed in item i), the sublevel set theory requires a more thorough treatment than that of \cite{Bourgain1991}, and we address certain anomalies present in \cite{Bourgain1991} (see \S\ref{subsec: counterexample}). Indeed, the resulting intricate uniform sublevel set estimates constitute one of the main novelties of this work and may be of independent interest.

%%%%%%%%%%%%%%%%%%%%%%%%%%%%%%%%%%%%%%%%%%%%%%%%%%%%%%%%%%%%%%%%%%%%%%%%%%%%%%%%%%%%%%%%%%%%%%%

%    Overview of the proof

%%%%%%%%%%%%%%%%%%%%%%%%%%%%%%%%%%%%%%%%%%%%%%%%%%%%%%%%%%%%%%%%%%%%%%%%%%%%%%%%%%%%%%%%%%%%%%%%

\subsection{Overview of the proof} As mentioned above, Theorem~\ref{thm: osc improve} is in fact a consequence of Theorem~\ref{thm: geom improve}. This implication is formally stated as follows.

\begin{proposition}[Bourgain~\cite{Bourgain1991}]\label{prop: geom red} 
 Let $n \geq 3$, $p(n) < p < \infty$ and $0 \leq \beta < \beta(n,p)$. There exists some $2 < q < q(n)$ and $0 \leq \alpha < \alpha_{\mathrm{LS}}(n;p)$ such that, for $\phi \colon \D^n \to \R$ a H\"ormander-type phase, 
\begin{enumerate}[I)]
    \item $\bK_{p \to 1}(\phi;  \beta) \Rightarrow \bH_{\infty \to q}(\phi)$;
    \item $\bN_{p \to p}(\phi;  \beta) \Rightarrow \bLS_s(\phi; \alpha)$ where $2 < s < q(n)$ satisfies $p = (s/2)'$.
\end{enumerate}
\end{proposition}

The above proposition originates from the groundbreaking work of Bourgain~\cite{Bourgain1991a} on the Fourier restriction problem. A version of Proposition~\ref{prop: geom red} for H\"ormander-type operators is implicit it \cite{Bourgain1991}, but as this only concerns I) in the $n=3$ case, for completeness we present a proof of Proposition~\ref{prop: geom red} in \S\ref{sec: geom red} below.

In light of the above, we focus on the proof of Theorem~\ref{thm: geom improve}. The argument is based on polynomial partitioning techniques of Guth~\cite{Guth2016, Guth2018}, incorporating various refinements from \cite{HRZ2022, Zahl2021}, and a variant of the \textit{polynomial Wolff axiom} theorem of Katz--Rogers~\cite{KR2018}. All these techniques reduce matters to proving certain uniform sublevel set estimates for real analytic functions: see item 4) below. These sublevel set estimates are the main new ingredient in the analysis. We sketch the steps of the proof below. \medskip

\noindent i) \textit{Duality.} A standard consequence of Kakeya/Nikodym maximal estimates are volume bounds of the form
    \begin{equation}\label{eq: overlap}
        \Big|\bigcup_{T \in \T} T \Big| \gtrsim_{\varepsilon} \delta^{\beta + \varepsilon} \sum_{T \in \T} |T|,
    \end{equation}
    where here $\T$ is a family of curved $\delta$-tubes. In the case of the Kakeya maximal operator in Theorem~\ref{thm: geom improve} I), we assume the tubes in $\T$ are \textit{direction-separated}. The precise definition of direction-separated is given in \S\ref{subsec: non concentration grains}; it agrees with the usual definitions in the case of classical Kakeya maximal operator associated to straight line tubes. Similar considerations play out for the Nikodym maximal operator, but for the purpose of this sketch we focus on the Kakeya case.
    
    Inequalities of the form \eqref{eq: overlap} measure the degree to which tubes $T \in \T$ can overlap. Although \eqref{eq: overlap} is weaker than the corresponding maximal function estimate, here we sketch a proof of \eqref{eq: overlap}, which is conceptually easier to work with. The \textit{bona fide} maximal estimates follow in a similar manner, using duality to work with a strengthened form of \eqref{eq: overlap} which takes into account multiplicity of tube intersections.\medskip
  
\noindent ii) \textit{Finding algebraic structure.} To study the set $K := \bigcup_{T \in \T} T$, we apply the polynomial partitioning method of Guth~\cite{Guth2016, Guth2018}. In particular, we use an algorithmic formulation of Guth's argument, similar to that introduced in \cite{HRZ2022} (see also \cite{Zahl2021}). Loosely speaking, the polynomial method is used to show that if $\T$ is a potential counterexample to \eqref{eq: overlap} (for appropriately chosen $\beta$), then the tubes $T \in \T$ must concentrate around some low degree algebraic surface. This reduction is similar to an argument used to study the classical Kakeya maximal function in \cite{HRZ2022}. However, for the purpose of proving Theorem~\ref{thm: geom improve}, the concentration properties are much, much weaker than those used in \cite{HRZ2022}. Consequently, on a technical level, the partitioning algorithm differs from that of \cite{HRZ2022} (and is, in many respects, simpler). We present the details in \S\ref{sec: partitioning}. \medskip
    
\noindent iii) \textit{Parametrising the tubes.} Having reduced to a situation where the tube family has algebraic structure, we apply a variant of the \textit{polynomial Wolff axiom} argument of Katz--Rogers~\cite{KR2018}. This allows us to parametrise our set $K$. It is useful to think of each $T \in \T$ as a union of curves $\Gamma_{y,\omega}$ which are $\delta$-scale perturbations of the core curve of $T$. In the Kakeya case, given a direction $y$, we can assign a centre $\omega(y)$ so that the curve $\Gamma_{y, \omega(y)}$ lies in $K$. This defines a map $y \mapsto \omega(y)$. Roughly speaking, if the tubes concentrate around a low degree algebraic surface, then this property can be used to show that $y \mapsto \omega(y)$ can be taken to be smooth. Thus, the map $\Phi \colon (y, t) \mapsto \Gamma_{y, \omega(y)}(t)$ is smooth and maps into $K$. We can think of this as a parametrisation of $K$. The rigorous justification for this step is highly nontrivial and, following \cite{KR2018}, relies on Gromov's algebraic lemma for parametrising semialgebraic sets (see Lemma~\ref{lem: Gromov}). We give the details in \S\ref{sec: non-concentration}.\medskip
    
\noindent iv)  \textit{Uniform sublevel set estimates.} To obtain a lower bound for $|K|$, it suffices to control the Jacobian determinant $\det J\Phi(y,t)$ of the parametrisation $\Phi \colon (y, t) \mapsto \Gamma_{y, \omega(y)}(t)$ introduced above. In particular, matters reduce to showing the set where $|\det J\Phi(y,t)|$ is small is itself small: we seek an exponent $\kappa > 0$ for which there is a sublevel set estimate of the form
    \begin{equation}\label{eq: sublevel sketch}
        |\{ (y, t) : |\det J\Phi(y,t)| < \sigma \}| \lesssim \sigma^{\kappa} \qquad \textrm{for all $\sigma > 0$.}
    \end{equation}
    The difficulty here is that the map $y \mapsto \omega(y)$ used to define $\Phi$ is non-explicit and somewhat arbitrary. In particular, the sublevel set estimate \eqref{eq: sublevel sketch} must be \textit{uniform} over all possible choices of map $y \mapsto \omega(y)$. The resulting uniform sublevel set estimates are rather subtle, and rely heavily on a variety of special properties of real analytic functions. They constitute the main new ingredient in the paper and are discussed in detail in \S\ref{sec: sublevel}.

%%%%%%%%%%%%%%%%%%%%%%%%%%%%%%%%%%%%%%%%%%%%%%%%%%%%%%%%%%%%%%%%%%%%%%%%%%%%%%%%%%%%%%%%%%%%%%%%

%    Examples and applications

%%%%%%%%%%%%%%%%%%%%%%%%%%%%%%%%%%%%%%%%%%%%%%%%%%%%%%%%%%%%%%%%%%%%%%%%%%%%%%%%%%%%%%%%%%%%%%%%

\section{Examples and applications}

%%%%%%%%%%%%%%%%%%%%%%%%%%%%%%%%%%%%%%%%%%%%%%%%%%%%%%%%%%%%%%%%%%%%%%%%%%%%%%%%%%%%%%%%%%%%%%%%

%    A counterexample under weakened hypotheses

%%%%%%%%%%%%%%%%%%%%%%%%%%%%%%%%%%%%%%%%%%%%%%%%%%%%%%%%%%%%%%%%%%%%%%%%%%%%%%%%%%%%%%%%%%%%%%%%

\subsection{A counterexample under weakened hypotheses}\label{subsec: counterexample} As discussed in \S\ref{subsec: context}, a programme of establishing estimates beyond the universal ranges, for both the H\"ormander operator $S^{\lambda}=S^{\lambda}[\phi; a]$ and the Kakeya maximal function $\mathcal{K}^{\delta}= \mathcal{K}^{\delta}[\phi; a]$, was initiated in the pioneering work of Bourgain~\cite{Bourgain1991}. In that paper Bourgain worked under a weak variant of Hypothesis I). In particular, given $\phi \colon \D^3 \to \R$ a translation-invariant phase, consider the following condition.\medskip

\noindent \textbf{Hypothesis w-I).} There do not exist scalars $\mu_{i,j}$, $1 \leq i, j \leq 2$ such that
   \begin{equation*}
    \det \partial_{yy}^2\phi(0,t;y)\big|_{y = 0} = \sum_{1 \leq i, j \leq 2} \mu_{i,j} \partial_{y_iy_j}^2 \phi(0,t;y)\big|_{y = 0} \qquad \textrm{for all $t \in (-1,1)$.}
\end{equation*} 

It is claimed in \cite[\S6]{Bourgain1991} that for any translation-invariant H\"ormander-type phase $\phi$ satisfying Hypothesis w-I) and all $p> p(3)=2$, there exists $\beta_{\phi}(p)< \beta(3; p)=1/p$ such that 
\begin{equation}\label{240511e3_3}
    \bK_{p\to 1}(\phi; \beta) \quad \textrm{holds for all} \quad \beta \geq \beta_{\phi}(p).
\end{equation}
Note that \eqref{240511e3_3} is (a slightly weaker version of) the conclusion of Theorem \ref{thm: n=3 osc improve}, which is Theorem \ref{thm: geom improve} in dimension $n=3$ (the exponents $p(n)$ and $\beta(n; p)$ are defined in \eqref{eq: bush exp}). Bourgain then combined \eqref{240511e3_3} and Proposition \ref{prop: geom red} and concluded that 
there exists $q_{\phi}<4$ such that 
\begin{equation*}
    \bH_{\infty \to q}(\phi) \ \ \text{for all } q\ge q_{\phi},
\end{equation*}
which is precisely the conclusion of Theorem \ref{thm: n=3 osc improve}.

Bourgain's Hypothesis w-I) is much simpler and often much weaker than the assumption of Theorem~\ref{thm: n=3 osc improve} (which coincides with the assumption of Theorem~\ref{thm: geom improve} when $n=3$).  In particular, Hypothesis w-I) involves only one frequency point $y=0$ whilst Hypothesis I) requires one to check almost all frequency points $y$.

Here we would like to clarify that Bourgain's Hypothesis w-I) is insufficient to guarantee \eqref{240511e3_3}. More precisely, we construct an explicit translation-invariant phase function $\phi$ satisfying Hypothesis w-I) (but not Hypothesis I))  and show that the conclusion \eqref{240511e3_3} fails for this phase function. 

\begin{example} For $n=3$, define the phase function 
\begin{equation*}
    \phi(x, t; y)=
    x \cdot y +
\frac{1}{2} t y_1^2+
\sum_{j=2}^{\infty} \frac{t^{j-1}y_2^j}{
j(j-1)
},
\end{equation*}
where $(x,t;y) \in \D^3$ and $y = (y_1, y_2) \in \B^2$, noting that the series converges within this domain. One can check directly that $\phi$ is a real analytic, translation-invariant H\"ormander-type phase. Compute
\begin{equation*}
    \det \partial^2_{yy} \phi(0, t; y)=
    \det
    \begin{bmatrix}
        t & 0\\
        0 & \frac{t}{1-ty_2}
    \end{bmatrix}
    = \frac{t^2}{1-ty_2} =
    t^2+ t^3 y_2+ t^4 y_2^2+\dots .
\end{equation*}
If we take $y=0$, then 
\begin{equation*}
  \partial^2_{y_1y_1} \phi(0, t; y)\big|_{y=0} = \partial^2_{y_2y_2} \phi(0, t; y)\big|_{y=0} = t \quad \textrm{and} \quad \partial^2_{y_1y_2} \phi(0, t; y)\big|_{y=0} \equiv 0
\end{equation*}
whilst
\begin{equation*}
    \det 
    \partial^2_{yy} \phi(0, t; y)\big|_{y=0} =
    t^2,
\end{equation*}
and so Hypothesis w-I) holds. However, whenever $y_2\neq 0$, one can always write 
\begin{align*}
    \det 
    \partial^2_{yy}\phi(0, t; y) &= 
    \Big(
    -\frac{1}{y_2}
    \Big) t+
    \Big(
    \frac{1}{y_2}
    \Big) \frac{t}{1-ty_2} \\
    &= \Big(-\frac{1}{y_2}\Bigr) \partial_{y_1y_1}^2\phi(0,t;y) + \Bigl(\frac{1}{y_2}\Bigr) \partial_{y_2y_2}^2\phi(0,t;y), 
\end{align*}
and therefore Hypothesis I) fails. \medskip

We show that the conclusion \eqref{240511e3_3} (and therefore also the stronger conclusion of Theorem \ref{thm: geom improve} in dimension $n=3$) fails for this phase function $\phi$. To do this, we construct a two-dimensional Kakeya set of curves associated to $\phi$. 

We restrict our discussion to the compact frequency region 
\begin{equation*}
   Y_{\circ} := \big\{
    y = (y_1, y_2) \in \B^2: y_1, \, y_2 \geq 1/2 \textrm{ and } |y| \leq 9/10 
    \big\}.
\end{equation*}
For $(t;y) \in (-1,1) \times Y_{\circ}$, consider the system of equations in $x = (x_1, x_2)$ given by
\begin{equation*}
  \partial_y\phi(x, t; y)=\omega(y) \qquad \textrm{for} \qquad  \omega(y):= \Big( 
\frac{y_1}{y_2},
\log \Big(\frac{y_1}{y_2}\Big)\Big).
\end{equation*}
Using the Taylor expansion 
\begin{equation*}
-\log (1-x)=x+\frac{x^2}{2}+\frac{x^3}{3}+\frac{x^4}{4}+\cdots,
\end{equation*}
we see that the solution can be explicitly written as
\begin{align*}
x_1(t; y)&=\frac{y_1}{y_2}-ty_1,\\
x_2(t; y)
&=\log \frac{y_1}{y_2}+\log (1-ty_2)=
\log \Big(
\frac{y_1}{y_2}-ty_1
\Big).
\end{align*}
In particular, all the $(x_1(t;y),x_2(t;y),t)$ simultaneously lie in the surface
\begin{equation*}
    M := \big\{ (u, \log u, t) : 1/100\leq u\leq 4 \textrm{ and } -1 < t < 1 \}
\end{equation*}
for all $(t; y) \in (-1,1) \times Y_{\circ}$. It is now easy to see that the conclusion \eqref{240511e3_3} fails, by testing the corresponding Kakeya maximal operator $\cK^{\delta}$ against the characteristic function of the $\delta$-neighbourhood of $M$.\medskip

The source of this anomaly appears to be the sublevel set estimates claimed in \cite[pp.364-365]{Bourgain1991} (that is, the estimate \eqref{eq: sublevel sketch} above), which also fail for the phase function in this example. In fact, by expressing the translation-invariant phase $\phi$ as in \eqref{eq: trans inv}, for the curves in our parametrisation,
\begin{equation*}
  \Phi(y, t) = \Gamma_{y,\omega(y)} (t) = (\omega(y) - \partial_y\psi(t;y), t),  
\end{equation*}
where we now consider arbitrary smooth $\omega \colon \B^2 \to \B^2$. Taking the Jacobian,
\begin{equation*}
    J\Phi(y,t) = \det
    \begin{pmatrix}
    \frac{\partial \Phi}{\partial y_1}(y,t) & \frac{\partial \Phi}{\partial y_2}(y,t) & \frac{\partial \Phi}{\partial t}(y,t)
    \end{pmatrix}
    = \frac{A(y) t^2 + B(y) t + C(y)}{1 - ty_2},
\end{equation*}
where, writing $\omega(y) = (\omega_1(y), \omega_2(y))$, we have
\begin{gather*}
  A(y) := 1 + y_2\partial_{y_2}\omega_2(y), \;\; B(y) := \partial_{y_1}\omega_1(y) + \partial_{y_2}\omega_2(y) + y_2 J\omega(y), \;\;
  C(y) := J\omega(y)
\end{gather*}
for $J\omega$ the Jacobian determinant of the map $y \to \omega(y)$. Therefore, 
\begin{equation*}
    \bigl|\{(y,t) : |J\Phi(y,t)| < \sigma \}\bigr| \sim \bigl|\{(y,t): |A(y) t^2 + B(y) t + C(y)| < \sigma \}\bigr|
\end{equation*}
cannot satisfy \eqref{eq: sublevel sketch} uniformly over all choices of smooth $\omega$, since $\omega$ can be chosen so that $A = B = C = 0$ on a set of positive measure in which case
\begin{equation*}
\bigl|\{(y,t): |A(y) t^2 + B(y) t + C(y)| < \sigma \}\bigr| \gtrsim 1.
\end{equation*}

\end{example}

In \S\ref{sec: sublevel} we develop a robust sublevel set theory, exploiting the full power of Hypothesis I) to avoid such counterexamples.

%%%%%%%%%%%%%%%%%%%%%%%%%%%%%%%%%%%%%%%%%%%%%%%%%%%%%%%%%%%%%%%%%%%%%%%%%%%%%%%%%%%%%%%%%%%%%%%%

%    Comparison with more recent results

%%%%%%%%%%%%%%%%%%%%%%%%%%%%%%%%%%%%%%%%%%%%%%%%%%%%%%%%%%%%%%%%%%%%%%%%%%%%%%%%%%%%%%%%%%%%%%%%

\subsection{Comparison with more recent results}

In this subsection, we discuss the connection between Theorem \ref{thm: pos-def} and the results in \cite{dai2023oscillatory} and \cite{guo2022dichotomy}. \medskip

We start with \cite{dai2023oscillatory}. Let us first recall the notion of \textit{contact order} introduced in \cite{dai2023oscillatory}. Let 
\begin{equation*}
    \phi(x, t; y)=\inn{x}{y}+ \psi(t; y),  \qquad x, y\in \B^2, \ t\in \B^1
\end{equation*}
be a translation-invariant phase function. Moreover, as in Theorem \ref{thm: pos-def} we will assume that $\phi$ is H\"ormander-type and positive-definite. Let $k\ge 4$ be an integer. The phase function $\phi$ is said to be of contact order $\le k$ at the origin $x=y=0$, $t=0$ if the matrix 
\begin{equation*}
    \begin{bmatrix}
        \partial_t \det \partial^2_{yy} \phi(0, 0; 0) & \partial_t  \partial^2_{y_1y_1} \phi(0, 0; 0) & \partial_t  \partial^2_{y_1 y_2} \phi(0, 0; 0) & \partial_t  \partial^2_{y_2 y_2} \phi(0, 0; 0)\\
        \partial^2_t \det \partial^2_{yy} \phi(0, 0; 0) & \partial^2_t  \partial^2_{y_1y_1} \phi(0, 0; 0) & \partial^2_t  \partial^2_{y_1 y_2} \phi(0, 0; 0) & \partial^2_t  \partial^2_{y_2 y_2} \phi(0, 0; 0)\\
        \vdots & \vdots & \vdots & \vdots\\
        \partial^k_t \det \partial^2_{yy} \phi(0, 0; 0) & \partial^k_t  \partial^2_{y_1y_1} \phi(0, 0; 0) & \partial^k_t  \partial^2_{y_1 y_2} \phi(0, 0; 0) & \partial^k_t  \partial^2_{y_2 y_2} \phi(0, 0; 0)
    \end{bmatrix}
\end{equation*}
has rank $4$. Under these assumptions, it is shown in \cite{dai2023oscillatory} that $\bH_{\infty \to q}(\phi)$ holds for all 
\begin{equation}\label{240527ebb}
    q\ge \frac{10}{3}- \frac{1}{9k-6}.
\end{equation}
It is elementary to see that the notion of contact orders is stable, in the sense that if $\phi$ is of contact order $\le k$ at the origin $x=y=0$, $t=0$, then (under a suitable definition) it is also of the same contact order at every other neighbouring point. In particular, if $\phi$ is of contact order $\le k$ at the origin $x=y=0$, $t=0$ for some  finite $k$, then the Kakeya non-compression Hypothesis I) automatically holds. In this sense, our assumption for Theorem  \ref{thm: pos-def} is a lot weaker than the finite contact order assumption in \cite{dai2023oscillatory}.  On the other hand, \cite{dai2023oscillatory} obtained a more quantitative range \eqref{240527ebb}; by contrast, Theorem \ref{thm: pos-def} only guarantees the existence of some $q< 10/3$ for which $\bH_{\infty \to q}(\phi)$ holds. \medskip

Let us briefly compare the proof techniques used in \cite{dai2023oscillatory} under the finite contact order hypothesis and the proof of Theorem \ref{thm: pos-def}. Both proofs rely crucially on certain sublevel set estimates. The fact that the notion of contact orders is stable makes the sublevel set estimates in \cite{dai2023oscillatory} rather elementary (see \cite[\S4.3]{dai2023oscillatory}). By contrast, the Kakeya non-compression Hypothesis I) is highly unstable, as pointed out in \S\ref{subsec: counterexample}. This makes the sublevel set estimates in the current paper much more sophisticated (see \S\S\ref{sec: non-concentration}-\ref{sec: sublevel} below). \medskip

Next, we discuss the relationship between Theorem \ref{thm: pos-def} and the results in \cite{guo2022dichotomy}. Indeed, \cite{guo2022dichotomy}, \cite{dai2023oscillatory} and the current paper can all be viewed as further developments on the sharp estimates for H\"ormander-type oscillatory integral operators obtained by Guth, Hickman and Iliopoulou \cite{GHI2019}. More precisely, for $\phi$ a positive-definite H\"ormander-type phase function, it is shown in \cite{GHI2019} that $\bH_{p\to p}(\phi)$ holds whenever
\begin{equation}\label{240824e3_3}
p \geqslant \begin{cases}2 \cdot \frac{3 n+1}{3 n-3} & \text { if } n \text { is odd, } \\ 2 \cdot \frac{3 n+2}{3 n-2} & \text { if } n \text { is even, }\end{cases}
\end{equation}
 In particular, the exponents in \eqref{240824e3_3} are sharp in the sense that for every $p$ that is strictly smaller than the right hand side of \eqref{240824e3_3}, we can always find an elliptic H\"ormander-type phase function $\phi$ such that $\bH_{p\to p}(\phi)$ fails. 

After the result in \cite{GHI2019}, it is a very natural question to ask whether there are special classes of H\"ormander-type phase functions for which we can prove $\bH_{p\to p}(\phi)$ for a range of $p$ that goes beyond the ranges in \eqref{240824e3_3}. Guo, Wang and Zhang \cite{guo2022dichotomy}, by extending the work of Bourgain \cite{Bourgain1991}, considered what they called \textit{Bourgain's condition}, and conjectured that if an elliptic H\"ormander-type phase function $\phi$ satisfies Bourgain's condition, then $\bH_{p\to p}(\phi)$ holds for all 
\begin{equation}\label{240824e3_5}
    p> \frac{2n}{n-1},
\end{equation}
the full range of the Fourier restriction conjecture. Moreover, they also proved $\bH_{p\to p}(\phi)$ for a range of $p$ that goes beyond the ranges in \eqref{240824e3_3}. 

Many natural phase functions (arising, for instance, in the study of the Bochner--Riesz conjecture) satisfy Bourgain's condition; however, it is an extremely strong hypothesis. For instance, if $\phi$ is translation-invariant, then the only way it can satisfy Bourgain's condition is if the phase is in fact linear in the $(x,t)$ variables (that is, we are in the constant coefficient case). One goal of \cite{dai2023oscillatory} and the current paper is to consider H\"ormander-type phase functions that do not satisfy Bourgain's condition,\footnote{For these phase functions, it is proven in \cite{guo2022dichotomy} that $\bH_{p\to p}(\phi)$ fails for some $p$ satisfying \eqref{240824e3_5}.} and to find certain weak curvature conditions that still allow us to prove something better than the results of Guth, Hickman and Iliopoulou \cite{GHI2019}. These are the condition of finite contact orders discussed above, and Kakeya non-compression condition from Definition \ref{dfn: n=3 lin comb K}.

%%%%%%%%%%%%%%%%%%%%%%%%%%%%%%%%%%%%%%%%%%%%%%%%%%%%%%%%%%%%%%%%%%%%%%%%%%%%%%%%%%%%%%%%%%%%%%%%

%    An application to Pierce--Yung operators

%%%%%%%%%%%%%%%%%%%%%%%%%%%%%%%%%%%%%%%%%%%%%%%%%%%%%%%%%%%%%%%%%%%%%%%%%%%%%%%%%%%%%%%%%%%%%%%%

\subsection{An application to Pierce--Yung operators} A significant application of our results is to the theory of \textit{Carleson operator of Radon-type}. The recent paper \cite{BGH_PY} establishes the following theorem.

\begin{theorem}[\cite{BGH_PY}]\label{thm: PY}
For $x, y\in \R$, consider the operator
\begin{equation*}
    \cC f(x,y) := \sup_{v \in \R} \Big|\mathrm{p.v.} \int_\R f(x-t, y-t^2) e^{i v t^3} \frac{\ud t}{t} \Big|,
\end{equation*}
initially defined for Schwartz functions $f \in \cS(\R^2)$. 
We have 
\begin{equation*}
\|\cC f\|_{L^p(\R^2)} \lesssim_{p} \|f\|_{L^p(\R^2)}
\end{equation*}
for all $1 < p < \infty$. 
\end{theorem}

We refer to \cite{PY2019} for the motivation for the definition of the operator $\cC$, where a version of Theorem~\ref{thm: PY} in dimensions $n \geq 3$ is established. Theorem~\ref{thm: PY} is the first \textit{bona fide} bound for an operator of this form for $n = 2$ and answers a question of Pierce and Yung; we refer to \cite{BGH_PY} for further details of the history of this problem and why the analysis of the $n = 2$ case differs from that of the $n \geq 3$ case. 

The key insight in the proof of Theorem~\ref{thm: PY} is that (suitably localised pieces of) the operator $\cC$ can be related to variable Schr\"odinger propagators $U^{\lambda}[\phi_{\pm}; a_{\pm}]$ satisfying the Nikdoym non-compression hypothesis from Definition~\ref{dfn: n=3 lin comb N}. Here the relevant phase function $\phi_{\pm}$ is in fact translation-invariant and is given by $\phi_{\pm}(x, t; y) := x \cdot y + \psi_{\pm}(t;y)$ where
\begin{equation*}
   \psi_{\pm}(t; y) := 
    -\frac{(y_2\pm\sqrt{y_2^2+ 3 t y_1})(
    6 t y_1+ y_2(y_2\pm\sqrt{y_2^2+ 3 t y_1}))}{27 t^2}.
\end{equation*}
We remark that this phase has signature $0$ in the sense of \cite{HI2022}, which prohibits the use of many established tools in the oscillatory integral operator literature. However, provided the amplitude $a_{\pm}$ is suitably localised, Theorem~\ref{thm: n=3 osc improve} II) can be applied to obtain a slight gain over the universal local smoothing estimates for $U^{\lambda}[\phi_{\pm}; a_{\pm}]$. This slight gain turns out to be crucial for the application, ensuring summability of certain estimates arising from frequency decomposition of $\cC f$. 

%%%%%%%%%%%%%%%%%%%%%%%%%%%%%%%%%%%%%%%%%%%%%%%%%%%%%%%%%%%%%%%%%%%%%%%%%%%%%%%%%%%%%%%%%%%%%%%%

%    Sharpness of the universal estimate

%%%%%%%%%%%%%%%%%%%%%%%%%%%%%%%%%%%%%%%%%%%%%%%%%%%%%%%%%%%%%%%%%%%%%%%%%%%%%%%%%%%%%%%%%%%%%%%%

\subsection{Sharpness of the universal estimate}\label{subsec: sharpness universal} As a final example, we exhibit operators $U_{\dagger}^{\lambda}$ and $\cN_{\dagger}^{\lambda}$ for which the universal estimates in Theorem~\ref{thm: uni osc} and Proposition~\ref{prop: univ Kak/Nik} are sharp. In particular, we describe a counterpart to Bourgain's construction in Theorem~\ref{thm: univ sharp} in the propagator / Nikodym setting.

The following example is due to Minxing Shen, whom we thank for allowing us to present it here. Let $n \geq 3$ be odd and define a phase function 
\begin{equation*}
\phi_{\dagger} \colon \D^n \to \R, \qquad \phi_{\dagger}(x,t; y) := \inn{x}{y} +\frac{1}{2} \inn{\bB(t)y}{y} + \frac{1}{2} \inn{\bC y}{y},
\end{equation*}
where
\begin{equation*}
\bB(t) := \underbrace{B(t) \oplus \cdots \oplus B(t)}_{\frac{n-1}{2}-\textrm{fold}} \qquad \textrm{for} \qquad B(t):=
\begin{pmatrix}
-t^2 &t
\\
t &0
\end{pmatrix}
\end{equation*}
and
\begin{equation*}
\bC := \underbrace{C \oplus \cdots \oplus C}_{\frac{n-1}{2}-\textrm{fold}} \qquad \textrm{for} \qquad C:=
\begin{pmatrix}
1 & 0
\\
0 & 0
\end{pmatrix}.
\end{equation*}
It is easy to see that $\phi_{\dagger}$ is a H\"ormander-type phase. Indeed, $\phi_{\dagger}$ is a translation-invariant phase and so H1) is immediate. On the other hand, the remaining condition H2) directly follows from the fact that $\det B'(t) = -1$ for all $t \in \R$.

Fixing $a \in C^{\infty}_c(\D^n)$ satisfying $a(0,0;0) = 1$, we let \begin{equation*}
    U_{\dagger}^{\lambda} := U^{\lambda}[\phi_{\dagger}; a] \qquad \textrm{and} \qquad  \cN^{\delta}_{\dagger} := \cN^{\delta}[\phi_{\dagger}] 
\end{equation*}
for all $\lambda \geq 1$. With these definitions, the variable Schr\"odinger propagator $U_{\dagger}^{\lambda}$ and corresponding maximal function $\cN^{\delta}_{\dagger}$ exhibit the following unfavourable behaviour.

\begin{theorem}\label{thm: sharp nik} For all $n \geq 3$ odd, the following hold:
\begin{enumerate}[i)]
    \item For all $2 \leq q \leq q(n)$, we have 
\begin{equation*}
\|U_{\dagger}^{\lambda}\|_{L^q(\R^{n-1}) \to L^q(\R^n)} \gtrsim \lambda^{\alpha_{\mathrm{LS}}(n;q)} \qquad \textrm{for all $\lambda \geq 1$.}
\end{equation*}
In particular, $\bLS_q(\phi; \alpha)$ \textbf{fails} for all $\alpha < \alpha_{\mathrm{LS}}(n;q)$.
    \item For all $p(n) \leq p \leq \infty$, we have
\begin{equation*}
    \|\cN^{\delta}_{\dagger}\|_{L^p(\R^n) \to L^1(\B^{n-1})} \gtrsim \delta^{-\beta(n; p)} \qquad \textrm{for all $0 < \delta < 1$.}
\end{equation*}
Thus, $\bN_{p \to 1}(\phi_{\dagger}, \beta)$ \textbf{fails} to hold for any $\beta < \beta(n; p)$ whenever $p(n) < p \leq \infty$.
\end{enumerate}
\end{theorem}

Here $\alpha_{\mathrm{LS}}(n;q)$ and $\beta(n; p)$ are the universal exponents defined in \eqref{eq: LS alpha def} and \eqref{eq: bush exp}, respectively. We recall that for the range $1 \leq p \leq p(n)$ the exponent $\beta(n,p)$ is automatically sharp for \textit{all} H\"ormander-type phases $\phi$ by testing the maximal inequality against a $\delta$-ball. 

\begin{proof}[Proof (of Theorem~\ref{thm: sharp nik})] It is easy to see that the curves $\Gamma_{y, \omega} := \{(\gamma_{y, \omega}(t), t) : t \in [-1,1]\}$ associated to the phase are parametrised by
\begin{equation*}
    \gamma_{y, \omega}(t) := \omega - \big(\bB(t) + \bC)y \qquad \textrm{for all $\omega$, $y \in \B^{n-1}$ and $t \in [-1,1]$.}
\end{equation*}
In particular, if we set $y(\omega) := \omega$ for $\omega \in \B^{n-1}$, then $\gamma_{y(\omega), \omega}(t) = \bM(t) \omega$, where
\begin{equation*}
     \bM(t) := \bI - \bB(t) - \bC = \underbrace{M(t) \oplus \cdots \oplus M(t)}_{\frac{n-1}{2}-\textrm{fold}} \qquad \textrm{for} \qquad M(t):=
\begin{pmatrix}
t^2 &-t
\\
-t &1
\end{pmatrix},
\end{equation*}
for $\bI$ the $(n-1) \times (n-1)$ identity matrix. Each matrix $M(t)$ is rank $1$ and diagonalisable with a basis of eigenvectors given by 
\begin{equation*}
    v(t) := \frac{1}{(1+ t^2)^{1/2}}
    \begin{bmatrix}
        1 \\
        t
    \end{bmatrix}
    \qquad \textrm{and} \qquad
w(t) := \frac{1}{(1+ t^2)^{1/2}}
    \begin{bmatrix}
       -t \\
        1
    \end{bmatrix}
\end{equation*}
corresponding to eigenvalues $0$ and $1 + t^2$, respectively. In particular, for fixed $t \in \R$, the image of $M(t)$ is equal to the span of $w(t)$. Thus, if we define
\begin{equation*}
S:=\left\{(x,t) \in [-2,2]^{n-1} \times [-1,1]:\;\begin{bmatrix}
    x_{2i-1} \\
    x_{2i}
\end{bmatrix}
\in \mathrm{span}\,w(t) \text{ for all }i=1,\ldots,\frac{n-1}{2}\right\},
\end{equation*}
then $\Gamma_{y(\omega), \omega} \subseteq S$ for all $\omega \in \B^{n-1}$. This tells us that $S$ is a Nikodym set for the curves $\Gamma_{y, \omega}$ associated to the phase: for every position $\omega \in \B^{n-1}$, there is a choice of direction $y(\omega)$ such that $\Gamma_{y(\omega), \omega} \subseteq S$. Furthermore, $S$ is a (compact piece of a) $\frac{n+1}{2}$-dimensional surface in $\R^n$, and therefore this shows that Nikodym sets of curves for $\phi_{\dagger}$ can have dimension as low as $\frac{n+1}{2}$. 
 
The above observation easily implies part ii) of the theorem. In particular, for $\delta > 0$ let $N_{\delta} S$ denote the $\delta$-neighbourhood of $S$. It then follows that 
\begin{equation*}
    \|\cN_{\dagger}^{\delta}\chi_{N_{\delta} S}\|_{L^1(\B^{n-1})} \gtrsim 1 \quad \textrm{and} \quad \|\chi_{N_{\delta} S}\|_{L^p(\R^n)} \sim \delta^{(n-1)/(2p)} = \delta^{\beta(n, p)}
\end{equation*}
for $p(n) \leq p \leq \infty$, yielding the desired operator norm bound.

Finally, we apply the above observations to study the propagator $U_{\dagger}^{\lambda}$. Let $\psi \in C^{\infty}_c(\B^{n-1})$ be a smooth cutoff function and for $\lambda \geq 1$ define a Schwartz function by $\hat{f}(y) := e^{-i\lambda |y|^2/2} \psi(y)$. 

Observe that
\begin{equation}\label{eq: sharp univ LS 1}
    U_{\dagger}^{\lambda}f(x,t) = \int_{\R^{n-1}} e^{i \lambda(\inn{x/\lambda}{y} - \frac{1}{2}\inn{\bM(t/\lambda)y}{y})}  a^{\lambda}(x,t;y)\psi(y)\,\ud y.
\end{equation}
By applying the rotation 
\begin{equation*}
   \bO(t/\lambda) = \underbrace{O(t/\lambda) \oplus \cdots \oplus O(t/\lambda)}_{\frac{n-1}{2}-\textrm{fold}} \in \mathrm{SO}(n-1) \quad \textrm{for} \quad O(t/\lambda) :=
\begin{pmatrix}
v(t/\lambda) & w(t/\lambda)
\end{pmatrix} \in \mathrm{SO}(2),
\end{equation*}
we can diagonalise the form $\inn{\bM(t/\lambda)y}{y}$; that is, 
\begin{equation*}
   \inn{\bM(t/\lambda)\bO(t/\lambda)y}{\bO(t/\lambda)y} = (1+t^2/\lambda^2) \sum_{j = 1}^{\frac{n-1}{2}} y_{2j}^2.
\end{equation*}
Furthermore, if $(x, t) \in S^{\lambda} := \{\lambda \by : \by \in S\}$, then there exist $r_1, \dots,r_{(n-1)/2} \in \R$ with $r_j = O(1)$ such that
\begin{equation*}
    \frac{1}{\lambda}\begin{bmatrix}
        x_{2j-1} \\
        x_{2j}
    \end{bmatrix}
    = 
    r_j w(t/\lambda) \qquad \textrm{for $1 \leq j \leq \frac{n-1}{2}$}.
\end{equation*}
In this case, the phase in \eqref{eq: sharp univ LS 1} satisfies
\begin{equation*}
    \inn{x/\lambda}{\bO(t/\lambda) y} - \frac{1}{2}\inn{\bM(t/\lambda)\bO(t/\lambda) y}{\bO(t/\lambda) y} =  \sum_{j = 1}^{\frac{n-1}{2}} r_j y_{2j} - \frac{1 + t^2/\lambda^2}{2}\sum_{j = 1}^{\frac{n-1}{2}} y_{2j}^2. 
\end{equation*}
Note that the above expression is independent of $y_1, y_3, \dots, y_{n-2}$. Thus, provided the support of $a$ is chosen sufficiently small, after applying a change of variables and the asymptotic expansion for oscillatory integrals with nondegenerate phase with respect to the variables $y_2$, $y_4, \cdots, y_{n-1}$  (see, for instance, \cite[Chapter VIII, Proposition 6, p.344]{Stein1993}), we deduce that
\begin{equation*}
    |U_{\dagger}^{\lambda}f(\bx)| \gtrsim \lambda^{-(n-1)/4} \qquad \textrm{for all $\bx \in S^{\lambda}$.}
\end{equation*}
On the other hand, a similar argument shows that
\begin{equation*}
    |\partial_{\bx} U_{\dagger}^{\lambda}f(\bx)| \lesssim \lambda^{-(n-1)/4} \qquad \textrm{for all $\bx \in \R^n$.}
\end{equation*}
Thus, we can conclude that there exists some constant $c > 0$, independent of $\lambda$, such that
\begin{equation*}
    |U_{\dagger}^{\lambda}f(\bx)| \gtrsim \lambda^{-(n-1)/4} \qquad \textrm{for all $\bx \in N_c S^{\lambda}$.}
\end{equation*}
Consequently,
\begin{equation}\label{eq: sharp univ LS 2}
    \|U_{\dagger}^{\lambda}f\|_{L^q(\R^n)} \gtrsim \lambda^{-(n-1)/4} | N_c S^{\lambda}|^{1/q} \gtrsim \lambda^{-(n-1)/4} \lambda^{(n+1)/(2q)}. 
\end{equation}

On the other hand, by non-stationary phase (integration-by-parts), we have $|f(x)| \lesssim_N (1 + |x|)^{-N}$ for $|x| \geq 2\lambda$ and all $N \in \N$. Furthermore, a simple application of van der Corput's lemma shows that $|f(x)| \lesssim \lambda^{-(n-1)/2}$ for $|x| \leq 2 \lambda$. Combining these observations,
\begin{equation}\label{eq: sharp univ LS 3}
    \|f\|_{L^q(\R^{n-1})} \lesssim \lambda^{-(n-1)(1/2 - 1/q)} \qquad \textrm{for $1 \leq q \leq \infty$.}
\end{equation}

The desired norm bounds for $U_{\dagger}$ now follow by comparing \eqref{eq: sharp univ LS 2} and \eqref{eq: sharp univ LS 3}.
\end{proof}

%%%%%%%%%%%%%%%%%%%%%%%%%%%%%%%%%%%%%%%%%%%%%%%%%%%%%%%%%%%%%%%%%%%%%%%%%%%%%%%%%%%%%%%%%%%%%%%%

%   Non-concentration

%%%%%%%%%%%%%%%%%%%%%%%%%%%%%%%%%%%%%%%%%%%%%%%%%%%%%%%%%%%%%%%%%%%%%%%%%%%%%%%%%%%%%%%%%%%%%%%%

\section{Non-concentration}\label{sec: non-concentration}

%%%%%%%%%%%%%%%%%%%%%%%%%%%%%%%%%%%%%%%%%%%%%%%%%%%%%%%%%%%%%%%%%%%%%%%%%%%%%%%%%%%%%%%%%%%%%%%%

%   Sublevel set hypotheses

%%%%%%%%%%%%%%%%%%%%%%%%%%%%%%%%%%%%%%%%%%%%%%%%%%%%%%%%%%%%%%%%%%%%%%%%%%%%%%%%%%%%%%%%%%%%%%%%

Before we begin, we introduce some notation used throughout this section. For $\phi \colon \D^{n-1} \to \R$ a non-degenerate phase, we let $\Psi$ be as defined in \eqref{eq: Psi} and fix compact balls $\Omega_{\phi}$, $Y_{\phi}$ and $I_{\phi}$ as in \S\ref{subsec: geometric max}. Given $\fp \colon \B^{n-1} \to \Omega_{\phi}$ and  $\fd \colon \B^{n-1} \to Y_{\phi}$ measurable, we define
\begin{equation}\label{eq: Upsilon}
   \Upsilon_{\fp}(t;y):=(\Psi(\fp(y);t;y), t ;y) \quad \textrm{and} \quad \Upsilon^{\fd}(\omega;t) := (\Psi(\omega;t;\fd(\omega));t;\fd(\omega)).
\end{equation}
With this notation in place, and recalling the $\cP(d, m)$ and $[M]_{\alpha, \beta}$ notation from \S\ref{subsec: beyond higher}, our key hypothesis reads as follows. 

\begin{definition}\label{dfn: hyp K N} Let $\phi \colon \D^{n-1} \to \R$ be a non-degenerate phase and $0 < \kappa < 1$.

\begin{enumerate}[I)]
    \item We let $\bS_{\tK}(\phi; m; \kappa)$ denote the statement:\medskip 

\noindent There exists a constant $C=C(\phi)$ such that the following holds. For all measurable $\fP \colon \B^{n-1} \to \mathrm{Mat}(\R, n-1)$ and $\fp \colon \B^{n-1} \to \Omega_{\phi}$, there exist $\alpha$, $\beta \in \cP(n-1,n-m)$ such that
\begin{equation*}
    \Big|\Big\{ (t;y) \in I_{\phi} \times Y_{\phi} : \Big|\det\Big[ (\partial_{yy}^2 \phi)(\Upsilon_{\fp}(t;y)) - \fP(y)\Big]_{\alpha, \beta}\Big| < \sigma \Big\}\Big| \leq C  \sigma^{\kappa}
\end{equation*}
uniformly over all $\sigma > 0$.\medskip

\item We let $\bS_{\tN}(\phi; m; \kappa)$ denote the statement: \medskip

\noindent There exists a constant $C=C(\phi)$ such that the following holds. For all measurable $\fD \colon \B^{n-1} \to \mathrm{Mat}(\R, n-1)$ and $\fd \colon \B^{n-1} \to Y_{\phi}$, there exist $\alpha$, $\beta \in \cP(n-1,n-m)$ such that
\begin{equation*}
    \sup_{\omega \in \Omega_{\phi}} \Big|\Big\{ t \in I_{\phi} : \Big|{\det}\Big[ I_{n-1}-
 (\partial_{yy}^2 \phi)(\Upsilon^{\fd}(\omega;t))\fD(\omega)\Big]_{\alpha, \beta} \Big| < \sigma \Big\}\Big| \leq C\sigma^{\kappa}
\end{equation*}
uniformly over all $\sigma > 0$. Here $I_{n-1}$ denotes the $(n-1)\times (n-1)$ identity matrix. \medskip
\end{enumerate} 
\end{definition}

For $n \geq 3$, we work with these conditions at the \textit{critical codimension} given by
\begin{equation}\label{eq: crit codim}
    m_{\crit}(n) := n - d_{\crit}(n) =
    \begin{cases}
        \frac{n-1}{2} & \textrm{if $n$ is odd,} \\
        \frac{n-2}{2} & \textrm{if $n$ is even,}
    \end{cases}
\end{equation}
where $d_{\crit}(n)$ is the critical dimension introduced in \S\ref{subsec: beyond higher}.

\begin{proposition}\label{prop: hyp K N} Let $n \geq 3$ and $\phi \colon \D^n \to \R$ be a real analytic, non-degenerate phase. 
\begin{enumerate}[I)]
    \item Suppose $\phi$ is translation-invariant and satisfies Hypothesis I). Then there exists some $0 < \kappa \leq 1$ such that $\bS_{\tK}(\phi; m_{\crit}(n); \kappa)$ holds; 
    \item Suppose $\phi$ satisfies Hypothesis II). Then there exists some $0 < \kappa \leq 1$ such that $\bS_{\tN}(\phi;m_{\crit}(n); \kappa)$ holds.
\end{enumerate}
\end{proposition}

The proof of Proposition~\ref{prop: hyp K N} is involved, and is deferred until \S\ref{sec: sublevel}. The argument relies on B\^ocher's theorem \cite{Bocher1900} on linear independence in the $C^{\omega}$ class, together with basic van der Corput sublevel set estimates.

%%%%%%%%%%%%%%%%%%%%%%%%%%%%%%%%%%%%%%%%%%%%%%%%%%%%%%%%%%%%%%%%%%%%%%%%%%%%%%%%%%%%%%%%%%%%%%%%

%   Non-concentration for grains

%%%%%%%%%%%%%%%%%%%%%%%%%%%%%%%%%%%%%%%%%%%%%%%%%%%%%%%%%%%%%%%%%%%%%%%%%%%%%%%%%%%%%%%%%%%%%%%%

\subsection{Non-concentration for grains}\label{subsec: non concentration grains}

A key ingredient in the proof of Theorem~\ref{thm: geom improve} is a bound on how families of tubes can concentrate in neighbourhoods of algebraic sets. To describe the setup, we first introduce some preliminary definitions. 

\begin{definition} Let $\T$ be a family of $\delta$-tubes associated to a non-degenerate phase $\phi$, as defined in \eqref{eq: delta tube}. Given $T = T_{y,\omega}^{\delta} \in \T$ for $(y, \omega) \in Y_{\phi} \times \Omega_{\phi}$, we let $\omega(T) := \omega$ and $y(T) := y$ denote the centre and direction of $T$, respectively. We say:
\begin{enumerate}[I)]
    \item $\T$ is \textit{direction-separated} if $|y(T_1) - y(T_2)| > \delta$ for all $T_1$, $T_2 \in \T$, $T_1 \neq T_2$;
    \item $\T$ is \textit{centre-separated} if $|\omega(T_1) - \omega(T_2)| > \delta$ for all $T_1$, $T_2 \in \T$, $T_1 \neq T_2$.
\end{enumerate}
\end{definition}

Our goal is to show that if $\T$ is a direction- or centre-separated family of $\delta$-tubes, then the tubes in $\T$ cannot cluster close to an algebraic surface. This property does not hold for general H\"ormander-type phases, but it does hold once we impose additional hypotheses such as those in the statement of Theorem~\ref{thm: osc improve}.

Given a finite family of non-zero polynomials $P_1, \dots, P_m \in \R[X_1, \dots, X_n]$, we let $Z(P_1, \dots, P_m)$ denote their common zero set
\begin{equation}\label{eq: variety}
    Z(P_1, \dots, P_m) := \bigcap_{j=1}^m\{x \in \R^n : P_j(x) = 0 \}.
\end{equation}
We say a set of the form \eqref{eq: variety} is a \textit{transverse complete intersection} if the vectors $\nabla P_1(x), \dots, \nabla P_m(x)$ are linearly independent for all $x \in Z(P_1, \dots, P_m)$. In this case, $Z(P_1, \dots, P_m)$ is a smooth submanifold of $\R^n$ of codimension $m$. We are interested in non-concentration properties with respect to neighbourhoods of such sets.

\begin{definition} Given $0 < \delta \leq \rho < 1$, $0 \leq m \leq n$ and $d \in \N$ we say a set $G \subseteq \R^n$ is a \textit{$(\delta, \rho, d)$-grain of codimension $m$} if 
\begin{equation*}
    G = N_{\delta}Z(P_1, \dots, P_m) \cap B(x_G, \rho)
\end{equation*}
for non-zero polynomials $P_1, \dots, P_m \in \R[X_1, \dots, X_n]$ of degree at most $d$ such that $Z(P_1, \dots, P_m)$ is a codimension $m$ transverse complete intersection and $x_G \in \R^n$. Here $N_{\delta}Z(P_1, \dots, P_m)$ denotes the (Euclidean) $\delta$-neighbourhood of $Z(P_1, \dots, P_m)$.
\end{definition}

We remark that, in the above definition, a $(\delta, \rho, d)$-grain of codimension $0$ is simply a $\rho$-ball in $\R^n$, so that the $\delta$ and $d$ parameters are superfluous in this special case. 

\begin{definition}\label{dfn: non-concentrated} Let $\phi \colon \D^n \to \R$ be a non-degenerate phase, $0 < \delta < 1$, $0 \leq m \leq n$ and $\kappa > 0$, $B = (B_1, B_2) \in [0, \infty)^2$. We say a family of $\delta$-tubes $\T$ associated to $\phi$ is \textit{$(\kappa, B)$-non-concentrated in grains of codimension $m$} if the following holds.\medskip

\noindent We have $\#\T \leq \delta^{-(n-1)}$ and, for all $d \in \N$ and $\varepsilon > 0$, the inequality
    \begin{equation}\label{eq: non-concentration}
        \#\big\{ T \in \T : |T \cap G| \geq \lambda |T| \big\} \lesssim_{\phi, d, \varepsilon} (\lambda/\rho)^{-B_1} \rho^{-B_2} \delta^{-(n-1) + \kappa - \varepsilon}
    \end{equation}
    holds whenever $G \subseteq \R^n$ is a $(\delta, \rho, d)$-grain of codimension $m$ and $\delta \leq \lambda \leq \rho \leq 1$.
\end{definition}

The following non-concentration inequality is a variant of the \textit{Polynomial Wolff Axiom} theorem of Katz--Rogers \cite{KR2018}. 

\begin{theorem}[Non-concentration for grains]\label{thm: non-concentration} Given $n \geq 2$, $1 \leq m \leq n-1$ and $0 < \kappa < 1$, there exist $ 0 < \kappa_{\circ} < 1$ and $B \in [0, \infty)^2$ such that the following holds whenever $\phi \colon \D^n \to \R$ is a non-degenerate phase and $0 < \delta < 1$. 
\begin{enumerate}[I)]
    \item If $S_{\tK}(\phi; m; \kappa)$ holds and $\T$ is a direction-separated family of $\delta$-tubes associated to $\phi$, then $\T$ is $(\kappa_{\circ}, B)$-non-concentrated in grains of codimension $m$. 
\item  If $S_{\tN}(\phi; m; \kappa)$ holds and $\T$ is a centre-separated family of $\delta$-tubes associated to $\phi$, then $\T$ is $(\kappa_{\circ}, B)$-non-concentrated in grains of codimension $m$.
\end{enumerate}
\end{theorem}

By combining Theorem~\ref{thm: non-concentration} with Proposition~\ref{prop: hyp K N}, we see that for $\phi \colon \D^n \to \R$ a non-degenerate phase satisfying the hypotheses of Theorem~\ref{thm: osc improve}, the associated tubes $\T$ satisfy non-concentration conditions with respect to gains at the critical codimension $m_{\crit}(n)$. This property is key to guaranteeing the improved oscillatory integral estimates. We remark that Theorem~\ref{thm: non-concentration} does not require translation invariance of the phase.

\begin{remark} Let $n \geq 2$ and $0 < \delta < 1$. We now specialise to the case the tubes in $\T$ are formed by straight lines, corresponding to the classical Kakeya and Nikodym maximal functions. Suppose $\T$ is either a direction- or centre-separated family of straight line $\delta$-tubes (associated, for instance, to the constant coefficient phase $\phi_{\mathrm{const}}$) and let $0 \leq m < n$. Note that we automatically have a trivial bound $\# \T \lesssim \delta^{-(n-1)}$ on the number of tubes in our family. Combining the polynomial Wolff axiom theorem of Katz--Rogers~\cite[Theorem 1.1]{KR2018} with Wongkew's theorem~\cite{Wongkew1993}, it follows that $\T$ is $(m, (n, m))$-non-concentrated in grains of codimension $m$. In other words,
\begin{equation*}
    \#\{T \in \T : |T \cap G| \geq \lambda |T| \} \lesssim_{d, \varepsilon} (\lambda/\rho)^{-n} \rho^{-m}\delta^{-(n-1) + m - \varepsilon}
\end{equation*}
whenever $G \subseteq \R^n$ is a $(\delta, \rho, d)$-grain of codimension $m$ and $\delta \leq \lambda \leq \rho \leq 1$. The most relevant situation is when $\lambda = \rho$; in this case, unless we are in a degenerate situation where $\rho$ is very close to $\delta$, the above estimate always beats the trivial bound. In general, however, the exponent $\kappa_{\circ}$ resulting from Theorem~\ref{thm: non-concentration} may be very small, whilst the components of $B$ may be large. In such cases, the non-concentration inequality \eqref{eq: non-concentration} is only effective for values of $\rho$ close to 1.
\end{remark}

The proof of Theorem~\ref{thm: non-concentration} uses the following `algebraic' variant of the inverse function theorem. 

\begin{lemma}[{Effective inverse function theorem \cite[Lemma 3.2]{CRW2003}}]\label{lem: effective inverse} For $1 \leq j \leq n$, let $P_j \in \R[X_1, \dots, X_n]$  and consider the mapping $P = (P_1, \dots, P_n) \colon \R^n \to \R^n$. Suppose that the Jacobian determinant $\det \partial_x P$ is not identically zero. Then there exists some $1 \leq M \leq \prod_{j=1}^n \deg P_j$ and open subsets $U_1, \dots, U_M$ of $\R^n$ such that:
\begin{enumerate}[i)]
    \item The $U_m$ are pairwise disjoint and cover $\R^n$ except for a set of measure zero;
    \item The restriction of $P$ to $U_m$ is a diffeomorphism onto its image for $1 \leq m \leq M$.
    \end{enumerate}
\end{lemma}

The advantage of Lemma~\ref{lem: effective inverse} is that it provides control over the number of open sets $M$, which does not follow from the classical inverse function theorem.

\begin{proof}[Proof (of Theorem~\ref{thm: non-concentration})] The proof follows the scheme introduced in \cite{KR2018}. We break the argument into several steps. Fix a $(\delta,\rho,d)$-grain $G$. Fix $\varepsilon > 0$ and let $\T$ be a collection of $\delta$-tubes associated to a non-degenerate phase $\phi$. We want to show the non-concentration of $\T$ in $G$, under the hypotheses of Theorem~\ref{thm: non-concentration}. By pigeonholing, we assume without loss of generality that there exists some $(\overline{y}, \overline{\omega}) \in Y_{\phi} \times \Omega_{\phi}$ such that each $T \in \T$ is of the form $T = T_{y, \omega}^{\delta}$ where $(y, \omega) \in Y_{\circ} \times \Omega_{\circ}$ for 
\begin{equation*}
\Omega_{\circ} := \big\{ \omega \in \Omega_{\phi} : |\omega - \overline{\omega}| \leq \delta^{\varepsilon/(6n)} \big\} \quad \textrm{and} \quad Y_{\circ} := \big\{y \in Y_{\phi} : |y - \overline{y}| \leq \delta^{\varepsilon/(6n)}\big\}.
\end{equation*}
 Under this hypothesis, it suffices to prove \eqref{eq: non-concentration} with $\varepsilon$ replaced with $2\varepsilon/3$. In the first steps, we shall not impose any further assumptions on $\T$ or $\phi$, so as to treat parts I) and II) of the theorem simultaneously. Later we shall assume $\T$ is either direction-separated or centre-separated, depending on which part of Theorem~\ref{thm: non-concentration} we wish to target.\medskip

%%%%%%%%%%%%%%%%%%%%%%%%%%%%%%%%%%%%%%%%%%%%%%%%%%%%%%%%%%%%%%%%%%%%%%%%%%%%%%%%%%%%%%%%%%%%%%%%

%                      Step 1: Approximation by a polynomial phase

%%%%%%%%%%%%%%%%%%%%%%%%%%%%%%%%%%%%%%%%%%%%%%%%%%%%%%%%%%%%%%%%%%%%%%%%%%%%%%%%%%%%%%%%%%%%%%%%

\noindent \underline{Step 1: Approximation by polynomial curves.} Given a $\delta$-tube $T_{y, \omega}^{\delta}$ associated to $\phi$, any interval $I \subseteq I_{\phi}$ and $\eta > 0$, we define the truncated $\eta$-tube
\begin{equation*}
T_{y, \omega}[\eta, I]  := \big\{ (x,t) \in \B^{n-1} \times I : |x - \gamma_{y, \omega}(t)| < \eta \big\}.
\end{equation*}
We also write $T[\eta, I] := T_{y, \omega}[\eta, I]$ if $T = T_{y, \omega}^{\delta}$, and we denote
\begin{equation*}
    \T[\eta, I] := \big\{ T[\eta, I]  : T \in \T \big\}.
\end{equation*}

Let $\cI$ be a finitely overlapping collection of open intervals of length $\delta^{\varepsilon/(10n)}$ which cover $I_{\phi}$, so that 
\begin{equation*}
    T = \bigcup_{I \in \cI} T[\delta, I] \qquad \textrm{for all $T \in \T$.}
\end{equation*}
    If $T \in \T$ satisfies $|T \cap G| \geq \lambda |T|$, then by pigeonholing there must exist some $I \in \cI$ such that $|T[\delta, I] \cap G| \geq \lambda [\#\cI]^{-1}|T|$. Writing $\tilde{\lambda} := \lambda [\#\cI]^{-1}$, noting that $\#\cI \sim \delta^{-\varepsilon/(10n)}$, we therefore have
\begin{equation}\label{eq: poly approx 1}
   \#\{T \in \T : |T \cap G| \geq \lambda |T| \} \lesssim \delta^{-\varepsilon/(10n)} \max_{I \in \cI} \#\{T \in \T : |T[\delta, I] \cap G| \geq \tilde{\lambda} |T| \}.
\end{equation}
Let $\bar{I} \in \cI$ be a fixed interval which realises the maximum on the right-hand side of \eqref{eq: poly approx 1}. Thus, replacing $\lambda$ with $\tilde{\lambda}$, it suffices to show that there exist $\kappa_{\circ} > 0$ and $B = (B_1, B_2) \in [0,\infty)^2$ with $B_1 \leq n$, say, such that
\begin{equation*}
    \#\{T \in \T : |T[\delta, \bar{I}\,] \cap G| \geq \lambda |T| \} \lesssim_{\phi, d, \varepsilon} (\lambda/\rho)^{-B_1} \rho^{-B_2} \delta^{-(n-1) + \kappa_{\circ} - \varepsilon/2}.
\end{equation*}

Let $\overline{t}$ denote the centre of our fixed interval $\bar{I}$ and $\widetilde{\Psi}$ the degree $K = K(\varepsilon) := \ceil{1000n^3/\varepsilon}$ Taylor polynomial of $\Psi$ centred at $(\overline{\omega}; \overline{t}; \overline{y})$, where $\overline{\omega}$ and $\overline{y} \in \B^{n-1}$ are the fixed points from the beginning of the argument. For each $T \in \T$ we have a $C^{\infty}$ map $\gamma_T \colon I_\phi \to \R^{n-1}$ which parametrises the core curve of $T$. In particular, if $T = T_{y,\omega}^{\delta}$, then $\gamma_T(t) := \Psi(\omega; t; y)$ and we define $\widetilde{\gamma}_T(t) := \widetilde{\Psi}(\omega; t; y)$, so that $\widetilde{\gamma}_T$ is a curve whose components are polynomials of degree $O_{\varepsilon}(1)$. We may assume without loss of generality that $\delta$ is chosen small, depending on $\phi$ and $\varepsilon$, so as to ensure that
\begin{equation*}
    |\gamma_T(t) - \widetilde{\gamma}_T(t)| < \delta \quad \textrm{and} \quad  |\gamma_T'(t) - \widetilde{\gamma}_T'(t)| < \delta \quad \textrm{for all $t \in \bar{I}$ and $T \in \T$.}
\end{equation*}
It follows that each $T \in \T$ satisfies $T[\delta, \bar{I}] \subseteq \widetilde{T}[2\delta, \bar{I}]$ where
\begin{equation*}
  \widetilde{T}[\eta, J]   := \big\{(x,t) \in \B^{n-1} \times J : |x - \widetilde{\gamma}_T(t)| < \eta \big\}
\end{equation*}
for $J \subseteq (-1,1)$ an interval and $\eta >0$. In particular, by replacing $2\delta$ with $\delta$ (after passing to a $(2\delta, 1, d)$-grain), our problem is now to show
\begin{equation}\label{eq: poly approx 2}
    \#\{T \in \T : |\widetilde{T} \cap G| \geq \lambda |T| \} \lesssim_{\phi, d, \varepsilon} (\lambda/\rho)^{-B_1} \rho^{-B_2} \delta^{-(n-1) + \kappa_{\circ} - \varepsilon/2}
\end{equation}
where $\kappa_{\circ}$ and $B$ are as above and we write $\widetilde{T} := \widetilde{T}[\delta, \bar{I}]$ for all $T \in \T$. 
\medskip

%%%%%%%%%%%%%%%%%%%%%%%%%%%%%%%%%%%%%%%%%%%%%%%%%%%%%%%%%%%%%%%%%%%%%%%%%%%%%%%%%%%%%%%%%%%%%%%%

%                      Step 2: Intersection vs containment

%%%%%%%%%%%%%%%%%%%%%%%%%%%%%%%%%%%%%%%%%%%%%%%%%%%%%%%%%%%%%%%%%%%%%%%%%%%%%%%%%%%%%%%%%%%%%%%%

\noindent \underline{Step 2: Intersection vs containment.} Suppose $T \in \T$ satisfies $|\widetilde{T} \cap G| \geq \lambda |T|$. It follows that there exists some translate $v_T \in B^{n-1}(0,\delta)$ such that the core curve $\widetilde{\Gamma}_T := \{(\widetilde{\gamma}(t), t) : t \in \bar{I} \}$ must satisfy
\begin{equation*}
    \cH^1\big(S_T\big) \gtrsim \lambda \qquad \textrm{where} \qquad S_T := (\widetilde{\Gamma}_T + (v_T,0)) \cap G. 
\end{equation*}
Recall that the components of the map $\widetilde{\gamma}_T$ are polynomials of degree $O_{\varepsilon}(1)$. Consequently, $S_T$ is a semialgebraic set of complexity $O_{d, \varepsilon}(1)$ and, by \cite[Theorem 1]{BPR1996}, has $O_{d,\varepsilon}(1)$ connected components. By pigeonholing, there exist some constant $c_{d, \varepsilon} > 0$, depending only on $n$, $d$ and $\varepsilon$, and an interval $J_T \subseteq \bar{I}$ of length $c_{d, \varepsilon}\lambda$ such that 
\begin{equation*}
  (\widetilde{\gamma}_T(t) + v_T, t) \in G \qquad \textrm{for all $t \in J_T$.}
\end{equation*}
In particular, we have the containment $\widetilde{T}[\delta, J_T] \subseteq N_{2\delta} G$.

Now let $\cJ(\lambda)$ be a finitely overlapping collection of open intervals of length $(c_{d, \varepsilon}/10) \lambda$ which cover $\bar{I}$. In light of the observations of the previous paragraph, 
\begin{equation}\label{eq: int vs cont 1}
    \#\{T \in \T : |\widetilde{T} \cap G| \geq \lambda |T| \} \lesssim_{d, \varepsilon} \sum_{J \in \cJ(\lambda)} \# \T[J]
\end{equation}
where 
\begin{equation*}
    \T[J] := \big\{T \in \T : \widetilde{T}[\delta, J] \subseteq N_{2\delta} G \big\}.
\end{equation*}

It is helpful to introduce the following terminology, corresponding to parts I) and II) of the theorem:\smallskip

\noindent Case I): $S_{\tK}(\phi;m;\kappa)$ holds and $\T$ is direction-separated;\smallskip

\noindent Case II): $S_{\tN}(\phi;m;\kappa)$ holds and $\T$ is centre-separated.\medskip

\noindent To prove the theorem, by \eqref{eq: poly approx 2} and \eqref{eq: int vs cont 1} it suffices to show that if either Case $\iota$) holds for $\iota \in \{\tI, \tII\}$, then
\begin{equation}\label{eq: desired}    \sum_{J\in\cJ(\lambda)} \#\mathbb{T}[J]\lesssim_{\phi,d,\varepsilon} \lambda^{-B_1}\delta^{-(n-1)+\kappa_\circ -\varepsilon/2}= (\lambda/\rho)^{-B_1}\rho^{-B_1}\delta^{-(n-1)+\kappa_\circ -\varepsilon/2}
\end{equation}
for some uniform constants $B_1 > 0$, $\kappa_\circ>0$ depending only on $\kappa$. We will in fact show this with
\begin{equation*}
    \begin{array}{lclcl}
      B_1 = 1+\frac{1}{\kappa+1},   & & \kappa_\circ = \frac{\kappa}{(n-1)\kappa+1} & & \text{in Case I,} \\[4pt]
      B_1 =  1+\frac{1}{\kappa},& &  \kappa_\circ= \frac{1}{(n-1)} & &\text{in Case II}.
    \end{array}
\end{equation*}

In either Case I) or Case II), let $\cJ_1(\lambda)$ be the set of $J\in \cJ(\lambda)$ with
\begin{equation*}
   \#\mathbb{T}[J] \leq \bC_{\phi,d,\varepsilon} \lambda^{1-B_1}\delta^{-(n-1)+\kappa_\circ -\varepsilon/2}
\end{equation*}
for $\bC_{\phi,d,\varepsilon} \geq 1$ an appropriately large constant specified at the end of the argument. Since $\# \cJ(\lambda)\lesssim_{d,\varepsilon}\lambda^{-1}$, it suffices to show
\begin{equation}\label{eq: poly approx reduced}
    \sum_{J\in\cJ_2(\lambda)} \#\mathbb{T}[J]\lesssim_{\phi,d,\varepsilon} \lambda^{-B_1}\delta^{-(n-1)+\kappa_{\circ} -\varepsilon/2}
\end{equation}
for $\cJ_2(\lambda):=\cJ(\lambda)\setminus\cJ_1(\lambda)$. We claim that \eqref{eq: poly approx reduced} will follow if we prove the estimate
    \begin{equation}\label{eq: int vs cont 2}
    |N_{C_\phi\delta} G \cap (\R^{n-1} \times J)|\gtrsim_{\phi, d, \varepsilon} \delta^{m-1}\lambda^{1+1/\kappa} (\delta^{n-1 + \varepsilon/2}\#\T[J])^{e(\iota)},
\end{equation}
for a constant $C_\phi$ depending only on $\phi$, for each $J\in \cJ_2(\lambda):=\cJ(\lambda)\setminus\cJ_1(\lambda)$, where $e(\tI) := 1+1/\kappa$ and $e(\tII) := 1$. Indeed, by Wongkew's theorem \cite{Wongkew1993}, we have the upper bound
\begin{equation}\label{eq: int vs cont 3}
    |N_{C_\phi\delta} G|\lesssim_{C_\phi,d} \rho^{n-m} \delta^m \leq \delta^m.
\end{equation}
Thus, once \eqref{eq: int vs cont 2} is established, it can be combined with \eqref{eq: int vs cont 3} to deduce that
\begin{equation*}
   \sum_{J \in \cJ(\lambda)} (\#\T[J])^{e(\iota)} \lesssim_{\phi, d,\varepsilon} (\lambda/\rho)^{-(1 + 1/\kappa)} \rho^{-(1 + 1/\kappa)} \delta^{-(n-1 + \varepsilon/2)e(\iota) + 1}.
\end{equation*}
Thus, by H\"older's inequality, 
\begin{equation*}
    \sum_{J \in \cJ(\lambda)} \#\T[J] \lesssim_{\phi, d,\varepsilon} (\lambda/\rho)^{-B_1} \rho^{-B_1} \delta^{-(n-1)+\kappa_{\star} - \varepsilon/2}
\end{equation*}
where $\kappa_{\star} := 1/e(\iota)$ and $B_1 := 1 + \kappa_{\star}/\kappa$ agrees with the above definition. It is easy to check that $\kappa_{\circ} \leq \kappa_{\star}$, so the desired bound \eqref{eq: desired} follows. \medskip

%%%%%%%%%%%%%%%%%%%%%%%%%%%%%%%%%%%%%%%%%%%%%%%%%%%%%%%%%%%%%%%%%%%%%%%%%%%%%%%%%%%%%%%%%%%%%%%%

%                      Step 3: Constructing the parametrisation

%%%%%%%%%%%%%%%%%%%%%%%%%%%%%%%%%%%%%%%%%%%%%%%%%%%%%%%%%%%%%%%%%%%%%%%%%%%%%%%%%%%%%%%%%%%%%%%%

\noindent \underline{Step 3: Constructing the parametrisation, part 1.} The goal for the remainder of the proof is to establish \eqref{eq: int vs cont 2} under the stated hypotheses. Let $J\in \cJ_2(\lambda)$. Following the strategy of \cite{KR2018}, we shall use the family of tubes $\T[J]$ to construct a parametrisation of the set $N_{10\delta} G$, and use this parametrisation to estimate the volume. 

We shall now assume $\T$ is a collection of $\delta$-tubes which is either direction-separated or centre-separated and, without loss of generality, that $\T[J] \neq \emptyset$. The two cases are similar and so, for the next two steps, we treat them in parallel. At this stage, we shall not make use of the additional assumptions featured in Case I) and Case II) defined above. 

Fix $J \in \cJ(\lambda)$ and let
    \begin{equation*}
        \cL:=\big\{(\omega; y)\in \Omega_{\circ} \times Y_{\circ} : \widetilde{T}_{y,\omega}[\delta/2, J]\subseteq N_{2\delta} G\big\}.
    \end{equation*}
    As in \cite{KR2018}, the Tarski--Seidenberg projection theorem (see Theorem~\ref{thm: Tarski}) can be used to show $\cL$ is a semialgebraic set of complexity $O_{d, \varepsilon}(1)$. Define the projections $\Pi_{\ctr} \colon \R^{n-1} \times \R^{n-1} \to \R^{n-1}$ and $\Pi_{\dir} \colon \R^{n-1} \times \R^{n-1} \to \R^{n-1}$ by
    \begin{equation*}
     \Pi_{\ctr} \colon (\omega; y) \mapsto \omega \quad \textrm{and} \quad \Pi_{\dir} \colon (\omega; y) \mapsto y, \qquad  (\omega; y) \in  \R^{n-1} \times \R^{n-1}.
    \end{equation*}
    If $(\omega; y) \in \Omega_{\circ} \times Y_{\circ}$ is such that $\widetilde{T}_{y,\omega}[\delta, J] \subseteq N_{2\delta}G$, then not only do we have $(\omega; y)\in \cL$, but also that 
    \begin{equation*}
       \big(\{\omega\}\times B^{n-1}(y,c_\phi\delta)\big) \cap (\Omega_{\circ} \times Y_{\circ}) \subseteq \cL; \quad \big(B^{n-1}(\omega,c_\phi\delta)\times \{y\}\big) \cap (\Omega_{\circ} \times Y_{\circ}) \subseteq \cL, 
    \end{equation*}
    for some constant $c_\phi>0$ depending only on $\phi$. Therefore,
\begin{itemize}
    \item $|\Pi_{\dir}(\cL)|\gtrsim_{\phi} \delta^{n-1}\#\T[J]$ if $\T$ is direction-separated;\smallskip
    \item $|\Pi_{\ctr}(\cL)|\gtrsim_{\phi} \delta^{n-1} \#\T[J]$ if $\T$ is centre-separated.
\end{itemize}
    
 By \cite[Lemma 2.2]{KR2018} (for the reader's convenience, this is reproduced as Corollary A.5 in the appendix of this article), we can find semialgebraic sections for $\cL$ under the projections $\Pi_{\ctr}$ and $\Pi_{\dir}$, of bounded complexity.  More precisely, there exist semialgebraic subsets $\cL_{\ctr}\subseteq \cL$ and $\cL_{\dir} \subseteq \cL$, each of complexity $O_{d, \varepsilon}(1)$, such that
 \begin{itemize}
     \item $\Pi_{\dir}(\cL_{\dir})=\Pi_{\dir}(\cL)$ and $\#\big[\cL_{\dir} \cap \big(\Pi_{\dir}\big)^{-1}\big(\{y\}\big)\big] = 1$ for all $y \in\Pi_{\dir}(\cL)$;
     \item $\Pi_{\ctr}(\cL_{\ctr})=\Pi_{\ctr}(\cL)$ and $\#\big[\cL_{\ctr} \cap \big(\Pi_{\ctr}\big)^{-1}\big(\{\omega\}\big)\big] = 1$ for all $\omega\in\Pi_{\ctr}(\cL)$.
 \end{itemize}

Both $\cL_{\dir}$ and $\cL_{\ctr}$ are compact semialgebraic subsets of $\Omega_{\circ} \times Y_{\circ}$ of complexity $O_{d, \varepsilon}(1)$. Furthermore,
\begin{itemize}
    \item $\dim \cL_{\dir} = n - 1$ if $\T$ is direction-separated;\smallskip
    \item $\dim \cL_{\ctr} = n - 1$ if $\T$ is centre-separated.\smallskip
\end{itemize}
Therefore, by Gromov's algebraic lemma (see Lemma~\ref{lem: Gromov}) and pigeonholing, for $K = K(\varepsilon) := \ceil{1000n^3/\varepsilon}$ there exists a $C^K$ function 
    \begin{equation*}
        h = (P, D) \colon [0,1]^{n-1}\mapsto \Omega_{\circ} \times Y_{\circ}
    \end{equation*}
with $\|h\|_{C^K}\leq 1$ such that the following holds:
\begin{itemize}
    \item If $\T$ is direction-separated, then
    \begin{equation*}
         h([0,1]^{n-1})\subseteq \cL_{\dir} \qquad \textrm{and} \qquad |h([0,1]^{n-1})|\sim_{d, \varepsilon} |\cL_{\dir}|.
    \end{equation*}
    \item If $\T$ is centre-separated, then 
    \begin{equation*}
        h([0,1]^{n-1})\subseteq \cL_{\ctr}  \qquad \textrm{and} \qquad  |h([0,1]^{n-1})|\sim_{d, \varepsilon} |\cL_{\ctr}|.
    \end{equation*}
\end{itemize}
Since  $|D([0,1]^{n-1})|\sim |\Pi_{\dir}(\cL)|$ and $|P([0,1]^{n-1})|\sim|\Pi_{\ctr}(\cL)|$, we have
\begin{itemize}
    \item $|D([0,1]^{n-1})|\gtrsim_{\phi,d, \varepsilon}\delta^{n-1}\#\T[J]$ if $\T$ is direction-separated; \smallskip
    \item $|P([0,1]^{n-1})|\gtrsim_{\phi,d, \varepsilon}\delta^{n-1}\#\T[J]$ if $\T$ is centre-separated.
\end{itemize}
Now observe that, for every $(\omega,y)\in \cL$, the point $\big(\widetilde{\Psi}(\omega; t; y),t\big)$ lies in $N_{2\delta} G$ for all $t\in J$ and, in turn, $(P(x), D(x))\in \cL$ for every $x\in [0,1]^{n-1}$. Thus, the map 
\begin{equation}\label{eq: const param 1}
  (x, t) \mapsto  \big(\widetilde{\Psi}(P(x); t; D(x)),t\big) \in  N_{2\delta} G, \qquad (x,t) \in [0,1]^{n-1} \times J
\end{equation}
can be thought of as parametrising a subset of $N_{2\delta} G$.\medskip

%%%%%%%%%%%%%%%%%%%%%%%%%%%%%%%%%%%%%%%%%%%%%%%%%%%%%%%%%%%%%%%%%%%%%%%%%%%%%%%%%%%%%%%%%%%%%%%%

%                     Step 4: Constructing the parametrisation II

%%%%%%%%%%%%%%%%%%%%%%%%%%%%%%%%%%%%%%%%%%%%%%%%%%%%%%%%%%%%%%%%%%%%%%%%%%%%%%%%%%%%%%%%%%%%%%%%

\noindent \underline{Step 4: Constructing the parametrisation, part 2.} In order for \eqref{eq: const param 1} to be an effective parametrisation, we require that it be a $O_{d,\varepsilon}(1)$ to $1$ mapping. To ensure this, rather than work with \eqref{eq: const param 1}, we work with a polynomial approximant. By pigeonholing, there exists an $x_B \in [0,1]^{n-1}$ such that $B := B(x_B,\delta^{\varepsilon/(10n)}) \cap [0,1]^{n-1}$ satisfies 
\begin{itemize}
    \item $|D(B)|\gtrsim \delta^{\varepsilon/2}|D([0,1]^{n-1})|$ if $\T$ is direction-separated; \smallskip
    \item $|P(B)|\gtrsim \delta^{\varepsilon/2}|P([0,1]^{n-1})|$  if $\T$ is centre-separated.
\end{itemize} 
Let $\widetilde{D}$ and $\widetilde{P}$ denote the degree $K-1$ Taylor polynomial of $D$ and $P$, respectively, centred at the point $x_B$. Since $\|h\|_{C^K} \leq 1$, by the Lagrange form of the remainder, 
\begin{equation*}
    |\widetilde{D}(x) - D(x)| < \delta^{10n} \qquad \textrm{and} \qquad |\widetilde{P}(x) - P(x)| < \delta^{10n} \qquad \textrm{for all $x \in B$.}
\end{equation*}
In particular, provided $\delta > 0$ is chosen sufficiently small, we can ensure that 
\begin{itemize}
    \item $|\widetilde{D}(B)|\gtrsim \delta^{\varepsilon/2}|D([0,1]^{n-1})|$ if $\T$ is direction-separated; \smallskip
    \item $|\widetilde{P}(B)|\gtrsim \delta^{\varepsilon/2}|P([0,1]^{n-1})|$  if $\T$ is centre-separated.
\end{itemize}
Furthermore, by \eqref{eq: const param 1} and the mean value theorem applied to $\widetilde{\Psi}$, the map 
\begin{equation}\label{eq: const param 2}
  (x, t) \mapsto  \big(\widetilde{\Psi}(\widetilde{P}(x); t; \widetilde{D}(x)),t\big) \in  N_{3\delta} G, \qquad (x,t) \in B \times J
\end{equation}
can be thought of as parametrising a subset of $N_{3\delta} G$. The advantage of the parametrisation \eqref{eq: const param 2} over \eqref{eq: const param 1} is that it is formed by polynomials of degree $O_{d, \varepsilon}(1)$ and therefore has controlled multiplicity.

If $\T$ is direction- (respectively, centre-) separated, then $|\widetilde{D}(B)| \gtrsim_{\phi, d, \varepsilon} \delta^{n-1} > 0$ (respectively, $|\widetilde{P}(B)| \gtrsim_{\phi, d, \varepsilon} \delta^{n-1} > 0$) and so, by Sard's theorem, the Jacobian determinant of $\widetilde{D}$ (respectively, $\widetilde{P}$) cannot vanish identically. Therefore, by the effective inverse function theorem from Lemma~\ref{lem: effective inverse}, there exists an open set $\widetilde{B} \subseteq B$ such that
\begin{subequations}
\renewcommand{\theequation}{\theparentequation-\Roman{equation}}
\begin{itemize}
    \item If $\T$ is direction-separated, then $\widetilde{D}|_{\widetilde{B}} \colon \widetilde{B} \to \widetilde{D}(\widetilde{B})$ is a diffeomorphism,
    \begin{equation}\label{eq: const param 3 I}
        |\widetilde{D}(\widetilde{B})|\gtrsim_{d, \varepsilon} \delta^{\varepsilon/2}|D([0,1]^{n-1})| \quad \textrm{and} \quad 
        \inf_{x \in \widetilde{B}} |\det \partial_x \widetilde{D}(x)| \gtrsim_\phi |\widetilde{D}(\widetilde{B})|;
    \end{equation}
    \item If $\T$ is centre-separated, then $\widetilde{P}|_{\widetilde{B}} \colon \widetilde{B} \to \widetilde{P}(\widetilde{B})$ is a diffeomorphism,
    \begin{equation}\label{eq: const param 3 II}
        |\widetilde{P}(\widetilde{B})|\gtrsim_{d, \varepsilon} \delta^{\varepsilon/2}|P([0,1]^{n-1})| \quad \textrm{and} \quad 
        \inf_{x \in \widetilde{B}} |\det \partial_x \widetilde{P}(x)| \gtrsim_\phi |\widetilde{P}(\widetilde{B})|.
    \end{equation}
\end{itemize}
\end{subequations}
More precisely, Lemma~\ref{lem: effective inverse} and pigeonholing amongst the resulting open sets produces an open set $U$ satisfying the first two properties. The set $U$ can then be further refined to ensure the Jacobian bound in \eqref{eq: const param 3 I} or \eqref{eq: const param 3 II}. Indeed, such a refinement is possible due to quantitative Sard estimate
\begin{equation*}
   \big|G\big(\{x \in U : |\det \partial_x G(x)| < \gamma \}\big)\big| =  \int_{\{x \in U : |\det \partial_x G(x)| < \gamma\}}|\det \partial_x G(x)|\,\ud x \leq \gamma|U|
\end{equation*}
for $U \subset \R^n$ open and $G \colon U \to \R^n$ an injective $C^1$ mapping. Specifically, this estimate is applied with $G=\widetilde{D}$ and $\gamma=c_\phi\cdot |\widetilde{D}(U)|$ to yield \eqref{eq: const param 3 I}, and with $G=\widetilde{P}$ and $\gamma=c_\phi\cdot |\widetilde{P}(U)|$ to yield \eqref{eq: const param 3 II}, for an appropriately small constant $c_\phi$ in both cases.
\medskip 
%%%%%%%%%%%%%%%%%%%%%%%%%%%%%%%%%%%%%%%%%%%%%%%%%%%%%%%%%%%%%%%%%%%%%%%%%%%%%%%%%%%%%%%%%%%%%%%%

%                      Step 5: Freezing variables

%%%%%%%%%%%%%%%%%%%%%%%%%%%%%%%%%%%%%%%%%%%%%%%%%%%%%%%%%%%%%%%%%%%%%%%%%%%%%%%%%%%%%%%%%%%%%%%%

\noindent \underline{Step 5: Freezing variables.} In light of \eqref{eq: const param 2}, we may bound
\begin{equation*}
        |N_{3\delta} G \cap (\R^{n-1} \times J)| \geq \int_J \big|\big\{\widetilde{\Psi}(\widetilde{P}(x),t,\widetilde{D}(x)): x \in \widetilde{B}\big\}\big|\,\ud t.
\end{equation*}
For each fixed $t \in J$, the next step is to apply the change of variables formula to estimate the image of $\widetilde{B}$ under our parametrisation. So far, we have only been working under the hypothesis that $\phi$ is a H\"ormander-type phase and $\T$ is either direction- or centre-separated. At this juncture, we exploit the additional hypotheses in Case I) and Case II), as defined in Step 2. We shall treat these two cases in parallel. To this end, we let $W_{\tI} := \widetilde{D}(\widetilde{B})$ and $W_{\tII} := \widetilde{P}(\widetilde{B})$.

\begin{itemize}
    \item In Case I), let
    \begin{equation*}
  \fp(y) := \widetilde{P}\circ\widetilde{D}^{-1}(y) \quad \textrm{and} \quad \fP(y) := \partial_y \fp(y),\quad \text{for all $y\in W_{\tI}$.}
\end{equation*}
    \item In Case II), let
  \begin{equation*}
  \fd(\omega) := \widetilde{D} \circ \widetilde{P}^{-1}(\omega) \quad \textrm{and} \quad \fD(\omega) := \partial_{\omega} \fd(\omega),\quad \text{for all $\omega \in W_{\tII}$.}
\end{equation*}
\end{itemize}
Here the inverses are well-defined smooth functions by the construction of $\widetilde{B}$.

By combining the observations of the previous steps, in Case $\iota$) for $\iota \in \{\tI, \tII\}$, we have
\begin{equation}\label{eq: bd Jacobian 1}
    |W_{\iota}| \gtrsim_{\phi, d, \varepsilon} \delta^{n-1 + \varepsilon/2} \# \T[J].
\end{equation}
Our additional hypotheses also imply the existence of $\alpha(\iota), \beta(\iota) \in \binom{[n-1]}{n - m}$ for which the following hold:
\begin{subequations}
\renewcommand{\theequation}{\theparentequation-\Roman{equation}}
\begin{itemize}
    \item In Case I), for $\Upsilon_{\fp}(t;y)$ as in \eqref{eq: Upsilon}, we define
    \begin{equation*}
A_{\tI}(t;y) := \Big[\fP(y) - (\partial_{yy}^2\phi)\big(\Upsilon_{\fp}(t;y)\big)\Big]_{\alpha(\tI), \beta(\tI)}.
\end{equation*}
Then we have a sublevel set estimate
\begin{equation}\label{eq: freeze 2 I}
    \left|\left\{(t;y)\in \bar{I} \times Y_\circ: |\det A_{\tI}(t; y)|<\sigma  \right\}\right| \lesssim_{\phi} \sigma^{\kappa} \qquad \textrm{for all $0<\sigma<1$.}
\end{equation}

\item In Case II), for $\Upsilon^{\fd}(\omega; t)$ as in \eqref{eq: Upsilon}, we define
    \begin{equation*}
A_{\tII}(t;\omega) := \Big[ I_{n-1} - (\partial_{yy}^2\phi)\big(\Upsilon^{\fd}(\omega; t)\big)\fD(\omega)\Big]_{\alpha(\tII), \beta(\tII)}.
\end{equation*}
Then we have a uniform sublevel set estimate
\begin{equation}\label{eq: freeze 2 II}
  \sup_{t \in \bar{I}}  \left|\left\{\omega \in X_\circ: |\det A_{\tII}(t; \omega)|<\sigma  \right\}\right| \lesssim_{\phi} \sigma^{\kappa} \qquad \textrm{for all $0<\sigma<1$.}
\end{equation}
\end{itemize}
\end{subequations}
Here we are using the submatrix notation introduced at the beginning of \S\ref{subsec: beyond higher}. To simplify the exposition, we assume in either case that $A_{\iota}$ is formed from the top-left $(n-m)\times (n-m)$-submatrix: that is, $\alpha(\iota) = \beta(\iota) = \{1, \dots, n-m\}$. The general case, where the $A_{\iota}$ are formed by deleting other rows and columns of the above matrices, is analogous.

Let $\Pi_m \colon \R^{n-1} \to \R^{m-1}$ denote the orthogonal projection 
\begin{equation*}
  \Pi_m \colon (y_1, \dots, y_{n-1}) \mapsto (y_{n-m +1}, \dots, y_{n-1}), \qquad   (y_1, \dots, y_{n-1}) \in \R^{n-1}. 
\end{equation*}
Let $\iota \in \{\tI, \tII\}$ and, given $v := (v_{n-m+1}, \dots, v_{n-1}) \in \Pi_m(W_{\iota})$, let 
\begin{equation*}
  U_{\iota}(v) :=  \big\{u \in \R^{n-m} :  (u, v) \in W_{\iota} \big\}.
\end{equation*}
In Case $\iota$), we define $\widetilde{\Psi}_{\iota, v} \colon U_{\iota}(v) \times J \times \R^{m-1} \to \R$ for all $v \in \Pi_m(W_{\iota})$ by
\begin{equation*}
   \widetilde{\Psi}_{\iota, v} (u;t;\eta) := \widetilde{\Psi} \circ V_{\iota}(t; u, v) + (\partial_\omega\widetilde{\Psi}) \circ V_{\iota}(t;u,v)\big(\sum_{j=n-m+1}^{n-1}\eta_je_j\big),
\end{equation*}
where $V_{\tI}(t;y):=(\fp(y);t;y)$ and $V_{\tII}(t;\omega):=(\omega;t;\fd(\omega))$. Here $e_j$ denotes the $j$th standard basis vector and $\eta = (\eta_{n-m+1}, \dots, \eta_{n-1}) \in \R^{m-1}$. 

For $\iota \in \{\tI, \tII\}$ and $v \in \Pi_m(W_{\iota})$, by \eqref{eq: const param 2}, the fact that $\|\partial_{\omega}\Psi\|_\infty\lesssim_\phi 1$, as well as the fact that $\widetilde{\Psi}$ is an appropriate polynomial approximation of $\Psi$ (see \eqref{eq:approx of Psi} for more details), in Case $\iota$) we have
\begin{equation*}
    \widetilde{\Psi}_{\iota, v}(u;t;\eta) \in N_{C_\phi\delta}G\qquad\textrm{for all $(u;t;\eta) \in U_{\iota}(v) \times J \times B(0,\delta)$,}
\end{equation*}
where here and below $B(0,\delta)$ denotes a ball in $\R^{m-1}$, and $C_\phi$ is a constant that depends only on $\phi$. Therefore,
\begin{equation}\label{eq: freeze 4}
    |N_{C_\phi\delta}G\cap (\R^{n-1}\times J)|\gtrsim \int_J|E_{\iota, v}(t)|\,\ud t \quad \textrm{for} \quad E_{\iota, v}(t) := \big\{ \widetilde{\Psi}_{\iota, v}(u;t;\eta): (u, \eta) \in U_{\iota}(v) \times B(0,\delta)\big\}.
\end{equation}
Furthermore, by Sard's theorem, $|E_{\iota, v}(t)| = |E_{\iota, v}^{\star}(t)|$ where
\begin{equation*}
     E_{\iota, v}^{\star}(t) := \big\{\widetilde{\Psi}_{\iota, v}(u;t;\eta): (u, \eta) \in U_{\iota}(v) \times B(0,\delta),\, \det J_{\iota, v}(u; t; \eta) \neq 0 \big\}
\end{equation*}
where $J_{\iota, v}(u; t; \eta)$ is the $(n-1) \times (n-1)$ Jacobian matrix of the mapping $(u; \eta) \mapsto \widetilde{\Psi}_{\iota, v}(u;t;\eta)$, for $t$ fixed.

For each $t\in J$, the change of variables formula\footnote{See, for instance, \cite[Theorem 3.2.3]{Federer1969} for the multiplicity version used here.} gives
\begin{equation}\label{eq: bd Jacobian 2}
    \int_{E_{\iota, v}^{\star}(t)} \#\fN_{\iota, v}(t;x)\,\ud x = \int_{B(0,\delta)}\int_{U_{\iota}(v)} \big| \det J_{\iota, v}(u; t; \eta)\big|\,\ud u \ud \eta,
\end{equation}
where $\fN_{\iota, v}(t;x)$ is the preimage set
\begin{equation*}
   \fN_{\iota, v}(t;x) := \big\{(u, \eta) \in U_{\iota}(v) \times B(0,\delta) :  \widetilde{\Psi}_{\iota, v}(u;t;\eta) = x,\, \det J_{\iota, v}(u; t; \eta) \neq 0 \big\}.
\end{equation*}
Fix $t \in J$ and $x \in E_{\iota, v}^{\star}(t)$ and suppose $(u, \eta) \in \fN_{\iota, v}(t;x)$. It follows that $(u, \eta)$ is an isolated root of the system of $\widetilde{\Psi}_{\iota, v}(\,\cdot\,;t;\,\cdot\,) - x$, using the terminology of \cite[Definition 5.1]{CKW2010}. Since the maps $\widetilde{\Psi}_{\iota, v}(\,\cdot\,;t;\,\cdot\,) - x$ are polynomials of degree $O_{d, \varepsilon}(1)$, by the formulation of Bezout's theorem in \cite[Theorem 5.2]{CKW2010} it follows that 
\begin{equation*}
    \#\fN_{\iota, v}(t;x) \lesssim_{d, \varepsilon} 1 \qquad \textrm{for all $x \in E_{\iota, v}^{\star}(t)$.} 
\end{equation*}
Combining these observations with \eqref{eq: bd Jacobian 2}, we obtain the bound
\begin{equation}\label{eq: bd Jacobian 3}
    |E_{\iota, v}(t)| \gtrsim_{d, \varepsilon} \int_{B(0,\delta)}\int_{U_{\iota}(v)} \big| \det J_{\iota, v}(u; t; \eta)\big|\,\ud u \ud \eta.
\end{equation}
We shall bound the right-hand side of \eqref{eq: bd Jacobian 3} using the sublevel set estimates from either hypothesis $S_{\tK}(\phi;m;\kappa)$ or $S_{\tN}(\phi;m;\kappa)$.\medskip

%%%%%%%%%%%%%%%%%%%%%%%%%%%%%%%%%%%%%%%%%%%%%%%%%%%%%%%%%%%%%%%%%%%%%%%%%%%%%%%%%%%%%%%%%%%%%%%%

%                      Step 6: Bounding the Jacobian

%%%%%%%%%%%%%%%%%%%%%%%%%%%%%%%%%%%%%%%%%%%%%%%%%%%%%%%%%%%%%%%%%%%%%%%%%%%%%%%%%%%%%%%%%%%%%%%%

\noindent \underline{Step 6: Bounding the Jacobian.} We will show that, up to a relabelling of variables, 
\begin{equation}\label{eq: step 6 bound}
|\det J_{\iota,v}(u;t;\eta)|\gtrsim_\phi |\det A_\iota(t;u,v)|-C_{\phi, \varepsilon}\delta|W_\iota|^{-(n-1)}
\end{equation}
for all $\eta\in B(0,\delta)$ and $u\in U_\iota(v)$, where $A_\iota(t;u,v)$ is the matrix defined at the start of Step 5 and $C_{\phi, \varepsilon} >0$ is a constant depending only on $\phi$ and $\varepsilon$.

Recall the notation for submatrices introduced in \S\ref{subsec: beyond higher}. For $\alpha := \{1,\ldots,n-m\}$ and $\beta := \{1, \dots, n-1\}$, we have
\begin{equation*}
J_{\iota, v}(u; t; \eta) = \widetilde{\fJ}_{\iota, v}(u; t) +\cE_{\iota,v}(u;t;\eta)
\end{equation*}
where $\widetilde{\fJ}_{\iota, v}(u; t)$, $\cE_{\iota,v}(u;t;\eta) \in \mathrm{Mat}(\R, n-1)$ with
\begin{equation*}
\widetilde{\fJ}_{\iota, v}(u; t):=\begin{bmatrix}
   [\partial_z(\widetilde{\Psi} \circ V_{\iota})(t;u, v)]_{\beta, \alpha}
& (\partial_\omega\widetilde{\Psi})\circ V_{\iota}(t;u,v)\cdot
\begin{pmatrix}
    0_{(n-m)\times (m-1)}\\
    I_{(m-1)\times (m-1)}
\end{pmatrix}
\end{bmatrix}
\end{equation*}
and $\cE_{\iota,v}(u;t;\eta)$ has $(i,k)$ entry $\big(\cE_{\iota,v}(u;t;\eta)\big)_{i,k}$ given by
\begin{equation*}
    \begin{cases}
\sum_{j=n-m+1}^{n-1}\partial_{z_k}\big((\partial_{\omega_j}\widetilde{\Psi}_i) \circ V_{\iota}\big)(t;u,v) \eta_j& \textrm{if $1 \leq k \leq n-m$,}\\
0& \textrm{if $n-m+1 \leq k \leq n-1$,}
    \end{cases}
\end{equation*}
for $1 \leq i \leq n-1$. Here $\widetilde{\Psi}_i$ denotes the $i$th entry of $\widetilde{\Psi}$ and $z$ represents the $y$ variable in Case I) and the $\omega$ variable in Case II).

To begin, we observe that 
\begin{equation}\label{eq: bd Jacobian 4}
\det J_{\iota, v}(u; t; \eta) = \det \big(\fJ_{\iota, v}(u; t)+\cE_{\iota,v}(u;t;\eta) \big)+ O(\delta^{10n})
\end{equation}
for all $(u;t;\eta) \in U_{\iota}(v) \times J \times B(0,\delta)$, where
\begin{equation*}
\fJ_{\iota, v}(u; t):=
\begin{bmatrix} 
   [\partial_z(\Psi \circ V_{\iota})(t;u, v)]_{\beta, \alpha}

& (\partial_\omega\Psi)\circ V_{\iota}(t;u,v)\cdot
\begin{pmatrix}
    0_{(n-m)\times (m-1)}\\
    I_{(m-1)\times (m-1)}
\end{pmatrix}
\end{bmatrix}.
\end{equation*}
Indeed, it is straightforward to see that, in Case I), the image of $\fp$ lies in the ball $2 \cdot \Omega_{\circ}$ concentric to $\Omega_{\circ}$ but with twice the radius. From the definition of $\widetilde{\Psi}$ and the Lagrange form of the remainder, we have
\begin{equation}\label{eq:approx of Psi}
   \begin{cases}
        |(\partial_y\Psi)(\fp(y); t; y) - (\partial_y\widetilde{\Psi})(\fp(y); t; y)| < \delta^{10n^2} \\
         |(\partial_\omega\Psi)(\fp(y); t; y) - (\partial_\omega\widetilde{\Psi})(\fp(y); t; y)| < \delta^{10n^2} 
   \end{cases} \qquad \textrm{for all $(y;t) \in W_{\tI} \times J$,}
\end{equation}
provided $\delta$ is chosen sufficiently small, depending only on $\phi$, $n$ and $\varepsilon$. On the other hand, $|\partial_y \fp(y)| \lesssim_{\phi, d, \varepsilon} \delta^{-n}$ for all $y \in W_{\tI}$, which follows from the Jacobian bound \eqref{eq: const param 3 I} and the bound $\|P\|_{C^K} \leq 1$ guaranteed by Gromov's lemma. Combining these observations with the chain rule, and applying a completely analogous argument in Case II), we readily obtain \eqref{eq: bd Jacobian 4}. 

We further show that
\begin{equation}\label{eq: bd error Jacobian}
    \det J_{\iota,v}(u;t;\eta)=\det \fJ_{\iota,v}(u;t) + O_{\phi, \varepsilon}(\delta |W_\iota|^{-(n-1)}).
\end{equation}
Due to \eqref{eq: bd Jacobian 4}, this holds as long as 
\begin{equation*}
  \big|\partial_{z_k}(\Psi_i \circ V_{\iota})(t;z)\big| \lesssim_{\phi} |W_\iota|^{-1} \quad \textrm{and} \quad   \big|\big(\cE_{\iota,v}(u;t;\eta)\big)_{i,k}\big| \lesssim_{\phi, \varepsilon} \delta |W_\iota|^{-1}
\end{equation*} for each $1 \leq i, k \leq n-1$, where $\Psi_i$ denotes the $i$th component of $\Psi$. We proceed to verify that this is indeed the case.
\begin{itemize}
    \item In Case I), using the bounds $\|\partial_{y_k}\partial_{\omega_j}\widetilde{\Psi}\|_{\infty} \lesssim_{\phi, \varepsilon} 1$ and $|\eta| < \delta$, we see that
\begin{equation*}
   \big|\partial_{y_k}(\Psi_i \circ V_{\tI})(t;y)\big| \lesssim_{\phi} 1 + \||\partial_{y_k}\fp|\|_{L^{\infty}(\widetilde{B})}, \quad  \big|\big(\cE_{\tI,v}(u;t;\eta)\big)_{i,k}\big| \lesssim_{\phi, \varepsilon} \delta + \delta\||\partial_{y_k}\fp|\|_{L^{\infty}(\widetilde{B})}.
\end{equation*}
To establish \eqref{eq: bd error Jacobian}, it therefore suffices to show 
\begin{equation*}
    |\partial_{y_k}\fp(y)|\lesssim_\phi |\widetilde{D}(\widetilde{B})|^{-1}=|W_{\tI}|^{-1} \qquad \textrm{for all $y\in \widetilde{B}$.}
\end{equation*}
However, since $\|P\|_{C^K}, \|D\|_{C^K}\leq 1$, the entries of the matrix
\begin{equation*}
   \big[\det\partial_y\widetilde{D}(y)\big] \partial_y\fp(y)=(\partial_y\widetilde{P})(\widetilde{D}^{-1}(y)){\rm adj}(\partial_y\widetilde{D}(y))
\end{equation*}
have magnitude $O_\phi(1)$. Thus, \eqref{eq: const param 3 I} implies the desired estimate. 

\item In Case II), using the bounds $\|\partial_{\omega_k}\partial_{\omega_j}\widetilde{\Psi}\|_{\infty} \lesssim_{\phi, \varepsilon} 1$ and $|\eta| < \delta$, we see that
\begin{equation*}
  \big|\partial_{\omega_k}(\Psi_i \circ V_{\tII})(t;\omega)\big| \lesssim_{\phi} 1 + \||\partial_{\omega_k}\fd|\|_{L^{\infty}(\widetilde{B})}, \;\; \big|\big(\cE_{\tII,v}(u;t;\eta)\big)_{i,k}\big| \lesssim_{\phi, \varepsilon} \delta + \delta\||\partial_{\omega_k}\fd|\|_{L^{\infty}(\widetilde{B})}.
\end{equation*}
To conclude the proof of \eqref{eq: bd error Jacobian}, it therefore suffices to show
\begin{equation*}
    |\partial_{\omega_k}\fd(\omega)|\lesssim_\phi |\widetilde{P}(\widetilde{B})|^{-1}=|W_{\tII}|^{-1} \qquad \textrm{for all $\omega \in \widetilde{B}$.}
\end{equation*} 
However, this follows in a completely analogous manner to Case I).
\end{itemize}
Finally we show that, up to relabelling variables,
\begin{equation*}
    |\det \fJ_{\iota,v}(u,t)|\sim_\phi|\det A_{\iota}(t;u,v)|.
\end{equation*}
Note that, once proved, this directly implies \eqref{eq: step 6 bound} when combined with \eqref{eq: bd error Jacobian}. It is during the proof of this last estimate that the significant differences between Case I) and Case II) manifest. 

By differentiating the defining equation $\partial_y\phi( \Psi(\omega; t; y),t;y)=\omega$ of $\Psi$ independently with respect to $\omega$ and $y$, we obtain
\begin{align*}
    \partial_{\omega}\Psi(\omega;t;y)&=(\partial_{xy}^2 \phi)^{-1}(\Psi(\omega;t;y), t; y),     \\
    \partial_y\Psi(\omega;t;y)&=-(\partial_{xy}^2 \phi)^{-1}(\Psi(\omega;t;y), t; y)(\partial^2_{yy}\phi)(\Psi(\omega;t;y), t; y).
\end{align*}
From this and the chain rule, we readily deduce the following formul\ae:
\begin{itemize}
    \item In Case I), for $\Upsilon_{\fp}(t;y) = (\Psi \circ V_{\tI}(t;y), t; y)$ as defined in \eqref{eq: Upsilon}, we have
    \begin{equation*}
\partial_y (\Psi\circ V_{\tI})(t;y) =(\partial_{xy}^2\phi)^{-1}\big(\Upsilon_{\fp}(t;y)\big) \big(\fP(y) - (\partial_{yy}^2\phi)\big(\Upsilon_{\fp}(t;y)\big)\big).
\end{equation*}
From this, we deduce that the matrix $\fJ_{\iota,v}(u;t)$ equals
\begin{equation*}
   (\partial^2_{xy}\phi)^{-1}\big(\Upsilon_{\fp}(t;u,v) \big)
    \begin{bmatrix}
 \big[\fP(u,v) - (\partial_{yy}^2\phi)\big(\Upsilon_{\fp}(t;u,v)\big)\big]_{\beta, \alpha} &
 \begin{pmatrix}
 0_{(n-m)\times(m-1)}\\
 I_{(m-1)\times (m-1)}
 \end{pmatrix}
    \end{bmatrix},
\end{equation*}
and therefore
\begin{equation*}
    |\det \fJ'_{\iota,v}(u;t)|\sim_\phi|\det \big[\fP(u,v) - (\partial_{yy}^2\phi)\big(\Upsilon_{\fp}(t;u,v)\big)\big]_{\alpha,\alpha}| = |\det A_{\tI}(t;u,v)|.
\end{equation*}
Indeed, up to relabelling the variables, $[\fP(u,v) - (\partial_{yy}^2\phi)\big(\Upsilon_{\fp}(t;u,v)\big)]_{\alpha,\alpha}$ equals $A_{\tI}(t;u,v)$.

\item In Case II), for $\Upsilon^{\fd}(\omega; t) = (\Psi\circ V_{\tII}(t;\omega), t; \fd(\omega))$ as defined in \eqref{eq: Upsilon}, we have
    \begin{equation*}
\partial_{\omega} (\Psi \circ V_{\tII})(t;\omega) =(\partial_{xy}^2\phi)^{-1}\big(\Upsilon^{\fd}(\omega; t)\big)\Big(I_{n-1} - (\partial_{yy}^2\phi)\big(\Upsilon^{\fd}(\omega; t)\big)\fD(\omega)\Big).
\end{equation*}
We deduce that the matrix $\fJ_{\iota,v}(u;t)$ equals
\begin{equation*}
    (\partial_{xy}^2\phi)^{-1}\big(\Upsilon^{\fd}(u,v; t)\big)\cdot
    \begin{bmatrix}
      \big[ I_{n-1} - (\partial_{yy}^2\phi)\big(\Upsilon^{\fd}(u,v; t)\big)\fD(u,v)\big]_{\beta, \alpha} &  
         \begin{pmatrix}
 0_{(n-m)\times(m-1)}\\
 I_{(m-1)\times (m-1)}
 \end{pmatrix}
    \end{bmatrix},
\end{equation*}
and therefore
\begin{equation*}
    |\det \fJ'_{\iota,v}(u;t)|\sim_\phi |\det [I_{n-1} - (\partial_{yy}^2\phi)\big(\Upsilon^{\fd}(u,v; t)\big)\fD(\omega)]_{\alpha,\alpha}|=|\det A_{\tII}(t;u,v)|.
\end{equation*}
Indeed, up to relabelling the variables, $[I_{n-1} - (\partial_{yy}^2\phi)\big(\Upsilon^{\fd}(u,v; t)\big)\fD(u,v)]_{\alpha,\alpha}$ equals $A_{\tII}(t;u,v)$.
\end{itemize}
\medskip

%%%%%%%%%%%%%%%%%%%%%%%%%%%%%%%%%%%%%%%%%%%%%%%%%%%%%%%%%%%%%%%%%%%%%%%%%%%%%%%%%%%%%%%%%%%%%%%%

%                      Step 7: Concluding the argument

%%%%%%%%%%%%%%%%%%%%%%%%%%%%%%%%%%%%%%%%%%%%%%%%%%%%%%%%%%%%%%%%%%%%%%%%%%%%%%%%%%%%%%%%%%%%%%%%

\noindent \underline{Step 7: Concluding the argument.} Due to \eqref{eq: step 6 bound}, for every $\sigma>2C_{\phi, \varepsilon} \delta |W_\iota|^{-(n-1)}$ we have that
\begin{equation*}
    |\det A_\iota (t;u,v)|\geq \sigma \quad \Longrightarrow \quad |\det J_{\iota,v}(u;t;\eta)|\gtrsim_\phi  |\det A_\iota (t;u,v)|.
\end{equation*}
Combining this with \eqref{eq: freeze 4} and \eqref{eq: bd Jacobian 3}, we obtain
\begin{equation*}
    |N_{C_\phi\delta} G \cap (\R^{n-1} \times J)| \gtrsim_\phi \delta^{m-1} \int_{\{(t; u) \in J \times U_{\iota}(v) : |\det A_{\iota}(t; u,v)| \geq \sigma\}}|\det A_{\iota}(t; u,v)|\,\ud t\ud u
\end{equation*}
whenever $\sigma > 2C_{\phi, \varepsilon} \delta |W_\iota|^{-(n-1)}$. 

The inequality in the previous display holds for all $v \in \Pi_m(W_{\iota})$. Integrating over this set, we deduce that 
\begin{equation}\label{eq: bd Jacobian 5}
    |N_{C_\phi\delta} G \cap (\R^{n-1} \times J)| \gtrsim_\phi \delta^{m-1} \int_{\{(t; z) \in J \times W_{\iota} : |\det A_{\iota}(t; z)| \geq \sigma\}} |\det A_{\iota}(t; z)|\, \ud t \ud z
\end{equation}
for all $\sigma > 2C_{\phi, \varepsilon} \delta |W_\iota|^{-(n-1)}$. We now treat Case I and Case II separately.\medskip

\noindent Case I). For $c_{\phi,d,\varepsilon} >0$ a constant, depending only on $\phi$, $d$ and $\varepsilon$, take $\sigma := c_{\phi,d,\varepsilon} \lambda^{1/\kappa} |W_{\tI}|^{1 /\kappa}$. Provided $c_{\phi,d,\varepsilon}$ is sufficiently small, by \eqref{eq: freeze 2 I} and the fact that $|J| \sim_{d,\varepsilon} \lambda$ we deduce that
\begin{equation*}
    |\{(t; y) \in J \times W_{\tI} : |\det A_{\tI}(t; y)| < \sigma\}| \leq \frac{|J \times W_{\tI}|}{2}.
\end{equation*} 
Furthermore, $\sigma> 2C_{\phi, \varepsilon} \delta |W_{\tI}|^{-(n-1)}$, as otherwise \eqref{eq: bd Jacobian 1} implies that
\begin{equation*}
    \#\mathbb{T}[J]\lesssim_{\phi,d,\varepsilon}\lambda^{-\frac{1}{(n-1)\kappa+1}} \delta^{-(n-1)+\frac{\kappa}{(n-1)\kappa+1}-\varepsilon/2} \leq \lambda^{-\frac{1}{\kappa+1}} \delta^{-(n-1)+\frac{\kappa}{(n-1)\kappa+1}-\varepsilon/2}
\end{equation*}
which, provided the constant $\bC_{\phi,d,\varepsilon}$ in the definition of $\cJ_1(\lambda)$ in Step 2 is large enough, contradicts the assumption that $J\in \cJ_2(\lambda)$. 

Combining the observations of the previous paragraph with \eqref{eq: bd Jacobian 5}, we obtain 
\begin{equation*}
  |N_{C_\phi\delta} G \cap (\R^{n-1} \times J)| \gtrsim_\phi \delta^{m-1} \sigma|J \times W_{\tI}|.
\end{equation*}
By \eqref{eq: bd Jacobian 1}, this further implies that 
\begin{equation*}
    |N_{C_\phi\delta} G \cap (\R^{n-1} \times J)|\gtrsim_{\phi, d,\varepsilon}\delta^{m-1}\lambda^{1+1/\kappa} |W_{\tI}|^{1+1/\kappa} \gtrsim_{\phi,d,\varepsilon}  \delta^{m-1}(\lambda\delta^{n-1 + \varepsilon/2}\#\T[J])^{(\kappa+1)/\kappa}.
\end{equation*}
Thus, we have precisely the desired bound \eqref{eq: int vs cont 2}. \medskip

\noindent Case II). For  $c_{\phi,d,\varepsilon} >0$ a constant, depending only on $\phi$, $d$ and $\varepsilon$, take $\sigma := c_{\phi,d,\varepsilon} \lambda^{1/\kappa}$. Provided $c_{\phi,d,\varepsilon}$ is sufficiently small, by \eqref{eq: freeze 2 II} and the fact that $|J| \sim_{d,\varepsilon} \lambda$ we deduce that 
\begin{equation*}
  \sup_{\omega \in \Omega_{\phi}}  |\{t \in J  : |\det A_{\tII}(t; \omega)| < \sigma\}| \leq \frac{|J|}{2}.
\end{equation*}
Furthermore, $\sigma>2C_{\phi, \varepsilon}\delta |W_{\tII}|^{-(n-1)}$, as otherwise \eqref{eq: bd Jacobian 1} implies that
\begin{equation*}
    \#\mathbb{T}[J]\lesssim_{\phi,d,\varepsilon}\lambda^{-\frac{1}{(n-1)\kappa}}\delta^{-(n-1) + 1/(n-1) -\varepsilon/2} \leq \lambda^{-1/\kappa}\delta^{-(n-1) + 1/(n-1) -\varepsilon/2},
\end{equation*}
which, provided the constant $\bC_{\phi,d,\varepsilon}$ in the definition of $\cJ_1(\lambda)$ is large enough, contradicts the assumption that $J\in \cJ_2(\lambda)$.

Combining the observations of the previous paragraph with \eqref{eq: bd Jacobian 5}, we have 
\begin{equation*}
  |N_{C_\phi\delta} G \cap (\R^{n-1} \times J)| \gtrsim_\phi \delta^{m-1} \sigma|J \times W_{\tII}|.
\end{equation*}
By \eqref{eq: bd Jacobian 1}, this further implies that 
\begin{equation*}
    |N_{C_\phi\delta} G \cap (\R^{n-1} \times J)|\gtrsim_{\phi, d,\varepsilon}\delta^{m-1}\lambda^{1+1/\kappa} |W_{\tII}| \gtrsim_{\phi,d,\varepsilon}  \delta^{m-1} \lambda^{1+1/\kappa} (\delta^{n-1 + \varepsilon/2}\#\T[J]).
\end{equation*}
Thus, we have precisely the desired bound \eqref{eq: int vs cont 2}.\medskip

Having deduced \eqref{eq: int vs cont 2} in either case, the result now follows from Wongkew's theorem \cite{Wongkew1993}, as discussed in Step 2 above.

\end{proof}

\section{Uniform sublevel set estimates}\label{sec: sublevel}

%%%%%%%%%%%%%%%%%%%%%%%%%%%%%%%%%%%%%%%%%%%%%%%%%%%%%%%%%%%%%%%%%%%%%%%%%%%%%%%%%%%%%%%%%%%%%%%%

%         Uniform sublevel set estimates

%%%%%%%%%%%%%%%%%%%%%%%%%%%%%%%%%%%%%%%%%%%%%%%%%%%%%%%%%%%%%%%%%%%%%%%%%%%%%%%%%%%%%%%%%%%%%%%%

\subsection{Sublevel sets and linear combinations of real analytic functions} Given a connected, nonempty open set $U \subseteq \R^n$, let $C^{\omega}(U)$ denote the space of real-valued, real analytic functions on $U$. We say $f \in C^{\omega}(U)$ is a \textit{linear combination} of $g_1, \dots, g_m \in C^{\omega}(U)$ if there exist $\mu_1, \dots, \mu_m \in \R$ such that
\begin{equation*}
f(u) = \sum_{j=1}^m \mu_j g_j(u) \qquad \textrm{for all $u \in U$.}
\end{equation*}
The key results of this section roughly state that if $f \in C^{\omega}(-1,1)$ is \textbf{not} a linear combination of $g_1, \dots, g_m \in C^{\omega}(-1,1)$, then for any fixed choice of $\mu_1, \dots, \mu_m \in \C$ the function $t \mapsto f(t) - \sum_{j=1}^m \mu_j g_j(t)$ can only be small on a small set. The precise statements involve analytic families of functions $f(\,\cdot\,;y)$ and $g_j(\,\cdot\,;y)$, $1 \leq j \leq m$, parametrised by $y \in \B^N$.

\begin{proposition}\label{prop: sublevel} Suppose $f \colon (-1,1) \times \B^N \to \R$ and $g_j \colon (-1,1) \times \B^N \to \R$, $1 \leq j \leq m$, are real analytic and such that $f(\,\cdot\,; y)$ is not a linear combination of $g_1(\,\cdot\,; y), \dots, g_m(\,\cdot\,; y)$ for $\cL^{N}$-almost every $y \in \B^N$.

Let $I_{\circ} \subseteq (-1,1)$ and $Y_{\circ} \subset \B^N$ be compact sets containing the origin. Then there exist an exponent $\kappa = \kappa_{\sE}  \in (0,1)$ and a constant $C = C_\sE \geq 1$, which depend only on the ensemble $\sE := \{f, g_1, \dots, g_m, I_{\circ}, Y_{\circ}\}$, such that sublevel set estimate
\begin{equation}\label{eq: sublevel}
 \Big|\Big\{(t; y) \in I_{\circ} \times Y_{\circ} : \big|f(t;y) - \sum_{j=1}^m \mu_j(y) g_j(t;y)\big| < \sigma\Big\}\Big|\leq C \sigma^{\kappa}
\end{equation}
holds uniformly over all $\sigma > 0$ and all choices of measurable functions $\mu_j \colon \B^N \to \R$, $1 \leq j \leq m$. 
\end{proposition}

We also establish a variant of Proposition~\ref{prop: sublevel}, which attains a stronger conclusion under stronger hypotheses. 

\begin{definition} Suppose $g_j \colon (-1,1) \times \B^N \to \R$, $1 \leq j \leq m$, are real analytic and $0 \leq r \leq m$. Define the infinite matrix
\begin{equation}\label{eq: rank hyp}
   \fB(s;y) := (b_{i,j}(s;y))_{\substack{ 0 \leq i < \infty \\ 1 \leq j \leq m}} \qquad \textrm{where} \qquad b_{i,j}(s;y) := \frac{(\partial_t^i g_j)(s;y) }{i!}.
\end{equation}
We say the tuple $(g_1, \dots, g_m)$ is \textit{of constant rank $r$} if 
\begin{equation*}
  \rank \fB(s;y) = r \qquad \textrm{for all $(s, y) \in (-1,1) \times \B^N$.}
\end{equation*}
\end{definition}

With the above definition, and recalling the notation $P_d[f,s]$ from \S\ref{sec: main results}, our second sublevel estimate reads as follows. 

\begin{proposition}\label{prop: sublevel slice}  Let $0 \leq r \leq m$ and suppose $f \colon (-1,1) \times \B^N \to \R$ and $g_j \colon (-1,1) \times \B^N \to \R$, $1 \leq j \leq m$, are real analytic such that:
\begin{enumerate}[a)]
    \item There exists some $d \in \N$ such that $P_d[f(\,\cdot\,; y),s]$ is not a linear combination of $P_d[g_1(\,\cdot\,; y),s], \dots, P_d[g_m(\,\cdot\,; y),s]$ for every $(s, y) \in (-1,1) \times \B^N$;
    \item $(g_1, \dots, g_m)$ is of constant rank $r$.
\end{enumerate}
Let $I_{\circ} \subseteq (-1,1)$ and $Y_{\circ} \subset \B^N$ be compact sets containing the origin. Then there exist an exponent $\kappa = \kappa_{\sE}  \in (0,1)$ and a constant $C = C_\sE \geq 1$, which depend only on the ensemble $\sE := \{f, g_1, \dots, g_m, I_{\circ}, Y_{\circ}\}$, such that sublevel set estimate
\begin{equation}\label{eq: sublevel slice}
    \Big|\Big\{t \in I_{\circ} : \big|f(t;y) - \sum_{j=1}^m \mu_j g_j(t;y)\big| < \sigma\Big\}\Big| \leq C \sigma^{\kappa}
\end{equation}
holds uniformly over all $\sigma > 0$, $y \in Y_{\circ}$ and all scalars $\mu_j \in \R$, $1 \leq j \leq m$. 
\end{proposition}

\begin{remark} The case $r=0$ corresponds to each $g_i \equiv 0$ and $f(\,\cdot\,;y) \not\equiv 0$ for every $y\in \B^N$. In this case, the bound \eqref{eq: sublevel slice} is the sublevel slice bound for $f$ alone which holds under the assumption that $f$ is real analytic and $f(\,\cdot\,;y) \not\equiv 0$ for every $y\in \B^N$. This is already a non-trivial result and can be proved using the Weierstrass Preparation Theorem; the argument below, in the case $r=0$, gives an alternative proof. The example $f(t;y) = y t$ shows the necessity that a) holds for every $y\in \B^N$ and not just for a.e. $y$.
\end{remark}

Clearly, the uniform slice estimate \eqref{eq: sublevel slice} from Proposition~\ref{prop: sublevel slice}, whenever it holds, is stronger than the averaged estimate \eqref{eq: sublevel} from Proposition~\ref{prop: sublevel}. However, it is not possible to prove \eqref{eq: sublevel slice} under hypothesis a) alone.

\begin{example} Consider the case $m = 1$, $N=1$ and the real analytic functions $f(t;y) := t^2$ and $g(t;y) := t^2y^2 + ty^3$. Then $f$ is not a linear combination of $g$. However, there is no uniform sublevel set estimate of the form
\begin{equation*}
    \Big|\Big\{t \in I_{\circ} : \big|f(t;y) - \mu g(t;y)\big| < \sigma\Big\}\Big| \leq C \sigma^{\kappa}.
\end{equation*}
Indeed, for $y \neq 0$, if we simply take $\mu := y^{-2}$, then matters are reduced to estimating the size of the sublevel set $S_{\sigma} := \{ t \in I_{\circ} : |t||y| < \sigma\}$. However, for $\sigma := |y|$, it follows that $S_{\sigma}$ is the whole interval $I_{\circ}$.  
\end{example}

% {\todo I suggest also add the following example for constant rank property.

% The constant rank property in Proposition \ref{prop: sublevel slice} b) is also crucial to deal with some technical issue as shown in the following example.

% Suppose $\vec v_1(t), \vec v_2(t)$, $t\in [-1,1]$ are two 1-parameter families of vectors in $\R^3$. Define $V(t)=\text{span}\{\vec v_1(t),\vec v_2(t)\}$. Suppose $\vec v_0\not\in V(t)$ for all $t$, then we hope a uniform separation 
% \[\dist(\vec v_0, V(t))\gtrsim_{\vec v_0, V} 1\] where the implied constant is independent of $t$. If we assume $\dim V(t)$ are the same for all $t$, then we have the uniform separation by continuity. Otherwise, there is a counterexample. Let $\{\vec e_1, \vec e_2, \vec e_3\}$ be standard basis in $\R^3$. Let $\vec v_0=\vec e_2$, $\vec v_1(t)=\vec e_1$, $\vec v_2(t)=\vec e_1+t\vec e_2+t^2\vec e_3$. We see $\vec v_0\not\in V(t)=\text{span}\{\vec v_1(t),\vec v_2(t)\}$, but $\dist( \vec v_0, V(t) )\le \dist(\vec e_2,\frac{1}{t}(\vec v_2(t)-\vec v_1(t)))\le t\rightarrow 0$ as $t\rightarrow 0$, which means we do not have a uniform  separation.

% }

Proposition~\ref{prop: sublevel} and Proposition~\ref{prop: sublevel slice} can be used to relate the hypotheses of Theorem~\ref{thm: osc improve} to the uniform sublevel set estimates $\bS_{\tK}(\phi, m; \kappa)$ and $\bS_{\tN}(\phi, m; \kappa)$ from Definition~\ref{dfn: hyp K N}. In particular, together they imply Proposition~\ref{prop: hyp K N}. 

\begin{proof}[Proof (of Proposition~\ref{prop: hyp K N})] Let $n \geq 3$ and suppose 
 $\phi \colon \D^n \to \R$ is a real analytic non-degenerate phase.\medskip

 \noindent I) Suppose $\phi$ is translation invariant and satisfies Hypothesis I). Our goal is to show $\bS_{\tK}(\phi; m_{\crit}(n); \kappa)$ holds for some $0 < \kappa \leq 1$. Under the translation invariant hypothesis, $\partial^2_{yy}\phi(x,t;y)$ is independent of $x$; therefore, the sublevel set condition simplifies significantly to the following statement.\medskip

\noindent \underline{\textbf{$\bS_{\tK}(\phi; m_{\crit}(n); \kappa)$:}}   For all measurable $\fP \colon \B^{n-1} \to \mathrm{Mat}(\R, n-1)$, there exist $\alpha$, $\beta \in \cP(n-1,d_{\crit}(n))$ such that
\begin{equation}\label{eq: hyp K N 1}
    \Big|\Big\{ (t;y) \in I_{\phi} \times Y_{\phi} : \Big|\det\Big[ \partial_{yy}^2 \phi(0,t;y) - \fP(y) \Big]_{\alpha, \beta} \Big| < \sigma \Big\}\Big| \leq C_{\phi} \sigma^{\kappa}
\end{equation}
holds uniformly over all $\sigma > 0$. Here the constant $C_{\phi}$ depends only on $\phi$ and, in particular, is independent of $\fP$. \medskip

Since $\phi$ satisfies Hypothesis I), there exist $\alpha$, $\beta \in \cP(n-1, d_{\crit}(n))$ such that $\cL^{n-1}(\cZ_{\mathrm{K}}(\phi; \alpha, \beta)) = 0$, where $\cZ_{\mathrm{K}}(\phi; \alpha, \beta)$ is the set introduced in Definition~\ref{dfn: lin comb K}. Fixing a choice of $\fP$ as above and expanding out the determinant in \eqref{eq: hyp K N 1}, we obtain
\begin{align*}
    \det [\partial_{yy}^2 \phi(0,t;y)]_{\alpha, \beta} +&(-1)^{d_{\rm crit}(n)}\det [\fP(y)]_{\alpha,\beta}\\
    &- \sum_{(\alpha', \beta') \in \cI(\alpha,\beta)} \mu_{\alpha', \beta'}(y) \det[\partial^2_{yy}\phi(0,t;y)]_{\alpha', \beta'},
\end{align*}
where each $\mu_{\alpha', \beta'} \colon \B^{n-1} \to \R$ is a measurable function, specifically a polynomial in the entries of $\fP$.

Given $y\in \B^{n-1} \setminus \cZ_{\mathrm{K}}(\phi; \alpha, \beta)$, the function $\det [\partial^2_{yy}\phi (0,\,\cdot\,;y)]_{\alpha, \beta}$ is not a $C^\omega$-linear combination of the functions $1$ and the $\det[\partial^2_{yy}\phi(0,\,\cdot\,;y)]_{\alpha', \beta'}$ for $(\alpha',\beta')\in \cI(\alpha,\beta)$ (that is, for all $\alpha' \subset \alpha$, $\beta' \subset \beta$ with $0 < |\alpha'| = |\beta'| < d_{\crit}(n)$, according to the notation of Section~\ref{sec: main results}). Indeed, if this failed for some $y_0\in\B^{n-1} \setminus \cZ_{\mathrm{K}}(\phi; \alpha, \beta)$, there would exist $\lambda \in \R$ and $\lambda_{\alpha', \beta'}\in\R$, $(\alpha',\beta')\in \cI(\alpha,\beta)$, such that
\begin{equation*}
    \det[\partial^2_{yy}\phi(0,t;y_0)]_{\alpha, \beta} = \lambda + \sum_{{(\alpha', \beta') \in \cI(\alpha,\beta)}} \lambda_{\alpha', \beta'} \det[\partial^2_{yy}\phi(0,t;y_0)]_{\alpha', \beta'}
\end{equation*}
holds for all $t \in (-1,1)$. Since $\phi(0,0;y) \equiv 0$ in \eqref{eq: trans inv}, setting $t=0$, the above implies that $\lambda = 0$, and hence that $\det\partial^2_{yy}\phi(0,\,\cdot\,;y_0)$ is a linear combination of the functions $\det[\partial^2_{yy}\phi(0,\,\cdot\,;y_0)]_{\alpha', \beta'}$ alone, contradicting the fact that $y_0\notin \cZ_{\mathrm{K}}(\phi; \alpha, \beta)$. Therefore, the hypothesis  $\cL^{n-1}(\cZ_{\mathrm{K}}(\phi; \alpha, \beta)) = 0$ is precisely what is required to invoke Proposition~\ref{prop: sublevel} in order to deduce the existence of some $\kappa = \kappa(\phi) \in (0,1)$ for which \eqref{eq: hyp K N 1} holds. \medskip

\noindent II) Suppose $\phi$ satisfies Hypothesis II). Our goal is to show $\bS_{\tN}(\phi; m_{\crit}(n); \kappa)$ holds for some $0 < \kappa \leq 1$. We recall the statement. \medskip

\noindent \underline{\textbf{$\bS_{\tN}(\phi; m_{\crit}(n); \kappa)$:}} For all $\fD \colon \B^{n-1} \to \mathrm{Mat}(\R, n-1)$ and $\fd \colon \B^{n-1} \to \B^{n-1}$ measurable, there exist $\alpha$, $\beta \in \cP(n-1,d_{\crit}(n))$ such that
\begin{equation}\label{eq: hyp K N 2}
   \sup_{\omega \in \Omega_{\phi}} \Big|\Big\{ t \in I_{\phi}  : \Big|\det\Big [(\partial_{yy}^2 \phi)(\Psi(\omega; t; \fd(\omega)),t;\fd(\omega))\fD(\omega) - I_{n-1} \Big]_{\alpha,\beta}\Big| < \sigma \Big\}\Big| \leq C_{\phi} \sigma^{\kappa}
\end{equation}
holds uniformly over all $\sigma > 0$. Here the constant $C_{\phi}$ depends only on $\phi$ and, in particular, is independent of $\fD$ and $\fd$. \medskip

Since $\phi$ satisfies Hypothesis II), there exist $\alpha$, $\beta \in \cP(n-1, d_{\crit}(n))$ and $d \in \N$ such that $\cZ_{\tN}(\phi; \alpha, \beta; d) = \emptyset$, where $\cZ_{\tN}(\phi; \alpha, \beta; d)$ is the set introduced in Definition~\ref{dfn: lin comb N}. Fix a choice of $\fD$ and $\fd$ as above and $\omega \in \Omega_{\phi}$ and set $y := \fd(\omega)$. As in Definition~\ref{dfn: lin comb N}, let
\begin{equation*}
  g_{\alpha',\beta'}(t;\omega;y) :=  \det[\partial_{yy}^2 \phi(\Psi(\omega;t;y),t;y)]_{\alpha', \beta'}
\end{equation*}
for $(\alpha',\beta') \in \cI(\alpha,\beta)$. Expanding out the determinant in \eqref{eq: hyp K N 2}, we obtain
\begin{equation*}
   (-1)^{d_{\rm crit}(n)} +  \sum_{(\alpha',\beta') \in \cI(\alpha,\beta)} \mu_{\alpha', \beta'}(\omega) g_{\alpha', \beta'}(t;\omega;y)
\end{equation*}
where each $\mu_{\alpha', \beta'} \colon \B^{n-1} \to \R$ is a measurable function which is a certain polynomial in the entries of $\fD$.

By condition a) of Hypothesis II), for any fixed $(\omega; s; y) \in \D^n$, the constant function $1$ is not a linear combination of the polynomials $P_d[g_{\alpha',\beta'}(\,\cdot\,;\omega;y), s]$ for $(\alpha',\beta') \in \cI(\alpha,\beta)$. In addition, condition b) of Hypothesis II) is precisely the remaining rank condition required to invoke Proposition~\ref{prop: sublevel slice} with $f \equiv 1$. This proposition therefore guarantees the existence of some $\kappa = \kappa(\phi) \in (0,1)$ for which \eqref{eq: hyp K N 2} holds. 
\end{proof}

The remainder of this section deals with the proofs of Proposition~\ref{prop: sublevel} and Proposition~\ref{prop: sublevel slice}. The arguments are postponed until \S\S\ref{subsec: lin ind proof 1}-\ref{subsec: lin ind proof 2}. Here we discuss some preliminaries on linear independence of functions and elementary sublevel set bounds.

%%%%%%%%%%%%%%%%%%%%%%%%%%%%%%%%%%%%%%%%%%%%%%%%%%%%%%%%%%%%%%%%%%%%%%%%%%%%%%%%%%%%%%%%%%%%%%%%

%         Linear independence in the real analytic category. 

%%%%%%%%%%%%%%%%%%%%%%%%%%%%%%%%%%%%%%%%%%%%%%%%%%%%%%%%%%%%%%%%%%%%%%%%%%%%%%%%%%%%%%%%%%%%%%%%

\subsection{Linear independence in the real analytic category}%\label{subsec: lin ind} 
Given a connected, nonempty open set $U \subseteq \R^n$, we say $g_1, \dots, g_m \in C^{\omega}(U)$ are \textit{linearly dependent} if there exist $\mu_1, \dots, \mu_m \in \R$ not all zero such that
\begin{equation*}
\sum_{j=1}^m \mu_j g_j(u) = 0 \qquad \textrm{for all $u \in U$.}
\end{equation*}
Otherwise, we say $g_1, \dots, g_m \in C^{\omega}(U)$ are \textit{linearly independent}. 

For a nonempty open interval $I \subseteq \R$ and $g_i \in C^{m-1}(I)$, $1 \leq i \leq m$, we define the \textit{Wronskian} $W(g_1, \ldots, g_m) \in C(I)$ by
\begin{equation*}
    W(g_1, \ldots, g_m) (t) :=
\det 
\begin{bmatrix}
g_1(t) & g_2(t) & \cdots & g_m(t) \\
g_1'(t) & g_2'(t) & \cdots & g_m' (t)\\
\vdots & \vdots &  & \vdots \\
g_1^{(m-1)}(t)& g_2^{(m-1)}(t) & \cdots & g_m^{(m-1)}(t)
\end{bmatrix}.
\end{equation*}
We recall the following classical theorem of B\^ocher~\cite{Bocher1900}.

\begin{theorem}[B\^ocher~\cite{Bocher1900}]\label{thm: Bocher} Let $I \subseteq \R$ be an open interval and $g_1, \dots, g_m \in C^{\omega}(I)$. The following are equivalent:
\begin{enumerate}[a)]
    \item $g_1, \dots, g_m$ are linearly dependent;
    \item The Wronskian $W(g_1, \dots, g_m)$ vanishes identically on $I$.
\end{enumerate}
\end{theorem}

It is trivial to show that if $g_1, \dots, g_m \in C^{m-1}(I)$ are linearly dependent, then $W(g_1, \dots, g_m) \equiv 0$. However, the converse fails in general, so the hypothesis that the $g_j$ are real analytic is essential in B\^ocher's theorem.

The following useful lemma relates linear independence to linear independence of the underlying Taylor polynomials in the finite-dimensional vector space $\R_{\leq d}[T]$ of real univariate polynomials of degree at most $d$. 

\begin{lemma}\label{lem: linear ind} Let $I \subseteq \R$ be a non-trivial open interval and $g_1, \dots, g_m \in C^{\omega}(I)$. Then the following are equivalent:
\begin{enumerate}[a)]
    \item The functions $g_1,\dots, g_m$ are linearly independent.
    \item Given any compact interval $I_{\circ} \subseteq I$, there exists some $d \geq m$ such that the polynomials $P_d[g_1,s], \ldots, P_d[g_m,s]$ are linearly independent for all $s \in I_{\circ}$.
    \item Given any compact interval $I_{\circ} \subseteq I$, there exist some $d \geq m$ and $s \in I_{\circ}$ such that the polynomials $P_d[g_1,s], \ldots, P_d[g_m,s]$ are linearly independent.
\end{enumerate} 
\end{lemma}

We emphasise that the choice of $d \in \N_0$ in b) is uniform over all $s \in I_{\circ}$. 

\begin{proof} It is clear that b) $\Rightarrow$ c), so that it suffices to show a) $\Rightarrow$ b) and c) $\Rightarrow$ a).\medskip

\noindent \textit{c) $\Rightarrow$ a).} Suppose a) fails so that $g_1,\dots, g_m$ are linearly dependent; that is, there exist scalars $\mu_1, \dots, \mu_m \in \R$, not all zero, such that $\mu_1 g_1 + \cdots + \mu_m g_m \equiv 0$ on $I$. Since forming the Taylor polynomial is a linear operation, it follows that $\mu_1P_d[g_1,s] +  \dots + \mu_mP_d[g_m,s] \equiv 0$ for all $d \in \N_0$ and all $s \in I$. This implies that c) fails.\medskip

\noindent \textit{a) $\Rightarrow$ b).} Suppose b) fails for some choice of compact interval $I_{\circ} \subseteq I$ where, without loss of generality, we assume $I_{\circ}$ has positive length. Thus, for each $d \geq m$ there exists some $s_d \in I_{\circ}$ such that the polynomials $P_d[g_1,s_d], \ldots, P_d[g_m,s_d]$ are linearly dependent. By compactness, we can pass to a strictly increasing sequence of integers $(d_k)_{k=1}^{\infty}$ such that $s_{d_k} \to s_{\circ}$ as $k \to \infty$ for some $s_{\circ} \in I_{\circ}$. Without loss of generality, we may assume $d_1 \geq m$.

Define $P_{j,k} := P_{d_k}[g_j; s_{d_k}]$ for each $1 \leq j \leq m$ and $k \in \N$. By the trivial implication in Theorem~\ref{thm: Bocher}, we have
\begin{equation*}
W(P_{1,k}, \dots, P_{m,k}) \equiv 0.
\end{equation*}
We wish to take a limit in $k$ to conclude that $W(g_1, \dots, g_m) \equiv 0$. To do this, we should be slightly careful since the Taylor polynomials $P_{j,k}$ are defined with respect to varying centres. 

Since the $g_j$ are real analytic on $I$, there exists some constant $C_g \geq 1$ such that
\begin{equation*}
   \max_{1 \leq j \leq m} \|g_j^{(\ell)}\|_{L^{\infty}(I_{\circ})} \leq C_g^{\ell+1} \ell! \qquad \textrm{for all $\ell \in \N_0$.}
\end{equation*}
By Taylor's theorem, for $0 \leq \ell \leq m-1$ we have
\begin{equation}\label{eq: lin ind 1}
\max_{1 \leq j \leq m} |g_j^{(\ell)}(t) - P_{j,k}^{(\ell)}(t)| \leq (2 C_g)^{d_k+2}|t-s_{d_k}|^{d_k+1 -\ell} \qquad \textrm{for all $t \in I_{\circ}$}
\end{equation}
whenever $k$ is sufficiently large. 

Let $\delta := 1/(8C_g) > 0$ and choose $K \in \N$ sufficiently large so that $|s_{d_k} - s_{\circ}| < \delta$ and \eqref{eq: lin ind 1} holds for all $k \geq K$. If $t \in I_{\circ}$ satisfies $|t - s_{\circ}| < \delta$, then $|t - s_{d_k}| \leq 2\delta$ for all $k \geq K$ and so \eqref{eq: lin ind 1} implies that
\begin{equation*}
  \max_{1 \leq j \leq m} |g_j^{(\ell)}(t) - P_{j,k}^{(\ell)}(t)| \leq (4C_g)^m 2^{-d_k}    \qquad \textrm{for all $k \geq K$ and $0 \leq \ell \leq m-1$.}
\end{equation*}
In particular, for each $1 \leq j \leq m$ and $0 \leq \ell \leq m-1$, we have $P_{j,k}^{(\ell)}(t) \to g_j^{(\ell)}(t)$ as $k \to \infty$ uniformly over $t \in U$, where $U \subseteq I_{\circ} \cap (s_{\circ}-\delta, s_{\circ}+\delta)$ is a choice of non-trivial open interval. Plugging this into the definition of the Wronskian, we see that
\begin{equation*}
    W(g_1, \dots, g_m) \equiv 0 \ \ \textrm{on $U$} \ \ \textrm{and hence} \ \ W(g_1, \dots, g_m) \equiv 0 \ \ \textrm{on $I$}
\end{equation*}
since $W(g_1, \dots, g_m)$ is analytic.
We now apply Theorem~\ref{thm: Bocher} to conclude that the functions $g_1, \dots, g_m$ are linearly dependent.
\end{proof}

It is useful to reformulate Lemma~\ref{lem: linear ind} in terms of matrices.

\begin{corollary}\label{cor: linear ind} Let $I \subseteq \R$ be a non-trivial open interval and $g_1, \dots, g_m \in C^{\omega}(I)$. Let $b_{i,j}(s) := g_j^{(i)}(s)/i!$ denote the coefficients of the power series expansion of $g_j$ centred at $s \in I$ and 
    \begin{equation*}
        B_{\leq d}(s) := (b_{i,j}(s))_{\substack{0 \leq i \leq d \\ 1 \leq j \leq m}} \in \mathrm{Mat}(\R, (d+1) \times m).
    \end{equation*}
Then the following are equivalent:
\begin{enumerate}[a)]
    \item The functions $g_1,\dots, g_m$ are linearly independent.
    \item Given any compact interval $I_{\circ} \subseteq I$, there exists some $d \geq m$ with the following property. For every $s \in I_{\circ}$ there exists an $m \times m$ minor $\fM_B(s)$ of $B_{\leq d}(s)$ with non-vanishing determinant. 
    \item Given any compact interval $I_{\circ} \subseteq I$, there exist some $d \geq m$ and $s \in I_{\circ}$ with the following property. There exists an $m \times m$ minor $\fM_B(s)$ of $B_{\leq d}(s)$ with non-vanishing determinant.
\end{enumerate} 
\end{corollary}

\begin{proof} This follows directly from Lemma~\ref{lem: linear ind}, since the polynomials $P_d[g_j,s]$ are linearly independent in the vector space $\R_{\le d}[T]$ if and only if their coefficient column vectors $(b_{i,j}(s))_{i=0}^d$ are linearly independent as vectors in $\R^{d+1}$.
\end{proof}

%%%%%%%%%%%%%%%%%%%%%%%%%%%%%%%%%%%%%%%%%%%%%%%%%%%%%%%%%%%%%%%%%%%%%%%%%%%%%%%%%%%%%%%%%%%%%%%%

%         Elementary sublevel set estimates. 

%%%%%%%%%%%%%%%%%%%%%%%%%%%%%%%%%%%%%%%%%%%%%%%%%%%%%%%%%%%%%%%%%%%%%%%%%%%%%%%%%%%%%%%%%%%%%%%%

\subsection{Elementary sublevel set estimates} A key ingredient in the proof of Proposition~\ref{prop: sublevel} is the following elementary inequality, which is valid in the $C^k$ category. 

\begin{lemma}[van der Corput estimate]\label{lem: van der Corput} For all $k \geq 1$ there exists a constant $C_k$ such that the following holds. If $I \subseteq \R$ is an open interval and $u \in C^k(I)$ satisfies $u^{(k)}(t) \geq 1$ for all $t \in I$, then
\begin{equation*}
    |\{t \in I : |u(t)| < \sigma \}| \leq C_k \sigma^{1/k} \qquad \textrm{for all $\sigma > 0$.}
\end{equation*}
The above bound holds for $k=0$ (so that $u(t) \geq 1$ on $I$) when $\sigma<1$. In this case, the left-hand side and right-hand side are both equal to zero.
\end{lemma}

The proof of Lemma~\ref{lem: van der Corput} is implicit in the proof of the oscillatory integral estimate in \cite[Proposition 2, p. 332]{Stein1993} (see also \cite[Lemma 3.3]{Christ1985} for an explicit treatment). 

The significance of Lemma~\ref{lem: van der Corput} is that it provides a \textit{uniform} estimate over all choices of interval $I$ and all choices of function $u$ satisfying $u^{(k)}(t) \geq 1$. This uniformity will play an essential role later in the proof of Proposition~\ref{prop: sublevel}. For now, we note that Lemma~\ref{lem: van der Corput} readily implies the following non-uniform bound for analytic functions of several variables.  

\begin{corollary}\label{cor: van der Corput} Suppose $f \in C^{\omega}(\B^n)$ is not identically zero and let $Y \subseteq \B^n$ be a compact, convex set. Then there exists an exponent $\kappa = \kappa_{f,Y} \in (0,1)$ and $C = C_{f,Y} \geq 1$ such that
\begin{equation}\label{eq: van der Corput}
    |\{y \in Y : |f(y)| < \sigma \}| \leq C \sigma^{\kappa} \qquad \textrm{for all $\sigma > 0$.}
\end{equation}
\end{corollary}

Note, in contrast with Lemma~\ref{lem: van der Corput}, here the exponent $\kappa$ and constant $C$ are non-uniform, in the sense that they are allowed to depend on $f$ and $Y$.

\begin{proof}[Proof (of Corollary~\ref{cor: van der Corput})] Since $f$ is not identically zero, by a compactness argument we can find a finite family of open balls $\{B_j\}_{j \in \cJ}$ and multiindices $(\alpha_j)_{j \in \cJ} \subseteq \N_0^n$ such that
\begin{equation*}
    Y \subseteq \bigcup_{j \in \cJ} B_j,  \quad B_j \subseteq \B^n\quad \textrm{and} \quad |\partial^{\alpha_j}_y f(y)| \geq c > 0 \quad \textrm{for all $y \in B_j$ and all $j \in \cJ$.}
\end{equation*}
Here the constant $c > 0$ depends on $f$ and $Y$. We may now apply a slicing argument on each ball $B_j$ to deduce that
\begin{equation}\label{eq: slicing}
    |\{ y \in B_j : |f(y)| < \sigma \}| \lesssim_{f,Y} \sigma^{1/|\alpha_j|}  \qquad \textrm{for all $\sigma > 0$ and $j \in \cJ$.}
\end{equation}
Indeed, \eqref{eq: slicing} follows by applying Lemma~\ref{lem: van der Corput} in one variable (after a suitable rotation of coordinates), as in \cite[Lemma 3.4]{Christ1985}. Summing over the finite indexing set $\cJ$ and choosing $\kappa$ to be the minimum of the exponents $1/|\alpha_j|$ concludes the proof. 
\end{proof}

By forming the intersection of the sublevel sets in \eqref{eq: van der Corput} over all $\sigma > 0$, we deduce the following well-known result.

\begin{lemma}\label{cor: meas zero} If $f \in C^{\omega}(\B^n)$ is not identically zero, then the zero set $Z(f) := \{x \in \B^n : f(x) = 0\}$ has $\cL^n$-measure zero. 
\end{lemma}

Finally, in order to apply Lemma~\ref{lem: van der Corput}, we shall make use of the following elementary inequality relating the size of the coefficients of a polynomial to pointwise lower bounds for its derivatives. 

\begin{lemma}\label{lem: poly bound} Let $P \in \R_{\leq d}[T]$ be a real polynomial of degree $d$ and $s \in [-1,1]$. There exists some $0 \leq k \leq d$ such that
\begin{equation*}
    |P^{(k)}(t)| \gtrsim_d \big(\sum_{i=0}^d |P^{(i)}(s)|^2 \big)^{1/2} \qquad \textrm{for all $t \in [-1,1]$.}
\end{equation*}
Here the implied constant depends only on the degree of $P$ and, in particular, not on the size of the coefficients. 
\end{lemma}

Lemma~\ref{lem: poly bound} can be proved by pigeonholing in the size of the terms of the Taylor expansion of $P$ around $s$. Alternatively, the result follows from the last displayed equation in \cite[p.124]{CRW1998} and equivalence of norms on the finite-dimensional vector space $\R_{\leq d}[T]$.

%%%%%%%%%%%%%%%%%%%%%%%%%%%%%%%%%%%%%%%%%%%%%%%%%%%%%%%%%%%%%%%%%%%%%%%%%%%%%%%%%%%%%%%%%%%%%%%%

%         Proof of Proposition~\ref{prop: sublevel}

%%%%%%%%%%%%%%%%%%%%%%%%%%%%%%%%%%%%%%%%%%%%%%%%%%%%%%%%%%%%%%%%%%%%%%%%%%%%%%%%%%%%%%%%%%%%%%%%

\subsection{Proof of Proposition~\ref{prop: sublevel}}\label{subsec: lin ind proof 1} With the above preliminaries in hand, we can turn to the proof of the first main result of this section.

\begin{proof}[Proof (of Proposition~\ref{prop: sublevel})] Throughout the argument, we shall say an object is \textit{$\sE$-admissible} if it depends only on $\sE = \{f, g_1, \dots, g_m, I_{\circ}, Y_{\circ}\}$. Fix measurable functions $\mu_i \colon \B^N \to \R$, $1 \leq i \leq m$, and define
\begin{equation*}
  F_{\mu} \colon (-1,1) \times \B^N \to \R, \qquad  F_{\mu}(t;y) := f(t;y) - \sum_{i=1}^m \mu_i(y) g_i(t;y). 
\end{equation*}
Our goal is to show that there exists an $\sE$-admissible exponent $\kappa \in (0,1)$ such that
\begin{equation*}
|E_{\mu}(\sigma)| \lesssim_{\sE} \sigma^{\kappa} \quad \textrm{for all $\sigma > 0$,} \quad \textrm{where} \quad E_{\mu}(\sigma) := \{(t; y) \in I_{\circ} \times Y_{\circ} : |F_{\mu}(t;y)| < \sigma\}.
\end{equation*}
Clearly we may assume $0 < \sigma < 1$. We carry out the argument in steps.\medskip

%%%%%%%%%%%%%%%%%%%%%%%%%%%%%%%%%%%%%%%%%%%%%%%%%%%%%%%%%%%%%%%%%%%%%%%%%%%%%%%%%%%%%%%%%%%%%%%%

%                      Step 1: Power series expansion and basic definitions.

%%%%%%%%%%%%%%%%%%%%%%%%%%%%%%%%%%%%%%%%%%%%%%%%%%%%%%%%%%%%%%%%%%%%%%%%%%%%%%%%%%%%%%%%%%%%%%%%

\noindent \underline{Step 1: Power series expansion and basic definitions.} Since the functions $f$ and $g_j$ are real analytic, given $s \in (-1,1)$ we may write
    \begin{equation}\label{eq: power series exp}
        f(t;y) = \sum_{i=0}^{\infty} a_i(s;y) (t-s)^i \quad \textrm{and} \quad g_j(t;y) = \sum_{i=0}^{\infty} b_{i, j}(s;y) (t-s)^i
    \end{equation}
    for $1 \leq j  \leq m$ and all $t \in (-1,1)$ belonging to a suitable neighbourhood of $s$. Note that the coefficient functions $a_i(s;\,\cdot\,)$, $b_{i,j}(s;\,\cdot\,)$ are real analytic functions on $\B^N$. For any index set $\cJ \subseteq \{1, \dots, m\}$, let $B_{\cJ}$ denote the matrix-valued function on $I \times \B^N$ given by 
    \begin{equation*}
        B_{\cJ}(s; y) := \big(b_{i,j}(s; y)\big)_{\substack{0 \leq i < \infty \\ j \in \cJ}} \qquad \textrm{for $y \in \B^N$}.
    \end{equation*}
    If $\cJ = \{1, \dots, m\}$, then, as in \eqref{eq: rank hyp}, we write $\fB$ for $B_{\cJ}$.
    
    There exists a maximal value of $0 \leq r \leq m$ with the following property. For some $y_{\star} \in \B^N$ there exists a choice of index set $\cJ(y_{\star}) \subseteq \{1, \dots, m\}$ with $\#\cJ(y_{\star}) = r$ such that the set of functions 
    \begin{equation*}
        \big\{ g_j(\,\cdot\,;y_{\star}) : j \in \cJ(y_{\star})\big\} 
    \end{equation*}
    is linearly independent. The case $r=0$ corresponds to each $g_j \equiv 0$ in which case we conclude $f \not\equiv 0$ by hypothesis and therefore the sublevel set bound \eqref{eq: sublevel} in Proposition \ref{prop: sublevel} follows from Corollary \ref{cor: van der Corput}. Hence we may assume $r\ge 1$.
    
    Fix $r$ and $y_{\star}$ with these properties and, by relabelling if necessary, assume without loss of generality that $\cJ(y_{\star}) = \{1, \dots, r\}$. Let 
    \begin{equation*}
       \cG(y) := \big\{ g_j(\,\cdot\,;y) : 1 \leq j \leq r \big\} \qquad \textrm{for all $y \in \B^N$}
    \end{equation*}
    and $B(s;y) := B_{\cJ(y_{\star})}(s;y)$. 

     Apply Corollary~\ref{cor: linear ind} to the collection of functions $\cG(y_{\star})$, the interval $I_{\circ}$ and $s=0$ (see part b). This guarantees the existence of an $r \times r$ minor $\fM_B(\,\cdot\,)$ of $B(0;\,\cdot\,)$ such that $\fM_B(y_{\star})$ has non-vanishing determinant. We emphasise that the choice of rows forming the minor $\fM_B(y)$ is independent of $y$.

     The determinant $\det \fM_B(\,\cdot\,)$ is an analytic function on $\B^N$. Since this function does not vanish identically, by Lemma~\ref{cor: meas zero} its zero set is of $\cL^N$ measure zero. Combining this with Corollary~\ref{cor: linear ind}, we have
    \begin{equation}\label{eq: sublevel 1}
       \cL^N(Y_{\mathrm{dep}}) = 0 \quad \textrm{for} \quad Y_{\mathrm{dep}} := \big\{ y \in \B^N : \cG(y) \textrm{ is a linearly dependent set}\big\}.
    \end{equation}
      
    By the maximality of $r$, for any $r+1 \leq k \leq m$ and any $y \in \B^N\setminus Y_{\mathrm{dep}}$, the function $g_k(\,\cdot\,;y)$ is a linear combination of the functions belonging to $\cG(y)$. We may therefore find an $\mathrm{Mat}(\R, r \times m)$-valued function $M$ defined on $\B^N\setminus Y_{\mathrm{dep}}$ and with real analytic entries such that
    \begin{equation}\label{eq: sublevel 0}
        \fB(s;y) = B(s;y)M(y) \qquad \textrm{for all $y \in \B^N\setminus Y_{\mathrm{dep}}$ and all $s \in (-1,1)$.}
    \end{equation}
  Indeed, by the above discussion, there exist functions $c_j^k \colon \B^N \setminus Y_{\mathrm{dep}} \to \R$ such that
  \begin{equation*}
    g_k(t;y) = c_1^k(y) g_1(t;y) + \cdots + c_r^k(y) g_r(t;y) \quad \textrm{for all $r+1 \le k \le m$, $y \in \B^N \setminus Y_{\mathrm{dep}}$.}  
  \end{equation*}
   Thus, by writing both sides of the above identity in terms of the power series \eqref{eq: power series exp}, we see that $M(y)$ can be expressed in block form as
    \begin{equation*}
    M(y) =
    \begin{bmatrix}
        I_{r\times r} & \fC(y)
    \end{bmatrix}
    \qquad \textrm{where} \qquad
    \fC(y) := \begin{bmatrix}
            c_1^{r+1}(y) & \cdots & c_1^m(y) \\
            \vdots &  & \vdots \\
            c_r^{r+1}(y) & \cdots & c_r^m(y)  
        \end{bmatrix},
    \end{equation*}
    for $I_{r \times r}$ the $r \times r$ identity matrix.
    
    Having obtained an expression for $M$, it remains to show that the  entries are real analytic. Fix $y \in \B^N \setminus Y_{\mathrm{dep}}$ and take $s=0$ in Corollary \ref{cor: linear ind} part b) to find an $r\times r$ minor $\fM(y)$ of $B(0;y)$ with nonzero determinant. Therefore $\fG(y) := \fM(y) \fC(y)$ is an $r \times (m-r)$ submatrix of $\fB(0;y)$. The entries of $\fG(y)$ consist of various derivatives of $g_k$ for $r+1 \leq k \leq m$ and hence are analytic in $y$. The fact that the determinant of $\fM(y)$ is not zero (and hence nonzero in a neighbourhood of $y$) implies by Cramer's rule that the entries of the inverse of $\fM$ are analytic in a neighbourhood of $y$. This shows the entries of $\fC$, and therefore the entries of $M$, are analytic near $y$.
    
 For each fixed $y \in \B^N$, the function $F_{\mu}(\,\cdot\,;y)$ is real analytic and so for each $s\in (-1,1)$ can be expressed in terms of a power series
        \begin{equation}\label{eq: sublevel 2}
          F_{\mu}(t;y) = \sum_{i = 0}^{\infty} c_{\mu,i}(s;y) (t-s)^i
        \end{equation}
for all $t \in (-1,1)$ belonging to a suitable neighbourhood of $s$. Here the functions $c_{\mu,i}(s;\,\cdot\,) \colon \B^N \to \R$ are merely measurable. We relate the vector-valued function $\vec{c}_{\mu} := (c_{\mu, i})_{i=0}^{\infty}$ to the coefficients $a_i$ and $b_{i,j}$ using the augmented matrix 
   \begin{equation*}
    A(s;y) := \begin{bmatrix}
            a_0(s;y) & b_{0, 1}(s;y) & \cdots & b_{0, r}(y) \\
            a_1(s;y) & b_{1, 1}(s;y) & \cdots & b_{1, r}(s;y) \\
            \vdots & \vdots & & \vdots 
        \end{bmatrix} \qquad \textrm{for $y \in \B^N$.}
    \end{equation*}
In particular, writing $\vec{\mu} := (\mu_j)_{j=1}^m$, by \eqref{eq: sublevel 0} we have the pointwise identity
\begin{equation}\label{eq: sublevel 3}
            \vec{c}_{\mu}(s;y) = A(s;y) W_{\mu}(y) \qquad \textrm{where} \qquad 
            W_{\mu}(y) := 
            \begin{bmatrix}
                1 \\
                -M(y)\vec{\mu}(y) 
            \end{bmatrix}  \in \R^{r+1}
        \end{equation}
for all $y \in \B^N\setminus Y_{\mathrm{dep}}$. In the analysis below, we will consider various $(r+1)\times (r+1)$ minors $\fM_{A,{\underline{d}}}(s,y)$ of the augmented matrix $A(s,y)$. These minors are of the form
\begin{equation*}
     \fM_{A,{\underline{d}}}(s;y) =  \begin{bmatrix}
            a_{d_1}(s;y) & b_{d_1, 1}(s;y) & \cdots & b_{ d_1,  r}(s;y) \\
            a_{d_2}(s;y) & b_{d_2, 1}(s; y) & \cdots & b_{d_2, r}(s; y) \\
            \vdots & \vdots & & \vdots \\
            a_{d_{r+1}}(s; y) & b_{d_{r+1},1}(s; y) & \cdots & b_{d_{r+1}, r}(s; y)
        \end{bmatrix}
    \end{equation*}
where the indices ${\underline{d}} = (d_1, d_2, \ldots, d_{r+1})$ satisfy $0 \le d_1 < d_2 < \cdots < d_{r+1}$.
\medskip
 
%%%%%%%%%%%%%%%%%%%%%%%%%%%%%%%%%%%%%%%%%%%%%%%%%%%%%%%%%%%%%%%%%%%%%%%%%%%%%%%%%%%%%%%%%%%%%%%%

%                      Step 2: Removing the singular set.

%%%%%%%%%%%%%%%%%%%%%%%%%%%%%%%%%%%%%%%%%%%%%%%%%%%%%%%%%%%%%%%%%%%%%%%%%%%%%%%%%%%%%%%%%%%%%%%%

\noindent \underline{Step 2: Removing the singular set.} In light of \eqref{eq: sublevel 1} and the hypotheses of the proposition, we may find some $y_{\circ} \in \B^N$ such that the following properties hold:
    \begin{enumerate}[i)]
        \item $\cG(y_{\circ})$ is a linearly independent set of functions;
        \item $f(\,\cdot\,; y_{\circ})$ does not lie in the linear span of $\cG(y_{\circ})$.
    \end{enumerate}
    Indeed, i) and ii) will in fact hold for $\cL^N$ almost every choice of $y_{\circ} \in \B^N$. Properties i) and ii) ensure that the set $\{f(\,\cdot\,; y_{\circ})\} \cup \cG(y_{\circ})$ is linearly independent. Thus, Corollary~\ref{cor: linear ind} implies that there exists some admissible $d \geq r+1 \geq 2$ (recall we are assuming $r\geq 1$) such that the following holds. For all $s \in I_{\circ}$, there is some choice of integer indices ${\underline{d}}(s) = (d_1(s), \ldots, d_{r+1}(s))$
satisfying $0 \leq d_1(s) < d_2(s) < \cdots < d_{r+1}(s) \leq d$ and some $(r+1) \times (r+1)$ minor $\fM_{A,{\underline{d}}(s)}(s,\,\cdot\,)$ of the augmented matrix $A(s,\,\cdot\,)$
    of the form described above whose determinant does not vanish identically as a function of $y$ (in particular, it is non-zero at $y = y_{\circ}$). We emphasise that, whilst the row indices $d_j = d_j(s)$ may depend on $s$, the upper bound $d$ for these indices is uniform over all $s \in I_{\circ}$.

Given $0<\sigma < 1$ and $s\in (-1,1)$, consider the \textit{singular set} 
\begin{equation}\label{eq: sublevel 3.5}
    Y_{\mathrm{sing}}(s; \sigma) := \big\{ y \in Y_{\circ} : |\det \fM_{A, {\underline{d}}(s)}(s; y)| < \sigma^{1/2d} \big\} \cup Y_{\mathrm{dep}}.
\end{equation}
We first note that the measure of $Y_{\mathrm{dep}}$ is zero by \eqref{eq: sublevel 1}. For the sublevel set, we
would like for this to be uniformly (in $s$) small with respect to $\sigma$. 
However, the following example shows this is not possible.

\begin{example} Consider the example $f(t;y) := t^2$ and $g(t;y) := y t$ where $N=1$ and $m=r=1$. The choice of integer indices ${\underline{d}}(s) = (d_1(s), d_2(s)) = (0,1)$ for all $s\in [-1/2, 1/2] \setminus\{0\}$ has the property that the $2\times 2$ minor $\fM_{A,{\underline{d}}(s)}(s;y)$ has determinant $-s^2y$, which does not vanish identically as a function of $y$. Moreover, for all $\sigma \ll 1$, we have
\begin{equation*}
 \sup_{s\in [-1/2, 1/2]} \, |Y_{\mathrm{sing}}(s; \sigma)|  \simeq_{\sE} 1,   
\end{equation*}
which is \textbf{not} small as $\sigma \to 0$.
\end{example}

Nevertheless, we will show that the measure of $Y_{\mathrm{sing}}(s; \sigma)$ is small outside a small set in $s$. In fact,
we claim there is an 
exponent $\kappa_1 \in (0,1)$ depending only on the $\sE$-admissible data and a set $E \subseteq I_{\circ}$ with measure $|E| \lesssim_{\sE} \sigma^{\kappa_1}$
such that
\begin{equation}\label{eq: sublevel 3.6}
   \sup_{s \in I_{\circ}\setminus E}|Y_{\mathrm{sing}}(s; \sigma)| \lesssim_{\sE} \ \sigma^{\kappa_1}.
\end{equation}
To see this, let 
$$
\cD := \bigl\{{\underline{d}}= (d_1,\ldots, d_{r+1}): 0 \le d_1 < \cdots < d_{r+1} \le d \ \ {\rm and} \ \det \fM_{A,{\underline{d}}} \not\equiv 0 \bigr\}
$$
and note that by our choice of integer indices above, 
$\{{\underline{d}}(s): s \in I_{\circ}\} \subseteq \cD$. 
For each ${\underline{d}} \in \cD$, the map $(s;y) \mapsto \det  \fM_{A,{\underline{d}}}(s;y)$ is a nonzero real analytic function and so Corollary~\ref{cor: van der Corput} implies there is an exponent $\rho_{{\underline{d}}} \in (0,1)$ depending only on the $\sE$-admissible data such that
$$
|\{(s,y) \in I_{\circ} \times Y_{\circ} : |{\rm det} \, \fM_{A,{\underline{d}}}(s;y)| < \sigma^{1/2d} \}| \ \lesssim_{\sE} \ \sigma^{\rho_{{\underline{d}}}}.
$$
Furthermore, note that
\begin{align*}
    \big\{ (s,y) \in I_{\circ} \times Y_{\circ} :  |\det\, &\fM_{A, {\underline{d}}(s)}(s; y)| < \sigma^{1/2d} \big\} \\
    &\subseteq \bigcup_{{\underline{d}} \in \cD} \bigl\{ (s,y) \in I_{\circ} \times Y_{\circ} : |{\rm det} \, \fM_{A,{\underline{d}}}(s;y)| < \sigma^{1/2d} \bigr\}
\end{align*}
and so 
\begin{equation*}
   \bigl|\bigl\{ (s,y) \in I_{\circ} \times Y_{\circ} : |\det \fM_{A, {\underline{d}}(s)}(s; y)| < \sigma^{1/2d} \bigr\}\bigr|  \ \lesssim_{\sE} \ \sigma^{2 \kappa_1} 
\end{equation*}
where $\kappa_1 := (1/2)\min_{{\underline{d}}\in \cD} \rho_{{\underline{d}}}$ \, (recall $\sigma<1$).
Let 
\begin{equation*}
 E := \{s \in I_{\circ}: |Y_{\mathrm{sing}}(s; \sigma)| > C \sigma^{\kappa_1}\}.    
\end{equation*}
Then
\begin{align*}
    \sigma^{2\kappa_1}  &\gtrsim_{\sE}  \bigl|\big\{ (s,y) \in I_{\circ} \times Y_{\circ} : |\det \fM_{A, {\underline{d}}(s)}(s; y)| < \sigma^{1/2d} \big\}\bigr|  \\
&\geq  \int_E |Y_{\mathrm{sing}}(s; \sigma)| \, \ud s  \\
&\geq C \sigma^{\kappa_1} |E|,
\end{align*}
which implies $|E| \lesssim_{\sE} \sigma^{\kappa_1}$ 
and \eqref{eq: sublevel 3.6} holds for the above choice of $E$, as claimed. 

Recall the set $E_\mu(\sigma)$ defined right above Step 1. We split $E_{\mu}(\sigma) = E^1_{\mu}(\sigma) \cup E^2_{\mu}(\sigma)$, where $E^1_{\mu}(\sigma) = \{(t,y) \in E_{\mu}(\sigma) : t \in I_{\circ}\setminus E) \}$, so that
$$
|E^2_{\mu}(\sigma)|  \le  |E| |Y_{\circ}| \ \lesssim_{\sE} \ \sigma^{\kappa_1}.
$$
It remains to treat $E^1_{\mu}(\sigma)$.

For each $s \in (-1,1)$, we contain the sublevel set $E^1_{\mu}(\sigma)$ within the disjoint union 
\begin{equation}\label{eq: sublevel 3.54}
  E^1_{\mu}(\sigma) \subseteq E_{\mu, \mathrm{sing}}(s; \sigma)  
 \cup E_{\mu, \mathrm{main}}(s; \sigma) 
\end{equation}
where
\begin{align}\label{eq: sublevel 3.55}
 E_{\mu, \mathrm{sing}}(s;\sigma) &:= \{(t;y) \in (I_{\circ} \setminus E) \times Y_{\circ} : y \in Y_{\mathrm{sing}}(s;\sigma) \}, \\
 \nonumber
 E_{\mu, \mathrm{main}}(s;\sigma) &:= E^1_{\mu}(\sigma) \setminus E_{\mu, \mathrm{sing}}(\sigma). 
\end{align}

Below, we shall use \eqref{eq: sublevel 3.6} to give a suitable bound for $|E_{\mu, \mathrm{sing}}(s;\sigma)|$. It then remains to estimate the main contribution $|E_{\mu, \mathrm{main}}(s;\sigma)|$.\medskip

%%%%%%%%%%%%%%%%%%%%%%%%%%%%%%%%%%%%%%%%%%%%%%%%%%%%%%%%%%%%%%%%%%%%%%%%%%%%%%%%%%%%%%%%%%%%%%%%

%                      Step 3: Bounds for the Taylor polynomial.

%%%%%%%%%%%%%%%%%%%%%%%%%%%%%%%%%%%%%%%%%%%%%%%%%%%%%%%%%%%%%%%%%%%%%%%%%%%%%%%%%%%%%%%%%%%%%%%%

\noindent \underline{Step 3: Bounds for the Taylor polynomial.} Fix $s \in I_{\circ}$ and for $y \in \B^N$, let $P_{\mu}(\,\cdot\,;s;y)$ denote the order $d$ Taylor polynomial of $F_{\mu}(\,\cdot\,;y)$ centred at $s$. Using the notation from \eqref{eq: sublevel 2}, we have
\begin{equation*}
    P_{\mu}(t;s;y) = \sum_{i=0}^d c_{\mu,i}(s;y) (t-s)^i.
\end{equation*}
Writing $\tilde{c}_{\mu,\ell} := c_{\mu, d_\ell(s)}$ for $1 \leq \ell \leq r+1$, we claim that  
\begin{equation}\label{eq: sublevel 3.75}
\big(\sum_{\ell=1}^r|\tilde{c}_{\mu,\ell}(s; y)|^2 \big)^{1/2} \gtrsim_{\sE} |\det \fM_{A,{\underline{d}}(s)}(s; y)|\big( 1 + |M(y)\vec{\mu}|\big) \quad \textrm{for all $y \in Y_{\circ} \setminus Y_{\mathrm{dep}}$.}
\end{equation}
Once this is established, Lemma~\ref{lem: poly bound} implies that for each $y \in Y_{\circ}\setminus Y_{\mathrm{dep}}$ there exists some $0 \leq k = k_{s,y} \leq d$ such that 
\begin{equation*}
    |\partial_t^k P_{\mu}(t;s;y)| \gtrsim_{\sE} |\det \fM_{A,{\underline{d}}(s)}(s;y)|\big( 1 + |M(y)\vec{\mu}|\big) \qquad \textrm{for all $t \in I_{\circ}$.}
\end{equation*}
In particular, recalling the definition of $Y_{\mathrm{sing}}(s;\sigma)$ from \eqref{eq: sublevel 3.5}, this implies that
\begin{equation}\label{eq: sublevel 4}
    |\partial_t^k P_{\mu}(t;s;y)| \gtrsim_{\sE} \sigma^{1/2d} \big( 1 + |M(y)\vec{\mu}|\big), \quad t \in I_{\circ},\, y \in Y_{\circ} \setminus  Y_{\mathrm{sing}}(s;\sigma). 
\end{equation}
Below we shall use this derivative bound, in conjunction with Lemma~\ref{lem: van der Corput}, to obtain favourable sublevel set estimates. 

Turning to the proof of \eqref{eq: sublevel 3.75}, it follows from the definition of $\fM_{A,{\underline{d}}(s)}(s;y)$ that $\tilde{c}_{\mu}(s;y) = \fM_{A,{\underline{d}}(s)}(s;y) W_{\mu}(y)$, where $\tilde{c}_{\mu} = (\tilde{c}_{\mu,\ell})_{\ell = 1}^{r+1}$. Hence 
\begin{equation*}
  W_{\mu}(y) =  [\fM_{A,{\underline{d}}(s)}(s;y)]^{-1} \tilde{c}_{\mu}(s;y)
\ \  {\rm and \ so} \ \
|\det \fM_{A,{\underline{d}}(s)}(s;y)||W_{\mu}(y)|  \lesssim_{\cE}  
|\tilde{c}_{\mu}(s;y)|   
\end{equation*}
by Cramer's rule and noting that the entries of $\fM_{A,{\underline{d}}(s)}(s;y)$ are bounded by a constant depending only on $\cE$. This clearly implies the claimed estimate \eqref{eq: sublevel 3.75}.
\medskip

%%%%%%%%%%%%%%%%%%%%%%%%%%%%%%%%%%%%%%%%%%%%%%%%%%%%%%%%%%%%%%%%%%%%%%%%%%%%%%%%%%%%%%%%%%%%%%%%

%                      Step 4: Controlling the remainder.

%%%%%%%%%%%%%%%%%%%%%%%%%%%%%%%%%%%%%%%%%%%%%%%%%%%%%%%%%%%%%%%%%%%%%%%%%%%%%%%%%%%%%%%%%%%%%%%%

\noindent \underline{Step 4: Controlling the remainder.} We now turn to controlling the remainder term
\begin{equation*}
    R_{\mu}(t;s;y) := F_{\mu}(t;y) - P_{\mu}(t;s;y) = \sum_{i = d + 1}^{\infty} c_{\mu,i}(s;y) (t-s)^i. 
\end{equation*}
In particular, we seek upper bounds for the derivatives $\partial_t^k R_{\mu}(t;s;y)$ for $0 \leq k \leq d$.

In light of \eqref{eq: sublevel 3}, the coefficient $c_{\mu,i}(s;y)$ corresponds precisely to the $i$th entry of $A(s;y) W_{\mu}(y)$. If we write
\begin{equation*}
\Sigma_{d,k}(s;y) := \sum_{i=d+1}^{\infty} \Big(|a_i(s;y)| + \sum_{j=1}^r  |b_{i,j}(s;y)|\Big) k! \binom{i}{k}|t-s|^{i-k},
\end{equation*}
then it follows from Cauchy--Schwarz and the embedding $\ell^1 \hookrightarrow \ell^2$ that 
\begin{equation*}
 |\partial^k_t R_{\mu}(t;s;y)| \leq |W_{\mu}(y)| \Sigma_{d,k}(s;y). 
\end{equation*}
Consequently, there exists some admissible constant $c_{\sE} > 0$ such that 
\begin{equation}\label{eq: sublevel 5}  
    |\partial^k_t R_{\mu}(t;s;y)| \lesssim_{\sE} \big(1 + |M(y) \vec{\mu}|\big)|t-s| \qquad \textrm{for $|t - s| < c_{\sE}$ and $0 \leq k \leq d$,}
\end{equation}
 where the inequality holds with an $\sE$-admissible constant uniformly over all $(s;y) \in I_{\circ}  \times Y_{\circ}$. Indeed, by a standard characterisation of real analytic functions of several variables (see, for example, \cite[p.24, p.281]{Hormander1990}), there exists some admissible $C_{\sE} \geq 1$ such that
\begin{equation*}
    |a_i(s;y)| + \sum_{j=1}^r  |b_{i,j}(s;y)|  \leq  C_{\sE}^{i+1} \qquad \textrm{for all $(s;y) \in I_{\circ} \times Y_{\circ}$ and all $i \in \N_0$} 
\end{equation*}
and so 
\begin{equation*}
\Sigma_{d,k}(s;y) 
\leq C_{\sE}^{k+2} \Big(
\sum_{i = d+1}^{\infty} \big(C_{\sE}|t-s|\big)^{i-k -1} i^k \Big)|t-s|
\lesssim_{\sE} |t-s|
\end{equation*}
whenever $|t-s| \leq c
_{\sE} := 1/(2C_{\sE})$.
\medskip

%%%%%%%%%%%%%%%%%%%%%%%%%%%%%%%%%%%%%%%%%%%%%%%%%%%%%%%%%%%%%%%%%%%%%%%%%%%%%%%%%%%%%%%%%%%%%%%%

%                      Step 5: Concluding the argument.

%%%%%%%%%%%%%%%%%%%%%%%%%%%%%%%%%%%%%%%%%%%%%%%%%%%%%%%%%%%%%%%%%%%%%%%%%%%%%%%%%%%%%%%%%%%%%%%%

\noindent \underline{Step 5: Concluding the argument.} 
Let $\fJ = \{J\}$ be a decomposition of $I_{\circ}$ into essentially disjoint intervals of equal length (we shall fix the common length later in the argument) and denote by $\fJ'$
the set of $J\in \fJ$ such that $J \cap (I_{\circ}\setminus E) \neq \emptyset$. For each $J \in \fJ'$, choose some $s_J \in J \cap (I_{\circ} \setminus E)$. Applying the partition \eqref{eq: sublevel 3.54} for each $s_J$, we may write
\begin{align}
\nonumber
|E^1_{\mu}(\sigma)| &= \sum_{J \in \fJ'} |E^1_{\mu}(\sigma) \cap (J \times \B^N)| \\
\label{eq: sublevel 5.5}
&= \sum_{J \in \fJ'} |E_{\mu, \mathrm{main}}(s_J;\sigma) \cap (J \times \B^N)| +  \sum_{J \in \fJ'} |E_{\mu, \mathrm{sing}}(s_J;\sigma) \cap (J \times \B^N)|  .
\end{align}
Recalling the definition from \eqref{eq: sublevel 3.55} and the bound \eqref{eq: sublevel 3.6}, we have
\begin{equation*}
    \sum_{J \in \fJ'}|E_{\mu, \mathrm{sing}}(s_J; \sigma)\cap (J \times \B^N)| \lesssim \sigma^{\kappa_1} \sum_{J \in \fJ'} |J| \lesssim \sigma^{\kappa_1}.
\end{equation*}
Combining this with \eqref{eq: sublevel 5.5}, we obtain
\begin{equation}\label{eq: sublevel 6}
   |E^1_{\mu}(\sigma)| \lesssim \sum_{J \in \fJ'}|E_{\mu, \mathrm{main}}(s_J; \sigma)\cap (J \times \B^N)| + \sigma^{\kappa_1}.
\end{equation}
It remains to estimate the main term on the right-hand side of \eqref{eq: sublevel 6}. Fixing $J \in \fJ'$, we apply Fubini's theorem to deduce that 
\begin{equation}\label{eq: sublevel 7}
    |E_{\mu, \mathrm{main}}(s_J; \sigma)\cap (J \times \B^N)| \leq \int_{Y_{\circ} \setminus Y_{\mathrm{sing}}(s_J;\sigma)}  |\{t \in J : |F_{\mu}(t;y)| < \sigma \}| \,\ud y. 
\end{equation}

Combining \eqref{eq: sublevel 4} with \eqref{eq: sublevel 5}, there exists some $\sE$-admissible $c_0 > 0$ with the following property: for all $y \in Y_{\circ} \setminus Y_{\mathrm{sing}}(s_J;\sigma)$ there exists some $0 \leq k_{J,y} \leq d$ such that
\begin{equation*}
    |\partial_t^{k_{J,y}} F_{\mu}(t;y)| \gtrsim_{\sE} \sigma^{1/2d} \big( 1 + |M(y)\vec{\mu}|\big) \gtrsim \sigma^{1/2d} \quad \textrm{for $|t-s_J| < 2c_0 \sigma^{1/2d}$.}
\end{equation*}
The van der Corput sublevel set estimate from Lemma~\ref{lem: van der Corput} implies that 
\begin{equation}\label{eq: sublevel 8}
    |\{t \in I : |F_{\mu}(t;y)| < \sigma \textrm{ and } |t-s_J| < 2c_0 \sigma^{1/2d} \}| \lesssim_{\sE} \sigma^{1/d - 1/2d^2}
\end{equation}
holds for each fixed $y \in Y_{\circ} \setminus Y_{\mathrm{sing}}(s_J;\sigma)$. We remark that the uniformity in the univariate van der Corput bound ensures that the above inequality holds with an $\sE$-admissible implied constant. 

Since $d \geq 2$ and we are free to assume $0 < \sigma < 1$, we may bound $\sigma^{1/d - 1/2d^2} < \sigma^{3/4d}$. Thus, if we choose the $J \in \fJ$ to have length $c_0\sigma^{1/2d}$ and apply \eqref{eq: sublevel 8} on each interval, we deduce that
\begin{equation}\label{eq: sublevel 9} 
   \sum_{J \in \fJ'} |\{t \in J : |F_{\mu}(t;y)| < \sigma \}| \lesssim_{\sE} \sigma^{\kappa_2} \qquad \textrm{for $\kappa_2 := \frac{1}{4d}$.}
\end{equation}
Here we have used the fact that $\#\fJ = O_{\sE}(\sigma^{-1/2d})$. We may now combine \eqref{eq: sublevel 6}, \eqref{eq: sublevel 7} and \eqref{eq: sublevel 9} to obtain
\begin{equation*}
    |E^1_{\mu}(\sigma)| \lesssim_{\sE} \sigma^{\kappa}, \qquad \textrm{where $\kappa := \min\{\kappa_1, \kappa_2\} \in (0,1)$}
\end{equation*}
is an $\sE$-admissible exponent. This concludes the proof. 
\end{proof}

%%%%%%%%%%%%%%%%%%%%%%%%%%%%%%%%%%%%%%%%%%%%%%%%%%%%%%%%%%%%%%%%%%%%%%%%%%%%%%%%%%%%%%%%%%%%%%%%

%         Preliminaries 

%%%%%%%%%%%%%%%%%%%%%%%%%%%%%%%%%%%%%%%%%%%%%%%%%%%%%%%%%%%%%%%%%%%%%%%%%%%%%%%%%%%%%%%%%%%%%%%%

\subsection{Proof of Proposition~\ref{prop: sublevel slice}}\label{subsec: lin ind proof 2} A variant of the argument used in the previous subsection also yields  Proposition~\ref{prop: sublevel slice}.

\begin{proof}[Proof (of Proposition~\ref{prop: sublevel slice})] As in the proof of Proposition~\ref{prop: sublevel}, we will say an object is \textit{$\sE$-admissible} if it depends only on $\sE := \{f, g_1, \dots, g_m, I_{\circ}, Y_{\circ}\}$. As before, we define
\begin{equation*}
  F_{\mu} \colon (-1,1) \times \B^N \to \R, \qquad  F_{\mu}(t;y) := f(t;y) - \sum_{j=1}^m \mu_j g_j(t;y)
\end{equation*}
and write
\begin{gather*}
   f(t;y) = \sum_{i=0}^{\infty} a_i(s;y) (t-s)^i, \quad  g_j(t;y) = \sum_{i=0}^{\infty} b_{i, j}(s;y) (t-s)^i, \\   F_{\mu}(t;y) = \sum_{i = 0}^{\infty} c_{\mu,i}(s;y) (t-s)^i.
\end{gather*}

By the hypotheses of the proposition, we can find some $d \in \N_0$ such that 
\begin{equation}\label{eq: poly non linear comb}
    P_d[f(\,\cdot\,;y),s] \textrm{ is not a linear combination of } P_d[g_j(\,\cdot\,;y), s],\, 1 \leq j \leq m,
\end{equation}
for all $(s,y) \in I_{\circ}\times Y_{\circ}$. By a simple compactness argument, after possibly increasing the value of $d$ (whilst still keeping this quantity admissible), we may assume without loss of generality that the matrix 
\begin{equation*}
  B_{\le d}(s;y) :=  (b_{i,j}(s;y))_{\substack{ 0 \leq i \leq d \\ 1 \leq j \leq m}}
\end{equation*}
satisfies $\rank  B_{\leq d}(s;y) = \rank \fB(s;y) = r$ for all $(s;y) \in I_{\circ} \times Y_{\circ}$, where the matrix $\fB$ is as defined in \eqref{eq: rank hyp}. If we let $\vec{a}(s;y) := (a_i(s;y))_{i=0}^d$, $\vec{c}_{\mu}(s;y) := (c_{\mu, i}(s;y))_{i=0}^d$ and $\vec{\mu} := (\mu_j)_{j=1}^m$, then it follows from the definitions that
\begin{equation}\label{eq: relating Taylor coeff}
    \vec{c}_{\mu}(s;y) = \vec{a}(s;y) - B_{\le d}(s;y)\vec{\mu}. 
\end{equation}

We can reinterpret \eqref{eq: poly non linear comb} as the statement that the vector $\vec{a}(s;y)$ lies outside the $r$-dimensional linear subspace $\Pi(s;y)$ given by the image of the matrix $B_{\le d}(s;y)$. In particular, we have
\begin{equation*}
   |\vec{c}_{\mu}(s;y)| \geq \dist\big(\vec{a}(s;y), \Pi(s;y)\big) =: \rho(s;y) > 0 
\end{equation*}
for all $\vec{\mu} \in \R^m$ and all $(s;y) \in I_{\circ} \times Y_{\circ}$. By continuity and compactness,
\begin{equation}\label{eq: small lambda bound}
    |\vec{c}_{\mu}(s;y)| \gtrsim_{\sE} 1 \qquad \textrm{for all $(s;y) \in I_{\circ} \times Y_{\circ}$.}
\end{equation}
Indeed, using the rank hypothesis, it is a simple matter to check that the distance function $\rho(s;y)$ is continuous in $(s;y)$, from which \eqref{eq: small lambda bound} follows. 

On the other hand, since $B_{\le d}(s;y)$ has rank $r$, it follows that the kernel $N(s;y) := \ker B_{\le d}(s;y)$ has dimension $m - r$. Let $V(s;y) := N(s;y)^{\perp}$ denote the orthogonal complement of the kernel in $\R^m$. It is straightforward consequence of the singular value decomposition of the matrix $B_{\le d}(s;y)$ that 
\begin{equation*}
|B_{\le d}(s;y) \vec{\mu}| \geq \sigma_{\min}(s;y)|\vec{\mu}| \qquad \textrm{for all $\vec{\mu} \in V(s;y)$,}
\end{equation*}
where $\sigma_{\min}(s;y)$ denotes the smallest non-zero singular value of $B_{\le d}(s;y)$. Moreover, if $\vec{\mu} \in \R^m$, then the orthogonal projection $\proj_{V(s;y)} \vec{\mu}$ of $\vec{\mu}$ onto $V(s;y)$ satisfies $B_{\le d}(s;y) \vec{\mu} = B_{\le d}(s;y) \proj_{V(s;y)} \vec{\mu}$. Thus, the previous display implies
\begin{equation*}
|B_{\le d}(s;y) \vec{\mu}| \geq \sigma_{\min}(s;y)|\proj_{V(s;y)}\vec{\mu}| \qquad \textrm{for all $\vec{\mu} \in \R^m$.}
\end{equation*}
Furthermore, the constant rank hypothesis implies that $\sigma_{\min}(s;y) = \sigma_r(s;y)$ where $\sigma_r(s;y) > 0$ denotes the $r$th largest singular value of $B(s;y)$, which is always non-zero. As a simple consequence of the min-max representation of the singular values, $\sigma_r(s;y)$ is a continuous function of $(s;y)$. Consequently, by compactness,
\begin{equation*}
   |B_{\le d}(s;y) \vec{\mu}| \gtrsim_{\sE} |\proj_{V(s;y)}\vec{\mu}| \qquad \textrm{for all $(s;y) \in I_{\circ} \times Y_{\circ}$ and all $\vec{\mu} \in \R^m$.} 
\end{equation*}
Note that the supremum of $|\vec{a}(s;y)|$ over all $(s;y) \in I_{\circ} \times Y_{\circ}$ is an $\sE$-admissible quantity. Thus, in light of \eqref{eq: relating Taylor coeff} and \eqref{eq: small lambda bound}, we therefore obtain the bound 
\begin{equation}\label{eq: large lambda bound}
    |\vec{c}_{\mu}(s;y)| \gtrsim_{\sE} 1 + |\proj_{V(s;y)}\vec{\mu}|  \qquad \textrm{for all $(s;y) \in I_{\circ} \times Y_{\circ}$.}
\end{equation}
 By combining \eqref{eq: large lambda bound} with Lemma~\ref{lem: poly bound}, for each $(s;y) \in I_{\circ} \times Y_{\circ}$ there exists some $k = k(s;y)$ satisfying $0 \leq k \leq d$ such that 
\begin{equation}\label{eq: leading term bound}
    |\partial_t^k P_d[F_{\mu}(\,\cdot\,;y),s](t)| \gtrsim_{\sE} 1 + |\proj_{V(s;y)}\vec{\mu}| \qquad \textrm{for all $t \in (-1,1)$.}
\end{equation}

It remains to control the remainder term. Write
\begin{align*}
    R_{\mu}(s;y;t) &:= F_{\mu}(t;y) - P_d[F_{\mu}(\,\cdot\,;y),s](t) \\
    &= \sum_{i=d+1}^{\infty} \Big((\partial_t^i f)(s;y) - \sum_{j=1}^m \mu_j (\partial_t^i g_j)(s;y)\Big)\frac{(t-s)^i}{i!}. 
\end{align*}
Observe that the expression $\sum_{j=1}^m \mu_j (\partial_t^i g_j)(s;y)/i!$ corresponds precisely to the $i$th entry of $\fB(s;y)\vec{\mu}$, where $\fB(s;y)$ is the matrix in \eqref{eq: rank hyp} with infinitely many rows. Since $\fB(s;y)$ and $B_{\le d}(s;y)$ have the same rank and the same number of columns, their kernels have the same dimension. Moreover, it is clear that $\ker \fB(s;y)$ is a subspace of $\ker B_{\le d}(s;y)$, from which we deduce that $\ker \fB(s;y) = \ker B_{\le d}(s;y) = N(s;y)$. If we write $\mu' := \proj_{V(s;y)}\vec{\mu}$ where $V(s;y) := N(s;y)^{\perp}$ as above, then it follows that
\begin{equation*}
    R_{\mu}(s;y;t) = \sum_{i=d+1}^{\infty} \Big((\partial_t^i f)(s;y) - \sum_{j=1}^m \mu_j' (\partial_t^i g_j)(s;y)\Big)\frac{(t-s)^i}{i!}. 
\end{equation*}
In particular, for all $0 \leq k \leq d$ we have, by differentiation of power series,
\begin{align}
\nonumber
   |\partial^k_t R_{\mu}(s;y;t)| &\leq |\proj_{V(s;t)}\vec{\mu}| \sum_{i=d+1}^{\infty} \Big(|(\partial_t^i f)(s;y)| + \sum_{j=1}^m  |(\partial_t^i g_j)(s;y)|\Big)\frac{|t-s|^{i-k}}{(i-k)!} \\
   \label{eq: error bound}
   &\lesssim \big(1 + |\proj_{V(s;t)}\vec{\mu}|\big)|t-s|,
\end{align}
where the last inequality holds with an admissible constant over all $(s;y) \in I_{\circ} \times Y_{\circ}$ whenever $|t-s| < c_{\sE}$ for some admissible constant $c_{\sE} > 0$. Indeed, this can be verified using a similar argument to that applied at the end of Step 4 in the proof of Proposition~\ref{prop: sublevel}.

Combining \eqref{eq: leading term bound} with \eqref{eq: error bound}, there exists some $\sE$-admissible $c_0 > 0$ such that that for each $(s;y) \in I_{\circ} \times Y_{\circ}$ there exists some $k = k(s;y) \in \N_0$ satisfying $0 \leq k \leq d$ such that
\begin{equation*}
    |\partial_t^k F_{\mu}(t;y)| \gtrsim 1 + |\proj_{V(s;y)}\vec{\mu}| \gtrsim_{\sE} 1 \qquad \textrm{for all $t \in I_{\circ}$ with $|t - s| < c_0$.}
\end{equation*}
This allows us to apply the sublevel set estimate from Lemma~\ref{lem: van der Corput} for each fixed $(s;y) \in I_{\circ} \times Y_{\circ}$ and $\vec{\mu} \in \R^m$ to deduce that
\begin{equation*}
    |\{t \in I_{\circ} : |F_{\mu}(t;y)| < \sigma \textrm{ and } |t-s| < c_0 \}| \lesssim_{\sE} \sigma^{1/d}.
\end{equation*}
The uniformity in the univariate van der Corput bound ensures that the above inequality holds with an $\sE$-admissible implied constant: in particular, the constant is uniform over all choices of $\vec{\mu} \in \R^m$, $(s;y) \in I_{\circ} \times Y_{\circ}$ and $\sigma  > 0$. By covering $I$ using $O_{\sE}(1)$ open intervals of length $2c_0$, we conclude the uniform sublevel set bound
\begin{equation*}
\sup_{y \in Y_{\circ}} \sup_{\mu \in \R^m}   |\{t \in I_{\circ} : |F_{\mu}(t;y)| < \sigma \}| \lesssim_{\sE} \sigma^{1/d}
\end{equation*}
for all $\sigma > 0$, as required.
\end{proof}

%%%%%%%%%%%%%%%%%%%%%%%%%%%%%%%%%%%%%%%%%%%%%%%%%%%%%%%%%%%%%%%%%%%%%%%%%%%%%%%%%%%%%%%%%%%%%%%%

%   Proof of the geometric estimates

%%%%%%%%%%%%%%%%%%%%%%%%%%%%%%%%%%%%%%%%%%%%%%%%%%%%%%%%%%%%%%%%%%%%%%%%%%%%%%%%%%%%%%%%%%%%%%%%

\section{Geometric maximal estimates I: Preliminaries} 

%%%%%%%%%%%%%%%%%%%%%%%%%%%%%%%%%%%%%%%%%%%%%%%%%%%%%%%%%%%%%%%%%%%%%%%%%%%%%%%%%%%%%%%%%%%%%%%%

%   Multiplicity formulation

%%%%%%%%%%%%%%%%%%%%%%%%%%%%%%%%%%%%%%%%%%%%%%%%%%%%%%%%%%%%%%%%%%%%%%%%%%%%%%%%%%%%%%%%%%%%%%%%

\subsection{Multiplicity formulation} Recall the maximal function estimates in Definition \ref{defKN}. Via $\ell^p$ duality, these estimates are equivalent to $L^p$ norm bounds for multiplicity functions associated to families of tubes. The latter bounds take the shape
\begin{equation}\label{eq: multiplicity bound}
      \big\|\sum_{T \in \T} \chi_T\big\|_{L^q(\R^n)} \lesssim_{\phi, \varepsilon} \delta^{-\beta - \varepsilon} \big(\sum_{T \in \T} |T| \big)^{1/p},
    \end{equation}
where $\T$ is a family of tubes associated to some phase $\phi$. To make this more precise, we introduce the following definitions.

\begin{definition} Let $\phi \colon \D^n \to \R$ be a non-degenerate phase, $1 \leq p, q \leq \infty$ and $\beta \geq 0$. 
\begin{enumerate}[I)]
    \item We let $\bK_{p \to q}^*(\phi; \beta)$ denote the statement:\medskip
    
    \noindent For all $\varepsilon > 0$, the inequality
    \eqref{eq: multiplicity bound} holds uniformly over all $0 < \delta < 1$ and all direction-separated families $\T$ of $\delta$-tubes associated to $\phi$.  \medskip
    
    \item We let $\bN_{p \to q}^*(\phi;\beta)$ denote the statement:\medskip
    
    \noindent For all $\varepsilon > 0$, the inequality
    \eqref{eq: multiplicity bound} holds uniformly over all $0 < \delta < 1$ and all centre-separated families $\T$ of $\delta$-tubes associated to $\phi$.
\end{enumerate} 
\end{definition}

By a standard duality argument, 
\begin{equation}\label{eq: Kak/Nik duality}
    \bK_{p \to q}(\phi; \beta) \Leftrightarrow  \bK_{q' \to p'}^*(\phi; \beta) \quad \textrm{and} \quad \bN_{p \to q}(\phi; \beta) \Leftrightarrow  \bN_{q' \to p'}^*(\phi; \beta)
\end{equation}
for all choices of exponent $1 \leq p, q \leq \infty$.

%%%%%%%%%%%%%%%%%%%%%%%%%%%%%%%%%%%%%%%%%%%%%%%%%%%%%%%%%%%%%%%%%%%%%%%%%%%%%%%%%%%%%%%%%%%%%%%%

%   Multilinear Kakeya/Nikodym estimates

%%%%%%%%%%%%%%%%%%%%%%%%%%%%%%%%%%%%%%%%%%%%%%%%%%%%%%%%%%%%%%%%%%%%%%%%%%%%%%%%%%%%%%%%%%%%%%%%

\subsection{Multilinear Kakeya/Nikodym estimates} Rather than work directly with inequalities of the form \eqref{eq: multiplicity bound}, we apply the standard \textit{broad-narrow} argument from \cite{BG2011} to reduce to a multilinear inequality involving transverse families of tubes.

\begin{definition} Let $n \geq 2$, $\phi \colon \D^n \to \R$ be a non-degenerate phase and $[\phi;a]$ a phase-amplitude pair. For $2 \leq k \leq n$, $\nu > \delta > 0$ and $G$ the generalised Gauss map defined in \eqref{eq: Gauss map}, we make the following definitions.
\begin{enumerate}[a)]
    \item We say a $k$-tuple of $\delta$-tubes $(T_1, \dots, T_k)$ associated to $\phi$ is \textit{$\nu$-transversal} if 
\begin{equation}\label{eq: nu transverse}
    \big| \bigwedge_{\ell = 1}^k G(\bx; y(T_{\ell})) \big| \geq \nu > 0 \qquad \textrm{for all $\bx \in T_1 \cap \cdots \cap T_k$.}
\end{equation}
(Here, recalling \eqref{eq: delta tube}, $y(T_l)$ is the direction of $T_l$.)
 For $X \subseteq \R^n$, we say $(T_1, \dots, T_k)$ is \textit{$\nu$-transversal on $X$} if the inequality \eqref{eq: nu transverse} holds on the more restrictive range $\bx \in T_1 \cap \cdots \cap T_k \cap X$.\footnote{We remark that $G(\bx;y(T_{\ell}))$ can be throught of as the tangent direction of $T_{\ell}$ at $\bx$. See \cite[Lemma 5.2]{GHI2019}.}
\item We say the families $\T_1,\dots, \T_k$ of $\delta$-tubes associated to $\phi$ are \textit{$\nu$-transversal} (respectively, $\nu$-transversal on $X$)  if every $k$-tuple $(T_1, \dots, T_k)$ with $T_{\ell} \in \T_{\ell}$, $1 \leq \ell \leq k$, is $\nu$-transversal. 
\end{enumerate}
\end{definition}

Given $K \geq 1$, we let $\cQ_{K^{-1}}$ denote a cover of $\B^{n-1} \times (-1,1)$ by essentially disjoint, axis-parallel cubes of side-length $K^{-1}$. The multilinear estimates take the shape
\begin{equation}\label{eq: k multiplicity bound}
     \Big\| \prod_{\ell = 1}^k \big|\sum_{T_{\ell} \in \T_{\ell}} \chi_{T_{\ell}}\big|^{1/k}\Big\|_{L^p(Q)} \lesssim_{\phi, \varepsilon} K^A \delta^{-\beta - \varepsilon} \Big(\sum_{T \in \T} |T| \Big)^{1/p}
    \end{equation}
where $Q \in \cQ_{K^{-1}}$ is a cube, $\T$ is a family of tubes associated to some phase $\phi$, the sets $\T_1,\dots, \T_k \subseteq \T$ are $K^{-(k-1)}$-transversal and $A = A_{n,k,p} > 0$. To make this more precise, we introduce the following definitions.

\begin{definition} Let $\phi \colon \D^n \to \R$ be a non-degenerate phase, $1 \leq p \leq \infty$ and $\beta \geq 0$. 
\begin{enumerate}[I)]
    \item We let $\bM\bK_{p, k}^*(\phi; \beta)$ denote the statement:\medskip
    
    \noindent For all $\varepsilon > 0$, $K \geq 1$ and $Q \in \cQ_{K^{-1}}$, the inequality
    \eqref{eq: k multiplicity bound} holds uniformly over all $0 < \delta < 1$, all direction-separated families $\T$ of $\delta$-tubes associated to $\phi$, and all $\T_1, \dots, \T_k \subseteq \T$ which are $K^{-(k-1)}$-transversal on $Q$.\medskip

    \item We let $\mathbf{MN}_{p, k}^*(\phi; \beta)$ denote the statement:\medskip
    
    \noindent For all $\varepsilon > 0$, $K \geq 1$ and $Q \in \cQ_{K^{-1}}$, the inequality
    \eqref{eq: k multiplicity bound}
    holds uniformly over all $0 < \delta < 1$, all centre-separated families $\T$ of $\delta$-tubes associated to $\phi$, and all $\T_1, \dots, \T_k \subseteq \T$ which are $K^{-(k-1)}$-transversal on $Q$.
\end{enumerate}
\end{definition}

Here and throughout what follows, the exponent $A = A_{n,k,p} > 0$ in \eqref{eq: k multiplicity bound} is understood to be a fixed quantity which depends only on $n$, $k$ and $p$ only. In particular, it is the same exponent as that appearing in Theorem~\ref{thm: uni multilin} below.

\begin{remark} Here we use the classical setup for our multilinear Kakeya/Nikodym estimates, in the spirit of the foundational work of Bennett--Carbery--Tao~\cite{BCT2006}. However, one could alternatively work with weaker \textit{broad norms} as in the Kakeya argument in \cite{HRZ2022}. The broad norm approach is slightly easier, but we opt for the multilinear formalism since it produces stronger intermediary results. 
\end{remark}

%%%%%%%%%%%%%%%%%%%%%%%%%%%%%%%%%%%%%%%%%%%%%%%%%%%%%%%%%%%%%%%%%%%%%%%%%%%%%%%%%%%%%%%%%%%%%%%%

%   Multilinear reduction

%%%%%%%%%%%%%%%%%%%%%%%%%%%%%%%%%%%%%%%%%%%%%%%%%%%%%%%%%%%%%%%%%%%%%%%%%%%%%%%%%%%%%%%%%%%%%%%%

\subsection{Multilinear reduction} In what follows we let $p'(n) := \frac{n+1}{n-1}$ denote the H\"older conjugate of the exponent from \eqref{eq: bush exp} and $s'(n;p)$ the H\"older conjugate of $s(n;p)$ from \eqref{eq: other bush exp}.

\begin{proposition}\label{prop: multilinear red} Let $n \geq 2$, $2 \leq k < \frac{n+3}{2}$ and $1 \leq p \leq p'(n)$. Given $0 \leq \beta < \beta(n;p')$, there exists some $0 \leq \beta_{\circ} < \beta(n;p')$ such that the following holds. If $\phi \colon \D^n \to \R$ is a H\"ormander-type phase, then:
\begin{enumerate}[I)]
    \item $\bM\bK_{p, k}^*(\phi; \beta) \Rightarrow \bK_{p \to p}^*(\phi; \beta_{\circ})$;
    \item $\bM\bN_{p, k}^*(\phi; \beta) \Rightarrow \bN_{p \to p}^*(\phi; \beta_{\circ})$.
\end{enumerate}
\end{proposition}

Proposition~\ref{prop: multilinear red} follows from a very simple variant  of the broad/narrow analysis of \cite{BG2011}. We remark that, since we are not looking to precisely quantify our estimates, the full machinery of \cite{BG2011} is not required. 

In order to present the proof of Proposition~\ref{prop: multilinear red}, we make some preliminary definitions and recall some observations from the literature. For $K \geq 1$, let $\Theta(K^{-1})$ be a finitely-overlapping covering of $\B^{n-1}$ by balls of radius $K^{-1}$. For each $\theta \in \Theta(K^{-1})$ let
\begin{equation*}
    \T_{\mathrm{K}}[\theta] := \big\{ T \in \T : y(T) \in \theta \big\}, \qquad \T_{\mathrm{N}}[\theta] := \big\{ T \in \T : \omega(T) \in \theta \big\}.
\end{equation*}
Given a linear subspace $V \subseteq \R^n$ and $\bx \in \R^n$, define
\begin{equation}\label{eq: V aligned caps}
    \Theta(K^{-1};V; \bx) := \big\{ \theta \in \Theta(K^{-1}) : \dist(G(\bx ;y_{\theta}), V) \leq C_{\phi} K^{-1} \big\},
\end{equation}
where $y_{\theta}$ denotes the centre of the $K^{-1}$-ball $\theta$ and $C_{\phi} \geq 1$ is an admissible constant, chosen large enough to satisfy the forthcoming requirements of the proof. 

For $2 \leq k \leq n$ a $k$-tuple $(\theta_1, \dots, \theta_k) \in \Theta(K^{-1})^k$ is \textit{transversal at $\bx \in \R^n$} if 
\begin{equation*}
    |G(\bz;y_1) \wedge \cdots \wedge G(\bz;y_k)| \geq K^{-(k-1)} \qquad \textrm{for all $y_{\ell} \in \theta_{\ell}$, $1 \leq \ell \leq k$ and $|\bz - \bx|_{\infty} < K^{-1}$.}
\end{equation*}
We let $\Theta^k_{\trans}(K^{-1}; \bx)$ denote the collection of $k$-tuples which are transversal at $\bx$. For simplicity, denote
\begin{equation*}
\mu_\mathbb{S}(\bx):=\sum_{T\in\mathbb{S}}\chi_T(\bx)
\end{equation*}
for any $\bx\in\R^n$ and $\mathbb{S}\subseteq\mathbb{T}$.
\begin{lemma}[Broad-narrow decomposition]\label{lem: b/n dec pointwise} Let $2 \leq k \leq n$, $\delta > 0$, $1 \leq K \ll \delta^{-1}$ and $\phi \colon \D^n \to \R$ a non-degenerate phase. Given $\T$ a family of $\delta$-tubes associated to $\phi$ and $\bx \in \R^n$, we have
\begin{align}\label{eq: b/n dec pointwise}
\mu_{\T}(\bx) \lesssim \max_{V \in \mathrm{Gr}(k-1,\R^n)} &\sum_{\theta \in \Theta(K^{-1}; V; \bx)}\mu_{\T_{\mathrm{K}}[\theta]}(\bx) \\
\nonumber
&\qquad + K^{n-1} \max_{(\theta_1, \dots, \theta_k) \in \Theta^k_{\trans}(K^{-1}; \bx)}  \prod_{\ell = 1}^k \mu_{\T_{\mathrm{K}}[\theta_{\ell}]}(\bx)^{1/k}.
\end{align}
\end{lemma}

Here $\mathrm{Gr}(d,\R^n)$ for $0 \leq d \leq n$ denotes the Grassmannian of all $d$-dimensional linear subspaces of $\R^n$. The proof is based on the following elementary lemma. 

\begin{lemma}\label{lem: transversality} Let $0 \leq d \leq n-1$ and suppose $\Omega \subseteq \R^n$ is a non-empty subset of the unit sphere $S^{n-1}$ which satisfies the following condition:
\begin{equation}\label{eq: broad points}
    \textrm{for all $W \in \mathrm{Gr}(d,\R^n)$, there exists some $\omega \in \Omega$ such that $\dist(\omega, W) \geq K^{-1}$.}
\end{equation}
Then there exist $\omega_1, \dots, \omega_{d+1} \in \Omega$ such that 
\begin{equation*}
    |\omega_1 \wedge \cdots \wedge \omega_{d+1}| \geq K^{-d}.
\end{equation*}
\end{lemma}

\begin{proof} We proceed by induction on $d$, noting that the base case $d = 0$ is immediate. 

Suppose the result holds for $d-1$ for some $1 \leq d \leq n-1$ and let $\Omega \subseteq S^{n-1}$ be a set of points satisfying \eqref{eq: broad points}. Since any linear subspace of dimension $d-1$ is contained in a linear subspace of dimension $d$, the set $\Omega$ also satisfies \eqref{eq: broad points} with $d$ replaced with $d-1$. Hence, we may apply the induction hypothesis to conclude that there exist some $\omega_1, \dots, \omega_d \in \Omega$ such that
\begin{equation*}
    |\omega_1 \wedge \cdots \wedge \omega_d| \geq K^{-d+1}. 
\end{equation*}

Let $W := \mathrm{span}\{\omega_1, \dots, \omega_d\}$, so that $W$ is a $d$-dimensional linear subspace. For any $\omega \in \Omega$, the `base $\times$ perpendicular height' formula for the volume of a parallelepiped gives
\begin{equation*}
  | \omega_1 \wedge \cdots \wedge \omega_d \wedge \omega | =   |\omega_1 \wedge \cdots \wedge \omega_d| |\dist(\omega, W)| \geq K^{-d+1} |\dist(\omega, W)|.
\end{equation*}
However, by hypothesis we can find some $\omega_{d+1} \in \Omega$ such that $\dist(\omega_{d+1}, W) \geq K^{-1}$. Plugging this value into the preceding display closes the induction. 
\end{proof}

We now turn to the proof of the broad-narrow decomposition.

\begin{proof}[Proof (of Lemma~\ref{lem: b/n dec pointwise})] The first step is to identify some $V_{\bx} \in \mathrm{Gr}(k-1,\R^n)$ which `captures' as many of the large values of $\mu_{\T_{\mathrm{K}}[\theta]}(\bx)$ as possible. More precisely, define the \textit{broad part} of the multiplicity function
\begin{equation}\label{eq: gen b/n 1} 
    \mu_{\T}^{\mathrm{Br}}(\bx) := \min_{V \in \mathrm{Gr}(k-1,\R^n)} \max_{\theta \notin \Theta(K^{-1}; V; \bx)} \mu_{\T_{\mathrm{K}}[\theta]}(\bx)
\end{equation}
and suppose $V_{\bx} \in \mathrm{Gr}(k-1,\R^n)$ realises the minimum in \eqref{eq: gen b/n 1}. For slightly technical reasons, we also define 
\begin{equation}\label{eq: gen b/n 2} 
    \widetilde{\Theta}(K^{-1}; V; \bx) := \big\{ \theta \in \Theta(K^{-1}; V; \bx) : \mu_{\T_{\mathrm{K}}[\theta]}(\bx) \geq \mu_{\T}^{\mathrm{Br}}(\bx) \big\}
\end{equation}
and further choose $V_{\bx}$ so that, of all spaces realising the minimum in \eqref{eq: gen b/n 1}, the space $V_{\bx}$ also maximises $\#\widetilde{\Theta}(K^{-1}; V; \bx)$. 

From the definition, the normals $G(\bx;y_{\theta})$ associated to the caps $\theta \in \Theta(K^{-1}; V_{\bx}; \bx)$ are aligned around the $(k-1)$-dimensional subspace $V_{\bx}$. However, they do not align around any lower dimensional subspace. \medskip

\noindent \textbf{Claim.} There does not exist an affine subspace $W \subseteq \R^n$ of dimension $k-2$ such that $\widetilde{\Theta}(K^{-1}; V_{\bx}; \bx) \subseteq \Theta(K^{-1}; W; \bx)$. \medskip

Assuming the claim, it is a simple matter to conclude the proof of Lemma~\ref{lem: b/n dec pointwise}. We define the collections of \textit{narrow} and \textit{broad} caps (for $\mu_{\T}$ at $\bx$) by
\begin{equation*}
    \mathcal{N}_{\bx} := \Theta(K^{-1}; V_{\bx}; \bx) \qquad \textrm{and} \qquad  \cB_{\bx} := \Theta(K^{-1}) \setminus \Theta(K^{-1}; V_{\bx}; \bx),
\end{equation*}
respectively. By the defining properties of $V_{\bx}$, we have
\begin{align*}
    \mu_{\T}(\bx) &\leq  \sum_{\theta \in \mathcal{N}_{\bx}} \mu_{\T_{\mathrm{K}}[\theta]}(\bx)  + \sum_{\theta \in \cB_x} \mu_{\T_{\mathrm{K}}[\theta]}(\bx)  \\
    &\lesssim \max_{V \in \mathrm{Gr}(k-1,\R^n)} \Big[ \sum_{\theta \in \Theta(K^{-1}; V; \bx)} \mu_{\T_{\mathrm{K}}[\theta]}(\bx) \Big]  + 
 K^{n-1} \mu_{\T}^{\mathrm{Br}}(\bx).   
\end{align*}

In view of the claim, there exists a $(k-1)$-tuple of caps 
\begin{equation*}
(\theta_{\bx,1}, \dots, \theta_{\bx,k-1}) \in \widetilde{\Theta}(K^{-1}; V_{\bx}; \bx)^{k-1}
\end{equation*}
which is transverse at $\bx$ in the sense that
\begin{equation*}
        \inf  \Big\{ \Big|\bigwedge_{\ell = 1}^{k-1} G(\bz;y_{\ell}) \Big| : y_{\ell} \in \theta_{x,\ell}, \, 1 \leq \ell \leq k -1, \, |\bz-\bx|_{\infty} < 2K^{-1} \Big\}\geq K^{-(k-2)}. 
\end{equation*}
To see this, we simply apply Lemma~\ref{lem: transversality} with
\begin{equation*}
    \Omega := \big\{G(\bx; y_{\theta}) : \theta \in \widetilde{\Theta}(K^{-1}; V_{\bx}; \bx)^{k-1}\big\},
\end{equation*}
using the claim to verify the hypothesis \eqref{eq: broad points} with $d=k-2$. The constant $C_{\phi}$ in the definition of \eqref{eq: V aligned caps} can be used to pass from bounds on the the size of wedge products of the $G(\bx; y_{\theta})$ to more general $G(\bz;y_{\theta})$. 

By definition, there exists a cap $\theta_{\bx, k} \notin \Theta(K^{-1}; V_{\bx}; \bx)$ such that 
\begin{equation*}
    \mu_{\T_{\mathrm{K}}[\theta_{\bx,k}]}(\bx) = \mu_{\T}^{\mathrm{Br}}(\bx).
\end{equation*}
It follows that the $k$-tuple $(\theta_{\bx,1}, \dots, \theta_{\bx, k})$ is transversal at $\bx$ and satisfies
\begin{equation*} 
    \mu_{\T}^{\mathrm{Br}}(\bx) \leq \prod_{\ell=1}^k \mu_{\T_{\mathrm{K}}[\theta_{\bx, \ell}]}(\bx)^{1/k} \leq \max_{(\theta_1, \dots, \theta_k) \in \Theta^k_{\trans}(K^{-1}; \bx)}  \prod_{\ell = 1}^k \mu_{\T_{\mathrm{K}}[\theta_{\ell}]}(\bx)^{1/k}.
\end{equation*}
\smallskip

It remains to prove the claim. We argue by contradiction, assuming such a subspace $W$ exists. Let $\theta^* := \theta_{\bx, k}$ be as above, so that $\theta^* \notin \Theta(K^{-1}; V_{\bx}; \bx)$ realises the maximum in the definition \eqref{eq: gen b/n 1} for $V = V_{\bx}$. Define $V_{\bx}^*$ to be a linear subspace of dimension $k-1$ which contains $W$ and $G(\bx; y^*)$, where $y^*$ is the centre of $\theta^*$.

First note that $V_{\bx}^*$ realises the minimum in \eqref{eq: gen b/n 1}. Indeed, by the definition \eqref{eq: gen b/n 2}, if $\theta \in \Theta(K^{-1}; V_{\bx}; \bx) \setminus  \widetilde{\Theta}(K^{-1}; V_{\bx}; \bx)$, then $\mu_{\T_{\mathrm{K}}[\theta]}(\bx) < \mu_{\T}^{\mathrm{Br}}(\bx)$. Consequently,
\begin{equation*}
    \mu_{\T}^{\mathrm{Br}}(\bx) =  \max_{\theta \notin \Theta(K^{-1}; V_{\bx}; \bx)} \mu_{\T_{\mathrm{K}}[\theta]}(\bx) = \max_{\theta \notin \widetilde{\Theta}(K^{-1}; V_{\bx}; \bx)} \mu_{\T_{\mathrm{K}}[\theta]}(\bx).
\end{equation*}
On the other hand, from the hypothesis on $W$ and the definition of $V_{\bx}^*$ we have $\widetilde{\Theta}(K^{-1}; V_{\bx}; \bx) \subseteq \Theta(K^{-1}; V_{\bx}^*; \bx)$ and so
\begin{equation*}
    \mu_{\T}^{\mathrm{Br}}(\bx) \geq \max_{\theta \notin \Theta(K^{-1}; V_{\bx}^*; \bx)} \mu_{\T_{\mathrm{K}}[\theta]}(\bx) \geq \mu_{\T_{\mathrm{K}}[\theta^*]}(\bx) = \mu_{\T}^{\mathrm{Br}}(\bx). 
\end{equation*}

In light of the maximality of $V_{\bx}$, it follows that $\# \widetilde{\Theta}(K^{-1}; V_{\bx}^*, \bx) \leq \# \widetilde{\Theta}(K^{-1}; V_{\bx},\bx)$. However, it is clear from the definitions that 
\begin{equation*}
    \widetilde{\Theta}(K^{-1}; V_{\bx}; \bx) \subseteq \widetilde{\Theta}(K^{-1}; V_{\bx}^*, \bx), 
\end{equation*}
whilst
\begin{equation*}
    \theta^* \in \widetilde{\Theta}(K^{-1}; V_{\bx}^*, \bx) \quad \textrm{and} \quad \theta^* \notin \widetilde{\Theta}(K^{-1}; V_{\bx}; \bx).
\end{equation*}
 Thus, $\# \widetilde{\Theta}(K^{-1}; V_{\bx}; \bx) < \# \widetilde{\Theta}(K^{-1}; V_{\bx}^*, \bx)$, which is a contradiction. 
\end{proof}

The pointwise estimate from Lemma~\ref{lem: b/n dec pointwise} implies the following $L^p$ variant. Here $\cQ_{K^{-1}}$ is defined as before \eqref{eq: k multiplicity bound}. 

\begin{corollary}[$L^p$ Broad-narrow decomposition \cite{BG2011}]\label{cor: b/n} Let $2 \leq k \leq n$, $1 \leq p < \infty$, $\delta > 0$, $1 \leq K \ll \delta^{-1}$ and $\phi \colon \D^n \to \R$ a H\"ormander-type phase. Suppose $\T$ is a family of $\delta$-tubes associated to $\phi$ and let $\mathrm{X} \in \{\mathrm{K}, \mathrm{N}\}$. For each $Q \in \cQ_{K^{-1}}$, there exist $\T_1[Q], \dots, \T_k[Q] \subseteq \T$ which are  $K^{-(k-1)}$-transversal on $Q$ such that
\begin{align}\label{eq: b/n}
   \big\|\sum_{T \in \T} \chi_T\big\|_{L^p(\R^n)} &\lesssim_\phi K^{(k-2)/p'}\Big(\sum_{\theta \in \Theta(K^{-1})} \big\|\sum_{T \in \T_{\mathrm{X}}[\theta]} \chi_T\big\|_{L^p(\R^n)}^p \Big)^{1/p} \\
   \nonumber\textbf{}
   & \qquad + K^{(n-1)(1 + k/p)} \Big(\sum_{Q \in \cQ_{K^{-1}}}  \Big\| \prod_{\ell = 1}^k \big|\sum_{T_{\ell} \in \T_{\ell}[Q]} \chi_{T_{\ell}}\big|^{1/k}\Big\|_{L^p(Q)}^p \Big)^{1/p}.
\end{align}
\end{corollary}

In what follows, we refer to \eqref{eq: b/n} as the \textit{$L^p$ broad-narrow decomposition}, and the first and second terms on the right-hand side of \eqref{eq: b/n} as the \textit{$L^p$ narrow} and \textit{$L^p$ broad} terms, respectively. 

\begin{proof}[Proof (of Corollary~\ref{cor: b/n})] We apply Lemma~\ref{lem: b/n dec pointwise} and take the $L^p$ norms of both sides of the resulting pointwise inequality. We apply the $L^p$ triangle inequality to analyse separately the contributions arising from the first and second terms on the right-hand side of \eqref{eq: b/n dec pointwise}.

Consider the contribution arising from the first term in \eqref{eq: b/n dec pointwise}. Let $V \in \mathrm{Gr}(k-1,\R^n)$ and $N_1, \dots, N_{n-k+1}$ be a basis for the orthogonal complement of $V$. The caps belonging to $\Theta(K^{-1}; V; \bx)$ lie in an $O(K^{-1})$ neighbourhood of the set
\begin{equation*}
 M_{\bx} := \big\{ y \in \B^{n-1} : \inn{G(\bx;y)}{N_j} = 0 \textrm{ for } 1 \leq j \leq n-k+1\big\}. 
\end{equation*}
Since $\phi$ is H\"ormander-type, it follows that $M_{\bx}$ is a $(k-2)$-dimensional submanifold. Thus, $\# \Theta(K^{-1}; V; \bx) \lesssim_\phi K^{k-2}$ and, consequently, 
\begin{equation}\label{eq: b/n 1}
    \sum_{\theta \in \Theta(K^{-1}; V; \bx)}\mu_{\T_{\mathrm{K}}[\theta]}(\bx) \lesssim_\phi K^{(k-2)/p'}  \Big(\sum_{\theta \in \Theta(K^{-1})}\mu_{\T_{\mathrm{K}}[\theta]}(\bx)^p \Big)^{1/p}.
\end{equation}
On the other hand, we claim that
\begin{equation}\label{eq: b/n 2}
    \sum_{\theta \in \Theta(K^{-1})}\mu_{\T_{\mathrm{K}}[\theta]}(\bx)^p \lesssim_\phi\sum_{\theta \in \Theta(K^{-1})}\mu_{\T_{\mathrm{X}}[\theta]}(\bx)^p \qquad \textrm{for $\mathrm{X} \in \{\mathrm{K},\mathrm{N}\}.$}
\end{equation}
Temporarily assuming this is the case, we combine \eqref{eq: b/n 1} and \eqref{eq: b/n 2} and take the $L^p(\R^n)$-norm of both sides of the resulting estimate. From this we see that the contribution of the first term in \eqref{eq: b/n dec pointwise} corresponds to the first term on the right-hand side of \eqref{eq: b/n}.

Turning to the proof of \eqref{eq: b/n 2}, it suffices to show that for any $\theta^* \in \Theta(K^{-1})$ and $\bx \in \R^n$ we have
\begin{equation}\label{eq: b/n 3}
    \#\big\{ \theta \in \Theta(K^{-1}) : \, \exists\, T \in \mathbb{T}_{\mathrm{N}}[\theta]\cap\mathbb{T}_{\mathrm{K}}[\theta^*] \textrm{ s.t. } \bx \in T \big\} \lesssim_\phi 1.
\end{equation}
Suppose that  $\bx \in T_j$ for $j = 1$, $2$, and $\bx = (x,t)$, $y_j := y(T_j)$ and $\omega_j := \omega(T_j)$. We therefore have $ |\Psi(\omega_1; t; y_1) - \Psi(\omega_2; t; y_2)| < 2\delta$, where $\Psi$ is as defined in \eqref{eq: Psi}. Now, in addition, suppose $T_j \in \T_{\mathrm{K}}[\theta^*]$ for $j = 1$, $2$, so that $|y_1 - y_2| \lesssim K^{-1}$. For $\delta \le K^{-1}$, it follows that 
\begin{equation*}
   |\omega_1 - \omega_2| = \big|(\partial_y \phi)(\Psi(\omega_1; t; y_1),t;y_1) - (\partial_y \phi)(\Psi(\omega_2; t; y_2),t;y_2)\big| \lesssim_\phi K^{-1}.
\end{equation*}
From these observations, we deduce that \eqref{eq: b/n 3}, and therefore also \eqref{eq: b/n 2}, holds. 

We now consider the contribution from the second term in \eqref{eq: b/n dec pointwise}. Given $Q \in \cQ_{K^{-1}}$ and $\bx \in Q$, we have $\Theta^k_{\trans}(K^{-1}; \bx) \subseteq \Theta^k(Q)$ where $\Theta^k(Q)$ denotes the set of all $k$-tuples $\vec{\theta} = (\theta_1, \dots, \theta_k) \in \Theta(K^{-1})^k$ such that
\begin{equation*}
    |G(\bz;y_1) \wedge \cdots \wedge G(\bz;y_k)| \geq K^{-(k-1)} \qquad \textrm{for all $y_{\ell} \in \theta_{\ell}$, $1 \leq \ell \leq k$, and $\bz \in Q$.}
\end{equation*}
We pointwise dominate the maximum appearing in the second term in \eqref{eq: b/n dec pointwise} by the corresponding $\ell^p$-norm. In view of the above and pigeonholing, the $L^p(Q)$-norm of the resulting expression can be bounded by
\begin{equation*}
    \Big(\sum_{\vec{\theta} \in \Theta^k(Q)} \Big\| \prod_{\ell = 1}^k \big|\sum_{T_{\ell} \in \T_{\mathrm{K}}[\theta_{\ell}]} \chi_{T_{\ell}}\big|^{1/k} \Big\|_{L^p(Q)}^p\Big)^{1/p} \lesssim K^{(n-1)k/p} \Big\| \prod_{\ell = 1}^k\big|\sum_{T_{\ell} \in \T_{\ell}[Q]} \chi_{T_{\ell}}\big|^{1/k} \Big\|_{L^p(Q)},
\end{equation*}
where $\T_{\ell}[Q] := \T_{\mathrm{K}}[\theta_{\ell, Q}]$ for $1 \leq \ell \leq k$ and some $(\theta_{1,Q}, \dots, \theta_{k,Q}) \in \Theta^k(Q)$. Taking the $\ell^p(\cQ_{K^{-1}})$-norm of the above estimate, we see that the contribution of the second term in \eqref{eq: b/n dec pointwise} corresponds to the second term on the right-hand side of \eqref{eq: b/n}.

\end{proof}

\begin{proof}[Proof (of Proposition~\ref{prop: multilinear red})] We treat cases I) and II) simultaneously, as the proofs are identical. Let $K \geq 1$ and apply Corollary~\ref{cor: b/n}, taking $\mathrm{X} := \mathrm{K}$ and $\mathrm{X} := \mathrm{N}$ in I) and II), respectively. We proceed to estimate the two terms on the right-hand side of \eqref{eq: b/n}.\medskip

\noindent\textit{$L^p$ Narrow term}. Since the tubes have $\delta$-separated directions in case I) and $\delta$-separated centres in case II), we see that  in either case 
\begin{equation}\label{eq: multilin red 1}
  \sum_{T \in \T_{\mathrm{X}}[\theta]} |T| \lesssim_\phi K^{-(n-1)} . 
\end{equation}

We apply the dual form of the universal Kakeya and Nikodym bounds from Proposition~\ref{prop: univ Kak/Nik}. In the regime $1 \leq p \leq p'(n)$, we have $\bK_{s \to p}^*\big(\phi; \beta(n;p')\big)$ and $\bN_{s \to p}^*\big(\phi; \beta(n;p')\big)$ for $s'(n; p') \leq s \leq \infty$. Taking $s = s'(n; p')$, and noting that this exponent satisfies $\frac{1}{s} = \frac{1}{p} + \frac{1}{2p'}$, we therefore deduce that
\begin{align*}
    \big\|\sum_{T \in \T_{\mathrm{X}}[\theta]} \chi_T\big\|_{L^p(\R^n)} &\lesssim_\phi \delta^{-\beta(n;p')} \Big(\sum_{T \in \T_{\mathrm{X}}[\theta]} |T|\Big)^{1/p + 1/(2p')} \\
    &\lesssim \delta^{-\beta(n;p')} K^{-(n-1)/(2p')} \Big(\sum_{T \in \T_{\mathrm{X}}[\theta]} |T| \Big)^{1/p},
\end{align*}
where the second step is by \eqref{eq: multilin red 1}. 
Summing over all $\theta \in \Theta(K^{-1})$, we deduce that
\begin{equation}\label{eq: multilin red 2}
     K^{(k-2)/p'} \Big(\sum_{\theta \in \Theta(K^{-1})} \big\|\sum_{T \in \T_{\mathrm{X}}[\theta]} \chi_T\big\|_{L^p(\R^n)}^p \Big)^{1/p} \lesssim_\phi \delta^{-\beta(n;p')} K^{-e(n,k;p)} \Big(\sum_{T \in \T} |T| \Big)^{1/p},
\end{equation}
where 
\begin{equation}\label{eq: multilin red 3}
    e(n,k;p) := \frac{n+3 - 2k}{2p'}.
\end{equation}
This is our estimate for the $L^p$ narrow term in \eqref{eq: b/n}.\medskip

\noindent \textit{$L^p$ Broad term}. Write $\beta := \beta(n; p') - \eta$ for some $\eta > 0$. Since the families of tubes $\T_1[Q], \dots, \T_k[Q]$ are $K^{-(k-1)}$-transversal on $Q$, applying either the hypothesis $\bM\bK_{p, k}^*(\phi; \beta)$ in case I) or the hypothesis $\bM\bN_{p, k}^*(\phi; \beta)$ in case II), we obtain
\begin{equation}\label{eq: multilin red 4}
   \Big\| \prod_{\ell = 1}^k \big|\sum_{T_{\ell} \in \T_{\ell}[Q]} \chi_{T_{\ell}}\big|^{1/k}\Big\|_{L^p(Q)} \lesssim_{\phi} K^A \delta^{-\beta(n;p')+ \eta} \Big(\sum_{T \in \T} |T| \Big)^{1/p} 
\end{equation}
for all $Q \in \cQ_{K^{-1}}$. Summing in $Q \in \cQ_{K^{-1}}$, which induces an additional factor of $K^{n/p}$, yields our estimate for the $L^p$ broad term in \eqref{eq: b/n}.\medskip

Applying Corollary~\ref{cor: b/n}, \eqref{eq: multilin red 2} and \eqref{eq: multilin red 4}, it follows that 
\begin{equation*}
    \big\|\sum_{T \in \T} \chi_T\big\|_{L^p(\R^n)} \lesssim_{\phi} \delta^{-\beta(n;p')} \big(K^{-e(n,k;p)} + K^{A+(n-1)(1+k/p) + n/p} \delta^{\eta}\big)\Big(\sum_{T \in \T} |T| \Big)^{1/p}.
\end{equation*}
The above inequality is valid for all choices of $K \geq 1$. Furthermore, under the condition $2 \leq k < \frac{n+3}{2}$, the exponent $e(n,k;p)$ defined in \eqref{eq: multilin red 3} is strictly positive. Thus, if we choose $K := \delta^{-\eta/E}$ where $E := 2(A + 2n^2)$, say, then it is clear that this yields an estimate of the desired form.
\end{proof}

%%%%%%%%%%%%%%%%%%%%%%%%%%%%%%%%%%%%%%%%%%%%%%%%%%%%%%%%%%%%%%%%%%%%%%%%%%%%%%%%%%%%%%%%%%%%%%%%

%  Universal multilinear estimate

%%%%%%%%%%%%%%%%%%%%%%%%%%%%%%%%%%%%%%%%%%%%%%%%%%%%%%%%%%%%%%%%%%%%%%%%%%%%%%%%%%%%%%%%%%%%%%%%

\subsection{Universal multilinear estimate} In light of Proposition~\ref{prop: multilinear red}, we shall henceforth focus on proving multilinear Kakeya/Nikodym estimates. The starting point is to consider certain \textit{universal} multilinear inequalities, which hold for all non-degenerate phases. These universal bounds follow from a variant of the Bennett--Carbery--Tao~\cite{BCT2006} theorem. 

It is convenient to work with a slightly more general setup, which allows us to treat both the Kakeya and Nikodym cases in a single framework. In what follows, we say a function $\gamma \colon \R \to \R^{n-1}$ is a \textit{$C^{\infty}$} (respectively, \textit{polynomial}) \textit{mapping} if the components $\gamma_j \colon \R \to \R$ are $C^\infty$ functions (respectively, polynomials) for $1 \leq j \leq n-1$.

\begin{definition} We say $\bdGamma$ \textit{defines a uniform family of $C^{\infty}$} (respectively, \textit{polynomial}) \textit{curves in $\R^n$} if $\bdGamma$ is a set of $C^{\infty}$ (respectively, polynomial) mappings $\gamma \colon (-1, 1) \to \R^{n-1}$ and there exists some sequence $(M_N)_{N \in \N} \subset (0,\infty)$ such that
\begin{equation}\label{eq: uniform family curves}
    \sup_{\gamma \in \bdGamma} \sup_{-1 < t < 1}|\gamma^{(N)}(t)| \leq M_N \qquad \textrm{for all $N \in \N$.}
\end{equation}
In this case, for each $\gamma \in \bdGamma$ we define a curve in $\R^n$ given by the graph
\begin{equation*}
    \Gamma_{\gamma} = \big\{(\gamma(t), t) : t \in (-1,1) \big\}.
\end{equation*}
\end{definition}

\begin{definition} Let $\bdGamma$ define a uniform family of $C^{\infty}$ curves in $\R^n$ and $0 < \delta < \nu < 1$. 
\begin{enumerate}[i)]
    \item Given $\gamma \in \bdGamma$, we define $T \subseteq \R^n$ the \textit{$\delta$-tube} associated to $\gamma$ by
\begin{equation*}%\label{eq: gen delta tube}
    T := \big\{(x,t) \in \B^{n-1} \times (-1,1) : |x - \gamma(t)| < \delta \big\}.
\end{equation*}
In this case, we write $\gamma_T := \gamma$ and $\Gamma_T := \Gamma_{\gamma} = \{(\gamma_T(t), t) : t \in (-1,1)\}$ is referred to as the \textit{core curve} of $T$. 
\item We say $\T$ is a \textit{family of $(\bdGamma, \delta)$-tubes} if $\T$ is a finite set of $\delta$-tubes associated to elements of $\bdGamma$. 
\item Given $2 \leq k \leq n$, we say $(T_1, \dots, T_k) \in \T^k$ is \textit{$\nu$-transversal} if
\begin{equation}\label{eq: gen nu transverse}
   \inf_{\substack{y_{\ell} \in \Gamma_{{T_\ell}},\, |\bx - y_{\ell}| < \delta \\ 1 \leq \ell \leq k}} \big| \bigwedge_{\ell = 1}^k T_{y_{\ell}} \Gamma_{T_{\ell}} \big| \geq \nu > 0 \qquad \textrm{for all $\bx \in T_1 \cap \cdots \cap T_k$.}
\end{equation}
where $T_{y_{\ell}} \Gamma_{T_{\ell}}$ denotes the tangent space to $\Gamma_{T_{\ell}}$ at $y_{\ell}$ and the infimum is taken over all tuples $(y_1, \dots, y_k) \in \Gamma_{T_1} \times \cdots \times \Gamma_{T_k}$ with $|\bx - y_{\ell}| < \delta$ for $1 \leq \ell \leq k$. For $X \subseteq \R^n$, we say $(T_1, \dots, T_k) \in \T^k$ is \textit{$\nu$-transversal on $X$} if the inequality \eqref{eq: gen nu transverse} holds on the more restrictive range $\bx \in T_1 \cap \cdots \cap T_k \cap X$.\footnote{We remark that \eqref{eq: gen nu transverse} agrees with \eqref{eq: nu transverse} for an appropriate choice of $\nu$ when the tubes are defined by a phase function.}
\item We say $\T_1,\dots, \T_k \subset \T$ are \textit{$\nu$-transversal} (respectively, \textit{$\nu$-transversal on $X$}) if every $k$-tuple $(T_1, \dots, T_k)$ with $T_{\ell} \in \T_{\ell}$, $1 \leq \ell \leq k$, is $\nu$-transversal  (respectively, $\nu$-transversal on $X$).
\end{enumerate}
\end{definition}

Having made these preliminary definitions, we can now state the universal multilinear theorem. 

\begin{theorem}[Universal multilinear estimate]\label{thm: uni multilin} For $2 \leq k \leq n$ and $1 \leq p \leq \frac{k}{k-1}$, there exists a constant $A = A_{n,k,p} \geq 1$ such that the following holds. Let $\bdGamma$ define a uniform family of $C^{\infty}$ curves in $\R^n$ and, for $0 < \delta < 1$, let $\T$ be a family of $(\bdGamma, \delta)$-tubes. Then, for all $\varepsilon > 0$, the inequality
\begin{equation}\label{eq: uni multilin}
  \Big\|\prod_{\ell = 1}^k \big| \sum_{T_{\ell} \in \T_{\ell}} \chi_{T_{\ell}} \big|^{1/k} \Big\|_{L^p(X \cap B_r)} \lesssim_{\bdGamma, \varepsilon} \nu^{-A} r^{k/p - (k - 1)} \delta^{k - 1 + (n-k)/p - \varepsilon} \prod_{\ell = 1}^k \big(\#\T_{\ell}\big)^{1/k}.
\end{equation}
holds whenever $\T_1, \dots, \T_k \subset \T$ are $K^{-1}$-transversal on $X$ for some $\nu  > 0$ and a measurable set $X \subseteq \R^n$ and $B_r \subseteq \R^n$ is a ball of radius $\delta < r < 1$. 
\end{theorem}

More precisely, one can ensure that the implicit constant in \eqref{eq: uni multilin} depends on $n$, $\varepsilon$ and the constant $M_1$ in \eqref{eq: uniform family curves} only.

\begin{proof}[Proof (of Theorem~\ref{thm: uni multilin})] By log convexity of $L^p$-norms, it suffices to only consider the cases $p = 1$ and $p = \frac{k}{k-1}$. The $p=1$ case follows simply by applying the multilinear H\"older inequality and exchanging the order of summation and integration, using the fact that $|T_{\ell} \cap B_r| \lesssim r \delta^{n-1}$ for any $T_{\ell} \in \T_{\ell}$. The $p = \frac{k}{k-1}$ case is a curved variant of the Bennett--Carbery--Tao theorem~\cite{BCT2006}: see, for instance, \cite{BG2011, Guth2015}. We remark that the result of \cite{Guth2015} ensures that the implicit constant depends on $\bdGamma$ in the manner described above. 
 \end{proof}

%%%%%%%%%%%%%%%%%%%%%%%%%%%%%%%%%%%%%%%%%%%%%%%%%%%%%%%%%%%%%%%%%%%%%%%%%%%%%%%%%%%%%%%%%%%%%%%%

%  Improved multilinear estimates via non-concentration

%%%%%%%%%%%%%%%%%%%%%%%%%%%%%%%%%%%%%%%%%%%%%%%%%%%%%%%%%%%%%%%%%%%%%%%%%%%%%%%%%%%%%%%%%%%%%%%%

\subsection{Improved multilinear estimates via non-concentration} Our goal is to establish multilinear estimates which go beyond the universal bounds of Theorem~\ref{thm: uni multilin}. To do so, we shall impose a non-concentration hypothesis on the underlying family of tubes $\T$, similar to that appearing in Definition~\ref{dfn: non-concentrated}.

\begin{definition} Let $\bdGamma$ define a uniform family of $C^{\infty}$ curves in $\R^n$, $0 < \delta < 1$, $0 \leq m \leq n$ and $\kappa > 0$, $B = (B_1, B_2) \in [0,\infty)^2$. We say a family $\T$ of $(\bdGamma, \delta)$-tubes is $(\kappa, B)$-non-concentrated in codimension $m$ grains if the following holds.\medskip

\noindent We have $\#\T \leq \delta^{-(n-1)}$ and, for all $d \in \N$ and $\varepsilon > 0$, the inequality
    \begin{equation}\label{eq: gen non-concentrated}
        \#\big\{ T \in \T : | T \cap G| \geq \lambda |T| \big\} \lesssim_{\bdGamma, \varepsilon, d} (\lambda/\rho)^{-B_1} \rho^{-B_2} \delta^{-(n-1) + \kappa - \varepsilon}
    \end{equation}
   holds whenever $G \subseteq \R^n$ is a $(\delta, \rho, d)$-grain of codimension $m$ for $\delta \leq \lambda \leq \rho \leq 1$.
\end{definition}

More precisely, the implied constant in \eqref{eq: gen non-concentrated} depends on $n$, $\varepsilon$, $d$ and a finite number (depending only on $n$ and $\varepsilon$) of the $M_N$ in \eqref{eq: uniform family curves} only.

\begin{theorem}[Multilinear Kakeya/Nikodym estimate]\label{thm: multilin Kak/Nik} For $2 \leq k \leq n$,  $1 < p < \frac{k}{k-1}$ and $\kappa > 0$, $B \in [0,\infty)^2$, there exists some $\eta(p) > 0$ such that the following holds. Suppose that $\bdGamma$ defines a uniform family of $C^{\infty}$ curves in $\R^n$ and, for $0 < \delta < 1$, that $\T$ is a family of $(\bdGamma, \delta)$-tubes which is $(\kappa, B)$-non-concentrated in codimension $n-k$ grains. Then the inequality
\begin{equation}\label{eq: multilin Kak/Nik}
    \Big\|\prod_{\ell = 1}^k \big| \sum_{T_{\ell} \in \T_{\ell}} \chi_{T_{\ell}} \big|^{1/k} \Big\|_{L^p(\R^n)} \lesssim_{\bdGamma,p} \nu^{-A} \delta^{-(n - k)/p' + \eta(p)} \Big(\sum_{T \in \T} |T| \Big)^{1/p}
\end{equation}
holds whenever $\T_1, \dots, \T_k \subset \T$ are $\nu$-transversal families for some $\nu > 0$. 
\end{theorem}

More precisely, the implied constant in \eqref{eq: multilin Kak/Nik} can be taken to depend only on $n$ and a finite number (depending only on $n$, $k$, $p$, $\kappa$ and $B$) of the quantities $M_N$ in \eqref{eq: uniform family curves}. The exponent $A = A_{n,k,p} \geq 1$ is the same as that appearing in Theorem~\ref{thm: uni multilin}. Finally, the choice of exponent $\eta(p)$ depends on $n$, $k$, $p$, $\kappa$ and $B$. We remark that the theorem is vacuously true for $k=n$, since there do not exist $(\bdGamma, \delta)$-tubes which are $(\kappa, B)$-non-concentrated in codimension $0$ grains. We also remark that there is a gain of a factor $\delta^{\eta(p)}$ in \eqref{eq: multilin Kak/Nik} compared with \eqref{eq: uni multilin}, as one may check the numerology by plugging in $\#\T=\#\T_\ell\sim\delta^{-(n-1)}, |T|\sim\delta^{n-1}$ and $r=1$.

We shall prove Theorem~\ref{thm: multilin Kak/Nik} in the next section, by adapting the polynomial partitioning arguments of Guth~\cite{Guth2018} and their refinements from \cite{HRZ2022}.

%%%%%%%%%%%%%%%%%%%%%%%%%%%%%%%%%%%%%%%%%%%%%%%%%%%%%%%%%%%%%%%%%%%%%%%%%%%%%%%%%%%%%%%%%%%%%%%%

%   Geometric maximal estimates

%%%%%%%%%%%%%%%%%%%%%%%%%%%%%%%%%%%%%%%%%%%%%%%%%%%%%%%%%%%%%%%%%%%%%%%%%%%%%%%%%%%%%%%%%%%%%%%%

\section{Geometric maximal estimates II: Polynomial partitioning}\label{sec: partitioning}

In this section, we adapt the polynomial partitioning arguments of \cite{Guth2018, HRZ2022} to apply to families of $(\bdGamma, \delta)$-tubes satisfying weak non-concentration hypotheses. At the end of the section, we apply the resulting partitioning algorithm to deduce Theorem~\ref{thm: multilin Kak/Nik}.

%%%%%%%%%%%%%%%%%%%%%%%%%%%%%%%%%%%%%%%%%%%%%%%%%%%%%%%%%%%%%%%%%%%%%%%%%%%%%%%%%%%%%%%%%%%%%%%%

%   Polynomial partitioning

%%%%%%%%%%%%%%%%%%%%%%%%%%%%%%%%%%%%%%%%%%%%%%%%%%%%%%%%%%%%%%%%%%%%%%%%%%%%%%%%%%%%%%%%%%%%%%%%

\subsection{Polynomial partitioning} Here we describe the main polynomial partitioning theorem from \cite{Guth2018}. Before presenting the statement, we recall some preliminary definitions. 

Given a polynomial $P \colon \R^n \to \R$ consider the collection  $\cell(P)$  
 of connected components of $\R^n \setminus Z(P)$. Each $O' \in \cell(P)$ is referred to as a \emph{cell} cut out by the variety $Z(P)$ and the cells are thought of as partitioning the ambient Euclidean space into a finite collection of disjoint regions. In order to account for the choice of scale $\delta > 0$ appearing in the definition of the $\delta$-tubes, it will be useful to consider the family of \emph{$\delta$-shrunken cells} defined by
\begin{equation}\label{shrunken cells}
    \cO := \big\{ O'\setminus N_{\delta}Z(P) : O' \in \cell(P) \big\}.
\end{equation}
Finally, given a finite set $A$, we say $A'$ is a
\textit{refinement} of $A$ if $A' \subseteq A$ and $\# A' \gtrsim \# A$. 

\begin{theorem}[Guth~\cite{Guth2018}]\label{thm: partitioning} Fix $0 < \delta < \rho$, $d \in \N$ and suppose $F \in L^1(\R^n)$ is non-negative and supported on a $(2\delta, \rho, d)$-grain $G$ of codimension $m$. There exists a dimensional constant $C_{\deg} \geq 1$ such that at least one of the following cases holds:\medskip

\paragraph{\underline{Cellular case}} There exists a polynomial $P \colon \R^n \to \R$ of degree at most $C_{\deg} d$ with the following properties:
\begin{enumerate}[i)]
    \item $\#\cell(P) \sim d^{n-m}$ and each $O \in \cell(P)$ has diameter at most $\rho/2$.
    \item One may pass to a refinement $\cell(P)'$ of $\cell(P)$ such that if $\cO$ is defined as in~\eqref{shrunken cells} with $\cell(P)$ replaced with $\cell(P)'$, then 
    \begin{equation*}
        \int_{\R^n} F  \lesssim d^{n-m} \int_{O} F \qquad \textrm{for all $O \in \cO$.}
    \end{equation*}
\end{enumerate}  
\paragraph{\underline{Algebraic case}} There exists a $(\delta, \rho, C_{\deg} d)$-grain $G'$ of codimension $m+1$ such that
    \begin{equation*}
        \int_{G} F \,\lesssim\,\log d \int_{G'} F.
    \end{equation*}
\end{theorem}

 The statement of this theorem does not explicitly appear in~\cite{Guth2018}, but it can be easily deduced from the argument described in Section 8.1 of that article together with simple pigeonholing arguments.

%%%%%%%%%%%%%%%%%%%%%%%%%%%%%%%%%%%%%%%%%%%%%%%%%%%%%%%%%%%%%%%%%%%%%%%%%%%%%%%%%%%%%%%%%%%%%%%%

%   Finding algebraic structure

%%%%%%%%%%%%%%%%%%%%%%%%%%%%%%%%%%%%%%%%%%%%%%%%%%%%%%%%%%%%%%%%%%%%%%%%%%%%%%%%%%%%%%%%%%%%%%%%

\subsection{Finding algebraic structure} Here we describe the main partitioning algorithm. We present this in two parts: roughly speaking, the first algorithm \texttt{[partition $m \to m+1$]} reduces a codimension $m$ arrangement of tubes (that is, the tubes have large intersection with a small neighbourhood of a low degree codimension $m$ variety) to a codimension $m+1$ arrangement. The second algorithm \texttt{[partition]} involves repeated application of \texttt{[partition $m \to m+1$]}. We formulate the algorithm in terms of general families of tubes formed around polynomial curves (which do not necessarily arise from a phase-amplitude pair).  

\begin{definition} Let $\bdGamma$ define a uniform family of polynomial curves in $\R^n$, $0 < \delta < 1$ and $E \in \N$. We say a family of $(\bdGamma, \delta)$-tubes $\T$ has \textit{degree bounded by $E$} if
\begin{equation*}
    \max\{\deg \gamma_{T,1}, \dots, \deg \gamma_{T,n-1}\} \leq E \qquad \textrm{for all $T \in \T$,}
\end{equation*}
where the $\gamma_{T,j}$ are the components of $\gamma_T$.
\end{definition}

Given any finite family $\T$ of $(\bdGamma, \delta)$-tubes in $\R^n$, as earlier we write
\begin{equation*}
    \mu_{\T} := \sum_{T \in \T} \chi_T.
\end{equation*}
Our algorithm essentially consists of repeated applications of Theorem~\ref{thm: partitioning} to $\mu_{\T}$.\medskip

%%%%%%%%%%%%%%%%%%%%%%%%%%%%%%%%%%%%%%%%%%%%%%%%%%%%%%%%%%%%%%%%%%%%%%%%%%%%%%%%%%%%%%%%%%%%%%%%

%   The first algorithm

%%%%%%%%%%%%%%%%%%%%%%%%%%%%%%%%%%%%%%%%%%%%%%%%%%%%%%%%%%%%%%%%%%%%%%%%%%%%%%%%%%%%%%%%%%%%%%%%

\noindent\textbf{The first algorithm.} With the above preliminaries in hand, we turn to describing the first partitioning algorithm \texttt{[partition $m \to m+1$]}.\medskip

\noindent\underline{\texttt{Input}.} \texttt{[partition $m \to m+1$]} will take as its input:
\begin{itemize}
     \item A set $\bdGamma$ defining a uniform family of polynomial curves in $\R^n$.
     \item Parameters $1 \leq p < \infty$, $0< \varepsilon_{\circ}$, $\kappa_{\ttiny} < 1$, $0<\delta \leq \rho \leq 1$ and $ 1 \leq E \leq d$. Furthermore, $d$ is chosen large, depending on $\varepsilon_{\circ}$ and $E$. 
    \item A family $\T$ of $(\bdGamma, \delta)$-tubes of degree bounded by $E \geq 1$.  
    \item An integer $2 \leq k \leq n$ and subfamilies $\T_1, \dots, \T_k \subseteq \T$.
    \item A $(\delta, \rho, d)$-grain $G_m$ in $\R^n$ of codimension $0 \leq m \leq n-2$ and a set $X := \R^{n-1} \times I$ for some choice of open interval $I \subseteq (-1,1)$.
\end{itemize}

Note that the process applies to essentially arbitrary families of $(\bdGamma, \delta)$-tubes: in particular, neither the direction-separated nor the centre-separated hypothesis appear at this stage, and nor does any transversality hypothesis.\medskip

%
%
% Output
%
\noindent\underline{\texttt{Output}.} \texttt{[partition $m \to m+1$]} outputs finite sequences 
\begin{equation*}
    (\cO_j)_{j=0}^J \qquad \textrm{and} \qquad (\T_{\ell}[O_j] : O_j \in \cO_j)_{j=0}^J \qquad \textrm{for $1 \leq \ell \leq k$,}
\end{equation*}
where:
\begin{itemize}
  \item Each $\cO_j$ is a family of measurable subsets of $G_m \cap X$, each of diameter at most $r_j := 2^{1-j}\rho$. The sets $O_j \in \cO_j$ are referred to as \emph{cells}.
\item $\T_{\ell}[O_j] \subseteq \T_{\ell}$ is a subfamily of $\delta$-tubes, for each $O_j \in \cO_j$ and $1 \leq \ell \leq k$. 
  \end{itemize}
Moreover, the above objects are defined so as to ensure that the following hold:\medskip

\noindent\underline{Property I.} 
\begin{equation}\tag*{$(\mathrm{I})_j$}
      \Big\|\prod_{\ell = 1}^k \big| \mu_{\T_{\ell}} \big|^{1/k} \Big\|_{L^p(X \cap G_m)}^p \leq d^{\varepsilon_{\circ}j} \sum_{O_{j} \in \cO_{j}} \Big\|\prod_{\ell = 1}^k \big| \mu_{\T_{\ell}[O_j]} \big|^{1/k}\Big\|_{L^p(O_j)}^p;
\end{equation}

\noindent\underline{Property II.} \ \ For all $1\le \ell \le k$,
\begin{equation}\tag*{$(\mathrm{II})_j$}
    \sum_{O_{j} \in \cO_{j}} \# \T_{\ell}[O_j] \leq d^{(1+\varepsilon_{\circ})j}  \#\T_{\ell};
\end{equation}

\noindent\underline{Property III.} \ \ For all $1\le \ell \le k$,
\begin{equation}\tag*{$(\mathrm{III})_j$}
     \max_{O_{j} \in \cO_{j}} \#\T_{\ell}[O_{j}] \leq  d^{-(n-m-1-\varepsilon_{\circ})j}  \#\T_{\ell}.
\end{equation}

\noindent\underline{\texttt{Initial step}.} Take $\cO_0 := \{G_m\}$ and $\T_{\ell}[G_m] := \T_{\ell}$ for $1 \leq \ell \leq k$. By definition, the grain $G_m$ is contained in  a ball of radius $\rho$ and therefore has diameter at most $r_0 := 2 \rho$. Furthermore, Property I, II and III all vacuously hold.\medskip

%
%
% Stopping conditions
%

\noindent\underline{\texttt{Stopping conditions}.} The algorithm terminates with the completion of its $j$th step, for some $j\in \mathbb{N}_0$, if either of two stopping conditions is met. These stopping conditions are labelled \texttt{[tiny]} and \texttt{[algebraic]}. To describe them, we let $C_{\ttiny}(\bdGamma, d)$ be a constant depending only on $M_1$ in \eqref{eq: uniform family curves} and $d$ and $C_{\alg} \geq 1$ be another constant which depends only on admissible parameters. Both $C_{\ttiny}(\bdGamma, d)$ and $C_{\alg}$ are chosen sufficiently large for the purposes of the forthcoming argument.\medskip

\noindent \underline{\texttt{Stop:[tiny]}.} There exist
\begin{itemize}
    \item $\cO_{\ttiny}$ a family of measurable subsets of $X$ with $\diam\,O \leq \delta^{\kappa_{\ttiny}}$ for all $O \in \cO_{\ttiny}$;
    \item $\T_{\ell}[O] \subseteq \T_{\ell}$ a subfamily of $\delta$-tubes for each $O \in \cO_{\ttiny}$ and $1 \leq \ell \leq k$,
\end{itemize}
such that the following hold:\medskip

\noindent \underline{Property I$_{\ttiny}$.}
    \begin{equation*}
        \sum_{O_{j} \in \cO_{j}} \Big\|\prod_{\ell = 1}^k \big| \mu_{\T_{\ell}[O_j]} \big|^{1/k}\Big\|_{L^p(O_j)}^p \leq C_{\ttiny}(\bdGamma, d) \sum_{O \in \cO_{\ttiny}} \Big\|\prod_{\ell = 1}^k \big| \mu_{\T_{\ell}[O]}  \big|^{1/k}\Big\|_{L^p(O)}^p;
    \end{equation*}
\noindent \underline{Property II$_{\ttiny}$.}
    \begin{equation*}
        \sum_{O \in \cO_{\ttiny}} \#\T_{\ell}[O] \leq C_{\ttiny}(\bdGamma, d)  \sum_{O_{j} \in \cO_{j}} \# \T_{\ell}[O_j];
    \end{equation*}
\noindent \underline{Property III$_{\ttiny}$.}
    \begin{equation*}
         \max_{O \in \cO_{\ttiny}} \#\T_{\ell}[O]  \leq C_{\ttiny}(\bdGamma, d)\max_{O_{j} \in \cO_{j}} \# \T_{\ell}[O_j].
    \end{equation*}

\noindent \underline{\texttt{Stop:[algebraic]}.} Let $r_j := 2^{1-j} \rho$ be as above and $C_{\deg} \geq 1$ be the dimensional constant appearing in the statement of Theorem~\ref{thm: partitioning}. There exist
\begin{itemize}
    \item $\cO_{\alg} \subseteq \cO_j$;
    \item $G_O$ a $(\delta, r_j, C_{\deg} d)$-grain of codimension $m+1$ for each $O \in \cO_{\alg}$
\end{itemize}
such that 
\begin{align}
\label{eq: stop alg}
    \sum_{O_{j} \in \cO_{j}} \Big\|\prod_{\ell = 1}^k  \big| \mu_{\T_{\ell}[O_j]} \big|^{1/k}&\Big\|_{L^p(O_j)}^p\\
\nonumber
&  \leq C_{\alg} \sum_{s = 1}^k \sum_{O \in \cO_{\alg}} \Big\|\big| \mu_{\T_s[G_O]}  \big|^{1/k}\prod_{\substack{1 \leq \ell  \leq k\\ \ell \neq s}}\big| \mu_{\T_{\ell}[O]} \big|^{1/k}\Big\|_{L^p(X \cap G_O)}^p
\end{align}
where 
\begin{equation}\label{eq: alg tubes}
    \T_s[G_O] := \Big\{ T_s \in \T_s[O] : |T_s \cap X \cap N_{2\delta}G_O| \geq \delta^{\kappa_{\ttiny}} |T_s| \Big\}
\end{equation}
for $O \in \cO_{\alg}$ and $1 \leq s \leq k$. \medskip

%
%
% Recursive step
%

%
%
% Recursive step
%

\noindent\underline{\texttt{Recursive step}.} Suppose for some $j \in \N_0$ we are at the $j$th stage of the recursive process, so in particular $\cO_j$ and the $(\T_{\ell}[O_j] : O_j \in \cO_j)_{\ell = 1}^k$ have been constructed and satisfy $(\mathrm{I})_j$, $(\mathrm{II})_j$ and $(\mathrm{III})_j$. Further suppose that neither stopping condition \texttt{[tiny]} nor \texttt{[algebraic]} is met. Our goal is to construct the next generation of objects $\cO_{j+1}$ and $(\T_{\ell}[O_{j+1}] : O_{j+1} \in \cO_{j+1})_{\ell = 1}^k$ so that $(\mathrm{I})_{j+1}$, $(\mathrm{II})_{j+1}$ and $(\mathrm{III})_{j+1}$ hold.

Fix $O_j \in \cO_j$ and recall that $O_j \subseteq X \cap G_m$. We apply the polynomial partitioning result from Theorem~\ref{thm: partitioning} with degree $d$ to the function
\begin{equation*}
    \prod_{\ell = 1}^k \big| \mu_{\T_{\ell}[O_j]}  \big|^{p/k} \chi_{O_j},
\end{equation*}
which, in particular, is supported in the codimension $m$ grain $G_m$. For each $O_j \in \cO_j$ either the cellular or the algebraic case holds, as defined in Theorem~\ref{thm: partitioning}. Let $\cO_{j,\cell}$ denote the subcollection of $\cO_j$ consisting of all cells for which the cellular case holds and $\cO_{j,\wall} := \cO_j \setminus \cO_{j,\cell}$.\medskip
 
\begin{claim}\label{claim: cell dominate} The cells in $\cO_{j,\cell}$ dominate, in the sense that 
\begin{equation}\label{eq: cell dom}
   \sum_{O_j \in \cO_j} \Big\|\prod_{\ell = 1}^k \big| \mu_{\T_{\ell}[O_j]} \big|^{1/k} \Big\|_{L^p(O_j)}^p \leq 2 \sum_{O_j \in \cO_{j,\cell}} \Big\|\prod_{\ell = 1}^k \big| \mu_{\T_{\ell}[O_j]}  \big|^{1/k} \Big\|_{L^p(O_j)}^p.
\end{equation} 
\end{claim}

We shall temporarily assume the above claim and show how to construct the cells $\cO_{j+1}$. Using the claim together with $(\mathrm{I})_j$, we have
 \begin{equation}\label{eq: cellular 1}
   \Big\|\prod_{\ell = 1}^k \big| \mu_{\T_{\ell}}  \big|^{1/k} \Big\|_{L^p(G_m\cap X)}^p \leq 2d^{\varepsilon_{\circ}j} \sum_{O_j \in \cO_{j,\cell}} \Big\|\prod_{\ell = 1}^k \big| \mu_{\T_{\ell}[O_j]} \big|^{1/k} \Big\|_{L^p(O_j)}^p.
\end{equation} 
The next generation of cells $\cO_{j+1}$ will now arise from the cellular decomposition guaranteed by Theorem~\ref{thm: partitioning}. Fix $O_j \in \cO_{j, \cell}$ so that there exists some polynomial $P \in \R[X_1, \dots, X_n]$ of degree at most $C_{\deg} d$ with the following properties:
\begin{enumerate}[i)]
    \item $\#\cell(P) \sim d^{n-m}$ and each $O \in \cell(P)$ has diameter at most $r_{j+1} = 2^{-j} \rho$. 
    \item One may pass to a refinement $\cell(P)'$ of $\cell(P)$ such that if 
    \begin{equation*}
        \cO_{j+1}'(O_j) := \big\{ \big(O\setminus N_{\delta}Z(P)\big) \cap X \cap G_m : O \in \cell(P)'\}
    \end{equation*}
    denotes the corresponding collection of $\delta$-shrunken cells, then 
    \begin{equation*}
    \Big\|\prod_{\ell = 1}^k \big| \mu_{\T_{\ell}[O_j]} \big|^{1/k} \Big\|_{L^p(O_j)}^p \lesssim\, d^{n-m} \Big\|\prod_{\ell = 1}^k \big| \mu_{\T_{\ell}[O_j]} \big|^{1/k} \Big\|_{L^p(O_{j+1})}^p
\end{equation*}
for all $O_{j+1} \in \cO_{j+1}'(O_j)$.
    \end{enumerate}
 Given $O_{j+1} \in \cO_{j+1}'(O_j)$, define
 \begin{equation*}
 \T_{\ell}[O_{j+1}] := \big\{ T_{\ell} \in \T_{\ell}[O_j] : T_{\ell} \cap O_{j+1} \neq \emptyset \big\}.
\end{equation*}

Recall that, for any $T \in \T$, the core curve $\Gamma_T$ of $T$ is parametrised by $(\gamma_T(t), t)$ for $t \in (-1,1)$, where the components of $\gamma_T$ are polynomials of degree at most $E$. Thus, by the Fundamental Theorem of Algebra, any $T \in \T$ can enter at most $Ed + 1$ cells $O_{j+1} \in \cO_{j+1}'(O_j)$. Consequently, 
\begin{equation}\label{eq: cellular 2}
    \sum_{O_{j+1} \in \cO_{j+1}'(O_j)} \#\T_{\ell}[O_{j+1}] \lesssim_E d \cdot \#\T_{\ell}[O_j],
\end{equation}
hence one may pass to a refinement of $\cO_{j+1}'(O_j)$ such that
 \begin{equation}\label{eq: cellular 3}
    \#\T_{\ell}[O_{j+1}] \lesssim_E d^{-(n-m -1)} \#\T_{\ell}[O_j] \qquad \textrm{for all $O_{j+1} \in \cO_{j+1}(O_j)$,}
\end{equation}
since $\#\cO_{j+1}'(O_j)\sim d^{n-m}$. Finally, define 
 \begin{equation*}
 \cO_{j+1} := \bigcup_{O_j \in \cO_{j, \cell}} \cO_{j+1}(O_j).
\end{equation*}
This completes the construction of $\cO_{j+1}$ and the $(\T_{\ell}[O_{j+1}] : O_{j+1} \in \cO_{j+1})_{\ell = 1}^k$; it remains to check that these objects satisfy the desired properties.\medskip

%
% Property I
%
%
\noindent\underline{Property I.} Fix $O_j \in \cO_{j,\cell}$ and observe that $\#\cO_{j+1}(O_j) \sim d^{n-m}$ and 
\begin{equation*}
    \Big\|\prod_{\ell = 1}^k \big| \mu_{\T_{\ell}[O_j]} \big|^{1/k} \Big\|_{L^p(O_j)}^p \lesssim d^{n-m} \Big\|\prod_{\ell = 1}^k \big| \mu_{\T_{\ell}[O_{j+1}]} \big|^{1/k} \Big\|_{L^p(O_{j+1})}^p
\end{equation*}
for all $O_{j+1} \in \cO_{j+1}(O_j)$. Averaging,
\begin{equation*}
    \Big\|\prod_{\ell = 1}^k \big| \mu_{\T_{\ell}[O_j]} \big|^{1/k} \Big\|_{L^p(O_j)}^p \lesssim \sum_{O_{j+1} \in \cO_{j+1}(O_j)}  \Big\|\prod_{\ell = 1}^k \big| \mu_{\T_{\ell}[O_{j+1}]}  \big|^{1/k} \Big\|_{L^p(O_{j+1})}^p.
\end{equation*}
Recalling~\eqref{eq: cellular 1}, one therefore deduces that
\begin{equation*}
    \Big\|\prod_{\ell = 1}^k \big| \mu_{\T_{\ell}} \big|^{1/k} \Big\|_{L^p(X \cap G_m)}^p \leq C d^{-\varepsilon_{\circ}} d^{\varepsilon_{\circ}(j+1)} \sum_{O_{j+1} \in \cO_{j+1}}  \Big\|\prod_{\ell = 1}^k \big| \mu_{\T_{\ell}[O_{j+1}]}  \big|^{1/k} \Big\|_{L^p(O_{j+1})}^p,
\end{equation*}
where $C \geq 1$ is an admissible constant which is an amalgamation of the implicit constants arising in the above argument. Provided $d$ is chosen large enough so as to ensure that the additional $d^{-\varepsilon_{\circ}}$ factor absorbs the unwanted constant $C$, one deduces $(\mathrm{I})_{j+1}$.\medskip

%
% Property II
%
%
\noindent\underline{Property II.} By the construction,
\begin{align*}
    \sum_{O_{j+1} \in \cO_{j+1}} \#\T_{\ell}[O_{j+1}] &\,=\, \sum_{O_j\in \cO_{j,\mathrm{cell}}}\; \sum_{O_{j+1} \in \cO_{j+1}(O_j)} \#\T_{\ell}[O_{j+1}]  \\
    &\lesssim_E d \sum_{O_j \in \cO_j} \#\T_{\ell}[O_j], 
\end{align*}
where the inequality follows from a term-wise application of~\eqref{eq: cellular 2}. Thus, $(\mathrm{II})_j$ implies that
\begin{equation*}
    \sum_{O_{j+1} \in \cO_{j+1}} \#\T_{\ell}[O_{j+1}] \leq C_E d^{-\varepsilon_{\circ}} d^{(1+\varepsilon_{\circ})(j+1)}  \#\T_{\ell}.
\end{equation*}
Provided $d$ is chosen sufficiently large, depending on $E$ and admissible parameters, one deduces $(\mathrm{II})_{j+1}$.\\

%
% Property III
%
%
\noindent\underline{Property III.} Let $O_j\in \cO_{j,\mathrm{cell}}$. Fix $O_{j+1} \in \cO_{j+1}(O_j)$ and recall from~\eqref{eq: cellular 3} that
\begin{equation*}
    \#\T_{\ell}[O_{j+1}] \lesssim_E d^{-(n - m - 1)}\#\T_{\ell}[O_{j}].
\end{equation*}
Thus, $(\mathrm{III})_j$ implies that
\begin{equation*}
    \#\T_{\ell}[O_{j+1}]\leq C_E d^{-\varepsilon_{\circ}} d^{-(n-m-1-\varepsilon_{\circ})(j+1)} \#\T_{\ell}.
\end{equation*}
Provided $d$ is chosen sufficiently large, one deduces $(\mathrm{III})_{j+1}$.\medskip

This concludes the description of the recursive step of \texttt{[partition $m \to m+1$]}, except that it remains to justify Claim~\ref{claim: cell dominate}. 

\begin{proof}[Proof (of Claim~\ref{claim: cell dominate})] If the condition \eqref{eq: cell dom} of the claim fails, then 
\begin{equation}\label{eq: alg dom}
   \sum_{O_j \in \cO_j} \Big\|\prod_{\ell = 1}^k \big| \mu_{\T_{\ell}[O_j]} \big|^{1/k} \Big\|_{L^p(O_j)}^p \leq 2 \sum_{O_j \in \cO_{j,\wall}} \Big\|\prod_{\ell = 1}^k \big| \mu_{\T_{\ell}[O_j]} \big|^{1/k} \Big\|_{L^p(O_j)}^p.
\end{equation} 
We shall show that this implies either the stopping condition \texttt{[tiny]} or \texttt{[algebraic]} is met, contrary to our initial assumption. For this, we use some ideas from the proof of \cite[Proposition 3.1]{GZ2018}.

Fix $O_j \in \cO_{j,\wall}$. By definition, there exists some $(\delta, r_j, C_{\deg} d)$-grain $G = G_{O_j}$ of codimension $m+1$ such that
\begin{equation}\label{eq: alg case 1}
   \Big\|\prod_{\ell = 1}^k \big| \mu_{\T_{\ell}[O_j]} \big|^{1/k} \Big\|_{L^p(O_j)}^p \lesssim \log d \int_{X \cap G} \Big|\sum_{(T_1, \dots, T_k) \in \T^{\otimes}[O_j]} \prod_{\ell = 1}^k \chi_{T_{\ell}}(\bx) \Big|^{p/k} \,\ud \bx,
\end{equation}
where $\T^{\otimes}[O_j] := \T_1[O_j] \times \cdots \times \T_k[O_j]$. Given a set $S \subseteq \R^n$ and $\bx \in S$, let $\CC(S;\bx)$ denote the Euclidean connected component of $S$ containing $\bx$. For each $\bx \in G$, we partition $\T^{\otimes}[O_j]$ into two sets:
\begin{align}
\nonumber
  \T^{\otimes}_{\ttiny}[O_j;\bx] &:=  \Big\{(T_1, \dots, T_k) \in \T^{\otimes}[O_j] : \max_{1 \leq \ell \leq k} \diam(\CC(T_{\ell} \cap X \cap G;\bx)) \leq C_{\bdGamma}\delta^{\kappa_{\ttiny}} \Big\}, \\
  \label{eq: alg tube set}
 \T^{\otimes}_{\alg}[O_j;\bx] &:= \T^{\otimes}[O_j] \setminus  \T^{\otimes}_{\ttiny}[O_j;\bx],
\end{align}
where $C_{\bdGamma} \geq 1$ is a constant which depends only on the quantity $M_1$ in \eqref{eq: uniform family curves}. Thus, we may dominate the right-hand side of \eqref{eq: alg case 1} by
\begin{equation}\label{eq: alg case 2}
    \int_{X \cap G} \Big|\sum_{\vec{T} \in \T^{\otimes}_{\ttiny}[O_j;\bx]} \prod_{\ell = 1}^k \chi_{T_\ell}(\bx)  \Big|^{p/k} \,\ud \bx + \int_{X \cap G} \Big|\sum_{\vec{T} \in \T^{\otimes}_{\alg}[O_j;\bx]}  \prod_{\ell = 1}^k \chi_{T_\ell}(\bx) \Big|^{p/k} \,\ud \bx.
\end{equation}
If the first term in \eqref{eq: alg case 2} is the largest of the two terms, then we shall see that the stopping condition \texttt{[tiny]} is met. If the second term is the largest, then we shall see \texttt{[algebraic]} is met. Either way, we shall arrive at a contradiction and thereby conclude the proof of the claim. To this end, we let $\cO_{\alg} \subseteq \cO_{j,\wall}$ denote the collection of cells $O_j \in \cO_{j,\wall}$ for which the second term in \eqref{eq: alg case 2} is the largest of the two terms. The analysis now divides into two cases.\medskip

\noindent \textbf{$\blacktriangleright$ Tiny-dominant case.} Suppose that the cells in $\cO_{\alg}$ do not dominate. More precisely, suppose that
\begin{equation}\label{eq: tiny dom}
   \sum_{O_j \in \cO_j} \Big\|\prod_{\ell = 1}^k \big| \mu_{\T_{\ell}[O_j]}  \big|^{1/k} \Big\|_{L^p(O_j)}^p \leq 4 \sum_{O_j \in \cO_{j,\wall} \setminus \cO_{\alg}} \Big\|\prod_{\ell = 1}^k \big| \mu_{\T_{\ell}[O_j]} \big|^{1/k} \Big\|_{L^p(O_j)}^p.
\end{equation}

Fix $O_j \in \cO_{j,\wall} \setminus \cO_{\alg}$ and let $G = G_{O_j}$ be the associated grain, satisfying \eqref{eq: alg case 1}. By the definition of $\cO_{\alg}$, we have
\begin{equation}\label{eq: tiny dom 1}
    \Big\|\prod_{\ell = 1}^k \big| \mu_{\T_{\ell}[O_j]}  \big|^{1/k} \Big\|_{L^p(O_j)}^p \lesssim \int_{X \cap G} \Big|\sum_{\vec{T} \in \T^{\otimes}_{\ttiny}[O_j;\bx]}  \prod_{\ell = 1}^k \chi_{T_\ell}(\bx) \Big|^{p/k} \,\ud \bx.
\end{equation}
Let $\cB_G$ be a finitely-overlapping cover of $G$ by balls of radius $\delta^{\kappa_{\ttiny}}/2$. We let $\cO_{\ttiny}(O_j) := \{B \cap X \cap G : B \in \cB_G\}$, so that the right-hand side of \eqref{eq: tiny dom 1} is bounded by 
\begin{equation*}
  \sum_{O \in \cO_{\ttiny}(O_j)} \int_O \Big|\sum_{\vec{T} \in \T^{\otimes}_{\ttiny}[O_j;\bx]}  \prod_{\ell = 1}^k \chi_{T_\ell}(\bx)  \Big|^{p/k} \,\ud \bx.
\end{equation*}
For each term in the sum, corresponding to a choice of $O \in \cO_{\ttiny}(O_j)$, we relax the summation over all tuples $\vec{T} = (T_1, \dots,  T_k) \in \T^{\otimes}_{\ttiny}[O_j;\bx]$ to summing over tuples $\vec{T} = (T_1, \dots,  T_k) \in \T_1[O] \times \cdots \times \T_k[O]$ where
\begin{equation*}
    \T_{\ell}[O] := \Big\{T_{\ell} \in \T_{\ell}[O_j] : \exists\, \bx \in T_{\ell} \cap O \textrm{ s.t. } \diam(\CC(T_{\ell} \cap X \cap  G; \bx)) \leq C_{\bdGamma}\delta^{\kappa_{\ttiny}}\Big\}. 
\end{equation*}

Let $\cO_{\ttiny}$ denote the disjoint union of the $\cO_{\ttiny}(O_j)$ over all $O_j \in \cO_{j,\wall} \setminus \cO_{\alg}$. We shall show that this collection, along with the tube families $\T_{\ell}[O]$ introduced above, realise the stopping condition \texttt{[tiny]}. Note that $\diam\,O < \delta^{\kappa_{\ttiny}}$ automatically holds for all $O \in \cO_{\ttiny}$, so it remains to demonstrate Property I$_{\ttiny}$, II$_{\ttiny}$ and III$_{\ttiny}$.\medskip

\noindent \underline{Property I$_{\ttiny}$.} By \eqref{eq: tiny dom} and the preceding observations,
\begin{equation*}
    \sum_{O_j \in \cO_j} \Big\|\prod_{\ell = 1}^k \big| \mu_{\T_{\ell}[O_j]}  \big|^{1/k} \Big\|_{L^p(O_j)}^p \leq C_{\ttiny}(\bdGamma, d) \sum_{O \in \cO_{\ttiny}} \Big\|\prod_{\ell = 1}^k \big| \mu_{\T_{\ell}[O]} \big|^{1/k} \Big\|_{L^p(O)}^p,
\end{equation*}
provided $C_{\ttiny}(\bdGamma, d) \geq 1$ is chosen sufficiently large, depending only on admissible parameters. This establishes Property I$_{\ttiny}$.\medskip

\noindent \underline{Property II$_{\ttiny}$.} We claim that, to prove Property II$_{\ttiny}$, it suffices to show
\begin{equation}\label{eq: tiny dom 2}
    \#\big\{O \in \cO_{\ttiny}(O_j) : T_{\ell} \in   \T_{\ell}[O]\big\} \lesssim_{C_{\bdGamma},d,E} 1, \qquad T_{\ell} \in \T_{\ell}[O_j], \;O_j \in \cO_{j,\wall} \setminus \cO_{\alg}.
\end{equation}
Indeed, it then follows that
\begin{equation*}
     \sum_{O \in \cO_{\ttiny}(O_j)} \#\T_{\ell}[O] = \sum_{T_{\ell} \in \T_{\ell}[O_j]} \#\big\{O \in \cO_{\ttiny}(O_j) : T_{\ell} \in   \T_{\ell}[O]\big\} \lesssim_{C_{\bdGamma}, d, E} \#\T_{\ell}[O_j]
\end{equation*}
for all $O_j \in \cO_{j,\wall} \setminus \cO_{\alg}$. Summing, we obtain
\begin{equation*}
    \sum_{O \in \cO_{\ttiny}} \#\T_{\ell}[O] = \sum_{O_j \in \cO_{j,\wall} \setminus \cO_{\alg}} \sum_{O \in \cO_{\ttiny}(O_j)} \#\T_{\ell}[O] \leq C_{\ttiny}(\bdGamma, d) \sum_{O_j \in \cO_j} \#\T_{\ell}[O_j],
\end{equation*}
provided $C_{\ttiny}(\bdGamma, d) \geq 1$ is chosen sufficiently large, depending only on $d$, $C_{\bdGamma}$ and admissible parameters (since $d \geq E$). This establishes Property II$_{\ttiny}$.

Turning to the proof of \eqref{eq: tiny dom 2}, fix $O_j \in \cO_{j,\wall} \setminus \cO_{\alg}$ and $T_{\ell} \in \T_{\ell}[O_j]$ and let $G = G_{O_j}$ be the grain associated to $O_j$. The goal is to show 
\begin{equation}\label{eq: tiny dom 3}
    \#\big\{O \in \cO_{\ttiny}(O_j) : \exists\, \bx \in T_{\ell} \cap O \textrm{ s.t. } \diam(\CC(T_{\ell} \cap X \cap  G; \bx)) \leq C_{\bdGamma}\delta^{\kappa_{\ttiny}}\big\} \lesssim_{C_{\bdGamma}, d,E} 1. 
\end{equation}
It is a simple consequence of the Tarski--Seidenberg projection theorem (see Theorem~\ref{thm: Tarski}) that $T_{\ell}$ and $G$ are both semialgebraic of complexity $O_{d,E}(1)$. It immediately follows that $T_{\ell} \cap X \cap G$ is also semialgebraic of complexity $O_{d,E}(1)$. Thus, by the main theorem in \cite{BPR1996}, the set $T_{\ell} \cap X \cap  G$ has $O_{d,E}(1)$ connected components. Let $W$ be one such connected component and further suppose that $\diam\,W  \leq C_{\bdGamma} \delta^{\kappa_{\ttiny}}$. Since $\cB_G$ is a finitely-overlapping family of balls of radius $\delta^{\kappa_{\ttiny}}/2$, it follows that $\#\{B \in \cB_G : W \cap B \neq \emptyset\} \lesssim_{C_{\bdGamma}} 1$. Thus, $W$ intersects at most $O_{C_{\bdGamma}}(1)$ of the sets $O \in \cO_{\ttiny}(O_j)$. Combining these observations, there can be at most $O_{C_{\bdGamma},d,E}(1)$ sets $O \in \cO_{\ttiny}(O_j)$ which intersect some connected component of $T_{\ell} \cap X \cap G$ of diameter at most $C_{\bdGamma} \delta^{\kappa_{\ttiny}}$. This is precisely the desired bound \eqref{eq: tiny dom 3}. 
\medskip

\noindent \underline{Property III$_{\ttiny}$.} Given $O \in \cO_{\ttiny}(O_j)$, by the definitions we see that $\T_{\ell}[O] \subseteq \T_{\ell}[O_j]$ and so Property III$_{\ttiny}$ is immediate. \medskip

\noindent Thus, we see that the stopping condition \texttt{[tiny]} is met, contrary to our earlier assumption.\medskip

\noindent \textbf{$\blacktriangleright$ Algebraic-dominant case.} Now suppose that \eqref{eq: tiny dom} fails. In light of \eqref{eq: alg dom}, it follows that 
\begin{equation}\label{eq: refined alg dom 1}
   \sum_{O_j \in \cO_j} \Big\|\prod_{\ell = 1}^k \big| \mu_{\T_{\ell}[O_j]}  \big|^{1/k} \Big\|_{L^p(O_j)}^p \leq 4 \sum_{O \in \cO_{\alg}} \Big\|\prod_{\ell = 1}^k \big| \mu_{\T_{\ell}[O]}  \big|^{1/k} \Big\|_{L^p(O)}^p.
\end{equation}
Fix $O \in \cO_{\alg}$ and let $G = G_O$ be the associated grain, satisfying \eqref{eq: alg case 1}. By the definition of $\cO_{\alg}$, we have
\begin{equation}\label{eq: refined alg dom 2}
    \Big\|\prod_{\ell = 1}^k \big| \mu_{\T_{\ell}[O]} \big|^{1/k} \Big\|_{L^p(O)}^p \lesssim \int_{X \cap G} \Big|\sum_{\vec{T} \in \T^{\otimes}_{\alg}[O;\bx]}  \prod_{\ell = 1}^k \chi_{T_\ell}(\bx) \Big|^{p/k} \,\ud \bx.
\end{equation}

Comparing the definition of $\T^{\otimes}_{\alg}[O;\bx]$ from \eqref{eq: alg tube set} with the definition of the sets $\T_s[G]$ from \eqref{eq: alg tubes}, we see that 
\begin{equation}\label{eq: refined alg dom 3}
   \T^{\otimes}_{\alg}[O;\bx] \subseteq \bigcup_{s = 1}^k  \Bigg(\prod_{\ell = 1}^{s-1} \T_{\ell}[O] \times \T_s[G] \times \prod_{\ell = s+1}^{k} \T_{\ell}[O] \Bigg).
\end{equation}
Indeed, to verify \eqref{eq: refined alg dom 3}, it suffices to show that if $T \in \T$ is such that $T \cap X \cap G$ has a connected component with diameter larger than $C_{\bdGamma}\delta^{\kappa_{\ttiny}}$, then $|T \cap X \cap N_{2\delta} G| \geq \delta^{\kappa_{\ttiny}}|T|$. Assuming there exists such a component, matters further reduce to finding some $t_0 < t_1$ such that $|t_1 - t_0| > 2 \delta^{\kappa_{\ttiny}}$ and $T \cap (\R^{n-1} \times [t_0, t_1]) \subseteq X \cap N_{2\delta}G$. Since connected semialgebraic sets are path connected (see, for instance, \cite[Theorem 5.23]{BPR2006}), there exist continuous functions $\overline{x} \colon [0,1] \to \R^{n-1}$ and $\overline{t} \colon [0,1] \to \R$ such that $t_0 := \overline{t}(0)$, $t_1 := \overline{t}(1)$ satisfy $t_0 \leq t_1$ and
\begin{equation*}
(\overline{x}(\sigma), \overline{t}(\sigma)) \in T \cap X \cap G, \quad  \sigma \in [0,1], \quad \big(|\overline{x}(0) - \overline{x}(1)|^2 + |\overline{t}(0) - \overline{t}(1)|^2\big)^{1/2} > C_{\bdGamma} \delta^{\kappa_{\ttiny}}. 
\end{equation*}
We show that $t_0$, $t_1$ satisfy the desired properties.
\begin{itemize}
    \item By the mean value theorem and the definition of $T$, we have
    \begin{equation*}
        |\overline{x}(0) - \overline{x}(1)| \leq \|\gamma_T'\|_{\infty}|t_0 - t_1| + 2 \delta.
    \end{equation*}
    Thus, $|t_0 - t_1| > 2 \delta^{\kappa_{\ttiny}}$ provided the constant $C_{\bdGamma}$ in \eqref{eq: alg tube set} is chosen appropriately. 
    \item Given any $(x,t) \in T$ with $\overline{t}(0) \leq t \leq \overline{t}(1)$, by the intermediate value theorem there exists some $\sigma \in [0,1]$ such that $(\overline{x}(\sigma), t) \in T \cap X \cap G$. In particular, $|\overline{x}(\sigma) - \gamma_T(t)| < \delta$, where $(\overline{x}(\sigma), t) \in G$. Thus, $|x - \overline{x}(\sigma)| < 2\delta$, and so $(x,t) \in N_{2 \delta}G$.  On the other hand, $[\overline{t}(0), \overline{t}(1)] \subseteq I$ and so $(x,t) \in X$. 
\end{itemize}

In light of \eqref{eq: refined alg dom 3}, we may bound 
\begin{equation}
\label{eq: refined alg dom 4}
     \int_{X \cap G} \Big|\sum_{\vec{T} \in \T^{\otimes}_{\alg}[O;\bx]} \prod_{\ell = 1}^k \chi_{T_\ell}(\bx) \Big|^{p/k} \,\ud \bx \lesssim \sum_{s = 1}^k \Big\|\big| \mu_{\T_s[G]} \big|^{1/k}\prod_{\substack{1 \leq \ell  \leq k\\ \ell \neq s}}\big| \mu_{\T_{\ell}[O]} \big|^{1/k}\Big\|_{L^p(X \cap G)}^p.
\end{equation}

Combining \eqref{eq: refined alg dom 4} with \eqref{eq: refined alg dom 2} and \eqref{eq: refined alg dom 1}, we obtain an inequality of a similar shape to that of \eqref{eq: stop alg}. In particular, provided $C_{\alg}$ is chosen sufficiently large, the stopping condition \texttt{[algebraic]} is met, contrary to our assumption.\medskip

In either the tiny-dominant or algebraic-dominant case we arrive at a contradiction. Thus, our initial assumption \eqref{eq: alg dom} must be false and therefore the claimed inequality \eqref{eq: cell dom} holds.
\end{proof}

The algorithm must terminate after at most $\ceil{\log_2 \delta^{-1}}$ steps. Indeed, suppose $J := \ceil{\log_2 \delta^{-1}}$ and the algorithm has reached the $J$th step. The cells $O_J \in \cO_J$ have diameter at most $2^{1-J}\rho  \leq 2\delta < \delta^{\kappa_{\ttiny}}$. Defining $\cO_{\ttiny} := \cO_J$, it follows that Property I$_{\ttiny}$, II$_{\ttiny}$ and III$_{\ttiny}$ vacuously hold with $j = J$. Hence, we realise all the conditions of stopping condition \texttt{[tiny]} and the algorithm terminates. \medskip

%%%%%%%%%%%%%%%%%%%%%%%%%%%%%%%%%%%%%%%%%%%%%%%%%%%%%%%%%%%%%%%%%%%%%%%%%%%%%%%%%%%%%%%%%%%%%%%%

%   The second algorithm

%%%%%%%%%%%%%%%%%%%%%%%%%%%%%%%%%%%%%%%%%%%%%%%%%%%%%%%%%%%%%%%%%%%%%%%%%%%%%%%%%%%%%%%%%%%%%%%%

\noindent\textbf{The second algorithm.} The algorithm \texttt{[partition $m \to m+1$]} is now applied repeatedly to arrive at a final decomposition. This process forms part of a second algorithm, referred to as \texttt{[partition]}.\medskip

\paragraph{\underline{\texttt{Input}}} \texttt{[partition]} will take as its input:

\begin{itemize}
    \item A set $\bdGamma$ defining a uniform family of polynomial curves in $\R^n$.
     \item An integer $2 \leq k \leq n$, parameters $\frac{k+1}{k} < p < \frac{k}{k-1}$, $0 < \varepsilon$, $\varepsilon_{\circ}$,  $\kappa_{\ttiny} < 1$, $0 < \delta \leq  1$ and $1 \leq E \leq d$. Furthermore, $\varepsilon_{\circ}$ is chosen sufficiently small, depending on $\varepsilon$, $p$, and $d$ is chosen large, depending on $\varepsilon_{\circ}$ and $E$.
    \item A family $\T$ of $(\bdGamma, \delta)$-tubes of degree bounded by $E$, satisfying $\#\T \leq \delta^{-(n-1)}$. 
     \item  A set $X := \R^{n-1} \times I$ for some choice of open interval $I \subseteq (-1,1)$.
  
    \item Subfamilies $\T_1, \dots, \T_k \subseteq \T$ which are $\nu$-transverse on $X$ for some $\nu > 0$. We also assume that the subfamilies are non-degenerate in the sense that
\begin{equation}\label{eq: non deg hyp}
\Big\|\prod_{\ell = 1}^k \big|\sum_{T_{\ell} \in \T_{\ell}} \chi_{T_{\ell}} \big|^{1/k}\Big\|_{L^p(X)} \geq \bC_{\bdGamma, \varepsilon, d} \nu^{-A} \delta^{-(n-k)/p' + \eta_{\ttiny}(p) - \varepsilon}  \Big(\sum_{T \in \T}|T|\Big)^{1/p}
\end{equation}
for
\begin{equation}\label{eq: delta exp}
   \eta_{\ttiny}(p) := (k/p - k + 1)\kappa_{\ttiny} > 0.
\end{equation}
Here $\bC_{\bdGamma, \varepsilon, d} \geq 1$ is a constant depending only on admissible parameters and $\varepsilon$, $d$ and $M_1$ from \eqref{eq: uniform family curves}, chosen sufficiently large so as to satisfy the forthcoming requirements of the argument, and $A = A_{k,n,p} \geq 1$ is as in the statement of Theorem~\ref{thm: uni multilin}.
\end{itemize}

In contrast with the previous algorithm, we now assume that the $\T_1, \dots, \T_k$ are \textit{transversal} families. However, the direction-separated and centre-separated hypotheses still do not appear. The hypothesis \eqref{eq: non deg hyp} is no real constraint: if this condition fails, then we shall see Theorem~\ref{thm: multilin Kak/Nik} essentially holds by default, provided $\varepsilon >0$ is chosen sufficiently small. We remark that the $\eta_{\ttiny}(p)$ exponent corresponds to the gain in Theorem~\ref{thm: uni multilin} when we localise to a ball of radius $r := \delta^{\kappa_{\ttiny}}$. \medskip

To describe the output, we introduce some further terminology. 

\begin{definition} Let $G$ be a grain and $\T_{\ell}[G] \subseteq \T_{\ell}$ for $1 \leq \ell \leq k$ be subfamilies of tubes. We say $(\T_{\ell}[G])_{\ell = 1}^k$ \textit{intersects significantly with $X \cap G$} if there exists some $1 \leq s \leq k$ such that
\begin{equation*}
    |T_s \cap X \cap N_{2\delta}G| \geq \delta^{\kappa_{\ttiny}} |T_s| \qquad \textrm{for all $T_s \in \T_s[G]$.}
\end{equation*}
\end{definition}

\paragraph{\underline{\texttt{Output}}} For $0 \leq m \leq n-k$, the $m$th step of the recursion will produce:
\begin{itemize}
    \item An $m+1$-tuple of:
    \begin{itemize}
        \item scales $(\rho_0, \dots, \rho_m)$ satisfying $1 =: \rho_0 \geq \rho_1 \geq  \dots \geq \rho_m \geq \delta^{\kappa_{\ttiny}}$; 
        \item parameters $(D_0, \dots, D_m)$ with $1 =: D_0 \leq D_j$ for $1 \leq j \leq m$;
    \end{itemize}
    Each $m$-tuple is formed by adjoining a component to the corresponding $(m-1)$-tuple from the previous stage.
    \item A family $\cG_m$ of $(\delta, \rho_m, d_m)$-grains of codimension $m$, where $d_m := (C_{\deg})^m d$,
    for $C_{\deg}$ the constant appearing in Theorem~\ref{thm: partitioning}.
    \item $\T_{\ell}[G_m] \subseteq \T_{\ell}$ a subfamily of $(\bdGamma, \delta)$-tubes, for each $G_m \in \cG_m$ and $1 \leq \ell \leq k$, such that $(\T_{\ell}[G_m])_{\ell = 1}^k$ intersects significantly with $X \cap G_m$.
\end{itemize}
These objects are chosen so that the following properties hold: \medskip

\noindent{\underline{Property 1}} 
\begin{equation}\tag*{$(1)_m$}
      \Big\|\prod_{\ell = 1}^k \big|\mu_{\T_{\ell}} \big|^{1/k}\Big\|_{L^p(X)}^p \lesssim \Big(\prod_{i = 1}^m D_i^{\varepsilon_{\circ}} \Big) \delta^{-m\varepsilon_{\circ}} \sum_{G_m \in \cG_m} \Big\|\prod_{\ell = 1}^k \big|\mu_{\T_{\ell}[G_m]} \big|^{1/k}\Big\|_{L^p(X \cap G_m)}^p.
\end{equation}
\noindent{\underline{Property 2}}  For $m > 0$, 
\begin{equation}\tag*{$(2)_m$}
    \sum_{G_m \in \cG_m} \# \T_{\ell}[G_m] \lesssim \Big(\prod_{i=1}^m D_i^{1 + \varepsilon_{\circ}}\Big) \# \T_{\ell}.
\end{equation}

\noindent{\underline{Property 3}}  For $m > 0$,
\begin{equation}\tag*{$(3)_m$}
    \max_{G_m \in \cG_m} \# \T_{\ell}[G_m] \lesssim \Big(\prod_{i=1}^m D_i^{-(n-i) + \varepsilon_{\circ}}\Big) \# \T_{\ell}.
\end{equation}

\paragraph{\underline{\texttt{Initial step}}} For the $m = 0$ step, we take $\cG_0 := \{\B^n\}$ and $\T_{\ell}[\B^n] := \T_{\ell}$ for $1 \leq \ell \leq k$, noting that the unit ball is a $(\delta, 1, 0)$-grain of codimension $0$ and that the families $(\T_{\ell})_{\ell = 1}^k$ of tubes intersect significantly with $\B^n$. With these definitions, Property 1, 2 and 3 all hold vacuously.\medskip

\paragraph{\underline{\texttt{($m+1$)th step}}} Let $0 \leq m \leq n-k-1$ and suppose that the recursive algorithm has ran through $m$ steps. We therefore have constructed a family $\cG_m$ of $(\delta, \rho_m, d_m)$-grains of codimension $m$ and subfamilies $\T_{\ell}[G_m] \subseteq \T_{\ell}$ of $(\bdGamma,\delta)$-tubes, for each $G_m \in \cG_m$ and $1 \leq \ell \leq k$, satisfying Properties 1, 2 and 3. We show how to construct the next generation of grains and tube families. 

For each $G_m \in \cG_m$, provided our input $d$ is large enough, we may apply \texttt{[partition $m \to m+1$]} with degree $d_m$ to bound the norm 
\begin{equation}\label{eq: partition me}
    \Big\|\prod_{\ell = 1}^k\big|\sum_{T_{\ell} \in \T_{\ell}[G_m]} \chi_{T_{\ell}} \big|^{1/k}\Big\|_{L^p(X \cap G_m)}.
\end{equation} 
One of two things can happen: either \texttt{[partition $m \to m+1$]} terminates due to the stopping condition \texttt{[tiny]} or it terminates due to the stopping condition \texttt{[algebraic]}. Let $\cG_{m,\ttiny}$ denote the set of all grains $G_m \in \cG_m$ for which \texttt{[partition $m \to m+1$]} terminates due to \texttt{[tiny]} and $\cG_{m,\alg} := \cG_m \setminus \cG_{m,\ttiny}$.

\begin{claim}\label{claim: grain dom} For $0 \leq m \leq n - k -1$, the grains $\cG_{m,\alg}$ dominate, in the sense that 
\begin{equation}\label{eq: grain dom}
   \sum_{G_m \in \cG_m}\Big\|\prod_{\ell=1}^k \big|\mu_{\T_{\ell}[G_m]} \big|^{1/k}\Big\|_{L^p(X \cap G_m)}^p \leq 2 \sum_{G_m \in \cG_{m,\alg}} \Big\|\prod_{\ell=1}^k \big|\mu_{\T_{\ell}[G_m]}  \big|^{1/k}\Big\|_{L^p(X \cap G_m)}^p.
\end{equation} 
\end{claim}

We shall temporarily assume the above claim and show how to construct the grains $\cG_{m+1}$. Let $G_m \in \cG_{m,\alg}$, so that \texttt{[partition $m \to m+1$]} applied to \eqref{eq: partition me} does not terminate owing to \texttt{[tiny]} and therefore terminates owing to \texttt{[algebraic]}. Consequently, 
  \begin{align}\label{recursive step property 1}
      \Big\|\prod_{\ell=1}^k \big|\mu_{\T_{\ell}[G_m]} \big|^{1/k}\Big\|_{L^p(X \cap G_m)}^p &\lesssim  D_{m+1}^{\varepsilon_{\circ}} \sum_{G \in \cG_{m+1}[G_m]} \Big\|\prod_{\ell=1}^k \big|\mu_{\T_{\ell}[G]} \big|^{1/k}\Big\|_{L^p(X \cap G)}^p; \\
\label{recursive step property 2}
    \sum_{G \in \cG_{m+1}[G_m]}  \#\T[G] &\lesssim D_{m+1}^{1 + \varepsilon_{\circ}} \#\T[G_m]; \\
\label{recursive step property 3}
\max_{G \in \cG_{m+1}[G_m]} \#\T[G] & \lesssim   D_{m+1}^{-(n-m-1) + \varepsilon_{\circ}} \#\T[G_m]
\end{align}
all hold for some choice of:
\begin{itemize}
    \item Scale $\rho_{m+1}$ satisfying $\rho_m \geq \rho_{m+1} \geq \delta^{\kappa_{\ttiny}}$ and parameter $D_{m+1}$;
    \item Family $\cG_{m+1}[G_m]$ of $(\delta, \rho_{m+1}, d_{m+1})$-
    grains of codimension $m+1$;
    \item $\T_{\ell}[G] \subseteq \T_{\ell}$ a subfamily of $(\bdGamma, \delta)$-tubes, for each grain $G \in \cG_{m+1}[G_m]$ and $1 \leq \ell \leq k$, such that $(\T_{\ell}[G])_{\ell = 1}^k$ intersects significantly with $X \cap G$.
\end{itemize}
Each inequality~\eqref{recursive step property 1},~\eqref{recursive step property 2} and~\eqref{recursive step property 3} is obtained by combining the definition of the stopping condition \texttt{[algebraic]} with Properties I, II and III from \texttt{[partition $m \to m+1$]}, respectively. Indeed, if $J$ denotes the stopping time for the application of \texttt{[partition $m \to m+1$]}, then we take
\begin{equation*}
    r_0 := 2\rho_m ,\qquad \rho_{m+1} := r_J/2 := 2^{-J}\rho_m\qquad \textrm{and}\qquad D_{m+1} := d_{m}^J,
\end{equation*}
using the notation from \texttt{[partition $m \to m+1$]}.

\begin{remark} The sets $\T_{\ell}[G]$ do not necessarily coincide with the sets defined \eqref{eq: alg tubes}. The two sets will coincide for one choice of index $\ell = s$ but for the remaining indices $\ell \neq s$ the $\T_{\ell}[G]$ will coincide with the sets $\T_{\ell}[O]$ as in \eqref{eq: stop alg}. 
\end{remark}

The $\rho_{m+1}$ and $D_{m+1}$ depend on the choice of $G_m$, but this dependence can be essentially removed by pigeonholing. In particular, the stopping time $J$ depends on $G_m$, but satisfies $J= O(\log \delta^{-1})$. Thus, since there are only logarithmically many possible different values, one may find a subset $\cG_{m,\alg}'$ of the $\cG_{m,\alg}$ over which the stopping time $J$ has a common value and, moreover, the inequality~\eqref{recursive step property 1} holds with $\cG_{m,\alg}$ replaced with $\cG_{m,\alg}'$ and an additional factor of $\delta^{-\varepsilon_{\circ}}$ on the right-hand side. Since both $\rho_{m+1}$ and $D_{m+1}$ are determined by~$J$, the desired uniformity is immediately inherited by these parameters. 

Let $\cG_{m+1}$ denote the union of the sets $\cG_{m+1}[G_m]$ over all $G_m \in \cG_{m,\alg}'$. It remains to observe that the desired properties hold for the newly constructed data:
\begin{itemize}
    \item $(1)_{m+1}$ in Property 1 follows from~\eqref{recursive step property 1} (with an additional factor of $\delta^{-\varepsilon_{\circ}}$, as per the above discussion) and $(1)_m$;
    \item $(2)_{m+1}$ in Property 2 follows from~\eqref{recursive step property 2} and $(2)_m$;
    \item $(3)_{m+1}$ in Property 3 follows from~\eqref{recursive step property 3} and $(3)_m$.
\end{itemize}
This concludes the description of the recursive step of \texttt{[partition]}, except that it remains to justify Claim~\ref{claim: grain dom}.

\begin{proof}[Proof (of Claim~\ref{claim: grain dom})] If the condition \eqref{eq: grain dom} of the claim fails, then 
  \begin{equation*}
      \sum_{G_m \in \cG_m} \Big\|\prod_{\ell = 1}^k\big|\mu_{\T_{\ell}[G_m]} \big|^{1/k}\Big\|_{L^p(X \cap G_m)}^p \leq 2 \sum_{G_m \in \cG_{m,\textrm{\texttt{tiny}}}}  \Big\|\prod_{\ell = 1}^k\big|\mu_{\T_{\ell}[G_m]} \big|^{1/k} \Big\|_{L^p(X \cap G_m)}^p
\end{equation*}
holds, where the right-hand summation is restricted to those $G_m \in \cG_m$ for which \texttt{[partition $m \to m+1$]} terminates owing to the stopping condition \texttt{[tiny]}.  

By combining $(\mathrm{I})_J$ (where $J$ denotes the stopping time) with Property I$_{\ttiny}$, II$_{\ttiny}$ and III$_{\ttiny}$ from \texttt{[partition $m \to m+1$]}, we deduce that there exist
\begin{itemize}
    \item $\cO_{\ttiny}$ a family of measurable subsets of $X$ with $\diam\,O \leq \delta^{\kappa_{\ttiny}}$ for all $O \in \cO_{\ttiny}$;
    \item $\T_{\ell}[O]$ a subfamily of $\delta$-tubes for each $O \in \cO_{\ttiny}$ and $1 \leq \ell \leq k$
\end{itemize}
such that
\begin{align}\label{eq: tiny endgame 1}
    \Big\|\prod_{\ell = 1}^k \big| \mu_{\T_{\ell}} \big|^{1/k} \Big\|_{L^p(X)}^p &\lesssim_{\bdGamma, d} \Big(\prod_{i = 1}^{m+1} D_i^{\varepsilon_{\circ}} \Big)\delta^{-m\varepsilon_{\circ}} \sum_{O \in \cO_{\ttiny}} \Big\|\prod_{\ell = 1}^k \big| \mu_{\T_{\ell}[O]}  \big|^{1/k}\Big\|_{L^p(O)}^p; \\
\label{eq: tiny endgame 2}
        \sum_{O \in \cO_{\ttiny}} \#\T_{\ell}[O] &\lesssim_{\bdGamma, d} \Big(\prod_{i = 1}^{m+1} D_i^{1 + \varepsilon_{\circ}} \Big)\#\T_{\ell};\\
\label{eq: tiny endgame 2.5}        
        \max_{O \in \cO_{\ttiny}} \#\T_{\ell}[O] &\lesssim_{\bdGamma, d} \Big(\prod_{i = 1}^{m+1} D_i^{-(n - i - \varepsilon_{\circ})} \Big) \#\T_{\ell}.
\end{align}
Note that here we have also used the properties $(1)_m$, $(2)_m$ and $(3)_m$ from the previous stage of the algorithm. The implicit constant depends on $M_1$ from \eqref{eq: uniform family curves}, but is otherwise independent of $\bdGamma$. 

Fixing $O \in \cO_{\ttiny}$, we apply the universal multilinear estimate from Theorem~\ref{thm: uni multilin} to the corresponding term in the right-hand sum in \eqref{eq: tiny endgame 1}. Since $\T_{\ell}[O] \subseteq \T_{\ell}$ for $1 \leq \ell \leq k$, it follows that $\T_1[O], \dots, \T_k[O]$ are $\nu$-transverse on $X$. Furthermore, the cell $O$ is contained in a ball of radius $4\delta^{\kappa_{\ttiny}}$. Thus, for all $\varepsilon > 0$, and in particular for the value appearing in the non-degeneracy condition~\eqref{eq: non deg hyp}, Theorem~\ref{thm: uni multilin} implies
\begin{equation*}
    \Big\|\prod_{\ell = 1}^k \big| \mu_{\T_{\ell}[O]} \big|^{1/k} \Big\|_{L^p(O)} \lesssim_{\bdGamma, \varepsilon} \nu^{-A}  \delta^{k - 1 + (n - k)/p + \eta_{\ttiny}(p) - \varepsilon/2}\prod_{\ell = 1}^k \big(\#\T_{\ell}[O]\big)^{1/k},
\end{equation*}
where $\eta_{\ttiny}(p)$ is as defined in \eqref{eq: delta exp} and the dependence on $\bdGamma$ is as above. 

Using \eqref{eq: tiny endgame 2.5} and the bound $\#\T_{\ell} \leq \#\T \leq \delta^{-(n-1)}$, we estimate
\begin{align*}
    \#\T_{\ell}[O] &\leq \big(\max_{O \in \cO_{\ttiny}}\#\T_{\ell}[O]\big)^{(1-1/p)}\big(\#\T_{\ell}[O]\big)^{1/p}\\
    &\lesssim_d \Big(\prod_{i=1}^{m+1} D_i^{-(n - i -\varepsilon_{\circ})(1-1/p)}\Big) \delta^{-(n-1)(1-1/p)} \big(\#\T_{\ell}[O]\big)^{1/p}.
\end{align*}
We therefore conclude the cell-wise bound
\begin{equation}\label{eq: tiny endgame 3}
    \Big\|\prod_{\ell = 1}^k \big| \mu_{\T_{\ell}[O]} \big|^{1/k} \Big\|_{L^p(O)} \lesssim_{\bdGamma, d,  \varepsilon}  M(\delta; D) \delta^{-\varepsilon/2} \prod_{\ell=1}^k\Big(\sum_{T \in \T_{\ell}[O]} |T| \Big)^{1/kp}
\end{equation}
for
\begin{equation*}
   M(\delta; D) :=    \nu^{-A} \delta^{-(n-k)/p' + \eta_{\ttiny}(p)}\Big(\prod_{i=1}^{m+1} D_i^{-(n - i -\varepsilon_{\circ})(1-1/p)}\Big).
\end{equation*}
Now take the $\ell^p$-norm of both sides of \eqref{eq: tiny endgame 3} over all $O \in \cO_{\ttiny}$. By applying H\"older's inequality to the resulting expression and \eqref{eq: tiny endgame 1},  \eqref{eq: tiny endgame 2}, we obtain
\begin{equation*}
   \Big\|\prod_{\ell = 1}^k \big| \mu_{\T_{\ell}} \big|^{1/k} \Big\|_{L^p(X)} \lesssim_{\bdGamma, d,  \varepsilon}  M(\delta; D) \delta^{-\varepsilon}\big(\prod_{i=1}^{m+1}  D_i^{(1 + \varepsilon_{\circ})/p + \varepsilon_{\circ}}\Big) \Big(\sum_{T \in \T} |T| \Big)^{1/p},
\end{equation*}
where we have assumed that $\varepsilon_{\circ} > 0$ is sufficiently small so that the $\delta^{-\varepsilon_{\circ}}$ powers can be absorbed into the $\delta^{-\varepsilon}$ factor.

Note that
\begin{equation*}
  M(\delta; D)\big(\prod_{i=1}^{m+1}  D_i^{(1 + \varepsilon_{\circ})/p + \varepsilon_{\circ}}\Big)  =    \nu^{-A} \delta^{-(n-k)/p' + \eta_{\ttiny}(p)}\Big(\prod_{i = 1}^{m+1} D_i^{-(n-i)+(n-i+1)/p + 2\varepsilon_{\circ}} \Big).
\end{equation*}
Provided $\varepsilon_{\circ} > 0$ is chosen sufficiently small, using the hypothesis $0 \leq m \leq n - k - 1$, we can ensure 
\begin{equation*}
  \frac{n-i + 1}{n-i} \leq \frac{k+1}{k} < p \qquad \textrm{and} \qquad (n-i) - \frac{n- i + 1}{p} > 2\varepsilon_{\circ} \quad \textrm{for $1 \leq i \leq m+1$.}  
\end{equation*}
Thus, we conclude that
\begin{equation*}
   \Big\|\prod_{\ell = 1}^k \big| \mu_{\T_{\ell}} \big|^{1/k} \Big\|_{L^p(X)} \lesssim_{\bdGamma,  \varepsilon, d}   \nu^{-A} \delta^{-(n-k)/p' + \eta_{\ttiny}(p) - \varepsilon} \Big(\sum_{T \in \T} |T| \Big)^{1/p}.  
\end{equation*}
However, provided $\bC_{\bdGamma, \varepsilon, d}$ in \eqref{eq: non deg hyp} is chosen sufficiently large, we see that this contradicts the non-degeneracy hypothesis. This contradiction arose from assuming \eqref{eq: grain dom} fails, so \eqref{eq: grain dom} must hold. 
\end{proof}

%%%%%%%%%%%%%%%%%%%%%%%%%%%%%%%%%%%%%%%%%%%%%%%%%%%%%%%%%%%%%%%%%%%%%%%%%%%%%%%%%%%%%%%%%%%%%%%%

%   The proof of the multilinear Kakeya/Nikodym estimate

%%%%%%%%%%%%%%%%%%%%%%%%%%%%%%%%%%%%%%%%%%%%%%%%%%%%%%%%%%%%%%%%%%%%%%%%%%%%%%%%%%%%%%%%%%%%%%%%

\subsection{The proof of the multilinear Kakeya/Nikodym estimate} We now apply the algorithm developed in the previous two subsections to conclude the proof of Theorem~\ref{thm: multilin Kak/Nik}.

\begin{proof}[Proof (of Theorem~\ref{thm: multilin Kak/Nik})] Let $2 \leq k \leq n$ and $\kappa > 0$, $B = (B_1, B_2) \in [0,\infty)^2$ and fix $\frac{k+1}{k} < p < \frac{k}{k-1}$. Suppose that $\bdGamma$ defines a uniform family of $C^{\infty}$ curves in $\R^n$ and, for $0 < \delta < 1$, that $\T$ is a family of $(\bdGamma, \delta)$-tubes which is $(\kappa, B)$-non-concentrated in codimension $(n-k)$ grains. Let $\T_1, \dots, \T_k \subset \T$ be $\nu$-transverse families for some $\nu > 0$. 

The first step is to form a family of polynomial approximants to the curves in $\bdGamma$. In particular, fix $\varepsilon > 0$ an admissible parameter, chosen sufficiently small to satsify the forthcoming requirements of the proof. By pigeonholing there exists an open interval $I \subseteq (-1,1)$ of length $\delta^{\varepsilon}$ such that
\begin{equation}\label{eq: alg endgame 1}
\Big\|\prod_{\ell = 1}^k \big| \sum_{T_{\ell} \in \T_{\ell}} \chi_{T_{\ell}} \big|^{1/k} \Big\|_{L^p(\R^n)} \lesssim    \delta^{-\varepsilon} \Big\|\prod_{\ell = 1}^k \big| \sum_{T_{\ell} \in \T_{\ell}} \chi_{T_{\ell}} \big|^{1/k} \Big\|_{L^p(X)} \qquad \textrm{for $X := \R^{n-1} \times I$.}
\end{equation}
Here we use the fact that each $T \in \T$ satisfies $T \subseteq \R^{n-1} \times (-1,1)$. 

Write $I = (a, a+\delta^{\varepsilon})$ for some $a \in (-1,1)$ and define
$N(\varepsilon) := \ceil{100n/\varepsilon}$. For each $C^{\infty}$ map $\gamma \in \bdGamma$, let $\widetilde{\gamma}$ denote the degree $N(\varepsilon)$ Taylor polynomial of $\gamma$ centred at $a$ and define $\widetilde{\bdGamma} := \{ \widetilde{\gamma} : \gamma \in \bdGamma\}$. We may assume without loss of generality that $\delta$ is chosen sufficiently small, depending on the constants $(M_N)_{N = 1}^{N(\varepsilon)}$ from \eqref{eq: uniform family curves}. In particular, we can ensure that
\begin{equation}\label{eq: alg endgame 2}
    |\gamma(t) - \widetilde{\gamma}(t)| < \delta \quad \textrm{and} \quad  |\gamma'(t) - \widetilde{\gamma}'(t)| < \delta \quad \textrm{for all $t \in I$ and all $\gamma \in \bdGamma$.}
\end{equation}
If $X := \R^{n-1} \times I$, then it follows that
\begin{equation}\label{eq: alg endgame 3}
    T \cap X \subseteq \widetilde{T} := \big\{(x,t) \in \B^{n-1} \times (-1,1) : |x - \widetilde{\gamma}_T(t)| < 2\delta \big\}.
\end{equation}
Note that $\widetilde{\T} := \{\widetilde{T} : T \in \T\}$ is a family of $(\widetilde{\bdGamma}, 2\delta)$-tubes. 

We shall apply \texttt{[partition]} with the following inputs: 
\begin{itemize}
    \item The set $\widetilde{\bdGamma}$, which defines a family of polynomial curves in $\R^n$.
     \item The parameters $2 \leq k \leq n$, $\frac{k+1}{k} < p < \frac{k}{k-1}$ and $\delta > 0$ as above. We further choose 
     \begin{equation}\label{eq: choose kappa tiny}
         \kappa_{\ttiny} := \frac{\kappa}{2(B_1 + B_2)} > 0
     \end{equation}
     and take $\varepsilon_{\circ} > 0$  to be small, depending on $p$ and $\varepsilon$, and $d \in \N$ to be large, depending on $\varepsilon_{\circ}$ and $N(\varepsilon)$.
    \item The set $\widetilde{\T} := \{\widetilde{T} : T \in \T\}$, which is a family of $(\widetilde{\bdGamma}, 2\delta)$-tubes of degree bounded by $N(\varepsilon)$ and satisfies $\#\widetilde{\T} \leq \delta^{-(n-1)}$. 
     \item The set $X := \R^{n-1} \times I$ as above.
    \item The subfamilies $\widetilde{\T}_{\ell} := \{\widetilde{T}_{\ell} : T \in \T_{\ell}\}$ for $1 \leq \ell \leq k$. Since the original tube families $\T_1, \dots, \T_k$ are $\nu$-transversal, it follows from the derivative bounds in \eqref{eq: alg endgame 2} that $\widetilde{\T}_1, \dots, \widetilde{\T}_k$ are $\nu/2$-transverse on $X$.
\end{itemize}
The quantities $\varepsilon_{\circ} > 0$ and $d \in \N$ are chosen admissibly, so as to ensure they satisfy the conditions required by \texttt{[partition]} and also the forthcoming requirements of the proof. To run \texttt{[partition]}, we shall temporarily assume that the tube families $\widetilde{\T}_1, \dots, \widetilde{\T}_k$ satisfy the non-degeneracy hypothesis \eqref{eq: non deg hyp}, with the fixed parameter $\varepsilon > 0$ as introduced above.

From the output of the $n-k$ stage, we deduce that there exist
\begin{itemize}
    \item $\cG$ a family of $(\delta, \rho, d)$-grains of codimension $n-k$ for some $\delta^{\kappa_{\ttiny}} \leq \rho \leq 1$;
    \item $\widetilde{\T}_{\ell}[G] \subseteq \widetilde{\T}_{\ell}$ subfamilies of $\delta$-tubes for each $G \in \cG$ and $1 \leq \ell \leq k$
\end{itemize}
with the following properties: \medskip

\noindent \underline{Property 1}. 
\begin{equation}\label{eq: alg endgame 4}
    \Big\|\prod_{\ell = 1}^k \big| \mu_{\,\widetilde{\T}_{\ell}}  \big|^{1/k} \Big\|_{L^p(X)}^p \lesssim \Big( \prod_{i = 1}^{n-k} D_i^{\varepsilon_{\circ}}\Big) \delta^{-(n-k)\varepsilon_{\circ}}  \sum_{G \in \cG} \Big\|\prod_{\ell = 1}^k \big|\mu_{\,\widetilde{\T}_{\ell}[G]}\big|^{1/k}\Big\|_{L^p(X \cap G)}^p.
\end{equation} 
\noindent{\underline{Property 2}}.
\begin{equation}\label{eq: alg endgame 5}
     \sum_{G \in \cG} \#\widetilde{\T}_{\ell}[G] \lesssim \Big( \prod_{i = 1}^{n-k} D_i^{1 +\varepsilon_{\circ}}\Big) \#\T.
\end{equation}
\noindent{\underline{Property 3}}.
\begin{equation}
\label{eq: alg endgame 6}
     \max_{G \in \cG} \#\widetilde{\T}_{\ell}[G] \lesssim \Big(\prod_{i = 1}^{n-k} D_i^{-(n-i)+\varepsilon_{\circ}}\Big) \#\T.
 \end{equation}
 Furthermore, we have the following crucial feature of the $\widetilde{\T}_{\ell}[G]$.\medskip
 
\noindent \underline{Property: significant intersection}. For each $G \in \cG$, there exists $1 \leq s = s_G  \leq k$ such that 
\begin{equation}\label{eq: alg endgame 7}
    |\widetilde{T} \cap X  \cap N_{2\delta}G| \geq \delta^{\kappa_{\ttiny}} |\widetilde{T}| \qquad \textrm{for all $\widetilde{T} \in \widetilde{\T}_s[G]$.}
\end{equation}

Fixing $G \in \cG$, since $\widetilde{\T}_{\ell}[G] \subseteq \widetilde{\T}_{\ell}$ for $1 \leq \ell \leq k$, it follows that $\widetilde{\T}_1[G], \dots, \widetilde{\T}_k[G]$ are $\nu/2$-transverse on $X$. Thus, the universal multilinear estimate from Theorem~\ref{thm: uni multilin} implies that
\begin{equation}\label{eq: alg endgame 8}
  \Big\|\prod_{\ell = 1}^k \big|\mu_{\,\widetilde{\T}_{\ell}[G]}\big|^{1/k}\Big\|_{L^p(X \cap G)} \lesssim_{\bdGamma, \varepsilon}  \nu^{-A}  \delta^{k - 1 + (n - k)/p - \varepsilon/3} \prod_{\ell = 1}^k \big(\#\widetilde{\T}_{\ell}[G]\big)^{1/k}.
\end{equation}

It is convenient to consider the union $\widetilde{\T}[G] := \widetilde{\T}_1[G] \cup \dots \cup \widetilde{\T}_k[G]$. Let $s = s_G$, so that the tube family $\widetilde{\T}_s[G]$ satisfies \eqref{eq: alg endgame 7}. By our hypotheses on $p$, we have $1-1/p<1/k$ and may therefore bound 
\begin{align}
\nonumber
    \prod_{\ell = 1}^k \big(\#\widetilde{\T}_{\ell}[G]\big)^{1/k} &\leq \big(\#\widetilde{\T}_s[G]\big)^{1-1/p} \big(\#\widetilde{\T}_s[G]\big)^{1/p - (k-1)/k} \big(\#\widetilde{\T}[G]\big)^{(k-1)/k} \\
    \label{eq: alg endgame 9}
    &\leq \big(\#\widetilde{\T}_s[G]\big)^{1-1/p}\big(\#\widetilde{\T}[G]\big)^{1/p}.
\end{align}

In view of \eqref{eq: alg endgame 2}, we have $\widetilde{T} \cap X \subseteq 2 \cdot T$ for each $T \in \T$, where $2 \cdot T$ denotes the tube which shares the same core curve as $T$, but with width $2\delta$ rather than $\delta$. In particular, \eqref{eq: alg endgame 7} implies that
\begin{equation*}
    |2 \cdot T \cap N_{2\delta} G | \geq \delta^{\kappa_{\ttiny}} |T| \qquad \textrm{whenever $T \in \T$ is such that $\widetilde{T} \in \widetilde{\T}_s[G]$.}
\end{equation*}
Since we are assuming $\T$ is $(B,\kappa)$-non-concentrated in codimension $n-k$ grains, recalling \eqref{eq: gen non-concentrated}, we conclude that
\begin{equation}\label{eq: alg endgame 10}
    \#\widetilde{\T}_s[G] \lesssim_{d,\varepsilon}
 \delta^{-\kappa_{\ttiny} (B_1 + B_2)}\delta^{-(n-1) + \kappa - \varepsilon/3} \leq \delta^{-(n-1) + \kappa/2 - \varepsilon/3}
\end{equation}
where we used the definition \eqref{eq: choose kappa tiny}. Recall, we also have the bound from \eqref{eq: alg endgame 6} for $\#\widetilde{\T}_s[G]$. Thus, given $0 \leq \sigma \leq 1$, we may take a weighted geometric mean of these bounds \eqref{eq: alg endgame 6} and \eqref{eq: alg endgame 10} to deduce that
\begin{equation*}
  \#\widetilde{\T}_s[G]   \lesssim_{d,\varepsilon} \Big(\prod_{i = 1}^{n-k} D_i^{-(n - i -\varepsilon_{\circ})\sigma} \Big) \delta^{\kappa(1-\sigma)/2} \delta^{-(n-1) - \varepsilon/3}.
\end{equation*}
 Combining this with \eqref{eq: alg endgame 8} and \eqref{eq: alg endgame 9}, we obtain
 \begin{equation}\label{eq: alg endgame 11}
   \Big\|\prod_{\ell = 1}^k \big|\mu_{\,\widetilde{\T}_{\ell}[G]}\big|^{1/k}\Big\|_{L^p(X \cap G)}  \lesssim_{\bdGamma,\varepsilon} M(\delta; D)  \delta^{-2\varepsilon/3}\Big(\sum_{T \in \widetilde{\T}[G]}|T|\Big)^{1/p}
\end{equation}
where
\begin{equation*}
M(\delta; D) := \nu^{-A} \delta^{-(n-k)/p' + \eta_{\alg}(\sigma,p)} \prod_{i = 1}^{n-k} D_i^{-(n - i -\varepsilon_{\circ})\sigma(1-1/p)}
\end{equation*}
for
\begin{equation*}
    \eta_{\alg}(\sigma,p) := \frac{\kappa(1-\sigma)}{2p'}.
\end{equation*}
 
 On the other hand, by \eqref{eq: alg endgame 5}, we have
\begin{equation}\label{eq: alg endgame 12}
  \sum_{G \in \cG} \#\widetilde{\T}[G] \lesssim \Big( \prod_{i = 1}^{n-k} D_i^{1 +\varepsilon_{\circ}}\Big) \#\T. 
\end{equation}
We may assume that $0< \varepsilon_{\circ} < \varepsilon/(3n)$. Taking the $\ell^p$ sum of both sides of \eqref{eq: alg endgame 11} over all $G \in \cG$ and combining the resulting inequality with \eqref{eq: alg endgame 1}, \eqref{eq: alg endgame 4} and \eqref{eq: alg endgame 12}, we then obtain 
\begin{equation*}
    \Big\|\prod_{\ell = 1}^k \big| \sum_{T_{\ell} \in \T_{\ell}} \chi_{T_{\ell}} \big|^{1/k} \Big\|_{L^p(\R^n)}  \lesssim_{\bdGamma,\varepsilon} M(\delta; D) \Big( \prod_{i = 1}^{n-k} D_i^{(1 +\varepsilon_{\circ})/p + \varepsilon_{\circ}}\Big) \delta^{-\varepsilon} \Big(\sum_{T \in \T} |T| \Big)^{1/p}
\end{equation*}
where
\begin{equation*}
  M(\delta; D) \Big( \prod_{i = 1}^{n-k} D_i^{(1 +\varepsilon_{\circ})/p + \varepsilon_{\circ}}\Big) \leq \nu^{-A} \delta^{-(n-k)/p' + \eta_{\alg}(\sigma,p)} \prod_{i = 1}^{n-k} D_i^{-(n - i)\sigma + ((n-i)\sigma + 1)/p + 2 \varepsilon_{\circ}}. 
\end{equation*}

Recall that $\frac{k+1}{k} < p < \frac{k}{k-1}$. Thus, provided $\varepsilon_{\circ} > 0$ is chosen sufficiently small (depending on $p$), we can find some $0 < \sigma_{\circ} < 1$ such that $k\sigma_{\circ} - (k\sigma_{\circ} + 1)/p \geq 2 \varepsilon_{\circ}$. This further implies that
\begin{equation*}
  (n-i)\sigma_{\circ} - \frac{(n-i)\sigma_{\circ} + 1}{p} \geq 2\varepsilon_{\circ} \qquad \textrm{for $1 \leq i \leq n - k$.} 
\end{equation*}
Since $1 - \sigma_{\circ} > 0$ and $p > 1$, it follows that $\eta_{\alg}(p) := \eta_{\alg}(\sigma_{\circ}, p) > 0$. Thus, by choosing $\varepsilon > 0$ sufficiently small, we can ensure that
\begin{equation*}
    \Big\|\prod_{\ell = 1}^k \big| \sum_{T_{\ell} \in \T_{\ell}} \chi_{T_{\ell}} \big|^{1/k} \Big\|_{L^p(\R^n)}  \lesssim_{\bdGamma}  \nu^{-A} \delta^{-(n-k)/p' + \eta_{\alg}(p)/2}\Big(\sum_{T \in \T} |T| \Big)^{1/p}.
\end{equation*}

All of the above is under the hypothesis that the tube families $\widetilde{\T}_1, \dots, \widetilde{\T}_k$ satisfy the non-degeneracy hypothesis \eqref{eq: non deg hyp}. Suppose this condition is not satisfied. In this case, using the containment \eqref{eq: alg endgame 3} and the pigeonholing \eqref{eq: alg endgame 1}, we have
\begin{equation*}
\Big\|\prod_{\ell = 1}^k \big|\sum_{T_{\ell} \in \T_{\ell}} \chi_{T_{\ell}} \big|^{1/k}\Big\|_{L^p(\R^n)} \lesssim_{\bdGamma,\varepsilon} \nu^{-A} \delta^{-(n-k)/p' + \eta_{\ttiny}(p) - 2\varepsilon}  \Big(\sum_{T \in \T} |T| \Big)^{1/p}.
\end{equation*}
Thus, by choosing $\varepsilon > 0$ sufficiently small in the above estimate, in either case we obtain the desired bound for $\frac{k+1}{k} < p < \frac{k}{k-1}$ with the exponent
\begin{equation*}
    \eta(p) := \frac{\min\{\eta_{\alg}(p), \eta_{\ttiny}(p) \}}{2}.
\end{equation*}
Finally, we extend the range of $p$ by interpolating with the trivial $L^1$ bound.
\end{proof}

%%%%%%%%%%%%%%%%%%%%%%%%%%%%%%%%%%%%%%%%%%%%%%%%%%%%%%%%%%%%%%%%%%%%%%%%%%%%%%%%%%%%%%%%%%%%%%%%

%                                         Geometric maximal_estimates III: Concluding the argument

%%%%%%%%%%%%%%%%%%%%%%%%%%%%%%%%%%%%%%%%%%%%%%%%%%%%%%%%%%%%%%%%%%%%%%%%%%%%%%%%%%%%%%%%%%%%%%%%

\section{Geometric maximal estimates III: Concluding the argument}

In this section we apply the somewhat general theory of the previous two sections to conclude the proof of Theorem~\ref{thm: geom improve}. 

%%%%%%%%%%%%%%%%%%%%%%%%%%%%%%%%%%%%%%%%%%%%%%%%%%%%%%%%%%%%%%%%%%%%%%%%%%%%%%%%%%%%%%%%%%%%%%%%

%         Linear estimates under sublevel set hypotheses

%%%%%%%%%%%%%%%%%%%%%%%%%%%%%%%%%%%%%%%%%%%%%%%%%%%%%%%%%%%%%%%%%%%%%%%%%%%%%%%%%%%%%%%%%%%%%%%%

\subsection{Linear estimates under sublevel set hypotheses}\label{subsec: lin under sublevel} Recall the definition of $\bS_{\tK}(\phi; m; \kappa)$ and $\bS_{\tN}(\phi; m; \kappa)$ from Definition~\ref{dfn: hyp K N} and the critical codimension $m_{\crit}(n)$ from \eqref{eq: crit codim}. Using duality, we first note that Theorem~\ref{thm: geom improve} is implied by the following proposition.  

\begin{proposition}\label{prop: geom improve redux} Let $n \geq 3$ be odd. For all $1 < p < p'(n)$ and $0 < \kappa \leq 1$, there exists some $0 < \beta_{\kappa}(p') < \beta(n;p')$ such that the following holds. For $\phi \colon \D^n \to \R$ a real analytic, translation-invariant H\"ormander-type phase:
\begin{enumerate}[I)]
    \item $\bS_{\tK}(\phi; m_{\crit}(n); \kappa) \Rightarrow \bK_{p \to p}^*(\phi; \beta_{\kappa}(p'))$;
    \item $\bS_{\tN}(\phi; m_{\crit}(n); \kappa) \Rightarrow \bN_{p \to p}^*\big(\phi; \beta_{\kappa}(p'))$.
\end{enumerate}
\end{proposition}

\begin{proof}[Proof (Proposition~\ref{prop: geom improve redux} $\Rightarrow$ Theorem~\ref{thm: geom improve})]
 By the duality \eqref{eq: Kak/Nik duality} and interpolation, Theorem~\ref{thm: geom improve} is equivalent to the following statement. For all $1 < p < p'(n)$ there exists some $\beta_{\phi}(p') < \beta(n;p')$ such that
 \begin{enumerate}[I)]
     \item If $\phi$ satisfies Hypothesis I), then $\bK_{p \to p}^*(\phi; \beta_{\phi}(p'))$ holds;
    \item If $\phi$ satisfies Hypothesis II), then $\bN_{p \to p}^*(\phi; \beta_{\phi}(p'))$ holds. 
 \end{enumerate}
  On the other hand, we know from Proposition~\ref{prop: hyp K N} that
\begin{enumerate}[I)]
    \item Under Hypothesis I), $\bS_{\tK}(\phi; m_{\crit}(n); \kappa)$ holds for some $0 < \kappa \leq 1$; 
    \item Under Hypothesis II), $\bS_{\tN}(\phi;m_{\crit}(n); \kappa)$ holds for some $0 < \kappa \leq 1$.
\end{enumerate}
Combining these facts concludes the reduction.
\end{proof}

%%%%%%%%%%%%%%%%%%%%%%%%%%%%%%%%%%%%%%%%%%%%%%%%%%%%%%%%%%%%%%%%%%%%%%%%%%%%%%%%%%%%%%%%%%%%%%%%

%         Mutlilinear estimates under sublevel set hypotheses

%%%%%%%%%%%%%%%%%%%%%%%%%%%%%%%%%%%%%%%%%%%%%%%%%%%%%%%%%%%%%%%%%%%%%%%%%%%%%%%%%%%%%%%%%%%%%%%%

\subsection{Mutlilinear estimates under sublevel set hypotheses} By the multilinear reduction in Proposition~\ref{prop: multilinear red}, we may reduce Proposition~\ref{prop: geom improve redux} to the following multilinear estimate. 

\begin{proposition}[Multilinear Kakeya/Nikodym estimate]\label{prop: multi Kak/Nik} Let $n \geq 3$ be odd and $1 \leq m \leq n-2$. For all $1 < p < \frac{n-m}{n-m-1}$ and $0 < \kappa \leq n-1$, there exists some $0 < \beta_{m,\kappa}(p') < m/p'$ such that the following holds. For $\phi \colon \D^n \to \R$ a real analytic, translation-invariant H\"ormander-type phase:
\begin{enumerate}[I)]
    \item $\bS_{\tK}(\phi; m; \kappa) \Rightarrow \bM\bK_{p, n-m}^*(\phi; \beta_{m,\kappa}(p'))$;
    \item $\bS_{\tN}(\phi; m; \kappa) \Rightarrow \bM\bN_{p, n-m}^*(\phi; \beta_{m,\kappa}(p'))$.
\end{enumerate}
\end{proposition}

\begin{proof}[Proof (Proposition~\ref{prop: multi Kak/Nik} $\Rightarrow$ Proposition~\ref{prop: geom improve redux})] If we choose $m = m_{\crit}(n)$, then for $n \geq 3$ we have $m \leq n - 2$ and
\begin{equation*}
  \frac{n-m}{n-m-1} = p'(n),\qquad  n - m = d_{\crit}(n)  < \frac{n+3}{2}, \qquad \frac{m}{p'} = \beta(n; p'),
\end{equation*}
using the parity of $n$. The reduction now immediately follows from Proposition~\ref{prop: multilinear red}.
\end{proof}

Combining this with the reduction from \S\ref{subsec: lin under sublevel}, to prove Theorem~\ref{thm: geom improve}, it suffices to show Proposition~\ref{prop: multi Kak/Nik}. However, the latter is a direct consequence of the multilinear theory developed in \S\ref{sec: partitioning} and the non-concentration properties established in \S\ref{sec: non-concentration}. 

\begin{proof}[Proof (of Proposition~\ref{prop: multi Kak/Nik})] Let $\T$ be a family of $\delta$-tubes associated to $\phi$. Suppose that we are in either one of the following cases:
\begin{enumerate}[I)]
    \item $\bS_{\tK}(\phi;m; \kappa)$ holds and $\T$ is direction-separated; 
    \item $\bS_{\tN}(\phi;m; \kappa)$ holds and $\T$ is centre-separated. 
\end{enumerate}
By Theorem~\ref{thm: non-concentration}, there exists some $\kappa_{\circ} > 0$ and $B \in [0,\infty)^2$ such that $\T$ is $(\kappa_{\circ}, B)$-non-concentrated on codimension $m$ grains. We may therefore apply Theorem~\ref{thm: multilin Kak/Nik} to the family of tubes $\T$ with $k := n - m$. From this we conclude that, for all $1 < p < \frac{n-m}{n-m-1}$ there exists some $\eta(p) > 0$ such that the inequality 
\begin{equation*}
    \Big\|\prod_{\ell = 1}^{n-m} \big| \sum_{T_{\ell} \in \T_{\ell}} \chi_{T_{\ell}} \big|^{1/(n-m)} \Big\|_{L^p(Q)} \lesssim_{\phi} K^A \delta^{-m/p' + \eta(p)} \Big(\sum_{T \in \T} |T| \Big)^{1/p}
\end{equation*}
holds whenever $Q \in \cQ_{K^{-1}}$ is a cube and $\T_1, \dots, \T_{n-m} \subset \T$ are $K^{-(n-m-1)}$-transversal on $Q$ for some $K \geq 1$. Note that Theorem~\ref{thm: multilin Kak/Nik} involves a global estimate, but the localised version easily follows by a rescaling argument. Letting $\beta_{m, \kappa}(p') := m/p' - \eta(p) < m/p'$, the desired result follows from the definitions. 
\end{proof}

%%%%%%%%%%%%%%%%%%%%%%%%%%%%%%%%%%%%%%%%%%%%%%%%%%%%%%%%%%%%%%%%%%%%%%%%%%%%%%%%%%%%%%%%%%%%%%%%

%                                         From geometric maximal estimates to oscillatory integral operators

%%%%%%%%%%%%%%%%%%%%%%%%%%%%%%%%%%%%%%%%%%%%%%%%%%%%%%%%%%%%%%%%%%%%%%%%%%%%%%%%%%%%%%%%%%%%%%%%

\section{From geometric maximal functions to oscillatory operators}\label{sec: geom red}

In this section we discuss the proof of Proposition~\ref{prop: geom red}, which follows a classical argument of Bourgain~\cite{Bourgain1991a, Bourgain1991}. 

%%%%%%%%%%%%%%%%%%%%%%%%%%%%%%%%%%%%%%%%%%%%%%%%%%%%%%%%%%%%%%%%%%%%%%%%%%%%%%%%%%%%%%%%%%%%%%%%

%                 Universal estimates revisited

%%%%%%%%%%%%%%%%%%%%%%%%%%%%%%%%%%%%%%%%%%%%%%%%%%%%%%%%%%%%%%%%%%%%%%%%%%%%%%%%%%%%%%%%%%%%%%%%

\subsection{Universal estimates revisited}\label{subsec: loc osc} 

We begin by clarifying the proof of the universal estimates in Theorem~\ref{thm: uni osc}, which was briefly sketched in \S\ref{sec: background}. 

\begin{proof}[Proof (of Theorem~\ref{thm: uni osc})] Since I) follows directly by interpolating classical estimates of H\"ormander~\cite{Hormander1973} and Stein~\cite{Stein1986}, we focus on II).

Let $\phi \colon \D^n \to \R$ be a H\"ormander-type phase and $U^{\lambda} := U^{\lambda}[\phi;a]$ for some choice of amplitude function. By Theorem~\ref{thm: uni osc} I) and Plancherel's theorem, 
\begin{equation}\label{eq: pseduo loc 1}
    \|U^{\lambda}f\|_{L^q(\R^n)} \lesssim_p \lambda^{\alpha_{\mathrm{H}}(n,q)} \|f\|_{L^2(\R^{n-1})} \qquad \textrm{for all $2 \leq q \leq \infty$. }
\end{equation}
Given $f \in \cS(\R^{n-1})$, we may write
\begin{equation*}
    U^{\lambda}f(x, t) = \int_{\R^{n-1}} K^{\lambda}(x, t;y) f(y)\,\ud y
\end{equation*}
where 
\begin{equation*}
   K^{\lambda}(x,t ; y) := \int_{\widehat{\R}^{n-1}} e^{i (\phi^{\lambda}(x, t;\xi) - \inn{y}{\xi})} a^{\lambda}(x, t;\xi)\,\ud \xi. 
\end{equation*}
There exists a constant $C_{\phi} \geq 1$ such that
\begin{equation*}
    \big|\partial_{\xi} \big[\phi^{\lambda}(x, t;\xi) - \inn{y}{\xi}]\big| \geq |y| - \lambda |(\partial_{\xi}\phi)(x/\lambda, t/\lambda;\xi)| \geq |y|/2 \qquad \textrm{whenever $|y| \geq C\lambda$} 
\end{equation*}
holds for all $(x,t;\xi) \in \supp a^{\lambda}$. A repeated integration-by-parts argument then shows that 
\begin{equation*}
    |K^{\lambda}(x,t;y)| \lesssim_N |y|^{-N}  \qquad \textrm{whenever $|y| \geq C\lambda$.}
\end{equation*}
Thus, we may write 
\begin{equation}\label{eq: pseduo loc 2}
   U^{\lambda}f(x,t) = U^{\lambda}f_{\mathrm{loc}}(x,t) + \lambda^{-100n} \|f\|_{L^q(\R^{n-1})}  \quad \textrm{where} \quad f_{\mathrm{loc}} := f \cdot \chi_{B(0, C\lambda)}.
\end{equation}

Applying \eqref{eq: pseduo loc 1} with $f$ replaced with $f_{\mathrm{loc}}$ and combining the resulting inequality with \eqref{eq: pseduo loc 2}, we deduce that
\begin{equation*}
    \|U^{\lambda}f\|_{L^q(\R^n)} \lesssim_q \lambda^{\alpha_{\mathrm{H}}(n,q)} \|f_{\mathrm{loc}}\|_{L^2(\R^{n-1})} + \lambda^{-10n} \|f\|_{L^q(\R^{n-1})}. 
\end{equation*}
Here we have also used the fact that $U^{\lambda}f$ is supported on a set of measure $O(\lambda^n)$ to control the error term. Finally, we apply H\"older's inequality to deduce that
\begin{equation*}
    \|U^{\lambda}f\|_{L^q(\R^n)} \lesssim_p \lambda^{\alpha_{\mathrm{H}}(n,q) + (n-1)(\frac{1}{2} - \frac{1}{q})} \|f\|_{L^q(\R^{n-1})}.
\end{equation*}
Recalling the definition of $\alpha_{\mathrm{LS}}(n,q)$ from \eqref{eq: LS alpha def}, we see this is precisely the desired inequality. 
\end{proof}

%%%%%%%%%%%%%%%%%%%%%%%%%%%%%%%%%%%%%%%%%%%%%%%%%%%%%%%%%%%%%%%%%%%%%%%%%%%%%%%%%%%%%%%%%%%%%%%%

%                 Reverse Square function bounds

%%%%%%%%%%%%%%%%%%%%%%%%%%%%%%%%%%%%%%%%%%%%%%%%%%%%%%%%%%%%%%%%%%%%%%%%%%%%%%%%%%%%%%%%%%%%%%%%

\subsection{Reverse square function bounds}

Let $0 < \rho < 1$ and $\Theta(\rho)$ be a finitely-overlapping family of $\rho$-balls covering $\B^{n-1}$, with centres in the lattice $\rho \Z^{n-1}$. We refer to the $\theta \in \Theta(\rho)$ as \textit{$\rho$-caps}. Let $(\psi_{\theta})_{\theta \in \Theta(\rho)}$ be a smooth partition of unity adapted to the family of $\rho$-caps $\Theta(\rho)$, so that each $\psi_{\theta}$ satisfies $|\partial_y^{\beta} \psi_{\theta}(y)| \lesssim_{\beta} \rho^{-|\beta|}$ for $\beta \in \N_0^{n-1}$. Given $f \in L^1(\B^{n-1})$, we write $f_{\theta} := f \cdot \psi_{\theta}$ for all $\theta \in \Theta(\rho)$, so that $f = \sum_{\theta \in \Theta(\rho)} f_{\theta}$.

\begin{proposition}[Reverse square function \cite{Bourgain1991}]\label{prop: vc sf} Let $1 \leq R \leq \lambda$ and $S^{\lambda} := S^{\lambda}[\phi;a]$ be a H\"ormander-type operator. Given $\varepsilon > 0$ and $N \in \N$, there exists $M = M(N,\varepsilon)$ and symbols $a_{\theta, \gamma}$, indexed by $\theta \in \Theta(R^{-1/2})$ and $\gamma \in \N_0^n$ with $|\gamma| \leq M$, satisfying 
\begin{equation*}
\supp a_{\theta, \gamma} \subseteq 2 \cdot \theta \quad \textrm{and} \quad |\partial_{\bx}^{\alpha} \partial_y^{\beta} a_{\theta, \gamma}(\bx;y)| \lesssim_{N, \varepsilon, \alpha, \beta} R^{-|\beta|/2}, \quad \alpha \in \N_0^n, \beta \in \N_0^{n-1},
\end{equation*}
and such that for $S_{\theta, \gamma}^{\lambda} := S^{\lambda}[\phi; a_{\theta, \gamma}]$ the following holds. For all $2 \leq q \leq q(n)$, we have
    \begin{equation}\label{eq: vc sf}
    \|S^{\lambda}f\|_{L^q(B_R)} \lesssim_{N,\varepsilon} R^{\frac{n-1}{4}(\frac{1}{2} - \frac{1}{q}) + \varepsilon} \sum_{\substack{ \gamma \in \N_0^n \\ |\gamma| \leq M}} \Big\|\big(\sum_{\theta \in \Theta(R^{-1/2})} |S_{\theta, \gamma}^{\lambda}f_{\theta}|^2 \big)^{1/2}\Big\|_{L^q(B_{2R})} + R^{-N} \|f\|_{L^2(\B^{n-1})},
\end{equation}
for any $R$-ball $B_R$ and $B_{2R}$ is the $2R$-ball concentric to $B_R$.   
\end{proposition}

Proposition~\ref{prop: vc sf} is a consequence of Stein's oscillatory integral theorem \cite{Stein1986} and essentially appears in \cite{Bourgain1991}. The precise formulation used here is similar to \cite[Proposition 8.2.2]{Sogge2017}. Whilst \cite[Proposition 8.2.2]{Sogge2017} treats operators with homogeneous phase functions, the argument easily adapts to give the desired result for H\"ormander-type operators. 
 
%%%%%%%%%%%%%%%%%%%%%%%%%%%%%%%%%%%%%%%%%%%%%%%%%%%%%%%%%%%%%%%%%%%%%%%%%%%%%%%%%%%%%%%%%%%%%%%%

%                                          Wave packet decomposition

%%%%%%%%%%%%%%%%%%%%%%%%%%%%%%%%%%%%%%%%%%%%%%%%%%%%%%%%%%%%%%%%%%%%%%%%%%%%%%%%%%%%%%%%%%%%%%%%

\subsection{Wave packet decomposition}%\label{subsec: wp} 
Let $\phi \colon \D^n \to \R$ be a H\"ormander-type phase and $1 \leq R \leq \lambda$. Given $f \in L^1(\B^{n-1})$, write $f_{\theta} := f \cdot \psi_{\theta}$ for all $\theta \in \Theta(R^{-1/2})$ as in the previous subsection. Let $\tilde{\psi}_{\theta} \in C^{\infty}_c(\R^{n-1})$ be a family of functions satisfying $\tilde{\psi}_{\theta}(y) = 1$ if $y \in \supp \psi_{\theta}$ and $\supp \tilde{\psi}_{\theta} \subseteq 2 \cdot \theta$, where $2 \cdot \theta$ denotes the cap concentric to $\theta$ but with twice the radius. We may choose the functions $\tilde{\psi}_{\theta}$ so that 
\begin{equation*}
%\label{eq: wave packet 1}
    |\partial_y^{\beta} \tilde{\psi}_{\theta}(y)| \lesssim_{\beta} R^{|\beta|/2} \qquad \textrm{for all $\beta \in \N_0^{n-1}$.}
\end{equation*}

By performing a Fourier series decomposition, we may express $f_{\theta} = f_{\theta} \cdot \tilde{\psi}_{\theta}$ as
\begin{equation*}
    f_{\theta}(y) = \sum_{\omega \in R^{1/2}\Z^{n-1}} f_{\theta, \omega}(y) \quad \textrm{where} \quad f_{\theta,\omega}(y) := R^{(n-1)/2} (f_{\theta})\;\widecheck{}\;(\omega)e^{-i\inn{\omega}{y}}\tilde{\psi}_{\theta}(y)
\end{equation*}
for Fourier coefficients 
\begin{equation*}
    (f_{\theta})\;\widecheck{}\;(\omega) := \frac{1}{(2\pi)^{n-1}} \int_{\R^{n-1}} f_{\theta}(y) e^{i\inn{y}{\omega}}\,\ud y. 
\end{equation*}

At this juncture it is useful to introduce some additional notation. Fix a ball $B_R = B(\bx_R, R) \subseteq \R^n$ and define the recentred phase
\begin{equation*}
%\label{eq: recentre}
    \phi_B(\bx;y) := \phi(\bx_B + \bx;y) - \phi(\bx_B;y).
\end{equation*}
Given $\varepsilon > 0$ and $(\theta, \omega) \in \Theta(R^{-1/2}) \times R^{1/2}\Z^{n-1}$, define the associated \textit{$R$-tube} (over $B_R$ with error $\varepsilon)$ by
\begin{equation}\label{eq: tube}
    T_{\theta,\omega} := \big\{ \bx \in \R^{n-1} \times (-2R, 2R) : |\partial_y \phi_B^{\lambda}(\bx;y_{\theta}) - \omega| < R^{1/2 + \varepsilon/100n} \big\}; 
\end{equation}
here $y_{\theta}$ denotes the centre of the cap $\theta$. We let $\bT(R)$ denote the set of all such $R$-tubes and write 
\begin{equation*}
%\label{eq: wp}
    f_T := e^{i \phi^{\lambda}(\bx_B;y)} f_{\theta,\omega}  \qquad \textrm{for} \quad T = T_{\theta, \omega} \in \bT(R).
\end{equation*}
Note that both $\bT(R)$ and the $f_T$ depend on a choice of $R$-ball $B$ (and the definition of the tubes depends on a choice of error parameter $\varepsilon > 0$), but we suppress this in the notation. Fixing $\theta \in \Theta(R^{-1/2})$, we also let $\bT_{\theta}(R) := \{T_{\theta, \omega} : \omega \in R^{1/2}\Z^{n-1}\}$ be the set of all tubes associated to $\theta$. With these definitions, 
\begin{equation*}
    f = \sum_{\theta \in \Theta(R^{-1/2})} f_{\theta} = e^{-i \phi^{\lambda}(\bx_B;y)}\sum_{\theta \in \Theta(R^{-1/2})} \sum_{T \in \bT_{\theta}(R)} f_T = e^{-i \phi^{\lambda}(\bx_B;y)}\sum_{T \in \bT(R)} f_T;
\end{equation*}
we refer to this as the \textit{wave packet decomposition of $f$}. 

For $T = T_{\theta, \omega} \in \bT(R)$, observe that
\begin{equation*}%\label{eq: wave packet 2}
    |(f_{\theta})\;\widecheck{}\;(\omega)|^2 \sim R^{-\frac{n-1}{2}}\|f_T\|_{L^2(\R^{n-1})}^2
\end{equation*}
and moreover, by Plancherel's theorem, 
\begin{equation}\label{eq: wave packet 3}
  \sum_{T \in \bT_{\theta}(R)} \|f_T\|_{L^2(\R^{n-1})}^2 \sim  \|f_{\theta}\|_{L^2(\R^{n-1})}^2.
\end{equation}
Finally, we observe localisation properties of the functions $S^{\lambda}f_T$.  

\begin{lemma}[Spatial/temporal localisation of the wave packets]\label{lem: local wp}
Let $1 \leq R \leq \lambda$, $\theta \in \Theta(R^{-1/2})$ with centre $y_{\theta}$ and suppose $a_{\theta} \in C^{\infty}_c(\D^n)$ satisfies
\begin{equation*}
\supp a_{\theta} \subseteq 2 \cdot \theta \quad \textrm{and} \quad |\partial_{\bx}^{\alpha} \partial_y^{\beta} a_{\theta}(\bx;y)| \lesssim_{\alpha, \beta} R^{-|\beta|/2}, \quad \alpha \in \N_0^n, \beta \in \N_0^{n-1}. 
\end{equation*}
If $\phi \colon \D^n \to \R$ is a H\"ormander-type phase and $T = T_{\theta,\omega} \in \bT_{\theta}(R)$ and $f_T$ are defined with respect to $\phi$ as above, we have
\begin{equation}\label{eq: local wp}
   |S^{\lambda}[\phi;a_{\theta}]f_T(\bx_B +\bx)| \lesssim_N R^{-\frac{n-1}{4}}\|f_T\|_{L^2(\R^{n-1})} \big(1 + R^{-1/2}|\partial_y \phi_B^{\lambda}(\bx;y_{\theta}) - \omega| \big)^{-N}
\end{equation}
for all $\bx \in B(0, 2R)$ and all $N \in \N$.
\end{lemma}

The bound \eqref{eq: local wp} follows from a straightforward non-stationary phase (integration-by-parts) argument. Minor variants of Lemma~\ref{lem: local wp} appear throughout the literature on oscillatory integral operators and we omit the details. 

%%%%%%%%%%%%%%%%%%%%%%%%%%%%%%%%%%%%%%%%%%%%%%%%%%%%%%%%%%%%%%%%%%%%%%%%%%%%%%%%%%%%%%%%%%%%%%%%

%                                         Kernel estimates

%%%%%%%%%%%%%%%%%%%%%%%%%%%%%%%%%%%%%%%%%%%%%%%%%%%%%%%%%%%%%%%%%%%%%%%%%%%%%%%%%%%%%%%%%%%%%%%%

\subsection{Kernel estimates}
%\label{subsec: kernel} 
Given $[\phi; a]$ a phase-amplitude pair, $U^{\lambda}[\phi; a]$ can be expressed as
\begin{equation*}
    U^{\lambda}[\phi; a]f(x, t) = \int_{\R^{n-1}} K^{\lambda}[\phi; a](x, t;\omega) f(\omega)\,\ud \omega,
\end{equation*}
where the kernel $K^{\lambda}$ is given by
\begin{equation*}
    K^{\lambda}[\phi; a](x,t;\omega) := \int_{\widehat{\R}^{n-1}} e^{i (\phi^{\lambda}(x,t;y) - \inn{\omega}{y})} a^{\lambda}(x,t;y)\,\ud y. 
\end{equation*}
We note the following pointwise estimates for the Fourier localised kernels.

\begin{lemma}[Kernel estimate]\label{lem: ker est} Let $\lambda \geq 1$, $\theta \in \Theta(\lambda^{-1/2})$ with centre $y_{\theta}$ and suppose $a_{\theta} \in C^{\infty}_c(\D^n)$ satisfies
\begin{equation*}
\supp a_{\theta} \subseteq 2 \cdot \theta \quad \textrm{and} \quad |\partial_{\bx}^{\alpha} \partial_y^{\beta} a_{\theta}(\bx;y)| \lesssim_{\alpha, \beta} \lambda^{-|\beta|/2}, \quad \alpha \in \N_0^n, \beta \in \N_0^{n-1}. 
\end{equation*}
If $\phi \colon \D^n \to \R$ is a H\"ormander-type phase, then the inequality 
\begin{equation}\label{eq: ker est}
   |K^{\lambda}[\phi; a_{\theta}](\bx;\omega)| \lesssim_N \lambda^{-(n-1)/2}\big(1 + \lambda^{-1/2}|\partial_y \phi^{\lambda}(\bx;y_{\theta}) - \omega| \big)^{-N}
\end{equation}
holds for all $\bx \in \R^n$ and all $N \in \N$.
\end{lemma}

The bound \eqref{eq: ker est} follows from a non-stationary phase argument, similar to that used in the proof of Lemma~\ref{lem: local wp}. We omit the details.

%%%%%%%%%%%%%%%%%%%%%%%%%%%%%%%%%%%%%%%%%%%%%%%%%%%%%%%%%%%%%%%%%%%%%%%%%%%%%%%%%%%%%%%%%%%%%%%%

%                                          Oscillatory integrals vs maximal functions

%%%%%%%%%%%%%%%%%%%%%%%%%%%%%%%%%%%%%%%%%%%%%%%%%%%%%%%%%%%%%%%%%%%%%%%%%%%%%%%%%%%%%%%%%%%%%%%%

\subsection{Oscillatory operators vs geometric maximal functions} 

By a standard duality argument, the square function appearing in the right-hand side of \eqref{eq: vc sf} can be estimated in terms of the corresponding Kakeya maximal function. This observation leads to the following lemma.  

\begin{lemma}[Bourgain \cite{Bourgain1991}]\label{lem: Hor vs Kak} Let $n \geq 2$ and $p(n) < p < \infty$. For  any H\"ormander-type phase $\phi \colon \D^n \to \R$ and $\eta > 0$, $\varepsilon > 0$, we have
\begin{enumerate}[I)]
    \item $\bK_{p \to 1}\big(\phi;  \beta(n, p) - \eta\big) \Rightarrow \bH_{\infty \to s}(\phi; \alpha_{\mathrm{H}}(n, s) - \eta/4 + \varepsilon)$,
    \item $\bN_{p \to p}\big(\phi;  \beta(n, p) - \eta\big) \Rightarrow \bLS_s(\phi; \alpha_{\mathrm{LS}}(n, s) - \eta/4 + \varepsilon)$,
\end{enumerate}
where $2 < s < q(n)$ satisfies $p = (s/2)'$.
\end{lemma}

Here the exponents $\alpha_{\mathrm{H}}(n,s)$ and $\alpha_{\mathrm{LS}}(n,s)$ are as defined in \eqref{eq: H alpha def} and \eqref{eq: LS alpha def}, respectively, and $\beta(n,p)$ is as defined in \eqref{eq: bush exp}. In particular, for the range of exponents $p$ and $s$ in Lemma~\ref{lem: Hor vs Kak}, we have
\begin{equation}\label{eq: alpha def}
    \alpha_{\mathrm{H}}(n, s) = \frac{1}{2} -\frac{n+1}{2}\Big(\frac{1}{2} - \frac{1}{s}\Big)  \quad \textrm{and} \quad \alpha_{\mathrm{LS}}(n, s) = \frac{1}{s} + \frac{n-1}{2}\Big(\frac{1}{2} - \frac{1}{s}\Big)
\end{equation}
whilst
\begin{equation*}
    \beta(n, p) := \frac{n-1}{2p}.
\end{equation*}
We remark that only Lemma~\ref{lem: Hor vs Kak} I) is treated in \cite{Bourgain1991} (for $n=3$), but II) follows by a similar argument (as does the extension to $n \geq 2$). 

\begin{proof}[Proof (of Lemma~\ref{lem: Hor vs Kak})]  Let $\varepsilon > 0$ and $N := 100n$ and $1 \leq R \leq \lambda$. Apply Proposition~\ref{prop: vc sf} to find some $M = M(n,\varepsilon) \in \N$ and symbols $a_{\theta,\gamma}$, indexed by $\theta \in \Theta(R^{-1/2})$ and $\gamma \in \N_0^n$ with $|\gamma| \leq M$, such that \eqref{eq: vc sf} holds with $\varepsilon$ replaced with $\varepsilon/2$. \medskip

\noindent I) Assume $\bK_{p \to 1}\big(\phi;  \beta(n, p) - \eta\big)$. Fix $f \in C^{\infty}_c(\B^{n-1})$ and an $R$-ball $B_R = B(\bx_R, R)$ and, by scaling, further assume that $\|f\|_{L^{\infty}(\B^{n-1})} = 1$. By pigeonholing, we can find some choice of $\gamma_{\circ} \in \N^n_0$, depending on $f$, so that $S_{\theta}^{\lambda} := S_{\theta,\gamma_{\circ}}^{\lambda}$ satisfies 
\begin{equation*}
    \sum_{\substack{\gamma \in \N_0^n \\ |\gamma| \leq M}} \Big\|\big(\sum_{\theta \in \Theta(R^{-1/2})} |S_{\theta,\gamma}^{\lambda}f_{\theta}|^2 \big)^{1/2}\Big\|_{L^s(B_{2R})} \lesssim \Big\|\big(\sum_{\theta \in \Theta(R^{-1/2})} |S_{\theta}^{\lambda}f_{\theta}|^2 \big)^{1/2}\Big\|_{L^s(B_{2R})};
\end{equation*}
here $B_{2R}$ is the $2R$-ball concentric to $B_R$. We now apply the wave packet decomposition to $f$. By our scaling and combining \eqref{eq: wave packet 3} and Lemma~\ref{lem: local wp}, we may assume that
\begin{equation*}
|S_{\theta}^{\lambda}f_T(\bx_B + \bx)| 
    \lesssim_{N, \varepsilon} R^{-\frac{n-1}{2}}a_T \chi_T(\bx) + R^{-N}\|f\|_{L^{\infty}(\R^{n-1})},\qquad \bx \in B(0,2R), T \in \bT(R),
\end{equation*}
where the coefficients $a_T \geq 0$ satisfy
\begin{equation*}
    \sum_{T \in \bT_{\theta}(R)} |a_T|^2 = 1 \qquad \textrm{for all $\theta \in \Theta(R^{-1/2})$.}
\end{equation*}
Thus, by translating the variables, 
\begin{equation*}
    \Big\|\big(\sum_{\theta \in \Theta(R^{-1/2})} |S_{\theta}^{\lambda}f_{\theta}|^2 \big)^{1/2}\Big\|_{L^s(B_{2R})} \lesssim  R^{-\frac{n-1}{2}} \Big\|\sum_{T \in \bT(R)} |a_T|^2\chi_T \Big\|_{L^{s/2}(\R^n)}^{1/2}
\end{equation*}
holds up to a negligible error term. Combining this with the square function estimate from Proposition~\ref{prop: vc sf}, we deduce that
\begin{equation}\label{eq: osc vs geom 1}
    \|S^{\lambda}f\|_{L^s(B_R)} \lesssim_{N, \varepsilon} R^{\frac{n-1}{4}(\frac{1}{2} - \frac{1}{s}) - \frac{n-1}{2} + \frac{\varepsilon}{2}} \Big\|\sum_{T \in \bT(R)} |a_T|^2\chi_T \Big\|_{L^{s/2}(\R^n)}^{1/2}.
\end{equation}

The expression appearing in the norm on the right-hand side of \eqref{eq: osc vs geom 1} looks similar to the multiplicity function appearing in the dual formulation of the Kakeya maximal bound, except that it involves a sum over many parallel tubes. To deal with this, we use an elementary probability trick that the authors learned from the course notes of Tao~\cite{Tao_notes}. For each $\theta \in \Theta(R^{-1/2})$, we choose a tube $T_{\theta} \in \bT_{\theta}(R)$ independently at random, where each tube $T \in \bT_{\theta}(R)$ has probability $|a_T|^2$ of being selected. It follows that
\begin{equation*}
    \E\big[\chi_{T_{\theta}}(\bx)\big] = \sum_{T \in \bT_{\theta}(R)} |a_T|^2 \chi_T(\bx)
\end{equation*}
and, by linearity of expectation,
\begin{equation}\label{eq: osc vs geom 2}
    \E\Big[\sum_{\theta \in \Theta(R^{-1/2})} \chi_{T_{\theta}}(\bx)\Big] = \sum_{T \in \bT(R)} |a_T|^2 \chi_T(\bx)
\end{equation}

By combining \eqref{eq: osc vs geom 2} and Minkowski's inequality,
\begin{equation*}
    \Big\|\sum_{T \in \bT(R)} |a_T|^2\chi_T \Big\|_{L^{s/2}(\R^n)} \leq \E \Big\|\sum_{\theta \in \Theta(R^{-1/2})} \chi_{T_{\theta}} \Big\|_{L^{s/2}(\R^n)}.
\end{equation*}
However, by applying the dual form of the Kakeya inequality and scaling, we have 
\begin{equation}\label{eq: osc vs geom 3}
    \Big\|\sum_{\theta \in \Theta(R^{-1/2})} \chi_{T_{\theta}} \Big\|_{L^{s/2}(\R^n)} \lesssim R^{\frac{2n}{s} + \varepsilon}\|\cK^{R^{-1/2}}\|_{L^r(\R^n) \to L^1(\B^{n-1})}.
\end{equation}
Here we assume that the amplitude $a$ in the operator is sufficiently localised, so that the resulting rescaled tubes can be parametrised by $(y, \omega) \in Y_{\phi} \times \Omega_{\phi}$ as in \S\ref{subsec: geometric max}. Thus, \eqref{eq: osc vs geom 1} and \eqref{eq: osc vs geom 3} give
\begin{equation}\label{eq: osc vs geom 4}
    \|S^{\lambda}f\|_{L^s(\R^n)} \lesssim_{\varepsilon} R^{\frac{n-1}{4}(\frac{1}{2} - \frac{1}{s}) - \frac{n-1}{2} + \frac{n}{s} + \varepsilon} \|\cK^{R^{-1/2}}\|_{L^p(\R^n) \to L^1(\B^{n-1})}^{1/2}.
\end{equation}

Unpacking the statement $\bK_{p \to 1}\big(\phi;  \beta(n, p) - \eta\big)$, we have
\begin{equation*}
    \|\cK^{R^{-1/2}}\|_{L^p(\R^n) \to L^1(\B^{n-1})}^{1/2} \lesssim_{\phi,a} R^{\beta(n,p)/4 - \eta/4} = R^{\frac{n-1}{4}(\frac{1}{2} - \frac{1}{s}) - \frac{\eta}{4}}. 
\end{equation*}
Combining this with \eqref{eq: osc vs geom 4}, simplifying the resulting $R$ exponent using \eqref{eq: alpha def} and recalling our choice of normalisation for $f$, we conclude that
\begin{equation*}
     \|S^{\lambda}f\|_{L^s(\R^n)} \lesssim_{\varepsilon} R^{\alpha_{\mathrm{H}}(n, s) - \frac{\eta}{4} + \varepsilon}\|f\|_{L^{\infty}(\B^{n-1})}.
\end{equation*}
The desired result now follows from the definition of $\bH_{\infty \to s}(\phi; \alpha_{\mathrm{H}}(n, s) - \eta/4 + \varepsilon)$.\medskip 

\noindent II) Now assume $\bN_{p \to p}\big(\phi;  \beta(n, p) - \eta\big)$. We fix $R := \lambda$, write $U_{\theta,\gamma}^{\lambda} := U^{\lambda}[\phi; a_{\theta,\gamma}]$ and let $P_{\theta}$ denote the Littlewood--Paley projector with multiplier $\psi_{\theta}$ for each $\theta \in \Theta(\lambda^{-1/2})$. Fixing $f \in \cS(\R^{n-1})$, we may apply the square function estimate to deduce that 
\begin{equation}\label{eq: prop vs Nik 1}
  \|U^{\lambda}f\|_{L^s(\R^n)} \lesssim_{N,\varepsilon} \lambda^{\frac{n-1}{4}(\frac{1}{2} - \frac{1}{s}) + \frac{\varepsilon}{2}}  \sum_{\substack{\gamma \in \N_0^n \\ |\gamma| \leq M}} \Big\|\big(\sum_{\theta \in \Theta(\lambda^{-1/2})} |U_{\theta,\gamma}^{\lambda}\circ P_{\theta}f|^2 \big)^{1/2}\Big\|_{L^s(\R^n)}
\end{equation}
holds up to the inclusion of a rapidly decaying error term.

As in the proof of part I), we apply pigeonholing to find some choice of $\gamma_{\circ} \in \N^n_0$, depending on $f$, so that $U_{\theta}^{\lambda} := U_{\theta,\gamma_{\circ}}^{\lambda}$ satisfies
\begin{equation}\label{eq: prop vs Nik 2}
    \sum_{\substack{\gamma \in \N_0^n \\ |\gamma| \leq M}} \Big\|\big(\sum_{\theta \in \Theta(R^{-1/2})} |U_{\theta,\gamma}^{\lambda}\circ P_{\theta}f|^2 \big)^{1/2}\Big\|_{L^s(\R^n)} \lesssim \Big\|\big(\sum_{\theta \in \Theta(R^{-1/2})} |U_{\theta}^{\lambda}\circ P_{\theta}f|^2 \big)^{1/2}\Big\|_{L^s(\R^n)}.
\end{equation}
Let $K^{\lambda}_{\theta}$ denote the kernel of $U_{\theta}^{\lambda}$ for each $\theta \in \Theta(R^{-1/2})$. By Lemma~\ref{lem: ker est}, we have $\|K^{\lambda}_{\theta}(\bx;\,\cdot\,)\|_{L^1(\R^{n-1})} \lesssim 1$ and so, by Cauchy--Schwarz, 
\begin{equation*}
  |U_{\theta}^{\lambda}\circ P_{\theta}f(\bx)|^2 \lesssim \int_{\R^{n-1}} |K^{\lambda}_{\theta}(\bx;\omega)| |P_{\theta}f(\omega)|^2\,\ud \omega.
\end{equation*}

For each $N \in \N$ we define a maximal function
\begin{equation*}
    \cN_N^{\lambda^{-1/2}} g(\omega) := \max_{\substack{\theta \in \Theta(\lambda^{-1/2}) \\ y_{\theta} \in Y_{\phi}}} \lambda^{-\frac{n+1}{2}} \int_{B(0, \rho\lambda)} \big(1 + \lambda^{-1/2}|\partial_y \phi^{\lambda}(\bx;y_{\theta}) - \omega| \big)^{-N} |g(\bx)|\,\ud \bx
\end{equation*}
for $g \in L^p(\R^n)$. Here $Y_{\phi}$ is the set appearing in \S\ref{subsec: geometric max} and $0 < \rho < 1$ depends on the size of the support of the underlying amplitude $a$. By combining \eqref{eq: prop vs Nik 1} and \eqref{eq: prop vs Nik 2} with a standard duality argument and Lemma~\ref{lem: ker est}, we obtain
\begin{equation}\label{eq: prop vs Nik 3}
    \|U^{\lambda}f\|_{L^s(\R^n)}^2 \lesssim_{N,\varepsilon}  \lambda^{\frac{n-1}{2}(\frac{1}{2} - \frac{1}{s}) + 1 + \varepsilon} \int_{\R^{n-1}} \sum_{\theta \in \Theta(\lambda^{-1/2})} |P_{\theta}f(\omega)|^2  \cN_N^{\lambda^{-1/2}} g(\omega)\,\ud \omega
\end{equation}
for some choice of $g \in L^p(\R^n)$ satisfying $\|g\|_{L^p(\R^n)} \leq 1$. Note that we pick up an extra factor of $\lambda$ owing to the normalisation in the kernel estimate in Lemma~\ref{lem: ker est}. In addition, here (as in the proof of I) we assume that the amplitude $a$ in the operator is sufficiently localised.

It is well-known (see, for instance, \cite{RdF1983}) that
\begin{equation}\label{eq: prop vs Nik 4}
    \Big\| \big(\sum_{\theta \in \Theta(\lambda^{-1/2})} |P_{\theta}f|^2 \big)^{1/2}  \Big\|_{L^s(\R^{n-1})} \lesssim \|f\|_{L^s(\R^{n-1})}.
\end{equation}
Thus, matters reduce to showing
\begin{equation}\label{eq: prop vs Nik 5}
    \|\cN_N^{\lambda^{-1/2}}\|_{L^p(\R^n) \to L^p(B^{n-1}(0,\lambda))} \lesssim_{N, \varepsilon} \lambda^{\frac{2}{s} - 1 + \varepsilon} \|\cN^{\lambda^{-1/2}}\|_{L^p(\R^n) \to L^p(\B^{n-1})}
\end{equation}
for $N$ sufficiently large. Indeed, once we have \eqref{eq: prop vs Nik 5}, we may apply H\"older's inequality, \eqref{eq: prop vs Nik 4} and \eqref{eq: prop vs Nik 5} to estimate the right-hand side of \eqref{eq: prop vs Nik 3}, giving
\begin{equation*}
    \|U^{\lambda}f\|_{L^s(\R^n)} \lesssim \lambda^{\frac{n-1}{4}(\frac{1}{2} - \frac{1}{s}) + \frac{1}{s} + \varepsilon} \|\cN^{\lambda^{-1/2}}\|_{L^p(\R^n) \to L^p(\B^{n-1})}^{1/2}\|f\|_{L^s(\R^{n-1})}.
\end{equation*}
Unpacking the definition of $\bN_{p \to p}\big(\phi;  \beta(n, p) - \eta\big)$, we deduce that
\begin{equation*}
    \|U^{\lambda}f\|_{L^s(\R^n)} \lesssim \lambda^{\frac{n-1}{2}(\frac{1}{2} - \frac{1}{s}) + \frac{1}{s} - \frac{\eta}{4} + \varepsilon} \|f\|_{L^s(\R^{n-1})}.
\end{equation*}
The desired result now follows from the definition of $\bLS_s(\phi; \alpha_{\mathrm{LS}}(n, s) - \eta/4 + \varepsilon)$.

Turning to the proof of \eqref{eq: prop vs Nik 5}, by choosing $N$ large enough we can ensure
\begin{equation*}
\cN_N^{\lambda^{-1/2}}g(\omega) \lesssim \sup_{y \in \B^{n-1}} \fint_{T_{y,\omega}} |g| + \lambda^{-10n} \|g\|_{L^r(\R^n)}
\end{equation*}
where $T_{y,\omega}$ is the $\lambda$-tube with error $\varepsilon$ as defined in \eqref{eq: tube} and $\omega \in B(0,\lambda)$. Note that any $\lambda$-tube with error $\varepsilon$ can be covered by $O(\lambda^{\varepsilon/100})$ parallel $\lambda$-tubes with error $0$. Combining this observation with a simple rescaling, and recalling that $-1/p = 2/s - 1$, we obtain the desired estimate \eqref{eq: prop vs Nik 5}. 
\end{proof}

%%%%%%%%%%%%%%%%%%%%%%%%%%%%%%%%%%%%%%%%%%%%%%%%%%%%%%%%%%%%%%%%%%%%%%%%%%%%%%%%%%%%%%%%%%%%%%%%

%                                         The local-to-global principle

%%%%%%%%%%%%%%%%%%%%%%%%%%%%%%%%%%%%%%%%%%%%%%%%%%%%%%%%%%%%%%%%%%%%%%%%%%%%%%%%%%%%%%%%%%%%%%%%

\subsection{The local-to-global principle} The final ingredient of the proof of Proposition~\ref{prop: geom red} is the following `local-to-global' lemma.

\begin{lemma}[Bourgain \cite{Bourgain1991}]\label{lem: loc to glob} Let $2 \leq s \leq q(n)$, $\alpha s < \tfrac{n+1}{2}$ and 
\begin{equation*}
   q(n, s, \alpha) := 2 + \frac{2s}{n+1 - 2\alpha s}.  
\end{equation*}
For $q > q(n, s, \alpha)$ and $\phi \colon \D^n \to \R$ a H\"ormander-type phase, 
\begin{equation*}
    \bH_{\infty \to s}(\phi; \alpha) \Rightarrow \bH_{\infty \to q}(\phi).
\end{equation*}
\end{lemma}

The $n=3$ case of Lemma~\ref{lem: loc to glob} is implicit in \cite{Bourgain1991}. For completeness, we present a full proof,  following the general scheme of \cite{Bourgain1991}.

\begin{proof}[Proof (of Lemma~\ref{lem: loc to glob})] Fix $q > q(n, s, \alpha)$ and let $f \in  L^{\infty}(\B^{n-1})$. Let $\fR z$ denote the real part of $z \in \C$. For $\sigma > 0$, define the superlevel set
\begin{equation*}
    F_{\sigma} := \big\{ \bx \in \R^n : \fR(S^{\lambda}f(\bx)) > \sigma\big\}.
\end{equation*}

By Chebyshev's inequality,
\begin{equation*}
    \sigma|F_{\sigma}| \leq \fR \int_{F_{\sigma}} S^{\lambda}f(\bx)\,\ud \bx = \fR\inn{S^{\lambda}f}{\chi_{F_{\sigma}}}.
\end{equation*}
Thus, by duality and the Cauchy--Schwarz inequality, we have
\begin{equation}\label{eq: loc glob 2a}
    \sigma|F_{\sigma}| \lesssim \|S^{\lambda, *} \chi_{F_{\sigma}}\|_{L^2(\B^{n-1})}\|f\|_{L^2(\B^{n-1})},
\end{equation}
where $S^{\lambda, *}$ is the adjoint operator
\begin{equation*}
   S^{\lambda, *}g(y) :=  \int_{B^n} e^{-i \phi^{\lambda}(\bx; y)}\overline{a^{\lambda}(\bx;y)} g(\bx)\,\ud \bx .
\end{equation*}

By expanding out the square of the $L^2$-norm in \eqref{eq: loc glob 2a}, we have
\begin{equation*}
   \|S^{\lambda, *}\chi_{F_{\sigma}}\|_{L^2(\B^{n-1})}^2 = \inn{\cS^{\lambda}\chi_{F_{\sigma}}}{\chi_{F_{\sigma}}}
\end{equation*}
where
\begin{equation*}
 \cS^{\lambda}g(\bx) := S^{\lambda}S^{\lambda, *}g(\bx) = \int_{\R^n} K^{\lambda}(\bx,\bz) g(\bz)\,\ud \bz  
\end{equation*}
for the kernel 
\begin{equation*}
    K^{\lambda}(\bx,\bz) := \int_{\R^{n-1}} e^{i(\phi^{\lambda}(\bx;y) - \phi^{\lambda}(\bz;y))} a^{\lambda}(\bx;y) \overline{a^{\lambda}(\bz;y)}\,\ud y.
\end{equation*}
The hypotheses H1) and H2), together with a stationary phase argument, imply that
\begin{equation}\label{eq: loc glob 1}
    |K^{\lambda}(\bx,\bz)| \lesssim (1 + |\bx - \bz|)^{-(n-1)/2}.
\end{equation}

We now perform a Whitney decomposition in $\R^n \times \R^n$ with respect to the diagonal $\{(\bx, \bx) : \bx \in \R^n\}$. In particular, for each $k \in \Z$ we let $\cQ_k$ denote the set of dyadic cubes in $\R^n$ of side-length $2^k$. Given $Q_1, Q_2 \in \cQ_k$, we write $Q_1 \sim Q_2$ if $Q_1$ and $Q_2$ are disjoint, but their parent intervals are not disjoint. It then follows that
\begin{equation*}
    1 = \sum_{k \in \Z} \sum_{\substack{Q_1, Q_2 \in \cQ_k \\ Q_1 \sim Q_2}} \chi_{Q_1}(\bx) \chi_{Q_2}(\bz)
\end{equation*}
and so we may decompose 
\begin{equation}\label{eq: loc glob 2}
   \|S^{\lambda, *}\chi_{F_{\sigma}}\|_{L^2(\B^{n-1})}^2 = \sum_{k \in \Z} \sum_{\substack{Q_1, Q_2 \in \cQ_k \\ Q_1 \sim Q_2}} \inn{\cS^{\lambda}\chi_{F_{\sigma} \cap Q_1}}{\chi_{F_{\sigma} \cap Q_2}}. 
\end{equation}

Fix $k \in \Z$ and $Q_1, Q_2 \in \cQ_k$ satisfying $Q_1 \sim Q_2$. Using \eqref{eq: loc glob 1}, we may estimate
\begin{align*}
|\inn{\cS^{\lambda}\chi_{F_{\sigma} \cap Q_1}}{\chi_{F_{\sigma} \cap Q_2}}| &\leq \sup_{\substack{\bx \in Q_2 \\ \bz \in Q_1}} |K^{\lambda}(\bx, \bz)| |F_{\sigma} \cap Q_1| |F_{\sigma} \cap Q_2| \\
&\lesssim \min\{1,2^{-k(n-1)/2}\} |F_{\sigma} \cap Q_1| |F_{\sigma} \cap Q_2|,
\end{align*}
since the hypothesis $Q_1, Q_2 \in \cQ_k$ implies $\dist(Q_1, Q_2) \gtrsim 2^k$. Summing this bound therefore yields 
\begin{equation}\label{eq: loc glob 3}
  \sum_{\substack{Q_1, Q_2 \in \cQ_k \\ Q_1 \sim Q_2}} |\inn{\cS^{\lambda}\chi_{F_{\sigma} \cap Q_1}}{\chi_{F_{\sigma} \cap Q_2}}| \lesssim \min\{1,2^{-k(n-1)/2}\} M |F_{\sigma}| \max_{Q \in \cQ_k}|F_{\sigma}\cap Q| ,
\end{equation}
where $M := \max_{Q_1 \in \cQ_k} \#\{Q_2 \in \cQ_k : Q_1 \sim Q_2\}$ satisfies $M \lesssim 1$.

On the other hand, using Cauchy--Schwarz, we may bound
\begin{equation*}
    |\inn{\cS^{\lambda}\chi_{F_{\sigma} \cap Q_1}}{\chi_{F_{\sigma} \cap Q_2}}| \leq \|S^{\lambda, *}\chi_{F_{\sigma} \cap Q_1}\|_{L^2(\B^{n-1})}\|S^{\lambda, *}\chi_{F_{\sigma} \cap Q_2}\|_{L^2(\B^{n-1})}. 
\end{equation*}
We now apply H\"ormander's $L^2$ estimate to deduce that
\begin{equation*}
    |\inn{\cS^{\lambda}\chi_{F_{\sigma} \cap Q_1}}{\chi_{F_{\sigma} \cap Q_2}}| \lesssim 2^k |F_{\sigma} \cap Q_1|^{1/2}|F_{\sigma} \cap Q_2|^{1/2}
\end{equation*}
Summing this bound and multiple applications of Cauchy--Schwarz therefore yields
\begin{equation}\label{eq: loc glob 4}
  \sum_{\substack{Q_1, Q_2 \in \cQ_k \\ Q_1 \sim Q_2}} |\inn{\cS^{\lambda}\chi_{F_{\sigma} \cap Q_1}}{\chi_{F_{\sigma} \cap Q_2}}| \lesssim 2^k M |F_{\sigma}|.
\end{equation}

Assume $\bH_{\infty \to s}(\phi; \alpha)$, where $\alpha$ and $s$ satisfy the hypotheses of the lemma. Let $E \subseteq \B^{n-1}$ be measurable and take $f := i^m \chi_E$ in the general framework, where $i$ is the imaginary unit and $ 0 \leq m  \leq 3$. We estimate the local quantity $|F_{\sigma} \cap Q|$ appearing in \eqref{eq: loc glob 3} using the hypothesised local estimate. In particular, by unpacking the definition of $\bH_{\infty \to s}(\phi; \alpha)$ and Chebyshev's inequality, 
\begin{equation*}
    \sigma^s |F_{\sigma} \cap Q| \leq \|S^{\lambda}\chi_E\|_{L^s(Q)}^s \lesssim 2^{ks\alpha} \qquad \textrm{for all $Q \in \cQ_k$ and all $k\ge 0$.} 
\end{equation*}
When $k<0$, we trivially have
\begin{equation*}
    \sigma^s |F_{\sigma} \cap Q| \leq \|S^{\lambda}\chi_E\|_{L^s(Q)}^s\le |Q|\|S^{\lambda}\chi_E\|_{L^{\infty}(Q)}^s \lesssim |Q|\le 1.
\end{equation*}
Combining this with \eqref{eq: loc glob 3}, we obtain
\begin{equation}\label{eq: loc glob 5}
\begin{split}
    &\sum_{\substack{Q_1, Q_2 \in \cQ_k \\ Q_1 \sim Q_2}} |\inn{\cS^{\lambda}\chi_{F_{\sigma} \cap Q_1}}{\chi_{F_{\sigma} \cap Q_2}}| \lesssim 2^{-k(\frac{n-1}{2} - s\alpha)} \sigma^{-s} M |F_{\sigma}| \qquad \textrm{when $k\ge 0$,}\\
  &\sum_{\substack{Q_1, Q_2 \in \cQ_k \\ Q_1 \sim Q_2}} |\inn{\cS^{\lambda}\chi_{F_{\sigma} \cap Q_1}}{\chi_{F_{\sigma} \cap Q_2}}| \lesssim  \sigma^{-s} M |F_{\sigma}| \qquad \textrm{when $k< 0$.}
\end{split}
\end{equation}

Together, \eqref{eq: loc glob 2}, \eqref{eq: loc glob 5} and \eqref{eq: loc glob 4}, and the fact that $M \lesssim 1$, imply that
\begin{equation}\label{eq: loc glob 6}
\begin{split}
     \|S^{\lambda, *}\chi_{F_{\sigma}}\|_{L^2(\B^{n-1})}^2 &\lesssim  \sum_{k \in \Z,k\ge 0} \min\{2^{-k(\frac{n-1}{2} - s\alpha)} \sigma^{-s}, 2^k\} |F_{\sigma}|+\sum_{k \in \Z,k<0} \min\{ \sigma^{-s}, 2^k\} |F_{\sigma}|\\
     &\lesssim  \sum_{k \in \Z,k\ge 0} \min\{2^{-k(\frac{n-1}{2} - s\alpha)} \sigma^{-s}, 2^k\} |F_{\sigma}|+\min\{ \sigma^{-s}, 1\} |F_{\sigma}|\\
     &\lesssim  \sum_{k \in \Z} \min\{2^{-k(\frac{n-1}{2} - s\alpha)} \sigma^{-s}, 2^k\} |F_{\sigma}|.
\end{split}
\end{equation}
Here in the last step, we absorb the second term into the case $k=0$.
Let $k_0 \in \Z$ be such that, for $k = k_0$, the two terms in the minimum are comparable; this is equivalent to choosing
\begin{equation*}
    2^{k_0} \sim \sigma^{-q + 2 + \eta} \qquad \textrm{where} \quad \eta := q - q(n, s, \alpha) > 0.
\end{equation*}
By partitioning the range of summation on the right-hand side of \eqref{eq: loc glob 6} according to whether $k > k_0$ or $k \leq k_0$, we deduce that
\begin{equation*}
    \|S^{\lambda, *}\chi_{F_{\sigma}}\|_{L^2(\B^{n-1})}^2 \lesssim \sigma^{-q + 2 + \eta} |F_{\sigma}|.
\end{equation*}
Finally, we combine the above inequality with \eqref{eq: loc glob 2a} to obtain
\begin{equation*}
  \sigma^2 |F_{\sigma}|^2 \lesssim  \sigma^{-q + 2 + \eta} |F_{\sigma}|,
\end{equation*}
which rearranges to give 
\begin{equation}\label{eq: loc glob 7}
|F_{\sigma}| \lesssim  \sigma^{-q + \eta}.    
\end{equation}
 Note that $\|S^{\lambda}\chi_E\|_{\infty} \lesssim_a 1$, where $a$ denotes the amplitude associated to $S^{\lambda}$. Thus, in \eqref{eq: loc glob 7} it suffices to consider $0 < \sigma \lesssim_a 1$ only, since otherwise the left-hand side is zero. Since $\eta > 0$, the bound \eqref{eq: loc glob 7} therefore implies
\begin{equation*}
   |\{\bx \in \R^n : |S^{\lambda}\chi_E(\bx)|  > \sigma \}|\lesssim_{\phi, a, q} \sigma^{-q} \qquad \textrm{for all $\sigma > 0$.}
\end{equation*}
The desired result now follows by restricted weak-type interpolation. 
\end{proof}

%%%%%%%%%%%%%%%%%%%%%%%%%%%%%%%%%%%%%%%%%%%%%%%%%%%%%%%%%%%%%%%%%%%%%%%%%%%%%%%%%%%%%%%%%%%%%%%%

%                Concluding the proof of Proposition~\ref{prop: geom red}

%%%%%%%%%%%%%%%%%%%%%%%%%%%%%%%%%%%%%%%%%%%%%%%%%%%%%%%%%%%%%%%%%%%%%%%%%%%%%%%%%%%%%%%%%%%%%%%%

\subsection{Concluding the proof of Proposition~\ref{prop: geom red}} It remains to combine our earlier lemmas to conclude the proof of the geometric reduction. 

\begin{proof}[Proof (of Proposition~\ref{prop: geom red})]  Fix $p(n) < p < \infty$ and $0 \leq \beta < \beta(n; p)$, so that $\eta := \beta(n, p) - \beta >0$. Let $2 < s < q(n)$ satisfy $p = (s/2)'$.\medskip

\noindent I) Take $\alpha := \alpha_{\mathrm{H}}(n,s) - \eta/4$, where $\alpha_{\mathrm{H}}(n,s)$ is as in \eqref{eq: alpha def}, so that 
\begin{equation*}
    n + 1 - 2 \alpha s = \frac{(n-1+\eta)s}{2} > 0.
\end{equation*}
It follows that the exponent $q(n, s, \alpha)$ as featured in Lemma~\ref{lem: loc to glob} satisfies
\begin{equation*}
   q(n, s, \alpha) := 2 + \frac{4}{n - 1 + \eta} < q(n) = 2 + \frac{4}{n - 1}.
\end{equation*} 
Furthermore, $\alpha s < \frac{n+1}{2}$ for $n \geq 3$. Combining Lemma~\ref{lem: Hor vs Kak} and Lemma~\ref{lem: loc to glob}, given any $q(n, s, \alpha) < q < q(n)$ we may choose $\varepsilon > 0$ sufficiently small so that
\begin{equation*}
    \bK_{p \to 1}(\phi;  \beta) \Rightarrow \bH_{\infty \to s}(\phi; \alpha_{\mathrm{H}}(n, s) - \eta/4 + \varepsilon) \Rightarrow \bH_{\infty \to q}(\phi).
\end{equation*}
This concludes the proof of I). \medskip

\noindent II) Taking $\alpha := \alpha_{\mathrm{LS}}(n,s) - \eta/8$, Lemma~\ref{lem: Hor vs Kak} already ensures that
\begin{equation*}
     \bN_{p \to p}(\phi;  \beta) \Rightarrow \bLS_s(\phi; \alpha),
\end{equation*}
which concludes the proof of II).
\end{proof}

%%%%%%%%%%%%%%%%%%%%%%%%%%%%%%%%%%%%%%%%%%%%%%%%%%%%%%%%%%%%%%%%%%%%%%%%%%%%%%%%%%%%%%%%%%%%%%%%

%                Appendix

%%%%%%%%%%%%%%%%%%%%%%%%%%%%%%%%%%%%%%%%%%%%%%%%%%%%%%%%%%%%%%%%%%%%%%%%%%%%%%%%%%%%%%%%%%%%%%%%

\appendix

%%%%%%%%%%%%%%%%%%%%%%%%%%%%%%%%%%%%%%%%%%%%%%%%%%%%%%%%%%%%%%%%%%%%%%%%%%%%%%%%%%%%%%%%%%%%%%%%

%                 Tools from real algebraic geometry

%%%%%%%%%%%%%%%%%%%%%%%%%%%%%%%%%%%%%%%%%%%%%%%%%%%%%%%%%%%%%%%%%%%%%%%%%%%%%%%%%%%%%%%%%%%%%%%%

\section{Tools from real algebraic geometry}
%\label{sec: semialg}

For the reader's convenience, here we recall the definitions and results from real algebraic geometry that play a role in our arguments in \S\ref{sec: non-concentration}.

\begin{definition} A set $S \subset \R^n$ is \emph{semialgebraic} if there exists a finite collection of polynomials $P_{i,j}$, $Q_{i,j} \colon \R^n \to \R$ for $1 \leq i \leq r$, $1 \leq j \leq s$ such that
\begin{equation}\label{eq: semialgebraic}
    S = \bigcup_{i=1}^r \big\{x \in \R^n : P_{i,1}(x) = \cdots = P_{i,s}(x) = 0,\, Q_{i,1}(x) > 0, \dots, Q_{i,s}(x) > 0\big\}.
\end{equation}
Given a semialgebraic set $S \subset \R^n$ the \textit{complexity} of $S$ is 
\begin{equation*}
    \inf\Big( \sum_{i,j} \deg P_{i,j} + \deg Q_{i,j} \Big)
\end{equation*}
where the infimum is taken over all possible representations of $S$ of the form~\eqref{eq: semialgebraic}.
\end{definition}

\begin{definition} Let $S \subseteq \R^n$ and $T\subseteq \R^m$ be semialgebraic sets. We say $f \colon T\rightarrow S$ is a \textit{semialgebraic map} if its graph is a semialgebraic subset of $\mathbb{R}^{m+n}$. The \textit{dimension} of $S$ is the largest $d \in \N_0$ such that there exists an injective semialgebraic map from $(0, 1)^d$ to $S$.
\end{definition}

We briefly recount three key results concerning semialgbraic sets, used in \S\ref{sec: non-concentration}. 

\subsubsection*{Wongkew's theorem} We make considerable use of the following theorem of Wongkew \cite{Wongkew1993}, which bounds the volume of neighbourhoods of algebraic varieties.

\begin{theorem}[Wongkew \cite{Wongkew1993}]
%\label{thm: Wongkew} 
Suppose $\bZ$ is an $m$-dimensional variety in $\R^n$ with $\deg \bZ \leq d$. For any $0 < \rho \leq \lambda$ and $\lambda$-ball $B_{\lambda}$ the neighbourhood $N_{\!\rho}(\bZ\cap B_{\lambda})$ can be covered by $O_d( (\lambda/\rho)^{m})$ balls of radius $\rho$.
\end{theorem}

\subsubsection*{The Tarski--Seidenberg projection theorem} A fundamental result in the theory of semialgebraic sets is the Tarski--Seidenberg projection theorem, which is also referred to as \textit{quantifier elimination}. A useful reference for this material is~\cite{BPR2006}. 

\begin{theorem}[Tarski--Seidenberg]\label{thm: Tarski}
Let $\Pi $ be the orthogonal projection of $\mathbb{R}^n$ into its first $n-1$ coordinates. Then for every $E\ge 1$, there is a constant $C(n,E) > 0$ so that, for every  semialgebraic $S \subset \mathbb{R}^n$ of complexity at most $E$, the projection $\Pi(S)$ has complexity at most $C(n,E)$. 
\end{theorem}

We repeatedly use Theorem~\ref{thm: Tarski} to form semialgebraic \textit{sections}.

\begin{corollary}%\label{cor: Tarski} 
Let $S \subset \mathbb{R}^{2n}$ be a compact semialgebraic set of complexity at most $E$. Let~$\Pi$ be the orthogonal projection into the final $n$ coordinates 
$(\mathbf{a},\mathbf{d}) \mapsto \mathbf{d}.$
Then there is a constant $C(n,E)>0$, depending only on $n$ and $E$, and a semialgebraic set $Z$, of complexity at most~$C(n,E)$, so that
$$Z \subset S,\quad\quad
\Pi(Z)= \Pi(S),$$
and so that for each $\mathbf{d},$ there is at most one $\mathbf{a}$ with
$(\mathbf{a},\mathbf{d}) \in Z.$
\end{corollary}

This is Lemma 2.2 from \cite{KR2018}. It is a direct consequence of Theorem~\ref{thm: Tarski}, as discussed in \cite{KR2018}. 

\subsubsection*{Gromov's algebraic lemma} The final key tool is the existence of useful parameterisations of semialgebraic sets, as guaranteed by the following lemma. 

\begin{lemma}[Gromov]\label{lem: Gromov} For all integers $E, n, r \geq 1$, there exists $M(E, n, r) < \infty$ with the following
properties. For any compact semialgebraic  set $A \subset [0,1]^n$ of dimension $m$ and complexity at most~$E$, there exists
an integer $N\le M(E,n,r)$ and $C^r$ maps $\phi_1,\dots,\phi_N: [0,1]^m \longrightarrow [0,1]^n$ such that
$$\bigcup_{j=1}^N  \phi_j ([0,1]^m) =A\quad \text{and}\quad \|\phi_j\|_{C^r}:=\max_{|\alpha|\le r}\|\partial^\alpha\phi_j\|_{\infty} \leq 1.$$
\end{lemma}

 This result was originally stated by Gromov. Detailed proofs were later given by Pila and Wilkie~\cite{PW2006} and Burguet~\cite{Burguet2008}.\\

\subsection*{Acknowledgements}

Shaoming Guo is partly supported by the Nankai Zhide Foundation, NSFC Grant No. 12426204, and NSF-2044828. Jonathan Hickman is supported by New Investigator Award UKRI097. Marina Iliopoulou is supported by an H.F.R.I. grant (H.F.R.I. Project Number: 23652, title: `Algebraic techniques in harmonic analysis and discrete geometry'), in the framework of the “3rd
Call for H.F.R.I.’s Research Projects to Support Faculty Members \& Researchers”. James Wright is supported by a Leverhulme Research Fellowship: project no. RF-2023-709$\backslash$9. We are grateful to the anonymous referees for very helpful and detailed comments and suggestions which have undoubtedly improved the presentation of this manuscript.

%%%%%%%%%%%%%%%%%%%%%%%%%%%%%%%%%%%%%%%%%%%%%%%%%%%%%%%%%%%%%%%%%%%%%%%%%%%%%%%%%%%%%%%%%%%%%%%%

%                                          REFERENCES

%%%%%%%%%%%%%%%%%%%%%%%%%%%%%%%%%%%%%%%%%%%%%%%%%%%%%%%%%%%%%%%%%%%%%%%%%%%%%%%%%%%%%%%%%%%%%%%%

\bibliography{reference}
\bibliographystyle{amsplain}

\end{document}